\documentclass{amsart}
\usepackage[utf8]{inputenc}

\title{A para-controlled approach to the stochastic Yang-Mills equation in two dimensions}

\usepackage{SYM}
\usepackage{SYM-diagrams}
\usepackage{hyperref}

\author[Bjoern Bringmann]{Bjoern Bringmann}
\address{Bjoern Bringmann, School of Mathematics, Institute for Advanced Study, Princeton, NJ 08540 \& Department of Mathematics, Princeton University, Princeton, NJ 08544}
\email{bringmann@princeton.edu}

\author[Sky Cao]{Sky Cao}
\address{Sky Cao, School of Mathematics, Institute for Advanced Study, Princeton, NJ 08540}
\email{skycao@ias.edu}


\begin{document}

\begin{abstract}
\noindent 
We consider the stochastic Yang-Mills heat equation on the two-dimensional torus. Using regularity structures, Chandra, Chevyrev, Hairer, and Shen previously proved both the local well-posedness and gauge-covariance of this model. In this article, we revisit their results using para-controlled calculus. One of the main ingredients is a new coordinate-invariant perspective on vector-valued stochastic objects. 
\end{abstract}

\maketitle

\tableofcontents

\section{Introduction}

Yang-Mills theory has played an important role in mathematics over the past several decades. The starting point is a principal $G$-bundle $P \ra M$ over a compact oriented Riemannian manifold $(M, g)$, where $G$ is a compact Lie group. The Yang-Mills action of a connection $A$ on $P$ is defined as
\begin{equation}\label{intro:eq-action}
S_{\mrm{YM}}(A) := \int_M \|F_A\|^2 d\mrm{vol}_g,
\end{equation}
where $\mrm{vol}_g$ is the volume form associated to $(M, g)$, $\|\cdot\|$ is an $\mrm{Ad}$-invariant norm on $\frkg$, and $F_A$ is the curvature two-form of $A$. It turns out that $S_{\mrm{YM}}$ is invariant under an infinite-dimensional symmetry group $\mc{G}$ which acts on the space of connections. This property is known as {\it gauge invariance} or {\it gauge symmetry} and it underpins all of Yang-Mills theory. Gauge symmetry is inherently nonlinear, which makes the study of Yang-Mills on the one hand more complicated, but on the other hand much richer. For instance, an understanding of the critical points of $S_{\mrm{YM}}$ has led to significant advances in geometry and topology, especially in four dimensions \cite{DK97, FU12}. The critical points are characterized as solutions to the Yang-Mills equations, which are given by
\[ D_A^* F_A = 0, \]
where $D_A$ is the covariant derivative associated to $A$ and $D_A^\ast$ its adjoint. A natural way to construct such critical points is to consider long-time limits of solutions to the Yang-Mills gradient flow, which is the gradient flow of the Yang-Mills action:
\begin{equation}\label{intro:eq-gradient-flow}
 \ptl_t A = -D_A^* F_A. 
\end{equation}
Constructing local and global solutions to this evolution equation is an interesting PDE problem in itself, and there have been many advances over the years which analyze the Yang-Mills gradient flow and other related equations \cite{KM95, KT99, O15, OT22, Rad92, Str94, Wal19}. 
\\

In probability and mathematical physics, the Yang-Mills problem is to construct the following probability measure on the space of connections:
\[ d\mu_{\mrm{YM}}(A) = Z^{-1} \exp(-S_{\mrm{YM}}(A)) dA. \]
To date, the measure has only been constructed in dimension two \cite{C19}, whereas the main goal is to construct the measure in dimension four. There are of course many important partial results towards the construction of the measure in dimensions three and four; we refer to \cite{Ch19, JW06} for a comprehensive list of references. \\

One of the approaches towards constructing the Yang-Mills measure is the method of  stochastic quantization \cite{PW81}, which aims to construct the measure as the unique invariant measure of the associated Langevin equation. This has been very successful for $\Phi^4$ theories, see e.g. introduction in \cite{GH21}, as well as the works \cite{AK2017, HM2018, MW2020, RZZ2017}. The Langevin equation corresponding to the Yang-Mills action \eqref{intro:eq-action} is given by
\begin{equation}\label{intro:eq-Langevin}
\partial_t A = - D_A^\ast F_A + \xi, \qquad A(0)=A_0,
\end{equation}
where $\xi$ is a $\frkg^d$-valued space-time white noise. The two most important features of the Langevin equation \eqref{intro:eq-Langevin} are its 
\begin{equation*}
 \textup{(I) gauge-covariance} \qquad \text{and} \qquad  \textup{(II) singular stochastic forcing.}
\end{equation*}

Gauge-covariance is a physical necessity and, as described above, shared by both the Yang-Mills action and Yang-Mills measure. In the context of \eqref{intro:eq-Langevin}, gauge-covariance formally requires the following: Let $\bA$ be the data-to-solution map of \eqref{intro:eq-Langevin}, which sends the initial data $A_0$ to the solution $A$. Then, for any initial data $A_0$ and any initial gauge-transformation $g_0$, there exists a time-dependent gauge-transformation $g$ such that $\bA^g(A_0)$ and $\bA(A_0^{g_0})$ have the same law. 

While gauge-covariance is dictated by physical considerations, it creates a fundamental mathematical challenge: The gradient flow \eqref{intro:eq-gradient-flow} and the Langevin equation \eqref{intro:eq-Langevin} are \emph{not} parabolic. To see this, we note that for any stationary classical solution $A \in C^2(\T^d\rightarrow \frkg)$ of \eqref{intro:eq-gradient-flow}, such as $A=0$, and any gauge-transformation $g \in C^2(\T^d\rightarrow G)$, $A^g$  is also a stationary classical solution of \eqref{intro:eq-gradient-flow}. But since  $g \in C^2(\T^d\rightarrow G)$ is arbitrary, this implies that \eqref{intro:eq-gradient-flow} has stationary classical solutions which are not smooth, and therefore cannot be parabolic.   A common approach to address lack of parabolicity in geometric evolution equations is the so-called \emph{DeTurck trick}, which was originally used for the Ricci-flow in \cite{DeT83} and first applied to the (deterministic) Yang-Mills gradient flow in \cite{D85}. The idea is to use a suitable time-dependent gauge-transformation of $A$, which can be used to transform the Langevin equation \eqref{intro:eq-Langevin} into 
\begin{equation}\label{intro:eq-SYM-coordinate-free}
\partial_t A =  - D_A^\ast F_A - D_A D_A^\ast A +  \xi, \qquad A(0)=A_0. 
\end{equation}
The evolution equation \eqref{intro:eq-SYM-coordinate-free} is called the stochastic Yang-Mills heat equation. In coordinates, it can be written as 
\begin{equation}\tag{$\mrm{SYM}$}\label{intro:eq-SYM}
\partial_t A^i = \Delta A^i + \big[ A^j , 2 \partial_j A^i - \partial_i A^j \big] + \big[ \big[ A^i, A^j \big], A_j \big] + \xi^i, \qquad A^i(0)=A_0^i, 
\end{equation}
where $1\leq i \leq d$. As is evident from \eqref{intro:eq-SYM}, the stochastic Yang-Mills heat equation is a nonlinear stochastic heat equation, and therefore parabolic. 

A separate issue in \eqref{intro:eq-Langevin} is the singular stochastic forcing $\xi$, which is not altered by the DeTurck trick and therefore still present in \eqref{intro:eq-SYM-coordinate-free}. Since $\xi$ is $\frkg^d$-valued space-time white noise, its regularity is just below $-1-\frac{d}{2}$. The stochastic Yang-Mills heat equation \eqref{intro:eq-SYM-coordinate-free} therefore belongs to the realm of singular stochastic partial differential equations, which has seen significant advances in the last decade.\\

In the seminal article \cite{CCHS22}, Chandra, Chevyrev, Hairer, and Shen studied\footnote{In \cite[(2.1)]{CCHS22}, the stochastic Yang-Mills heat equation contains an additional counter-term of the form $CA$, where $C$ is a linear map on $\frkg$. This counter-term is necessary in \cite{CCHS22} to obtain a gauge-covariant limit and will be discussed in more detail below, see e.g. the end of Subsection \ref{section:introduction-main} and Section \ref{section:gauge-covariance}.} the stochastic Yang-Mills heat equation \eqref{intro:eq-SYM-coordinate-free} in dimension $d=2$. The first main theorem \cite[Theorem 2.1]{CCHS22} is the construction of a natural state space $\Omega^1_\alpha$ for \eqref{intro:eq-SYM-coordinate-free}, which strikes a delicate balance: On the one hand, the state space $\Omega^1_\alpha$ is large enough to contain solutions of \eqref{intro:eq-SYM-coordinate-free}. On the other hand, it is small enough so that both Wilson-loop observables and the action of the gauge transformations can be defined on $\Omega^1_\alpha$. The second and third main theorem \cite[Theorem 2.4 and Theorem 2.9]{CCHS22}, which are of a more dynamical nature, are the local well-posedness of \eqref{intro:eq-SYM-coordinate-free} and the gauge-covariance in law of the solution. The proof is based on the theory of  regularity structures, which is a general approach to singular stochastic partial differential equations. It was first introduced by Hairer in \cite{H14}, and has since been further developed in \cite{BCCH21,BHZ19,CH16,HS23, BH2023}. In order to treat \eqref{intro:eq-SYM-coordinate-free}, \cite{CCHS22} further extended the theory of regularity structures by developing a new basis-free framework. The basis-free framework is necessary to make use of the geometric nature of \eqref{intro:eq-SYM-coordinate-free}, and it is used to compute the renormalization terms \cite[Section 6.2]{CCHS22} and prove gauge-covariance \cite[Section 7]{CCHS22}. Recently, Chandra, Chevyrev, Hairer, and Shen \cite{CCHS22+} also extended their results to the stochastic Yang-Mills heat equation \eqref{intro:eq-SYM-coordinate-free} in dimension $d=3$, which is even more singular. See also the survey \cite{Chevyrev2022} for an overview of the two works \cite{CCHS22, CCHS22+}. For a related construction of the state space in dimension $d=3$, we also refer the reader to  \cite{CC21a,CC21b}.\\ 

While regularity structures provide the most complete picture, there are at least three further approaches to local well-posedness of singular parabolic stochastic partial differential equations: The para-controlled calculus of \cite{BB19,GIP15}, the renormalization-group methods of \cite{D21,K16}, and the diagram-free approach of \cite{LOT21,LOTT21,OSSW18,OW19}. We focus our attention on para-controlled calculus, as it is most relevant to this article. It has previously been applied to prove local well-posedness of the KPZ equation \cite{GP17}, parabolic $\Phi^4_3$-model \cite{CC18}, and three-dimensional stochastic Navier-Stokes equation \cite{ZZ15}. One possible advantage of para-controlled calculus over regularity structures is that it is closer to classical methods for (deterministic) PDEs. Due to this, para-controlled calculus is not only more accessible to certain mathematicians, but can sometimes be combined more easily with other techniques from PDEs, such as energy estimates. The latter aspect has been useful in establishing global well-posedness of singular parabolic SPDEs, such as the parabolic $\Phi^4_3$-model \cite{MW17} and, more recently, the stochastic Navier-Stokes equation \cite{HR23}. \\

In this article, we revisit the stochastic Yang-Mills heat equation \eqref{intro:eq-SYM-coordinate-free} in two dimensions, which was studied using regularity structures in \cite{CCHS22}, from the perspective of para-controlled calculus. Our focus lies on the dynamical results \cite[Theorem 2.4 and 2.9]{CCHS22} rather than the constructions of the state space and Markov process on gauge orbits \cite[Theorem 2.1 and 2.13]{CCHS22}, since the latter do not rely on regularity structures and would therefore be similar in the para-controlled setting.  Our primary hope is that this article may serve as an entry point to the beautiful methods and results of \cite{CCHS22} and, more generally, to geometric stochastic evolution equations as a whole. We also hope that our analysis may be useful to establish a global theory of \eqref{intro:eq-SYM-coordinate-free}, and discuss this further in Subsection \ref{section:global} below.

\subsection{Main results}\label{section:introduction-main}

In order to state our first main result, we need to introduce additional notation. We first choose a regularity parameter 
\begin{equation}\label{intro:eq-kappa}
0<\kappa\ll 1,
\end{equation}
which will be fixed throughout the rest of this article. 
For any $N\in \dyadic$, we let $P_{\leq N}$ be the Littlewood-Paley operator from \eqref{prelim:eq-LWP} below. We then define a frequency-truncated stochastic Yang-Mills heat equation by 
\begin{equation}\tag{$\mrm{SYM}_N$}\label{intro:eq-AN}
\begin{aligned}
    \partial_t A^i_{\leq N} &= \Delta A^i_{\leq N} + \big[ A^j_{\leq N} , 2 \partial_j A^i_{\leq N} - \partial_i A^j_{\leq N} \big] + \big[ \big[ A^i_{\leq N}, A^j_{\leq N} \big], A_{\leq N,j} \big] + P_{\leq N}\xi^i, \\ 
    A_{\leq N}(0)&=A_0. 
\end{aligned}
\end{equation}
We emphasize that the subscript of $A_{\leq N}$ is only used as an index and that the only frequency-truncation in \eqref{intro:eq-AN} is the $P_{\leq N}$-operator acting on $\xi$. In our first theorem, we prove that the solutions of \eqref{intro:eq-AN} converge locally in time as $N\rightarrow \infty$.

\begin{theorem}[Local well-posedness]\label{intro:thm-lwp}
For all $S\geq 1$, there exists a random time $\tau=\tau_S$ which is almost surely positive and such that the following holds: 

For all $A_0 \in \Cs_x^{-\kappa}(\T^2\rightarrow \frkg^2)$ satisfying $\| A_0 \|_{\Cs_x^{-\kappa}(\T^2)}\leq S$ and all $N\in \dyadic$, the solution $A_{\leq N}$ of \eqref{intro:eq-AN} exists on $[0,\tau]\times \T^2$. Furthermore,  the solutions $A_{\leq N}$ converge to a unique limit $A$ in $ C_t^0 \Cs_x^{-\kappa}([0,\tau] \times \T^2 \rightarrow \frkg^2)$.
\end{theorem}

\begin{figure}
\centering
\begin{tabular}{|P{\bigcolwidth}|P{\colwidth}|P{\colwidth}|P{\colwidth}|P{\colwidth}|}
\hline 
Object &  $\linear$ &  $\quadratic$ &  $X$ & $Y$  \\ \hline
Regularity & $0-$ &  $1-$ & $1-$ & $2-$ \\ \hline
\end{tabular}
\centering
    \caption{We display the four terms from our Ansatz in \eqref{intro:eq-ansatz}  and their regularities.}
    \label{figure:intro-regularities}
\end{figure}

As mentioned above, Theorem \ref{intro:thm-lwp} was first proven using regularity structures in \cite{CCHS22}. We now briefly discuss the main ideas behind our alternative proof, which instead relies on para-controlled calculus. Our para-controlled Ansatz for \eqref{intro:eq-AN}, which is explained in more detail in Section \ref{section:para-controlled-ansatz}, is of the form 
\begin{equation}\label{intro:eq-ansatz}
A_{\leq N} = \linear[\leqN][r][] + \quadratic[\leqN][r][] + X_{\leq N} + Y_{\leq N}, 
\end{equation}
where $\linear[\leqN][r][]$ and $\scalebox{0.85}{$\quadratic[\leqN][r][]$}$ are linear and quadratic stochastic objects, $X_{\leq N}$ is a para-controlled component, and $Y_{\leq N}$ is a nonlinear remainder. The spatial regularities of the four terms in \eqref{intro:eq-ansatz} are listed in Figure \ref{figure:intro-regularities} and provide important analytical information on our Ansatz. Compared to earlier works using para-controlled calculus, the main novelty of this article lies in our treatment of the linear and quadratic stochastic objects in \eqref{intro:eq-ansatz} and their nonlinear interactions. To illustrate this, we focus on the cubic term 
\begin{equation*}
\big[ \big[ \linear[\leqN][r][i], \linear[\leqN][r][j] \big], \linear[\leqN][l][j] \big],
\end{equation*}
which appears after inserting the Ansatz \eqref{intro:eq-ansatz} into \eqref{intro:eq-AN}. While the stochastic Yang-Mills heat equation \eqref{intro:eq-AN} does not require a renormalization itself (see Remark \ref{ansatz:rem-cancellation}), several of the individual nonlinear interactions require a renormalization. The renormalized version of the cubic term is given by
\begin{equation}\label{intro:eq-renormalized-cubic}
\big[ \big[ \linear[\leqN][r][i], \linear[\leqN][r][j] \big], \linear[\leqN][l][j] \big] -  \sigma_{\leq N}^2 \Kil \big( \linear[\leqN][r][i] \big). 
\end{equation} 
In \eqref{intro:eq-renormalized-cubic}, $\sigma_{\leq N}^2$ is a logarithmically-divergent renormalization constant and $\Kil$ is the Killing-map, which is a natural geometric operator from the theory of Lie algebras (see Definition \ref{def:Kil}). In order to analyze \eqref{intro:eq-renormalized-cubic}, one possible approach is to choose a basis $(E_a)_{a=1}^{\dim \frkg}$ of the Lie algebra $\frkg$ and use the basis expansion $\linear[\leqN][r][]=\linear[\leqN][r][a] E_a$. After inserting the basis expansion into \eqref{intro:eq-renormalized-cubic}, it is then possible to use the scalar theory of \cite{GIP15}. However, this approach is coordinate-dependent, and obscures the geometric nature of the iterated Lie bracket and Killing-map in \eqref{intro:eq-renormalized-cubic}. 
Instead of using a basis expansion, we follow an approach that is coordinate-independent and partially inspired by \cite{CCHS22}. It relies on the following two-step procedure: 
\begin{enumerate}[label=(\roman*)]
    \item\label{item:tensor} First, analyze the tensor product $\linear[\leqN][r][i]\otimes \linear[\leqN][r][j] \otimes \linear[\leqN][r][j]$. This analysis relies on a product formula for vector-valued multiple stochastic integrals (Lemma \ref{lemma:multiple-stochastic-integral-tensor-product}), which allows us to decompose the tensor product into non-resonant and resonant components. The resonant components involve tensor contractions, which are coordinate-independent. Indeed, these tensor contractions are essentially given by the quadratic Casimir, which is known to be a coordinate-invariant quantity. 
    \item\label{item:insert} Second, view the iterated Lie bracket in \eqref{intro:eq-renormalized-cubic} as a function of the tensor product $\linear[\leqN][r][i]\otimes \linear[\leqN][r][j] \otimes \linear[\leqN][r][j]$. The resonant component of the tensor product, which involves tensor contractions, then naturally produces contracted iterated Lie-brackets. In turn, the contracted iterated Lie-brackets then promptly lead to the Killing-map, which is included in the renormalization in \eqref{intro:eq-renormalized-cubic}. 
\end{enumerate}
This two-step procedure is not only coordinate-independent, but also offers a major practical advantage: The same tensor product from Step \ref{item:tensor} can be used for different iterated Lie-brackets in Step \ref{item:insert}, which avoids unnecessary repetition. In our article, we will use the tensor product 
\begin{equation*}
    \linear[\leqN][r][i]\otimes \linear[\leqN][r][j]
\end{equation*} 
to simultaneously treat 
\begin{equation*}
\big[ \linear[\leqN][r][i], \linear[\leqN][r][j] \big] \qquad \text{and} \qquad
\frkg \ni E  \mapsto \big[ \big[ E , \linear[\leqN][r][i] \big], \linear[\leqN][r][j] \big] - \delta^{ij} \sigma_{\leq N}^2 \Kil \big( E \big). 
\end{equation*}

\begin{remark}
The pure Yang-Mills measure in $d=2$ has remarkable integrability properties which have been used and studied extensively \cite{Mig75,L03,S97}. Our argument, however, does not rely on integrability and can be generalized to the two-dimensional Yang-Mills-Higgs equation, which is non-integrable. 
\end{remark}

Before beginning to discuss our second main result, we make some definitions that will simplify the notation  in the ensuing discussion. These definitions involve Besov spaces which will be defined in Section \ref{section:function-spaces}.

\begin{definition}[Solution map for spatially smooth driving terms]
Let $A_0 \in \Cs_x^{-\kappa}$, let $T > 0$, and let $\eta \in \Cs_{tx}^{\beta}((0, T))$ be a space-time distribution on $(0, T) \times \T^2$, where $\beta = (\beta_t, \beta_x) \in \R$, $\beta_t > -1$, and $\beta_x \geq 0$. Let $A \in C_t^0 \Cs_x^{-\kappa}([0, \tau) \times \T^2)$ be the maximal mild solution to the Yang-Mills heat equation driven by $\eta$, with initial data $A_0$:
\[ \ptl_t A^i = \Delta A^i + \big[A^j, 2\ptl_j A^i - \ptl_i A^j\big] + \big[ \big[ A^i, A^j\big], A_j\big] + \eta, ~~ i \in [2], ~~ A(0) = A_0. \]
Then, we define $\bA(A_0, \eta) := A$.
\end{definition}

\begin{definition}[Solution map]
Let $A_0 \in \Cs_x^{-\kappa}$, let $T_0 > 0$, and let $\eta \in \Cs_{tx}^{\beta}((0, T_0))$ be a space-time distribution on $(0, T_0) \times \T^2$, where $\beta = (\beta_t, \beta_x) \in \R$ with $\beta_t > -1$. For any $T \in [0, T_0]$ such that 
\[ \lim_{\substack{N \toinf}} \bA(A_0, P_{\leq N} \eta) \text{ exists in $C_t^0 \Cs_x^{-\kappa}([0, T] \times \T^2)$}, \]
define $\bA(A_0, \eta) \in C_t^0 \Cs_x^{-\kappa}([0, T] \times \T^2)$ to be the above limit. In this case, we say that $\bA(A_0, \eta)$ exists on $[0, T]$. We say that $\bA(A_0, \eta)$ exists on $[0, T)$ if it exists on $[0, T']$ for all $T' < T$. 
\end{definition}

\begin{remark} It is clear from the definition that, whenever $\bA(A_0,\eta)$ exists on $[0,T)$ or $[0,T]$, then it is necessarily unique. 
In terms of the solution map $\bA$, our local well-posedness theorem (Theorem \ref{intro:thm-lwp}) states that, for all $S \geq 1$, there is random time $\tau_S>0$ such that the solution $\bA(A_0, \xi)$ exists on $[0, \tau_S]$ for all initial data $A_0$ satisfying $\|A_0\|_{\Cs_x^{-\kappa}(\T^2)} \leq S$. 
\end{remark}

We now begin to introduce the second main result of this paper, which is gauge-covariance of solutions to the stochastic Yang-Mills equation. First, given a connection $A : \T^2 \ra \frkg^2$ and a gauge transformation $g : \T^2 \ra G$, the gauge transformation $A^g$ is the connection given by (here $[2] := \{1, 2\}$)
\beq\label{eq:gauge-transformation-def} A^g = \Ad_g A - (dg) g^{-1},
\text{ i.e. } A^g_i = \Ad_g A_i - (\ptl_i g) g^{-1}, ~~ i \in [2].\eeq
Now for expository purposes, we first discuss gauge-covariance of the associated deterministic PDEs. Given a solution $A$ to the Yang-Mills gradient flow
\[ \ptl_t A = -D_A^* F_A, ~~ A(0) = A_0, \]
and a gauge transformation $g_0 : \T^2 \ra G$, we have that $\gauged{A} = A^{g_0}$ is again a solution to the Yang-Mills gradient flow:
\[ \ptl_t \gauged{A} = -D_{\gauged{A}}^* F_{\gauged{A}}, ~~ \gauged{A} = A_0^{g_0}. \]
In other words, the Yang-Mills gradient flow is {\it gauge-covariant}. Next, given a solution $A$ to the Yang-Mills heat flow (i.e. with DeTurck term)
\beq\label{eq:ricci-deturck} \ptl_t A = -D_A^* F_A - D_A D_A^* A, ~~ A(0) = A_0, \eeq
and a gauge transformation $g_0$ as before, there exists a time-dependent gauge transformation $g = g(t)$, with $g(0) = g_0$, such that with $\gauged{A} = A^{g}$, we have that
\[ \ptl_t \gauged{A} = -D_{\gauged{A}}^* F_{\gauged{A}} - D_{\gauged{A}} D_{\gauged{A}}^* \gauged{A}, ~~ \gauged{A} = A_0^{g_0}.\]
This is again a form of gauge-covariance. Even though the gauge transformation is now time-dependent, the important thing is that $A$ and $\gauged{A}$ are still the exact same flow on the orbit space (i.e. the space of gauge equivalence classes of connections). The time-dependent gauge transformation $g$ is specified by
\beq\label{eq:intro-gauge-transformation} (\ptl_t g)g^{-1} = - D_{A^g}^* ((dg) g^{-1}), ~~ g(0) = g_0.\eeq
This turns out to be a nonlinear parabolic\footnote{This may not be so clear in the present form but is explained in Section \ref{section:gauge-covariance}, see e.g. equation \eqref{eq:gauge-transformation}.} PDE, which is locally well-posed under mild assumptions on the initial data $g_0$. Next, we introduce a smooth driving term $\eta : \R \times \T^2 \ra \frkg^2$ and consider a solution $A$ to
\[ \ptl_t A = -D_A^* F_A - D_A D_A^* A + \eta, ~~ A(0) = A_0. \]
If we let $g$ again be defined as the solution to \eqref{eq:intro-gauge-transformation} (where now $A$ is the solution to the above equation) and again let $\gauged{A} = A^g$, then one can show\footnote{For a more detailed discussion, see the beginning of Section \ref{section:gauge-covariance}.} that $\gauged{A}$ solves
\beq\label{eq:gauged-transformed-equation-with-driving-term} \ptl_t \gauged{A} = -D_{\gauged{A}}^* F_{\gauged{A}} - D_{\gauged{A}} D_{\gauged{A}}^* \gauged{A} + \Ad_g \eta, ~~ \gauged{A} = A_0^{g_0}.\eeq
This suggests the form of gauge covariance that the stochastic Yang-Mills equation should satisfy: if $A$ is the solution \eqref{intro:eq-SYM}, then $\gauged{A} = A^g$ should be the solution to \eqref{intro:eq-SYM} with $\xi$ replaced by $\Ad_g \xi$, where $g$ satisfies \eqref{eq:intro-gauge-transformation}. More succinctly, $\bA(A_0, \xi)^g = \bA(A_0^{g_0}, \Ad_g \xi)$. This is the content of our second main result.

\begin{theorem}[Gauge-covariance]\label{thm:gauge-covariance}
For all $S \geq 1$, there exists a random time $\tau = \tau_S$ which is almost surely positive and such that the following holds for any $A_0 \in \Cs_x^{-\kappa}(\T^2 \ra \frkg^2)$ and $g_0 \in \Cs_x^{1-\kappa}(\T^2 \ra G)$ satisfying $\|A_0\|_{\Cs_x^{-\kappa}(\T^2)}, \|g_0\|_{\Cs_x^{1-\kappa}(\T^2)} \leq S$:
\begin{enumerate}[label=(\Roman*)]
    \item\label{introduction:item-gauge-1} The solutions $A = \bA(A_0, \xi)$ to \eqref{intro:eq-SYM} and $g$ to \eqref{eq:intro-gauge-transformation} exist on $[0, \tau]$,
    \item \label{introduction:item-gauge-2} $\Ad_g \xi := \lim_{N \toinf} \Ad_g P_{\leq N} \xi$ exists in $\Cs_{tx}^{(-1/2-100\kappa,-1-100\kappa)}((0, \tau))$ and $\bA(A_0^{g_0}, \Ad_g \xi)$ exists on $[0, \tau]$,
    \item \label{introduction:item-gauge-3} $(\bA(A_0, \xi))^{g} = \bA(A_0^{g_0}, \Ad_g \xi)$ on $[0, \tau]$.
\end{enumerate}
\end{theorem}

\begin{remark}\label{intro:rem-gauge-regularity}
Given the space $\Cs_x^{-\kappa}(\T^2 \ra \frkg^2)$ of connections, the natural associated space of gauge transformations is $\Cs_x^{1-\kappa}(\T^2 \ra G)$, since then $(A, g) \mapsto A^g$ is a map $\Cs_x^{-\kappa}(\T^2 \ra \frkg^2) \times \Cs_x^{1-\kappa}(\T^2 \ra G) \ra \Cs_x^{-\kappa}(\T^2 \ra \frkg^2)$.
\end{remark}

\begin{remark}
As will be clear from the proof, the claims in \ref{introduction:item-gauge-2} and \ref{introduction:item-gauge-3} hold for any time $T$ less than the maximum simultaneous time of existence of $A$ and $g$.
\end{remark}

\begin{remark}
Since $g$ satisfies \eqref{eq:intro-gauge-transformation}, it is an adapted process with respect to the filtration generated by $\xi$. At least formally, this implies that $\Ad_g \xi \stackrel{d}{=} \xi$. Therefore, in principle, the gauge covariance statement of Theorem \ref{thm:gauge-covariance} should lead to a result about gauge-covariance in law of solutions to the stochastic Yang-Mills equation, similar to the gauge-covariance results of \cite{CCHS22, CCHS22+}. However, since $A$ and $g$ may blow up in finite time,  making this precise is slightly technical. Since the main focus of our article lies on PDE-techniques, and Theorem \ref{thm:gauge-covariance} already addresses PDE-aspects of gauge-covariance, we omit such a statement here.
\end{remark}

The main claim of Theorem \ref{thm:gauge-covariance} is \ref{introduction:item-gauge-3}, which is the actual statement of gauge-covariance. Indeed, in the course of proving this point, \ref{introduction:item-gauge-2} will also directly follow (and \ref{introduction:item-gauge-1} will just follow from Theorem \ref{intro:thm-lwp} and standard local well-posedness theory for nonlinear parabolic equations). The main difficulty in proving gauge-covariance is that solutions to the stochastic Yang-Mills equation are defined as limits $\bA(A_0, \xi) = \lim_{N \toinf} \bA(A_0, P_{\leq N} \xi)$ of solutions to mollified equations. When we apply the computation which gives \eqref{eq:intro-gauge-transformation} at finite $N$, we obtain\footnote{To be precise, this is not exactly true, since $g$ is defined in terms of $A = \bA(A_0, \xi)$ and not $A_{\leq N} = \bA(A_0, P_{\leq N} \xi)$. However, it is approximately true, so we ignore it here for simplicity.} that $\bA(A_0, P_{\leq N} \xi)^g = \bA\big(A_0^{g_0}, \Ad_g (P_{\leq N} \xi)\big)$. Thus in order to show that $\bA(A_0, \xi)^g = \bA(A_0^{g_0}, \Ad_g \xi)$, we need to show that
\beq\label{eq:intro-limits-equal} \lim_{N \toinf} \bA\big(A_0^{g_0}, \Ad_g (P_{\leq N} \xi)\big) = \lim_{N \toinf} \bA\big(A_0^{g_0}, P_{\leq N} (\Ad_g \xi)\big).\eeq
Ultimately, the key difficulty is to understand the behavior of $\Ad_g (P_{\leq N} \xi) - P_{\leq N}(\Ad_g \xi)$. At finite $N$, one does not expect this difference to be zero, because $\Ad_g$ and $P_{\leq N}$ do not in general commute. On the other hand, since mollification by $P_{\leq N}$ looks more and more like the identity as $N$ increases, we should expect that in the $N \toinf$ limit, $P_{\leq N}$ and $\mrm{Ad}_g$ {\it do commute}. It would perhaps not be too hard to show some qualitative form of commutativity, such as the convergence of $\Ad_g (P_{\leq N} \xi) - P_{\leq N}(\Ad_g \xi)$ to zero in the sense of space-time distributions. However, the claim \eqref{eq:intro-limits-equal} requires a more precise quantitative version of this approximate commutativity at finite values of $N$ (Proposition \ref{prop:linear-objects-close-gauge-covariance}), which is the most technical part of the article. Once this is done, we can rely on the local well-posedness theory from the proof of Theorem \ref{intro:thm-lwp} in order to obtain \eqref{eq:intro-limits-equal}. For a more detailed discussion, we refer the reader to the beginning of Section \ref{section:gauge-covariance}.

\begin{remark}[Comparison with \cite{CCHS22}]\label{remark:intro-cchs-comparision-gauge-covariance}
The claim \ref{introduction:item-gauge-3} in Theorem \ref{thm:gauge-covariance} corresponds to \cite[Theorem 2.9(i)]{CCHS22}, and in particular the statement in Theorem 2.9(i) that $(\bar{A}, [\bar{g}])$, $(B, [g])$ converge in probability to the same limit. On the other hand, our proof strategy is quite different from that of \cite{CCHS22}. At a basic level, the difference is that \cite{CCHS22} compares 
$\mrm{Ad}_g (P_{\leq N} \xi)$ with $P_{\leq N}(\mrm{Ad}_g \xi)$, whereas this article compares $\Ad_g^{-1} P_{\leq N}(\Ad_g \xi)$ with $P_{\leq N} \xi$. 
In the former approach, one is led to consider a system of equations involving the connection $A$ and the gauge transformation $g$. In our approach, we avoid having to consider such a system, which we find to be more convenient for our purposes, since it allows us to rely more on the local well-posedness theory from the proof of Theorem \ref{intro:thm-lwp}. For more details, see Remark \ref{remark:comparison-cchs-gauge-covariance}.

Another difference with \cite{CCHS22} is that \eqref{intro:eq-AN} does not contain a renormalization, whereas \cite[(1.8), (2.5), and (2.6)]{CCHS22} contain additional counterterms. The reason is that our approximation of $\xi$ only involves a mollification in the spatial variables, but does not include a mollification in the time variable, which eliminates a probabilistic resonance. For a more detailed discussion, we refer to Remark \ref{remark:no-extra-resonace} below.

Lastly, it would be natural to expect that the solution we constructed in Theorem \ref{intro:thm-lwp}, shown to be gauge-covariant in Theorem \ref{thm:gauge-covariance}, coincides with the gauge-covariant solution constructed in \cite[Theorem 2.9]{CCHS22}. The techniques of \cite[Section 8]{CS23} may possibly be relevant to this problem, but we do not pursue this further here.
\end{remark}

\subsection{Open problems} 

We now briefly discuss open problems related to the stochastic Yang-Mills heat equation \eqref{intro:eq-SYM-coordinate-free}. 

\subsubsection{Global well-posedness of the stochastic Yang-Mills heat equation}\label{section:global}

While Theorem \ref{intro:thm-lwp} gives the local existence of solutions to \eqref{intro:eq-SYM-coordinate-free}, we do not address the issue of global existence. Even in the deterministic setting, global existence of solutions to the Yang-Mills heat flow \eqref{eq:ricci-deturck} (note this includes the DeTurck term), is a difficult problem.
In fact, when the initial data is a pure gauge, i.e. is of the form $A_0 = 0^{g_0} = -(dg_{0}) g_0^{-1}$ (here $g_0 : \T^2 \ra G$ is a gauge transformation), then the Yang-Mills heat flow evolves as $- (dg(t)) g(t)^{-1}$, where $g(t)$ evolves as harmonic map heat flow with initial data $g_0$. Examples of finite-time blowup of harmonic map heat flow with domain $D^2$ (the 2D unit disk) and range $S^2$ have been constructed \cite{CDY1992}, and thus one might expect to be able to do the same for maps $\T^2 \ra G$, which would then show that the Yang-Mills heat flow can have finite-time blowup in 2D. Thus global existence of the Yang-Mills heat flow should be interpreted as on the orbit space, i.e. as global existence modulo gauge transformations (on the other hand, note that the Yang-Mills gradient flow \eqref{intro:eq-gradient-flow} exists globally in dimensions $2, 3$ \cite{Rad92} and $4$ \cite{Wal19}, and is gauge equivalent to the Yang-Mills heat flow, whenever the latter exists).

In a recent breakthrough \cite{CS23}, Chevyrev and Shen proved the invariance of the two-dimensional Yang-Mills measure under the two-dimensional stochastic Yang-Mills flow \eqref{intro:eq-SYM-coordinate-free}. Their argument relies on a lattice approximation, properties of the two-dimensional Yang-Mills measure, and Bourgain's globalization argument \cite{B94}. As a corollary \cite[Corollary 2.19]{CS23}, they also obtain the global existence of solutions to the stochastic Yang-Mills heat flow on the orbit space, i.e., modulo gauge transformations.

Despite the significant advances in \cite{CS23}, it is still an interesting problem to prove the global existence of solutions to \eqref{intro:eq-SYM-coordinate-free} using only PDE arguments (and without using invariance). This approach would be closer to the spirit of stochastic quantization \cite{PW81}, and may lead to a new PDE-based construction of the two-dimensional Yang-Mills measure. In the past, the para-controlled approach has been successfully used to prove the global existence for solutions of scalar singular stochastic partial differential equations, see e.g. \cite{MW17,HR23}. We therefore believe that the para-controlled approach may also be helpful in proving the global existence of solutions to \eqref{intro:eq-SYM-coordinate-free}, but this is of course a difficult problem.

\subsubsection{Random geometric wave equations} 

The hyperbolic Yang-Mills equation, which is the hyperbolic counterpart to the Yang-Mills gradient flow \eqref{intro:eq-gradient-flow}, is given by
\begin{equation}\label{intro:eq-hyperbolic-YM}
D_\alpha F^{\alpha \beta}=0 \qquad (t,x) \in \mathbb{R}_t\times \mathbb{T}^d_x.
\end{equation}
Here, $0\leq \alpha,\beta\leq d$ are space-time coordinates which are raised and lowered with respect to the Minkowski metric. The hyperbolic Yang-Mills equation has been studied extensively and we refer the reader to \cite{KM95,KT99,O15,OT19} and the references therein. Motivated by the aforementioned progress on stochastic geometric parabolic equations \cite{CC21a,CC21b,CCHS22,CCHS22+,CS23,S21}, it is an interesting problem to prove the well-posedness of stochastic versions of \eqref{intro:eq-hyperbolic-YM}, which may include random initial data and/or stochastic forcing. At the moment, this problem is open in all spatial dimensions. \\

In \cite{BR23}, the first author and Rodnianski considered a simpler model for a gauge-covariant wave equation in two dimensions \cite[(1.3)]{BR23}, which is perhaps the model most closely related to \eqref{intro:eq-SYM-coordinate-free} in the dispersive literature. This model is limited to the Abelian setting $G=U(1)$, but includes an additional Higgs-field $\phi$. The main result \cite[Theorem 1.3]{BR23} proves the well-posedness of this gauge-covariant wave equation, and therefore corresponds to Theorem \ref{intro:thm-lwp}. As explained in \cite[Theorem 1.7]{BR23}, there remain significant challenges in treating stochastic versions of the full hyperbolic Yang-Mills equation \eqref{intro:eq-hyperbolic-YM}, and further developments await. 
We also point the reader to \cite{KLS20}, which treats the Maxwell-Klein-Gordon equation in four dimensions with random initial data below the energy space, and \cite{BLS24,BJ25}, which consider the $(1+1)$-dimensional wave maps with Brownian initial data. 

\subsection{Overview of the article}

At the end of this introduction, we include a short overview of the article, which also serves as a reader's guide. In Section \ref{section:preliminaries}, we first introduce regularity parameters which will be used throughout the article. We then recall basic results and notation from both differential geometry and harmonic analysis. Given the breadth of this material, we encourage all readers of this article to read (or at least skim) this section. In Section \ref{section:para-controlled-ansatz}, we introduce the para-controlled Ansatz for  \eqref{intro:eq-AN}. While our treatment of the para-controlled Ansatz is self-contained, we still emphasize that \eqref{intro:eq-AN} is among the more difficult models to which para-controlled calculus has been applied. For readers not familiar with para-controlled calculus, it may therefore be helpful to consult \cite{CC18,GIP15} before reading Section \ref{section:para-controlled-ansatz}. In Section~\ref{section:multiple-stochastic-integrals}, we set-up vector-valued multiple stochastic integrals and, in particular, prove a product formula for vector-valued multiple stochastic integrals (Lemma \ref{lemma:multiple-stochastic-integral-tensor-product}). The important aspect of this product formula is that it is coordinate-independent, i.e., it does not require a choice of basis for the Lie-algebra~$\frkg$. In Section~\ref{section:objects}, we estimate the stochastic objects of~\eqref{intro:eq-AN}. The main result of this section, which collects all stochastic estimates, is contained in Proposition~\ref{objects:prop-enhanced}. We note that our coordinate-independent approach to these stochastic estimates, which was briefly discussed below Theorem~\ref{intro:thm-lwp}, is best illustrated by Lemma~\ref{objects:lem-without} and  Lemma~\ref{objects:lem-trace-identities}, as well as their proofs. The rest of Section~\ref{section:objects} may be technically more difficult than Lemma~\ref{objects:lem-without} and  Lemma~\ref{objects:lem-trace-identities}, but it is based on similar ideas and can therefore be skipped on first reading. In Section~\ref{section:nonlinear-estimates}, we prove the well-posedness of the para-controlled stochastic Yang-Mills heat equation (from Section \ref{section:para-controlled-ansatz}), which directly implies Theorem \ref{intro:thm-lwp}. The proof consists of a standard combination of para-product estimates and our estimates of the stochastic objects (from Section \ref{section:objects}). We suggest that readers familiar with para-controlled calculus only read the statement of Proposition \ref{nonlinear:prop-wellposedness-para} and skip the rest of this section. Finally, in Section \ref{section:gauge-covariance}, we prove the gauge-covariance of solutions of \eqref{intro:eq-SYM}, i.e., Theorem \ref{thm:gauge-covariance}. Since the main idea of the proof has already been explained below Theorem \ref{thm:gauge-covariance}, we now only describe how the different technical aspects of the proof are distributed over Section \ref{section:gauge-covariance}. At the beginning of Section \ref{section:gauge-covariance}, we introduce~\eqref{eq:gauged-A}, which is a variant of the stochastic Yang-Mills equation \eqref{intro:eq-AN} containing the stochastic forcing term $\mrm{Ad}_{g}^{-1} P_{\leq N}(\mrm{Ad}_{g} \xi)$. The solution of \eqref{eq:gauged-A} is a gauge-transformation of $\bA\big(A_0^{g_0}, P_{\leq N} (\Ad_g \xi)\big)$, and understanding the solution of \eqref{eq:gauged-A} is the main ingredient in the proof of Theorem \ref{thm:gauge-covariance}. We also state the main proposition of this section (Proposition \ref{prop:main-gauge-covariance}). In Subsection \ref{section:space-time-besov-space}, we introduce space-time Besov spaces and para-product operators, which are an important technical ingredient. In Subsection~\ref{section:stochastic-estimate-gauge-covariance}, we prove additional estimates for our stochastic objects (from Section \ref{section:objects}). The most important result in this subsection is Lemma~\ref{lemma:Duhamel-linear-times-noise-besov-space-bound}, which is responsible for the absence of renormalization terms in \eqref{intro:eq-AN} (recall Remark~\ref{remark:intro-cchs-comparision-gauge-covariance}). In Subsection \ref{section:glinear}, we then examine the linear stochastic object corresponding to $\mrm{Ad}_{g}^{-1} P_{\leq N}(\mrm{Ad}_{g} \xi)$, which we refer to as the modified linear stochastic object. The main result of this subsection is Proposition \ref{prop:linear-objects-close-gauge-covariance}, which contains an expansion of the modified linear stochastic object in terms of the original stochastic objects. In Subsection~\ref{section:gauged-transformed-enhanced-data-set}, we then control the modified enhanced data set. The main ingredients are Proposition~\ref{prop:linear-objects-close-gauge-covariance} and an important cancellation, see the proof of Proposition~\ref{gauged:prop-enhanced}. Finally, in  Subsection~\ref{section:contraction-mapping-gauge-covariance}, we combine our earlier estimates with a contraction-mapping argument to obtain Proposition~\ref{prop:main-gauge-covariance}. We believe that it is difficult to read the individual parts of Section \ref{section:gauge-covariance} separately, and recommend that this section is read linearly.\\

\textbf{Acknowledgements:} The authors thank Hao Shen for several helpful conversations. We also thank the anonymous referee for valuable corrections and suggestions. S.C. was supported by the Minerva Research Foundation. 
\section{Preliminaries}\label{section:preliminaries}

\subsection{Parameters}\label{section:parameters}

We first introduce the following standard notation.

\begin{notation}\label{notation:lessim}
For quantities $A, B > 0$, we say $A \lesssim B$ (resp. $A \gtrsim B$) if there exists a constant $C > 0$ such that $A \leq C B$ (resp. $A \geq C^{-1} B$). The precise value of the constant $C$ can change from line to line. We say that $A \sim B$ if $A \lesssim B$ and $A \gtrsim B$.
\end{notation}

Next, in order to measure regularity and/or decay in frequency space, we introduce two parameters $\kappa,\eta\in(0,1)$ which satisfy 
\begin{equation}\label{prelim:eq-parameters}
0< \eta \ll \kappa \ll 1.
\end{equation}
For notational convenience, we also set 
\begin{equation}
\eta^\prime = \frac{\eta}{100}, ~~\kfactors = 100 \kappa,
\end{equation}
which will primarily be used in Section \ref{section:gauge-covariance}.

\subsection{Explicit geometric formulas}

For the reader's convenience we give explicit coordinate-wise formulas for the various geometric objects that appear in this paper. In the following, given a positive integer $n$, define $[n] := \{1, \ldots, n\}$.
The curvature $F_A : \T^2 \ra \frkg^{2 \times 2}$ of a connection $A$ is a 2-form given by
\[ (F_A)_{ij} = \ptl_i A_j - \ptl_j A_i + [A_i, A_j], ~~ i, j \in [d].\]
Given a connection $A$ and a 2-form $F$, we have that $D_A^* F : \T^2 \ra \frkg^2$ is a 1-form given by
\[ (D_A^* F)_i =  \ptl^j F_{ij} + [A^j, F_{ij}], ~~ i \in [d]\]
Given connections $A, B$, we have that $D_A^* B : \T^2 \ra \frkg$ is a 0-form given by
\[ D_A^* B = -\ptl_j B^j - [A_j, B_j]. \]
Given a connection $A$ and a 0-form $f$, we have that $D_A f : \T^2 \ra \frkg^2$ is a connection given by
\[ (D_A f)_i = \ptl_i f + [A_i, f], ~~ i \in [d].\]
We have the following covariance properties for the various quantities we just defined (recall the definition of gauge transformation \eqref{eq:gauge-transformation-def}):
\beq\label{eq:covariance-formulas} \Ad_g (F_A) = F_{A^g}, ~~ \Ad_g (D_A^* F) = D_{A^g}^* (\Ad_g F), ~~ \Ad_g (D_A^* B) = D_{A^g}^* (\Ad_g B), ~~ \Ad_g (D_A f) = D_{A^g} (\Ad_g f). \eeq

\subsection{Harmonic analysis}
For notational convenience, we define the complex exponential $\e\colon \R \rightarrow \C$ by $\e(y)=\exp(iy)$ for all $y\in \R$. Furthermore, for all $n\in \Z^2$, we define $\e_{n}\colon \T^2 \rightarrow \C$ by $\e_{n}(x):=  \e(n\cdot x)$, and we let $\langle n \rangle := \sqrt{1 + |n|^2}$, where $|n|$ is the Euclidean norm of $n$.
For any distribution $\phi \colon \T^2\rightarrow \C$, we define the Fourier transform $\widehat{\phi}\colon \Z^2\rightarrow \C$ by 
\begin{equation}
\widehat{\phi}(n) = \frac{1}{(2\pi)^2} \int_{\T^2} \phi(x) \overline{\e_{n}(x)} \dx
\end{equation}
for all $n\in \Z^2$. 
Let $\rho \colon \R^2 \rightarrow [0,1]$ be a smooth even cut-off function which is radially non-increasing, and which satisfies
\begin{equation*}
\rho\big|_{[-1,1]^2}=1 \qquad \text{and} \qquad 
\rho\big|_{\R^2 \backslash [-9/8,9/8]^2}=0.
\end{equation*}
For all dyadic $N\geq 1$, we define $\rho_{\leq N}\colon \R^2 \rightarrow \C$ by 
\begin{equation*}
\rho_{\leq N} (\xi) = \rho \big( \xi / N\big).    
\end{equation*}
Furthermore, we define
\begin{equation*}
\rho_1 = \rho_{\leq 1} \qquad \text{and} \qquad 
\rho_{N} = \rho_{\leq N} - \rho_{\leq N/2} 
\quad \text{for all } N\geq 2.
\end{equation*}
Note that $\rho_N \geq 0$ since $\rho$ is radially non-increasing. Finally, we define the Littlewood-Paley operators $(P_{\leq N})_N$ and $(P_{N})_N$ by 
\begin{equation}\label{prelim:eq-LWP}
\widehat{P_{\leq N} \phi}(n) = \rho_{\leq N}(n) \widehat{\phi}(n) 
\qquad \text{and} \qquad
\widehat{P_{N} \phi}(n) = \rho_{N}(n) \widehat{\phi}(n)
\end{equation}
for all distributions $\phi \colon \T^2 \rightarrow \C$ and all $n\in \Z^2$. The Littlewood-Paley operators in \eqref{prelim:eq-LWP} can also be extended from scalar-valued to $\frkg$-valued functions using a decomposition based on any basis of $\frkg$. Next, we recall the definition of H\"{o}lder-Besov spaces using the Littlewood-Paley operators.

\begin{definition}[H\"{o}lder-Besov space]\label{prelim:def-Hoelder}
For $\alpha \in \R$, define the H\"{o}lder-Besov norm
\[ \|f\|_{\Cs^\alpha} := \sup_{N \in \dyadic} N^\alpha \|P_N f\|_{L^\infty}. \]
We then define the space $\Cs^\alpha$ to be the closure of $C^\infty(\T^2)$ under the $\Cs^\alpha$-norm. We will also often write $\Cs_x^\alpha$ instead of $\Cs^\alpha$.
\end{definition}

Equipped with the Littlewood-Paley operators, we can now define the para-product operators. To this end, we first introduce the following notation.

\begin{notation}
For frequency scales $M, N \in \dyadic$, we say that $M \ll N$ (resp. $M \gg N$) if $M \leq N/16$ (resp. $M \geq 16N$). We say that $M \sim N$ if $N/16 < M < 16N$, and otherwise we write $M \nsim N$.
\end{notation}

Since symbols like $\sim$ have been defined slightly differently for general quantities in Notation \ref{notation:lessim}, we are committing a slight abuse of notation here, but it should always be clear from context which definition is being used.

\begin{definition}[Para-product operators]
For all smooth functions $\phi,\psi \colon \T^2 \rightarrow \C$, we define
\begin{alignat}{2}
\phi \parall \psi &:= \sum_{\substack{M,N\colon \\ M \ll N}} 
P_M \phi \, P_N \psi, \hspace{10ex}
\phi \parasim \psi &:= \sum_{\substack{M,N\colon \\ M \sim N}} 
P_M \phi \, P_N \psi, \\  
\phi \paragg \psi &:= \sum_{\substack{M,N\colon \\ M \gg N}} 
P_M \phi \, P_N \psi, \hspace{10ex}
\phi \paransim \psi &:= \sum_{\substack{M,N\colon \\ M \not \sim N}} 
P_M \phi \, P_N \psi.
\end{alignat}
Similarly, for all smooth $\frkg$-valued functions $\Phi,\Psi\colon \T^2 \rightarrow \frkg$, we define
\begin{alignat}{2}
 \Phi \parall \Psi &:= \sum_{\substack{M,N\colon \\ M \ll N}} 
 P_M \Phi \otimes P_N \Psi , \hspace{10ex}
 \Phi \parasim \Psi  &:= \sum_{\substack{M,N\colon \\ M \sim N}} 
 P_M \Phi \otimes P_N \Psi , \label{prelim:eq-para-g-1} \\  
 \Phi \paragg \Psi  &:= \sum_{\substack{M,N\colon \\ M \gg N}} 
 P_M \Phi \otimes P_N \Psi , \hspace{10ex}
 \Phi \paransim \Psi  &:= \sum_{\substack{M,N\colon \\ M \not \sim N}}
 P_M \Phi \otimes P_N \Psi . \label{prelim:eq-para-g-2} 
\end{alignat}
We emphasize that \eqref{prelim:eq-para-g-1} and \eqref{prelim:eq-para-g-2} are defined using tensor products of $P_M \Phi$ and $P_N \Psi$ in $\frkg$, and therefore the para-products in \eqref{prelim:eq-para-g-1} and \eqref{prelim:eq-para-g-2} map into $\frkg\otimes \frkg$. Since the paraproducts in \eqref{prelim:eq-para-g-1} and \eqref{prelim:eq-para-g-2} will often appear in conjunction with Lie brackets, we also write 
\begin{alignat}{2}
\big[ \Phi \parall \Psi \big]&:= \sum_{\substack{M,N\colon \\ M \ll N}} 
\big[ P_M \Phi , P_N \Psi \big], \hspace{10ex}
\big[ \Phi \parasim \Psi \big] &:= \sum_{\substack{M,N\colon \\ M \sim N}} 
\big[ P_M \Phi , P_N \Psi \big], \\  
\big[ \Phi \paragg \Psi \big] &:= \sum_{\substack{M,N\colon \\ M \gg N}} 
\big[ P_M \Phi , P_N \Psi \big], \hspace{10ex}
\big[ \Phi \paransim \Psi \big] &:= \sum_{\substack{M,N\colon \\ M \not \sim N}}
\big[ P_M \Phi , P_N \Psi \big].
\end{alignat}
\end{definition}

In the next lemma, we record estimates for the product and para-product operators in H\"{o}lder spaces. 

\begin{lemma}[Para-product estimates]\label{prelim:lem-para-product}
Let $\alpha,\beta,\gamma \in \R \backslash \{0\}$ and let $\phi,\psi \colon \T^2\rightarrow \C$. Then, we have the following estimates:
\begin{enumerate}[label=(\roman*)]
    \item (Product estimate) If $\gamma \leq \min(\alpha,\beta)$ and $\alpha+\beta>0$, then 
    \begin{equation*}
    \big\| \phi  \psi \big\|_{\Cs_x^\gamma}
    \lesssim \big\| \phi \big\|_{\Cs_x^\alpha} 
    \big\| \psi \big\|_{\Cs_x^\beta}. 
    \end{equation*}
    \item (Low$\times$high-estimate) If $\gamma \leq \min(\alpha,0)+\beta$, then
    \begin{equation*}
    \big\| \phi \parall \psi \big\|_{\Cs_x^\gamma}
    \lesssim \big\| \phi \big\|_{\Cs_x^\alpha} 
    \big\| \psi \big\|_{\Cs_x^\beta}. 
    \end{equation*}
    \item (High$\times$high-estimate) If $\gamma \leq \alpha+\beta$ and $\alpha+\beta>0$, then
    \begin{equation*}
    \big\| \phi \parasim \psi \big\|_{\Cs_x^\gamma}
    \lesssim \big\| \phi \big\|_{\Cs_x^\alpha} 
    \big\| \psi \big\|_{\Cs_x^\beta}. 
    \end{equation*}
    \item (High$\times$low-estimate) If $\gamma \leq \alpha+\min(0,\beta)$, then
    \begin{equation*}
    \big\| \phi \paragg \psi \big\|_{\Cs_x^\gamma}
    \lesssim \big\| \phi \big\|_{\Cs_x^\alpha} 
    \big\| \psi \big\|_{\Cs_x^\beta}. 
    \end{equation*}
\end{enumerate}
\end{lemma}

\begin{lemma}[\protect{\cite[Lemma 2.4]{GIP15}}]\label{prelim:lem-lh-hh} 
Let $\alpha,\beta,\gamma \in \R$ be such that $\alpha<1$, $\alpha+\beta+\gamma>0$, and $\beta+\gamma<0$. Then, it holds for all $f,g,h\colon \T^2 \rightarrow \C$ that
\begin{equation}
\big\| \big( f \parall g \big) \parasim h - f \big( g \parasim h \big) \big\|_{\Cs^{\alpha+\beta+\gamma}}
\lesssim \big\| f \big\|_{\Cs^\alpha} \big\| g \big\|_{\Cs^\beta} \big\| h \big\|_{\Cs^\gamma}. 
\end{equation}
\end{lemma}

\begin{lemma}\label{lemma:P-N-lo-hi-commutator}
Let $\alpha < 1$, $\beta \in \R$, $\delta > 0$. For $N \in \dyadic$, we have that
\[ \|P_N (f \parall g) - f \parall P_N g\|_{\Cs_x^{\alpha + \beta - \delta}} \lesssim N^{-\delta} \|f\|_{\Cs_x^\alpha} \|g\|_{\Cs_x^\beta}. \]
\end{lemma}

In the following definition, we introduce the Fourier multipliers $(Q^j_{>N})$ which will play an important role in our proof of gauge-covariance (see e.g. Proposition \ref{prop:linear-objects-close-gauge-covariance}). 

\begin{definition} Let $N\in \dyadic$ and let $j\in [2]$. We then define the operator $Q^j_{>N}$ by 
\begin{equation}\label{prelim:eq-Q}
    \widehat{Q^j_{>N} f}(n) = \frac{1}{\icomplex} \big( \partial^j \rho_{>N} \big)(n) \widehat{f}(n)
\end{equation}
for all smooth $f\colon \T^2 \rightarrow \mathbb{C}$ and $n\in \Z^2$. 
\end{definition}

We now state two lemmas regarding $Q^j_{>N}$ which will both be proven in Appendix \ref{section:proofs-commutator}. The first lemma shows that $Q^j_{>N}$ appears in the commutator of the Littlewood-Paley operator $P_{\leq N}$ and a para-multiplication operator. 

\begin{lemma}\label{lemma:P-N-commutator-space}
Let $\alpha < 2$, $\beta \in \R$, $\delta \in (0, 2)$. For $N \in \dyadic$, we have that
\[ \| P_{\leq N}(f \parall g) - f \parall P_{\leq N} g - \ptl_j f \parall Q^j_{> N} g\|_{\Cs_{x}^{\alpha + \beta - \delta}} \lesssim N^{-\delta} \|f\|_{\Cs_{x}^\alpha} \|g\|_{\Cs_x^\beta}. \]
\end{lemma}

The second lemma shows that $Q^j_{>N}$ can be used to gain either one power of $N$ or one spatial derivative.

\begin{lemma}[Estimate of $Q^\ell_{>N}$]\label{prelim:lem-Q}
Let $N\in \dyadic$ and $\ell \in [2]$. For all $\alpha,\beta \in \R$ satisfying $\alpha\leq \beta\leq \alpha+1$, it then holds that
\begin{equation}\label{prelim:eq-Q-estimate}
\big\| Q_{>N}^\ell f \big\|_{\Cs_x^\beta} \lesssim N^{\beta-(\alpha+1)} \big\| f \big\|_{\Cs_x^\alpha}. 
\end{equation}
\end{lemma}

\begin{notation}
Given $n_1, n_2 \in \Z^2$, we write $n_{12} = n_1 + n_2$. Similarly, given $n_1, n_2, n_3 \in \Z^2$, we write $n_{123} = n_1 + n_2 + n_3$. Additionally, in what follows, given some number of dyadic scales $N_0, \ldots, N_k \in \dyadic$, we write $N_{\max}$ to denote the maximum of $N_0, \ldots, N_k$. Note that $N_{\max}$ may be differ in various displays, depending on the set of dyadic scales which appear in any given display.
\end{notation}

In the following lemma, we record elementary counting estimates.

\begin{lemma}[Elementary counting estimates]\label{prelim:lem-counting}
For all $0<\delta\ll 1$ and all $N_0,N_1,N_2\in \dyadic$, it holds that 
\begin{equation}\label{prelim:eq-counting-e1}
\sum_{\substack{n_0,n_1,n_2 \in \Z^2 \colon \\ n_0 = n_1 + n_2 }} 
\rho_{N_0}(n_0) \rho_{N_1}(n_1) \rho_{N_2}(n_2) \lesssim N_0^{\delta} N_1^2 N_2^2 N_{\textup{max}}^{-\delta},
\end{equation}
\beq\label{eq:quadratic-stochastic-object-no-derivative-time-derivative-estimate} \sum_{n_1, n_2 \in \Z^2} \rho_{N_0}(n_{12}) \rho_{N_1}(n_1) \rho_{N_2}(n_2)\frac{1}{\fnorm{n_1}^2} \lesssim N_0^{2-\delta} N_{\mrm{max}}^{\delta}, \eeq
\beq\label{eq:quadratic-stochastic-object-one-derivative-combintarial-estimate} \sum_{n_1, n_2 \in \Z^2} \frac{|n_2
|^2}{\fnorm{n_1}^2 \fnorm{n_2}^2 \fnorm{n_{12}}^2} \rho_{N_1}(n_1) \rho_{N_2}(n_2) \rho_{N_0}(n_{12}) \max(\fnorm{n_1}, \fnorm{n_2}, \fnorm{n_{12}})^{-2} \lesssim N_{\max}^{-2} \lesssim N_0^{-2 + \delta} N_{\max}^{-\delta}.\eeq 
Given an additional dyadic scale $N_3\in \dyadic$, assuming that $N_0 \lesssim \max(N_1, N_2, N_3)$, we have that 
\beq\label{eq:cubic-object-no-derivative-combinatorial-estimate}\begin{split}
\sum_{n_1, n_2, n_3 \in \Z^3} \frac{\rho_{N_1}(n_1) \rho_{N_2}(n_2) \rho_{N_3}(n_3)}{\fnorm{n_1}^2 \fnorm{n_2}^2 \fnorm{n_3}^2} \rho_{N_0}(n_{123}) \lesssim  N_0^{\delta} N_{\mrm{max}}^{-\delta}.
\end{split}
\eeq
Given an additional dyadic scale $N_{12}\in \dyadic$ such that $N_{12} \sim N_3$, we have that 
\beq\label{eq:cubic-two-derivative-combinatorial-estimate}\begin{split}
\sum_{n_1, n_2, n_3 \in \Z^2} \frac{1}{\fnorm{n_1}^2 \fnorm{n_{12}}^2} &\min\big(\fnorm{n_1}^{-2}, \fnorm{n_2}^{-2}, \fnorm{n_{12}}^{-2}\big) ~\times \\
&\rho_{N_0}(n_{123}) \rho_{N_1}(n_1) \rho_{N_2}(n_2) \rho_{N_{12}}(n_{12}) \rho_{N_3}(n_3) \lesssim N_0^{\delta} N_{\mrm{max}}^{-\delta}.
\end{split}\eeq 
\end{lemma}
\begin{proof}
\emph{Proof of \eqref{prelim:eq-counting-e1}}
We first note that it is sufficient to treat the case $N_0 \geq N_1 \geq N_2$. This is because the left-hand side of \eqref{prelim:eq-counting-e1} is symmetric in $N_0$, $N_1$, and $N_2$ and the exponents on the right-hand side of \eqref{prelim:eq-counting-e1} are non-decreasing in $N_0$, $N_1$, and $N_2$. In the case $N_0 \geq N_1 \geq N_2$,  we have that 
\begin{align*}
\sum_{\substack{n_0,n_1,n_2 \in \Z^2 \colon \\ n_0 = n_1 + n_2 }} 
\rho_{N_0}(n_0) \rho_{N_1}(n_1) \rho_{N_2}(n_2)  
\lesssim \sum_{n_1,n_2 \in \Z^2} 
 \rho_{N_1}(n_1) \rho_{N_2}(n_2) \sim N_1^2 N_2^2 \sim N_0^{\delta} N_1^2 N_2^2 N_{\textup{max}}^{-\delta}. 
\end{align*}
This proves \eqref{prelim:eq-counting-e1}. 

\emph{Proof of \eqref{eq:quadratic-stochastic-object-no-derivative-time-derivative-estimate}} To prove \eqref{eq:quadratic-stochastic-object-no-derivative-time-derivative-estimate}, we may bound
\[\begin{split}
\sum_{n_1, n_2 \in \Z^2} \rho_{N_0}(n_{12}) \rho_{N_1}(n_1) \rho_{N_2}(n_2)\frac{1}{\fnorm{n_1}^2} &\leq N_1^{-2} \sum_{n_1, n_2 \in \Z^2} \rho_{N_0}(n_{12}) \rho_{N_1}(n_1) \rho_{N_2}(n_2) \\
&\lesssim N_1^{-2} \min(N_0, N_1, N_2)^2 \mrm{min}_2(N_0, N_1, N_2)^2, 
\end{split}\]
where $\min_2$ denotes the second-smallest number. If $N_0 = N_{\mrm{max}}$, then we may bound the above by
\[ N_1^{-2} N_1^2 N_2^2 = N_2^2 N_0^{2 - \delta} N_{\max}^{\delta - 2} \leq N_0^{2-\delta} N_{\max}^{\delta}.\]
If otherwise $N_0 < N_{\max}$, then we can obtain the bound
\[ N_1^{-2} N_0^2 \min(N_1, N_2)^2 \leq N_0^2 \leq N_0^{2-\delta} N_{\max}^{\delta}. \]
This proves \eqref{eq:quadratic-stochastic-object-no-derivative-time-derivative-estimate}.

\emph{Proof of \eqref{eq:quadratic-stochastic-object-one-derivative-combintarial-estimate}.} For \eqref{eq:quadratic-stochastic-object-one-derivative-combintarial-estimate}, observe that we may bound the quantity in question by
\[ N_{\max}^{-2} \sum_{n_1, n_2 \in \Z^2} \frac{\rho_{N_1}(n_1) \rho_{N_0}(n_{12})}{\fnorm{n_1}^2 \fnorm{n_{12}}^2} \lesssim N_{\max}^{-2}, \]
as desired.

\emph{Proof of \eqref{eq:cubic-object-no-derivative-combinatorial-estimate}} Without loss of generality, it suffices to assume that $N_1 \leq N_2 \leq N_3$. The quantity in question may then be bounded by
\[ \frac{ N_1^2 N_2^2 N_0^2}{N_1^2 N_2^2 N_3^2} = \frac{N_0^2}{N_3^2} \lesssim N_0^{\delta} N_{\max}^{-\delta}, \]
where we used the assumptions $N_0 \lesssim \max(N_1, N_2, N_3)$ and $N_1 \leq N_2 \leq N_3$ in the final step.

\emph{Proof of \eqref{eq:cubic-two-derivative-combinatorial-estimate}} For \eqref{eq:cubic-two-derivative-combinatorial-estimate}, note that the term in question may be upper bounded by
\beq N_{\mrm{max}}^{-2} \sum_{n_1 \in \Z^2} \frac{\rho_{N_1}(n_1)}{\fnorm{n_1}^2} \sum_{n_2 \in \Z^2} \frac{\rho_{N_2}(n_2) \rho_{N_{12}}(n_{12})}{\fnorm{n_{12}}^2} \sum_{n_3 \in \Z^2} \rho_{N_0}(n_{123}) \rho_{N_3}(n_3), \eeq 
which itself may be further bounded (here we use that $N_{12} \sim N_3$)
\[ N_{\mrm{max}}^{-2} \min(N_0, N_3)^2 \lesssim N_{\mrm{max}}^{-2} N_0^{\delta} N_3^{2 - \delta} \lesssim N_0^{\delta} N_{\mrm{max}}^{-\delta}. \qedhere\]
\end{proof}

\subsection{Function spaces and heat propagator}\label{section:function-spaces}

\begin{definition}[Function spaces]\label{prelim:def-function-space}
For $\alpha\in \R$, $\nu\geq 0$, and $T > 0$, we define the time-weighted H\"{o}lder space
\begin{equation}
\big\| f \big\|_{\Wc^{\alpha,\nu}([0,T])} := \sup_{t\in (0,T]} \min(t, 1)^\nu \big\| f(t) \big\|_{\Cs^\alpha}. 
\end{equation}
We also define the global-in-time space
\[ \|f\|_{\Wc^{\alpha, \nu}(\R)} := \sup_{t \in \R \backslash \{0\}} \min(|t|, 1)^\nu \|f\|_{\Cs^{\alpha}}.\]
For $\zeta \in [0,1)$, $\alpha \in \R$, $\nu \geq 0$, and $T > 0$, we also define
\begin{equation}
\big\| f \big\|_{\CWc^{\zeta,\alpha,\nu}([0,T])} := \sup_{t,s \in (0,T]} \min(t,s)^\nu  \frac{\big\| f(t) - f(s) \big\|_{\Cs^\alpha}}{|t-s|^\zeta}. 
\end{equation}
\end{definition}

The $\CWc^{\zeta,\alpha,\nu}$-norms will only play a secondary role in our estimates (see Remark \ref{prelim:rem-integral-commutator}), and we therefore encourage the reader to focus on the $\Wc^{\alpha,\nu}$-norms. The global-in-time norm $\Wc^{\alpha, \nu}(\R)$ is only needed for gauge-covariance (Section \ref{section:gauge-covariance}).

\begin{definition}[Solution spaces]\label{prelim:def-solution-space}
Let 
\begin{equation*}
\Reg := \big\{ 2\kappa, 3\kappa, 1-2\kappa, 1+ 2\kappa, 2-5\kappa \big\}
\end{equation*}
be a collection of regularity parameters and let $\theta:= \kappa$. For any $\alpha \in \Reg$, we define
\begin{equation}\label{prelim:eq-Salpha}
\big\| f \big\|_{\Sc^\alpha}:=
\sum_{\substack{\alpha^\prime \in \Reg \colon \\ \alpha^\prime \leq \alpha}} 
\big\| f \big\|_{\Wc^{\alpha^\prime,\frac{\alpha^\prime}{2}+\theta}} 
+ \big\| f \big\|_{C_t^0 \Cs^{-\kappa}_x}
+  \sum_{\substack{\alpha^\prime \in \Reg \colon \\ \alpha^\prime \leq \alpha}}   \big\| f \big\|_{\CWc^{\frac{\alpha^\prime+\kappa}{2},-\kappa,\frac{\alpha^\prime}{2}+\theta}}.
\end{equation}
\end{definition}
In \eqref{prelim:eq-Salpha}, the relationships between the different parameters are dictated by parabolic scaling. For example, consider a function satisfying $\| f\|_{\Sc^\alpha}\leq 1$ and a parameter $\alpha^\prime \in \Reg$ satisfying $\alpha^\prime \leq \alpha$. Then, $\| f(t) \|_{\Cs_x^{\alpha^\prime}}$ is controlled by $t^{-\alpha^\prime/2-\theta}$.  Thus, the more spatial derivatives $\alpha^\prime$ we attempt to control, the more inverse powers of $t$ are required, and the conversion rate is dictated by parabolic scaling. 

By interpolation, the sum in \eqref{prelim:eq-Salpha} is, of course, equivalent to the sum over $\alpha^\prime\in \{ 2\kappa, \alpha \}$. However, we decided to keep the sum over elements in $\Reg$, since it shows which regularities will be used in our estimates. \\

Equipped with our norms from Definition \ref{prelim:def-function-space} and Definition \ref{prelim:def-solution-space}, we now state estimates for the heat propagator and the corresponding Duhamel integral. To this end, we define 
\begin{equation}
\Hc_t := e^{-t (1-\Delta)}. 
\end{equation}
We also define the associated Duhamel integral as 
\begin{equation*}
\Duh [f] := \int_0^t \Hts f(s) \ds. 
\end{equation*}
It will often be convenient for us to integrate over $s\in (-\infty,t]$ rather than $s\in [0,t]$, and we therefore also define
\begin{equation*}
\Duhinf [f] := \int_{-\infty}^t \Hts f(s) \ds. 
\end{equation*}

\begin{lemma}[Heat propagator in $\Sc^{\alpha}$]\label{prelim:lem-heat-Salpha}
Let $\alpha \in \Reg$, $T\in (0,1]$, and $f\in \Cs_x^{-\kappa}$. Then, it holds that
\begin{equation}\label{prelim:eq-heat-Salpha}
\big\| \Hc_t f \big\|_{\Sc^\alpha([0,T])}
\lesssim \big\| f \big\|_{\Cs_x^{-\kappa}}. 
\end{equation}
\end{lemma}

The estimate \eqref{prelim:eq-heat-Salpha} follows from $\theta \geq \kappa$ and standard estimates for the heat propagator (see Lemma \ref{prelim:lem-heat-flow}), and we therefore omit the proof.

\begin{proposition}[Duhamel integral in $\Sc^\alpha$]\label{prelim:prop-Duhamel-Salpha}
Let $\alpha \in \Reg$, $\beta \leq -\kappa$, and $\nu \in [0,1)$. Furthermore, assume that 
\begin{equation}\label{prelim:eq-Duhamel-Salpha-a}
\alpha < \beta +2 \qquad \text{and} \qquad \delta:= \min \Big( 1 + \frac{\beta}{2}, 1 \Big) - \nu >0.
\end{equation}
Then, it holds that
\begin{equation}
\Big\| \Duh \big[ g \big] \Big\|_{\Sc^\alpha([0,T])} \lesssim T^\delta \big\| g \big\|_{\Wc^{\beta,\nu}([0,T])}. 
\end{equation}
\end{proposition}

Proposition \ref{prelim:prop-Duhamel-Salpha} is a standard Schauder-type estimate and similar estimates have been used in \cite{CC18,GIP15}. For the sake of completeness, we present a proof in Appendix \ref{section:proofs-para}. \\

The next lemma controls a commutator-term involving the  Duhamel integral and para-products. 

\begin{lemma}[Trilinear integral commutator estimate]\label{prelim:lem-trilinear-integral-commutator}
Let $T\in [0,1]$, let $\nu\geq \theta+2\kappa$, and let $f,g,h\colon [0,T] \times \T^2 \rightarrow \C$. Then, it holds that
\begin{align}\label{prelim:rem-trilinear-integral-commutator}
\Big\| \Big( \Duh \big( f \parall g \big) - f \parall \Duh \big( g\big) \Big) \parasim h \Big\|_{\Wc^{-4\kappa,\nu}([0,T])}
\lesssim \big\| f \big\|_{\Sc^{1-2\kappa}([0,T])} \big\| g \big\|_{\Wc^{-1-\kappa,0}([0,T])} \big\| h\big\|_{\Wc^{-1-\kappa,0}([0,T])}. 
\end{align}
\end{lemma}

A similar estimate has been shown in \cite[Proposition 2.7]{CC18}. For the sake of completeness, we present a proof of 
of Lemma \ref{prelim:lem-trilinear-integral-commutator} in Appendix \ref{section:proofs-para}. 

\begin{remark}\label{prelim:rem-integral-commutator}
The estimate \eqref{prelim:rem-trilinear-integral-commutator} is the only estimate in our analysis which requires the $\CWc$-terms in \eqref{prelim:eq-Salpha}, i.e., the H\"{o}lder-regularity in the time-variable. Since $\Duh ( f \parall g )(t)$ depends on the whole process $(f(s))_{0\leq s \leq t}$ whereas $(f \parall \Duh ( g))(t)$ depends only on $f(t)$, this cannot be avoided. 
In all other estimates, we only require the $\Wc$-terms in  \eqref{prelim:eq-Salpha}. 
\end{remark}

In the course of the proof of gauge covariance (Theorem \ref{thm:gauge-covariance}), we will need to work with space-time Besov spaces, which we next introduce. Besides entering into the definition of the solution map $\bA$ and the statement of Theorem \ref{thm:gauge-covariance}, these space-time Besov spaces will not be needed until Section \ref{section:gauge-covariance}.

Given dyadic scales $L, N \in \dyadic$, define the space-time Littlewood-Paley operator $P_{L, N}$ by
\[ \widehat{P_{L, N} \phi}(u, n) = \rho_{L}(u) \rho_{N}(n) \hat{\phi}(u, n), ~~ \phi : \R \times \T^2 \ra \R, ~~ u \in \R, n \in \Z^2. \]

\begin{remark}
In this paper, $L$ will always be used to denote a temporal frequency, while $M, N$ will be used to denote spatial frequencies.
\end{remark}

\begin{definition}[Space-time Besov space]
For $1 \leq p, q \leq \infty$, $\alpha = (\alpha_t, \alpha_x) \in \R$, define the space $\mc{B}_{p, q}^\alpha$ as the set of distributions $\phi \in \mc{D}'(\R \times \T^2)$  
such that
\[ \|\phi\|_{\mc{B}^\alpha_{p,q}} := \bigg(\sum_{L, N \in \dyadic} \big(L^{\alpha_t} N^{\alpha_x} \|P_{L, N} \phi\|_{L^p(\R \times \T^2)}\big)^q\bigg)^{1/q} < \infty,  \]
with the usual interpretation as $\ell_\infty$ norm when $q = \infty$. When $p = q = \infty$ we denote $\mc{B}_{p, q}^\alpha = \Cs^\alpha_{tx}$.
\end{definition}

Because the symbol of $P_{L, N}$ decomposes as the product $\rho_L(u) \rho_N(n)$, essentially all the basic results for Besov spaces carry over to these space-time Besov spaces, with straightforward modifications to the proofs. For instance, the paraproduct estimates carry over in the analogous way. Note that there are now nine possible paraproducts. We shall denote $\parall_{tx}, \parasim_{tx}, \paragg_{tx}$ as the paraproducts which are the same in both space and time. To denote paraproducts which are different, we will e.g. write $\parall_t \parasim_x$ (this example is low$\times$high in time and high$\times$high in space). Examples of space-time paraproduct estimates are:
\[ \|f \parall_t \paragg_x g\|_{\Cs_{tx}^{\gamma}} \lesssim \|f\|_{\Cs_{tx}^\alpha} \|g\|_{\Cs_{tx}^\beta}, ~~ \|f \parasim_t \parall_x g\|_{\Cs_{tx}^{\eta}} \lesssim \|f\|_{\Cs_{tx}^\alpha} \|g\|_{\Cs_{tx}^\beta}, \]
where in the first example we assume $\gamma_t \leq \min(\alpha_t, 0) + \beta_t$ and $\gamma_x \leq \alpha_x + \min(\beta_x, 0)$, and in the second example we assume $\alpha_t + \beta_t > 0$, $\eta_t < \alpha_t + \beta_t$, and $\eta_x \leq \min(\alpha_x, 0) + \beta_x$.

\begin{notation}
Paraproducts $\parall, \parasim, \paragg$ without subscripts will always be understood to be spatial paraproducts. 
\end{notation}

Next, we define space-time Besov spaces on finite time intervals. For $T > 0$, let $\mc{D}'((0, T) \times \T^2)$ be the space of distributions on $(0, T) \times \T^2$, i.e. the continuous dual of $C^\infty_c((0, T) \times \T^2)$.

\begin{definition}[Space-time Besov spaces on finite time intervals]\label{def:space-time-besov-space-finite-interval}
Let $1 \leq p, q \leq \infty$, $\alpha \in \R^2$. Define $\Cs_{tx}^\alpha((0, T)) = \Cs_{tx}^{\alpha}((0, T) \times \T^2)$ to be the set of distributions $f \in \mc{D}'((0, T) \times \T^2)$ such that there exists $g \in \Cs_{tx}^\alpha$ with $f = g ~\big|_{(0, T) \times \T^2}$. We define the corresponding norm
\[ \|f\|_{\Cs_{tx}^\alpha((0, T))} = \|f\|_{\Cs_{tx}^\alpha((0, T) \times \T^2)} := \inf \{ \|g\|_{\Cs_{tx}^\alpha} : g ~\big|_{(0, T) \times \T^2} = f\}. \]
\end{definition}

\subsection{Lie groups and Lie algebras}\label{section:lie-group-lie-algebra}

Throughout this paper, fix a compact Lie group $G$. As a consequence of the Peter-Weyl theorem, we may assume that $G \sse \unitary(N)$ for some $N \geq 1$, where $\unitary(N)$ is the unitary group.
Let $\Ad$ be the adjoint representation of $G$ on $\frkg$, and $\ad$ the adjoint representation of $\frkg$ on $\frkg$. Explicitly, $\Ad_g X = g X g^{-1}$ for $g \in G$, $X \in \frkg$, and $\ad(X)(Y) = [X,Y]$ for $X, Y \in \frkg$. Fix an $\Ad$-invariant inner product $\langle \cdot, \cdot \rangle_\frkg$ on $\frkg$. Let $(E_a)_{a=1}^{\dim\frkg}$ be an orthonormal basis of the Lie algebra $\frkg$. We let $(E_a^*)_{a =1}^{\dim\frkg}$ be the associated dual basis of $\frkg^*$.  Let $\cfrkg := \frkg \otimes \C$ be the complexification of $\frkg$, and let $\langle \cdot, \cdot \rangle_{\cfrkg}$ denote the extension of $\langle \cdot, \cdot \rangle_\frkg$ to $\cfrkg$. Note that (with respect to $\langle \cdot, \cdot \rangle_{\cfrkg}$) $(E_a)_{a=1}^{\dim \frkg}$ is also an orthonormal basis of $\cfrkg$. Let $m, n \geq 0$. We identify the space of linear operators $L(\cfrkg^m, \cfrkg^n) \cong \cfrkg^{* \otimes m} \otimes \cfrkg^{\otimes n}$. We denote by $\langle \cdot, \cdot \rangle_{\cfrkg^{* \otimes m} \otimes \cfrkg^{\otimes n}}$ the induced inner product on $\cfrkg^{*\otimes m} \otimes \cfrkg^{\otimes n}$.

\begin{definition}[Killing map]\label{def:Kil}
Define $\Kil : \frkg \ra \frkg$ by
\[ \Kil(E) := \big[ \big[E, E_a\big], E^a\big] = \big[E_a, \big[E^a, E\big]\big] = \mrm{ad}(E_a) \mrm{ad}(E^a) E. \]
\end{definition}

\begin{remark}
We call $\Kil$ the Killing map because it is intimately related to the Killing form $K$ of $\frkg$, as follows. First, recall that by definition, the Killing form is a symmetric bilinear form on $\frkg$ defined as $K(E, F) := \Tr(\mrm{ad}(E) \mrm{ad}(F))$. By using the skew-symmetry of $\mrm{ad}$ (which follows from the $\mrm{Ad}$-invariance of $\langle \cdot, \cdot \rangle_\frkg$) and the identity $\mrm{ad}(E) F = - \mrm{ad}(F) E$, we have that
\[\begin{split}
K(E, F) &= \langle \mrm{ad}(E) \mrm{ad}(F) E_a, E^a \rangle_\frkg = -\langle \mrm{ad}(F) E_a, \mrm{ad}(E) E^a \rangle_\frkg \\
&= -\langle \mrm{ad}(E_a) F, \mrm{ad}(E^a) E \rangle = \langle F, \mrm{ad}(E_a) \mrm{ad}(E^a) E \rangle_\frkg = \langle F, \Kil(E) \rangle_\frkg. \end{split}\]
Thus $\Kil$ is exactly $K$, if we view $K$ not as a bilinear map (i.e. an element of $\frkg^* \otimes \frkg^*$) but rather as an element of $\frkg^* \otimes \frkg \cong L(\frkg, \frkg)$.
\end{remark}

\begin{remark}
In \cite{CCHS22, CCHS22+}, the authors define the map $\mrm{ad}_{\mrm{Cas}}$, where $\mrm{Cas} = E_a \otimes E^a$ is the Casimir element of the universal enveloping algebra of $\frkg$. Explicitly, $\mrm{ad}_{\mrm{Cas}}(E) = [E_a, [E^a, E]]$ (see e.g. \cite[Section 5.1]{CCHS22+}). Thus we see that $\mrm{ad}_{\mrm{Cas}}$ is exactly $\Kil$.
\end{remark}

\begin{definition}[Pairings]\label{def:pairing}
Let $m \geq 1$. A pairing $\mc{P}$ of $[m]$ is a set of disjoint two-element subsets of $[m]$ (this is also often referred to as a matching). 
We let $\mc{M}(m)$ be the set of pairings of $[m]$. For $m_1, m_2 \geq 1$, let $\mc{M}(m_1, m_2)$ be the set of pairings $\mc{P}$ of $[m_1 + m_2]$ such that all elements of $\mc{P}$ contain an element of $[m_1]$ and an element of $\{m_1 + 1, \ldots, m_1 + m_2\}$. We say that $\mc{P} \in \mc{M}(m_1, m_2)$ is a complete pairing if all elements of $[m_1 + m_2]$ are paired, i.e. $\mc{P}$ is a partition of $[m_1 + m_2]$ (note this requires $m_1 = m_2$).
\end{definition}

\begin{remark}
Note that complete pairings $\mc{P} \in \mc{M}(m, m)$ can be identified with permutations of $[m]$.
\end{remark}

\begin{definition}[Tensor contraction]\label{def:tensor-contraction}
Let $m \geq 2$, $n \geq 0$, and consider a tensor space of the form $(\cfrkg^*)^{\otimes m} \otimes \cfrkg^{\otimes n}$. Given distinct indices $i, j \in [m]$, define the tensor contraction $\tensorcon{\{i, j\}}  : (\cfrkg^*)^{\otimes m} \otimes \cfrkg^{\otimes n} \ra (\cfrkg^*)^{\otimes (m-2)} \otimes \cfrkg^{\otimes n}$ as the linear map which contracts the $i$th and $j$th $\cfrkg^*$ component. More formally, it is the linear map which sends (suppose without loss of generality $i < j$)
\beq E_{a_1}^* \otimes \cdots \otimes E_{a_m}^* \otimes E^{b_1} \otimes \cdots \otimes E^{b_n} \mapsto \big(E_{a_i}^*(E_{a_j})\big) E^*_{a_1} \otimes \cdots \widehat{E^*_{a_i}} \otimes \cdots \widehat{E^*_{a_j}} \otimes E^*_{a_m} \otimes E^{b_1} \otimes \cdots \otimes E^{b_n}. \eeq 
More generally, given a pairing $\mc{P}$ of $[m]$, define $\tensorcon{\mc{P}} : (\cfrkg^*)^{\otimes (m - 2|\mc{P}|)} \otimes \cfrkg^{\otimes n}$ as the composition
\beq \tensorcon{\mc{P}} = \bigcirc_{\{i, j\} \in \mc{P}} \tensorcon{\{i, j\}}.\eeq 
(Note that the order of compositions does not matter.) 
\end{definition}

\begin{convention}\label{convention:iterated-tensor}
Given $m_1, m_2 \geq 1$, $n_1, n_2 \geq 1$, and $S \in \cfrkg^{*\otimes m_1} \otimes \cfrkg^{\otimes n_1}$, $T \in \cfrkg^{*\otimes m_2} \otimes \cfrkg^{\otimes n_2}$, we will view $S \otimes T \in \cfrkg^{*\otimes (m_1 + m_2)} \otimes \cfrkg^{\otimes (n_1 + n_2)}$, as opposed to as an element of $\cfrkg^{*\otimes m_1} \otimes \cfrkg^{\otimes n_1} \otimes \cfrkg^{*\otimes m_2} \otimes \cfrkg^{\otimes n_2}$. This is done to mesh with our notation for tensor contraction. 
\end{convention}

\begin{example}\label{example:1-2-tensor-contraction}
Let $S, T \in L(\cfrkg, \cfrkg) \cong \cfrkg^* \otimes \cfrkg$. We will view $S \otimes T \in \cfrkg^* \otimes \cfrkg^* \otimes \cfrkg \otimes \cfrkg$ (recall Convention \ref{convention:iterated-tensor}). Then
\beq \tensorcon{\{1, 2\}} (S \otimes T) = \delta^{cd} (S E_c) \otimes (T E_d), \eeq 
where on the right hand side, we view $S, T$ as linear maps $\cfrkg \ra \cfrkg$. 
\end{example}

\begin{notation}\label{notation:inner-product-as-linear-map}
We will (slightly abusing notation) write $\langle \cdot \rangle_{\cfrkg^{\otimes n}}$ for the linear map $\cfrkg^{\otimes n} \otimes \cfrkg^{\otimes n} \ra \C$ induced by the inner product $\langle \cdot, \cdot \rangle_{\cfrkg^{\otimes n}}$ on $\cfrkg^{\otimes n}$.
\end{notation}

\begin{lemma}\label{lemma:contracted-tensor-norm-bound}
Let $m, n \geq 1$, and let $S, T \in \cfrkg^{* \otimes m} \otimes \cfrkg^{\otimes n}$. Let $\mc{P} \in \mc{M}(m, m)$ be a complete pairing. Then
\[ \big|\langle \tensorcon{\mc{P}}(S \otimes T) \rangle_{\cfrkg^{\otimes n}}\big| \leq |S|_{\cfrkg^{* \otimes m} \otimes \cfrkg^{\otimes n}} |T|_{\cfrkg^{* \otimes m} \otimes \cfrkg^{\otimes n}} . \]
\end{lemma}
\begin{proof}
We express $S, T$ as a sum over elementary tensors:
\[ S = S^{a_1 \cdots a_m}_{b_1 \cdots b_n} E_{a_1}^* \otimes \cdots E_{a_m}^* \otimes E^{b_1} \cdots E^{b_n}, ~~ T = T^{a_1 \cdots a_m}_{b_1 \cdots b_n} E_{a_1}^* \otimes \cdots E_{a_m}^* \otimes E^{b_1} \cdots E^{b_n}.\]
Since $\mc{P}$ is a complete pairing, there is some permutation $\sigma$ of $[m]$ such that
\[ \tensorcon{\mc{P}}(S \otimes T) = \prod_{i \in [m]} \delta_{a_i c_{\sigma(i)}}  S^{a_1 \cdots a_m}_{b_1 \cdots b_n} T^{c_1 \cdots c_m}_{d_1 \cdots d_n} E^{b_1} \otimes \cdots \otimes E^{b_n} \otimes E^{d_1} \otimes \cdots \otimes E^{d_n}.\]
From this, we obtain
\[ \langle \tensorcon{\mc{P}}(S \otimes T) \rangle_{\cfrkg^{\otimes m}} = \prod_{i \in [m]} \delta_{a_i c_{\sigma(i)}}  \prod_{j \in [n]} \delta^{b_j d_j} S^{a_1 \cdots a_m}_{b_1 \cdots b_n} T^{c_1 \cdots c_m}_{d_1 \cdots d_n}. \]
By Cauchy-Schwarz, the right hand side above may be bounded
\[ \bigg(\sum_{\substack{a_i, b_j \\ i \in [n], j \in [m]}} \big|S^{a_1 \cdots a_m}_{b_1 \cdots b_n} \big|^2\bigg)^{1/2} \bigg(\sum_{\substack{a_i, b_j \\ i \in [n], j \in [m]}} \big|T^{a_1 \cdots a_m}_{b_1 \cdots b_n} \big|^2\bigg)^{1/2} = |S|_{\cfrkg^{* \otimes m} \otimes \cfrkg^{\otimes n}} |T|_{\cfrkg^{* \otimes m} \otimes \cfrkg^{\otimes n}}, \]
and the desired result follows.
\end{proof}


\section{Para-controlled Ansatz}\label{section:para-controlled-ansatz}

We recall that the frequency-truncated stochastic Yang-Mills heat flow is given by 
\begin{equation}\label{ansatz:eq-SYM-leqN}
    \partial_t A^i_{\leq N} = \Delta A^i_{\leq N} + \big[ A^j_{\leq N} , 2 \partial_j A^i_{\leq N} - \partial^i A_{\leq N,j} \big] + \big[ \big[ A^i_{\leq N}, A^j_{\leq N} \big], A_{\leq N,j} \big] + P_{\leq N}\xi^i.
\end{equation}
For technical reasons, it is convenient to subtract and add the massive term $A_{\leq N}^i$ to \eqref{ansatz:eq-SYM-leqN}. This allows us to replace the symbol $n \in \Z^2 \mapsto |n|^2$ of the negative Laplacian by $n \in \Z^2 \mapsto \langle n \rangle^2$, which is strictly positive. To be precise, we rewrite \eqref{ansatz:eq-SYM-leqN} as 
\begin{equation}\label{ansatz:eq-SYM-leqN-revised}
 \big( \partial_t - \Delta + 1 \big) A^i_{\leq N} = \big[ A^j_{\leq N} , 2 \partial_j A^i_{\leq N} - \partial^i A_{\leq N,j} \big] + \big[ \big[ A^i_{\leq N}, A^j_{\leq N} \big], A_{\leq N,j} \big] + A_{\leq N}^i +  P_{\leq N}\xi^i.
\end{equation}

\begin{figure}
\centering
\begin{tabular}{|P{\bigcolwidth}|P{\colwidth}|P{\colwidth}|P{\colwidth}|P{\colwidth}|P{\colwidth}|P{\colwidth}|}
\hline 
Object & $A$ & $\linear$ &  $\quadratic$ & $B$ & $X$ & $Y$  \\ \hline
Regularity & $-\kappa$ & $-\kappa$ & $1-2\kappa$ & $1-2\kappa$ & $1-2\kappa$ & $2-5\kappa$ \\ \hline
\end{tabular}
\centering
    \caption{In this figure, we display the terms in our para-controlled Ansatz and their regularities.}
    \label{figure:regularities}
\end{figure}

The main goal of this section is to derive the para-controlled Ansatz for \eqref{ansatz:eq-SYM-leqN}. The terms in this Ansatz, which will be defined momentarily, and the corresponding regularities are listed in Figure \ref{figure:regularities}. We first define the linear stochastic object $\linear[\leqN][r][]$, which is a $\frkg$-valued one-form, by 

\begin{equation}
\linear[\leqN][r][]=\linear[\leqN][l][i] \mathrm{d}x^i
\end{equation}
and 
\begin{equation}\label{ansatz:eq-linear}
\linear[\leqN][r][i] =  \Duhinf \Big( P_{\leq N} \xi^i \Big). 
\end{equation}
The reason for using the Duhamel integral over $(-\infty,t]$, rather than the Duhamel integral over $[0,t]$, is that \eqref{ansatz:eq-linear} yields a stationary stochastic object, which is convenient in many of our calculations. 
In Lemma \ref{lemma:linear-object}, we will show that the linear stochastic object $\linear[\leqN][r][]$ has spatial regularity $0-$. We now let $1\leq i,j,k\leq 2$. Then, we define the quadratic stochastic objects by 
\begin{equation}\label{ansatz:eq-quadratic}
 \quadratic[\leqN][d][i][j][k] = \Duhinf \Big( \Big[ \linear[\leqN][r][i], \partial_k \linear[\leqN][r][j] \Big] \Big).
\end{equation}
Furthermore, we define the combined quadratic stochastic object as 
\begin{equation}
\quadratic[\leqN][r][]= \quadratic[\leqN][l][i] \mathrm{d}x^i,
\end{equation}
where we define
\begin{equation}\label{ansatz:eq-quadratic-combined}
\quadratic[\leqN][r][i] = \quadratic[\leqN][l][i]
= 2 \quadratic[\leqN][d][j][i][j] - \quadratic[\leqN][d][j][j][i].
 \end{equation}
 On the right-hand side of \eqref{ansatz:eq-quadratic-combined}, we implicitly sum over all indices $1\leq j \leq 2$. This summation coincides with the summation in the nonlinearity of \eqref{ansatz:eq-SYM-leqN-revised}. In Lemma \ref{lemma:quadratic-stochastic-object-with-derivative} below, we will show that the quadratic object from \eqref{ansatz:eq-quadratic-combined}  has spatial regularity $1-$. We now define
 \begin{align}
 Z_{\leq N} &:= A_{\leq N} - \linear[\leqN][r][] - \quadratic[\leqN][r][], \\ 
B_{\leq N} &:= A_{\leq N} - \linear[\leqN][r][]. \label{ansatz:eq-BN}
 \end{align}
 The reason for also introducing $B_{\leq N}$, which is the sum of $\quadratic[\leqN][r][]$ and $Z_{\leq N}$, is that it can be used to simplify the right-hand side of \eqref{ansatz:eq-Z} below. The evolution equation for the nonlinear remainder $Z_{\leq N}$ is given by
 \begin{equation}\label{ansatz:eq-Z}
 \begin{aligned}
 (\partial_t - \Delta +1) Z_{\leq N}^i &= 
 \Big[ B_{\leq N}^j , 2 \partial_j \linear[\leqN][r][i] - \partial^i \linear[\leqN][l][j] \Big] 
 + \Big[ \linear[\leqN][r][j], 2 \partial_j B_{\leq N}^i - \partial^i B_{\leq N,j} \Big]
  \\ 
 &+ \Big[ B^j_{\leq N} , 2 \partial_j B^i_{\leq N} - \partial^i B_{\leq N,j} \Big] \\ 
 &+ \Big[ \Big[ \Big(\linear[\leqN][r][i] + B_{\leq N}^i\Big), 
 \Big(\linear[\leqN][r][j] + B_{\leq N}^j\Big) \Big],
 \Big(\linear[\leqN][l][j] + B_{\leq N,j}\Big) \Big] \\
 &+ \linear[\leqN][r][i] + B_{\leq N}^i. 
 \end{aligned}
 \end{equation}
The evolution equation \eqref{ansatz:eq-Z} cannot be solved using a direct contraction-mapping argument. To see this, we note that the right-hand side of \eqref{ansatz:eq-Z} contains low$\times$high-interactions of the form 
\begin{equation}\label{ansatz:eq-motivation-e1}
\Big[ B_{\leq N} \parall \partial \hspace{0.1ex} \linear[\leqN][r] \Big]. 
\end{equation}
Since $\partial \hspace{0.1ex} \linear[\leqN][r]$ has spatial regularity $(-1)-$, \eqref{ansatz:eq-motivation-e1} has  at most spatial regularity  $(-1)-$, and therefore $Z_{\leq N}$ has at most spatial regularity $1-$. However, the right-hand side of \eqref{ansatz:eq-motivation-e1} also contains high$\times$high-interactions of the form 
\begin{equation}\label{ansatz:eq-motivation-e2}
\Big[ B_{\leq N} \parasim \partial \hspace{0.1ex} \linear[\leqN][r] \Big] 
= \Big[ \quadratic[\leqN][r][] \parasim \partial \hspace{0.1ex} \linear[\leqN][r] \Big]  
+ \Big[ Z_{\leq N} \parasim \partial \hspace{0.1ex} \linear[\leqN][r] \Big]. 
\end{equation}
Since the sum of the regularities of $Z_{\leq N}$ and $\partial \hspace{0.1ex} \linear[\leqN][r]$ is negative, the regularity of $Z_{\leq N}$ alone is insufficient to control the second term in \eqref{ansatz:eq-motivation-e2}. To address this problem, we need to understand the random structure of $Z_{\leq N}$.  We therefore make the para-controlled Ansatz $Z_{\leq N}^j=X_{\leq N}^j + Y_{\leq N}^j$, where  $X_{\leq N}^j$ is the para-controlled component and $Y_{\leq N}^j$  is the smooth nonlinear remainder. They will be chosen as the solutions of the following system of equations: 
 \begin{equation}\label{ansatz:eq-X}
 (\partial_t - \Delta +1 ) X_{\leq N}^i = 
 \Big[ B_{\leq N,j} \parall \Big( 2 \partial^j \, \linear[\leqN][r][i] - \partial^i \, \linear[\leqN][r][j] \Big) \Big] 
 \end{equation}
 and 
 \begin{align}
 (\partial_t - \Delta +1 ) Y_{\leq N}^i 
 &= \Big[  B_{\leq N}^j  
 \parasim 
 \Big( 2 \partial_j \linear[\leqN][r][i] - \partial^i \linear[\leqN][l][j] \Big) \Big]
+ \Big[ \linear[\leqN][r][j] \parasim \Big( 2 \partial_j  B_{\leq N}^i - \partial^i B_{\leq N,j}  \Big) \Big] 
 \allowdisplaybreaks[4] \label{ansatz:eq-Y-1} \\ 
 &+ \Big[ B_{\leq N}^j
 \paragg 
 \Big( 2 \partial_j \linear[\leqN][r][i] - \partial_i \linear[\leqN][r][j] \Big) \Big] + \Big[ \linear[\leqN][r][j] \paransim \Big( 2 \partial_j B_{\leq N}^i - \partial^i B_{\leq N,j} \Big) \Big]  
  \allowdisplaybreaks[4]\label{ansatz:eq-Y-2} \\ 
 &+\Big[ B^j_{\leq N} , 2 \partial_j B^i_{\leq N} - \partial^i B_{\leq N,j} \Big]  \allowdisplaybreaks[4]\label{ansatz:eq-Y-3}  \\ 
 &+\Big[ \Big[ \Big(\linear[\leqN][r][i] + B_{\leq N}^i\Big), 
 \Big(\linear[\leqN][r][j] + B_{\leq N}^j\Big) \Big],
 \Big(\linear[\leqN][l][j] + B_{\leq N,j}\Big) \Big]
  \allowdisplaybreaks[4]\label{ansatz:eq-Y-4} \\
 &+ \linear[\leqN][r][i] + B_{\leq N}^i.  
  \allowdisplaybreaks[4]\label{ansatz:eq-Y-5}
 \end{align}
 
\begin{remark}
In the terms in \eqref{ansatz:eq-Y-2}-\eqref{ansatz:eq-Y-5}, the random structure of $B_{\leq N}$ is irrelevant. The only terms in which the random structure of $B_{\leq N}$ is needed are the two high$\times$high-interactions in \eqref{ansatz:eq-Y-1}. 
\end{remark}

In order to unify the treatment of the terms in \eqref{ansatz:eq-Y-1}, we utilize the product formula. More precisely, we rewrite the second summand in \eqref{ansatz:eq-Y-1} as 
\begin{align}
\Big[ \linear[\leqN][r][j] \parasim \Big( 2 \partial_j  B_{\leq N}^i - \partial^i B_{\leq N,j}  \Big) \Big]  \notag  
=& - \Big[  \Big( 2 \partial_j  B_{\leq N}^i - \partial^i B_{\leq N,j}  \Big) \parasim \linear[\leqN][r][j]  \Big]  \notag  \\
=& -2 \partial_j \Big[ B^i_{\leq N} \parasim   \linear[\leqN][r][j] \Big] + \partial^i \Big[ B_{\leq N,j} \parasim \linear[\leqN][r][j]   \Big] \label{ansatz:eq-product-formula-1} \\ 
&+2  \Big[ B^i_{\leq N} \parasim \partial^j \linear[\leqN][l][j] \Big] - \Big[ B^j_{\leq N} \parasim \partial^i \linear[\leqN][l][j] \Big]. \label{ansatz:eq-product-formula-2}
\end{align}
The two summands in \eqref{ansatz:eq-product-formula-1}, i.e., the terms containing total derivatives, are rather harmless. This is because the sum of the regularities of $B$ and $\linear$ is positive and the Duhamel integral easily compensates for the loss of one derivative. The two summands in \eqref{ansatz:eq-product-formula-2}, i.e., the terms without total derivatives, require a more detailed analysis. When combining them with the first summand from \eqref{ansatz:eq-Y-1}, we obtain 
\begin{align}
&\Big[  B_{\leq N}^j  
 \parasim 
 \Big( 2 \partial_j \linear[\leqN][r][i] - \partial^i \linear[\leqN][l][j] \Big) \Big] + 2  \Big[ B^i_{\leq N} \parasim \partial^j \linear[\leqN][l][j] \Big] - \Big[ B^j_{\leq N} \parasim \partial^i \linear[\leqN][l][j] \Big] \notag \\
 =&
 2 \Big[  B_{\leq N}^j  
 \parasim 
 \Big( \partial_j \linear[\leqN][r][i] -  \partial^i \linear[\leqN][l][j] \Big) \Big] + 2  \Big[ B^i_{\leq N} \parasim \partial^j \linear[\leqN][l][j] \Big]. 
 \label{ansatz:eq-product-formula-3}
\end{align}
By inserting our Ansatz for $B_{\leq N}$, we can further decompose
\begin{align}
&2 \Big[  B_{\leq N}^j  
 \parasim 
 \Big( \partial_j \linear[\leqN][r][i] -  \partial^i \linear[\leqN][l][j] \Big) \Big] + 2  \Big[ B^i_{\leq N} \parasim \partial^j \linear[\leqN][l][j] \Big] \notag \\
 =& 
 2 \Big[  \quadratic[\leqN][r][j] 
 \parasim 
 \Big( \partial_j \linear[\leqN][r][i] -  \partial^i \linear[\leqN][l][j] \Big) \Big] + 2  \Big[\quadratic[\leqN][r][i]  \parasim \partial^j \linear[\leqN][l][j] \Big]
 \label{ansatz:eq-hh-quadratic}  \\ 
 +&2 \Big[  X_{\leq N}^j  
 \parasim 
 \Big( \partial_j \linear[\leqN][r][i] -  \partial^i \linear[\leqN][l][j] \Big) \Big] + 2  \Big[ X^i_{\leq N} \parasim \partial^j \linear[\leqN][l][j] \Big] 
 \label{ansatz:eq-hh-X} \\
 +&2 \Big[  Y_{\leq N}^j  
 \parasim 
 \Big( \partial_j \linear[\leqN][r][i] -  \partial^i \linear[\leqN][l][j] \Big) \Big] + 2  \Big[ Y^i_{\leq N} \parasim \partial^j \linear[\leqN][l][j] \Big]
 \label{ansatz:eq-hh-Y}.
\end{align}
The terms in \eqref{ansatz:eq-hh-quadratic} 
are explicit stochastic objects, which are controlled in Section \ref{section:objects}. The terms in \eqref{ansatz:eq-hh-X} are the most complicated.
 Since $X_{\leq N}$ only has spatial regularity $1-$, they cannot be controlled using a product estimate, and we instead have to utilize the random structure of $X_{\leq N}$. More precisely, we rely on the double Duhamel trick, i.e., we insert the Duhamel integral of the right-hand side of \eqref{ansatz:eq-X} into \eqref{ansatz:eq-hh-X}. This yields 
\begin{align}
&2 \Big[  X_{\leq N}^j  
 \parasim 
 \Big( \partial_j \linear[\leqN][r][i] -  \partial^i \linear[\leqN][l][j] \Big) \Big] + 2  \Big[ X^i_{\leq N} \parasim \partial^j \linear[\leqN][l][j] \Big] 
 \notag \\ 
= &2 \bigg[ \Duh  \bigg( \Big[ B_{\leq N,k}  \parall \Big( 2 \partial^k \linear[\leqN][r][j] - \partial^j \linear[\leqN][r][k]  \Big) \Big] \bigg) \parasim \Big( \partial_j \linear[\leqN][r][i] - \partial^i \linear[\leqN][l][j] \Big) \bigg] \label{ansatz:eq-double-Duhamel-1} \\
 +&2 \bigg[ \Duh  \bigg( \Big[ B_{\leq N,k}  \parall \Big( 2 \partial^k \linear[\leqN][r][i] - \partial^i \linear[\leqN][r][k]  \Big) \Big] \bigg) \parasim  \partial^j \linear[\leqN][l][j]  \bigg]. \label{ansatz:eq-double-Duhamel-2}
\end{align}
Finally, due to the high-regularity of $Y_{\leq N}$, the terms in \eqref{ansatz:eq-hh-Y} can be controlled using paraproduct estimates. \\

In order to state the final form of the evolution equation for $Y_{\leq N}$, we have to discuss one final aspect, which concerns renormalizations. It turns out that \eqref{ansatz:eq-Y-4}, \eqref{ansatz:eq-hh-quadratic},  \eqref{ansatz:eq-double-Duhamel-1}, and \eqref{ansatz:eq-double-Duhamel-2} above all exhibit logarithmic divergences. In order to capture the logarithmic divergences, we make the following definition.

\begin{definition}[Renormalization constant]\label{ansatz:def-renormalization}
For all $N \in \dyadic$, we define
\begin{equation}
\sigma_{\leq N}^2 := \sum_{n\in \Z^2} \frac{\rho_{\leq N}^2(n)}{\langle n\rangle^2} \sim \log(N). 
\end{equation}
\end{definition}
Equipped with $\sigma_{\leq N}^2$, we now subtract divergent counterterms from \eqref{ansatz:eq-Y-4}, \eqref{ansatz:eq-hh-quadratic},  \eqref{ansatz:eq-double-Duhamel-1}, and \eqref{ansatz:eq-double-Duhamel-2}. More precisely, we consider  
\begin{align}
\Big[ \Big[ \Big(\linear[\leqN][r][i] + B_{\leq N}^i\Big), 
 \Big(\linear[\leqN][r][j] + B_{\leq N}^j\Big) \Big],
 \Big(\linear[\leqN][l][j] + B_{\leq N,j}\Big) \Big] 
 &- \sigma_{\leq N}^2 \Kil \big( \linear[\leqN][r][i] +B_{\leq N}^i\big) 
 \label{ansatz:eq-renormalization-1} \\ 
 2 \Big[  \quadratic[\leqN][r][j] 
 \parasim 
 \Big( \partial_j \linear[\leqN][r][i] -  \partial^i \linear[\leqN][l][j] \Big) \Big] + 2  \Big[\quadratic[\leqN][r][i]  \parasim \partial^j \linear[\leqN][l][j] \Big]
 &+ \sigma_{\leq N}^2 \Kil \big( \linear[\leqN][r][i] \big),
 \label{ansatz:eq-renormalization-2} \\ 
2 \bigg[ \Duh  \bigg( \Big[ B_{\leq N,k}  \parall \Big( 2 \partial^k \linear[\leqN][r][j] - \partial^j \linear[\leqN][r][k]  \Big) \Big] \bigg) \parasim \Big( \partial_j \linear[\leqN][r][i] - \partial^i \linear[\leqN][l][j] \Big) \bigg] &+ \frac{3}{2} \sigma_{\leq N}^2 \Kil\big( B_{\leq N}^i \big), 
\label{ansatz:eq-renormalization-3} \\ 
2 \bigg[ \Duh  \bigg( \Big[ B_{\leq N,k}  \parall \Big( 2 \partial^k \linear[\leqN][r][i] - \partial^i \linear[\leqN][r][k]  \Big) \Big] \bigg) \parasim  \partial^j \linear[\leqN][l][j]  \bigg] &- \frac{1}{2} \sigma_{\leq N}^2 \Kil\big( B_{\leq N}^i \big). 
\label{ansatz:eq-renormalization-4}
\end{align}
Despite the fact that the individual terms  \eqref{ansatz:eq-Y-4}, \eqref{ansatz:eq-hh-quadratic},  \eqref{ansatz:eq-double-Duhamel-1}, and \eqref{ansatz:eq-double-Duhamel-2} require a renormalization, the original equation \eqref{ansatz:eq-SYM-leqN-revised} does not need to be renormalized. The reason for this is the combinatorial identity
\begin{equation}
\sigma_{\leq N}^2 \Kil \big( \linear[\leqN][r][i] +B_{\leq N}^i\big) 
- \sigma_{\leq N}^2 \Kil \big( \linear[\leqN][r][i] \big) 
- \frac{3}{2} \sigma_{\leq N}^2 \Kil \big( B_{\leq N}^i\big) 
+ \frac{1}{2} \sigma_{\leq N}^2 \Kil \big( B_{\leq N}^i\big)=0. 
\end{equation}
In other words, the reason is that the divergent counterterms in \eqref{ansatz:eq-renormalization-1}-\eqref{ansatz:eq-renormalization-4} cancel each other. 

\begin{remark}\label{ansatz:rem-cancellation}
This cancellation was already observed in \cite{CCHS22} and only appears in the two-dimensional setting. It boils down to the condition 
\begin{equation}\label{ansatz:eq-cancellation}
- \frac{3(d-1)}{d} + \frac{1}{d} + (d-1)=0, 
\end{equation}
which is only satisfied in dimension $d=2$. For more details regarding \eqref{ansatz:eq-cancellation}, we refer to 
Remark \ref{objects:rem-with-two},
Remark \ref{objects:rem-Wick-d-two-derivatives},
Remark \ref{nonlinear:rem-wick-para}, and 
Remark \ref{nonlinear:rem-Wick-cubic}. 
\end{remark}

For expository purposes, we capture the final form of our evolution equations for $X_{\leq N}$ and $Y_{\leq N}$ in the following definition. 

\begin{definition}[Para-controlled stochastic Yang-Mills heat equation]\label{ansatz:def-paracontrolled}
Let $\tau>0$, let $N\in \dyadic$, and let $A_0\colon \T^d \rightarrow \frkg^2$. Furthermore, let $X_{\leq N},Y_{\leq N}\colon [0,\tau]\times \T^2 \rightarrow \frkg^2$ and define 
\begin{equation*}
B_{\leq N}  :=  \quadratic[\leqN][r] + X_{\leq N} + Y_{\leq N}. 
\end{equation*}
Then, we call $X_{\leq N}$ and $Y_{\leq N}$ a solution of the para-controlled stochastic Yang-Mills heat equation with initial data $A_0$ if the initial conditions
\begin{align}
X_{\leq N}(0)&=0 \qquad  \text{and} \qquad Y_{\leq N}(0)= A_0 - \linear[\leqN][r][] (0) - \quadratic[\leqN][r][](0), \label{ansatz:eq-initial} 
\end{align}
the evolution equation
\begin{align}
 (\partial_t - \Delta +1 ) X_{\leq N}^i &= 
 \Big[ B_{\leq N,j} \parall \Big( 2 \partial^j \, \linear[\leqN][r][i] - \partial^i \, \linear[\leqN][r][j] \Big) \Big], \label{ansatz:eq-new-X}
\end{align}
and the evolution equation
\begin{align}
&\hspace{2ex} (\partial_t - \Delta +1) Y_{\leq N}^i \notag \\ 
 &= 
 2 \Big[  \quadratic[\leqN][r][j] 
 \parasim 
 \Big( \partial_j \linear[\leqN][r][i] -  \partial^i \linear[\leqN][l][j] \Big) \Big] + 2  \Big[\quadratic[\leqN][r][i]  \parasim \partial^j \linear[\leqN][l][j] \Big]
 + \sigma_{\leq N}^2 \Kil \big( \linear[\leqN][r][i] \big)
\allowdisplaybreaks[4] \label{ansatz:eq-Y-new-1}  \\ 
 &+2 \bigg[ \Duh  \bigg( \Big[ B_{\leq N,k}  \parall \Big( 2 \partial^k \linear[\leqN][r][j] - \partial^j \linear[\leqN][r][k]  \Big) \Big] \bigg) \parasim \Big( \partial_j \linear[\leqN][r][i] - \partial^i \linear[\leqN][l][j] \Big) \bigg] + \frac{3}{2} \sigma_{\leq N}^2 \Kil\big( B_{\leq N}^i \big),  
 \allowdisplaybreaks[4] \label{ansatz:eq-Y-new-2} \\
 &+2 \bigg[ \Duh  \bigg( \Big[ B_{\leq N,k}  \parall \Big( 2 \partial^k \linear[\leqN][r][i] - \partial^i \linear[\leqN][r][k]  \Big) \Big] \bigg) \parasim  \partial^j \linear[\leqN][l][j]  \bigg] 
 - \frac{1}{2} \sigma_{\leq N}^2 \Kil\big( B_{\leq N}^i \big) \allowdisplaybreaks[4] \label{ansatz:eq-Y-new-3} \\ 
  &+2 \Big[  Y_{\leq N}^j  
 \parasim 
 \Big( \partial_j \linear[\leqN][r][i] -  \partial^i \linear[\leqN][l][j] \Big) \Big] + 2  \Big[ Y^i_{\leq N} \parasim \partial^j \linear[\leqN][l][j] \Big]
 \allowdisplaybreaks[4] \label{ansatz:eq-Y-new-4} \\ 
 &-2 \partial_j \Big[ B^i_{\leq N} \parasim   \linear[\leqN][r][j] \Big] + \partial^i \Big[ B_{\leq N,j} \parasim \linear[\leqN][r][j]   \Big] \allowdisplaybreaks[4] \label{ansatz:eq-Y-new-5} \\ 
 &+ \Big[ B_{\leq N}^j
 \paragg 
 \Big( 2 \partial_j \linear[\leqN][r][i] - \partial_i \linear[\leqN][r][j] \Big) \Big] + \Big[ \linear[\leqN][r][j] \paransim \Big( 2 \partial_j B_{\leq N}^i - \partial^i B_{\leq N,j} \Big) \Big] \allowdisplaybreaks[4] \label{ansatz:eq-Y-new-6} \\ 
 &+\Big[ B^j_{\leq N} , 2 \partial_j B^i_{\leq N} - \partial^i B_{\leq N,j} \Big]  
 \allowdisplaybreaks[4]\label{ansatz:eq-Y-new-7}  \\ 
 &+\Big[ \Big[ \Big(\linear[\leqN][r][i] + B_{\leq N}^i\Big), 
 \Big(\linear[\leqN][r][j] + B_{\leq N}^j\Big) \Big],
 \Big(\linear[\leqN][l][j] + B_{\leq N,j}\Big) \Big] 
 - \sigma_{\leq N}^2 \Kil \big( \linear[\leqN][r][i] +B_{\leq N}^i\big)  \label{ansatz:eq-Y-new-8} \\ 
 &+ \linear[\leqN][r][i] + B_{\leq N}^i
  \allowdisplaybreaks[4]\label{ansatz:eq-Y-new-9}
 \end{align}
 are satisfied. 
\end{definition}

As a direct consequence of the derivations in this section, we obtain the following proposition.

\begin{proposition}
Let $\tau>0$, let $N\in \dyadic$, and let $A_0\colon \T^d \rightarrow \frkg^2$. If $X_{\leq N},Y_{\leq N}\colon [0,\tau]\times \T^2\rightarrow \frkg^2$ solve the para-controlled stochastic Yang-Mills heat equation (Definition \ref{ansatz:def-paracontrolled}), then
\begin{equation*}
A_{\leq N}:= \linear[\leqN][r][] + \quadratic[\leqN][r][] + X_{\leq N} + Y_{\leq N}
\end{equation*}
solves the stochastic Yang-Mills heat equation \eqref{ansatz:eq-SYM-leqN-revised}. 
\end{proposition}

\section{Vector-valued multiple stochastic integrals}\label{section:multiple-stochastic-integrals}

In this section, we develop a theory of vector-valued multiple stochastic integrals, which is our main technical tool for analyzing various explicit stochastic objects.
It is an extension of the well-known scalar theory of multiple stochastic integrals (see e.g. \cite{M14,Nua06}). We take $d$ to be general in this section. We represent the $\frkg^d$-valued space-time white noise $\xi$ as
\begin{equation}
\xi_j(t,x) = \sum_{n \in \Z^2} \e_n(x) \hat{\xi}_j(t, n) = \sum_{n \in \Z^2} \e_n(x) dW_j(t, n), ~~ j \in [d],
\end{equation}
where $(W_j(\cdot, n), j \in [d], n \in \Z^2)$ are i.i.d. standard $\cfrkg$-valued Brownian motions, modulo the condition (to ensure $\frkg$-valuedness)
$W_j(\cdot, n) = \ovl{W_j(\cdot, -n)}$ for all $n \in \Z^2$.
This representation can be thought of as a $\cfrkg$-valued white noise (modulo the previous antisymmetry condition) on the index set $\indexset := [d] \times \Z^d \times \R$, where $\indexset$ is endowed with the measure $\lebI$, which is obtained by taking the product of counting measure on $[d] \times \Z^d$ with Lebesgue measure on $\R$. 
Recall
that $(E_a)_{a=1}^{\dim \frkg}$ is an orthonormal basis of $\frkg$ (and thus also of $\cfrkg$), which has been fixed in Section \ref{section:lie-group-lie-algebra}.

\begin{definition}[Vector-valued multiple stochastic integral]\label{def:multiple-stochastic-integral}
Let $k,m \geq 1$. Let $f \in L^2(\indexset^m, L(\cfrkg^{\otimes m}, \cfrkg^{\otimes k}))$. 
We then define the multiple stochastic integral by expanding $W = W^a E_a$:
\[\begin{split}
\int_{\indexset^m} f dW^{\otimes m} &:= \int_{\indexset^m} f(E_{a_1} \otimes \cdots \otimes E_{a_m}) dW^{a_1} \cdots dW^{a_m} \\
&:= \big( E_{b_1} \otimes \cdots \otimes E_{b_k} \big) \int_{\indexset^n} f_{a_1 \cdots a_m}^{b_1 \cdots b_k} dW^{a_1} \cdots dW^{a_m}\in \cfrkg^{\otimes k}. 
\end{split}\]
Note that the second line above is a linear combination of elements of $\cfrkg^{\otimes k}$, where the scalars are stochastic integrals whose definition is for instance given in \cite{Nua06,M14}.
\end{definition}

We quickly check that Definition \ref{def:multiple-stochastic-integral} is independent of the choice of basis.

\begin{lemma}
Let $(F_a)_{a \in [\dim \cfrkg]}$ be another basis for $\cfrkg$ and write $W = W_E^a E_a = W_F^a F_a$. For any $k,m \geq 1$, and any $f \in L^2(\indexset^m, L(\cfrkg^{\otimes m}, \cfrkg^{\otimes k}))$,
we have that
\[ (E_{b_1} \otimes \cdots \otimes E_{b_k}) \int_{\indexset^m} f_{a_1 \cdots a_m}^{b_1 \cdots b_k} dW^{a_1}_E \cdots dW^{a_m}_E \stackrel{a.s.}{=} (F_{d_1} \otimes \cdots \otimes F_{d_k})
\int_{\indexset^m} f_{c_1 \cdots c_m}^{d_1 \cdots d_k} dW^{c_1}_F \cdots dW^{c_m}_F.\]
\end{lemma}
\begin{proof}
Let $M$ be the change of basis matrix of $E$ and $F$, i.e. for all $a \in [\dim \cfrkg]$, $E_a = M_{ab} F^b$.
This implies $W = W^a_E E_a = W^a_E M_{ab} F^b$,
which implies that $W^b_F = W^a_E M_{ab}$. We thus have that
\[ (F_{d_1} \otimes \cdots \otimes F_{d_k})
\int_{\indexset^m} f_{c_1 \cdots c_m}^{d_1 \cdots d_k} dW^{c_1}_F \cdots dW^{c_m}_F \stackrel{a.s.}{=} M_{a_1 c_1} \cdots M_{a_m c_m} (F_{d_1} \otimes \cdots \otimes F_{d_k})
\int_{\indexset^m} f_{c_1 \cdots c_m}^{d_1 \cdots d_k} dW^{a_1}_E \cdots dW^{a_m}_E. \]
Again using that $E_a = M_{ab} F^b$, we have that
\[\begin{split}
(E_{b_1} \otimes \cdots \otimes E_{b_k}) f^{b_1 \cdots b_k}_{a_1 \cdots a_m} =  f(E_{a_1} \otimes \cdots \otimes E_{a_m}) &= M_{a_1 c_1} \cdots M_{a_m c_m} f(F_{c_1} \otimes \cdots \otimes F_{c_m}) \\
&= M_{a_1 c_1} \cdots M_{a_m c_m} (F_{d_1} \otimes \cdots \otimes F_{d_k}) f_{c_1 \cdots c_m}^{d_1 \cdots d_k} .
\end{split}\]
The desired result now follows upon combining the previous two displays.
\end{proof}

We now present two examples of how stochastic objects can be represented as multiple stochastic integrals. The two examples (and similar representations) will be used heavily in Section \ref{section:objects}. 

\begin{example}[Linear objects]\label{prelim:example-linear}
For $i \in [d]$, $t \in \R$, $x \in \T^2$, we may express $\linear[N][r][i](t, x)$ as a stochastic integral of the above form as follows. Define the function $f_{t, x, N}^i : \indexset \ra L(\cfrkg, \cfrkg)$ by
\begin{equation}\label{prelim:eq-example-linear}
f_{t, x, N}^i(\dimind, n, s) = \delta^{i}_{\dimind} \ind(s \leq t) \rho_{N}(n) \e_n(x) e^{-(t-s)\fnorm{n}^2} I_{\cfrkg},
\end{equation} 
where $I_{\cfrkg} \in L(\cfrkg, \cfrkg)$ is the identity map. With this definition, we see that
\[ \int_{\indexset} f_{t, x, N}^i dW = \sum_{\dimind \in [d]} \delta^{i}_{\dimind} \sum_{n \in \Z^2} \rho_{N}(n) \e_n(x) \int_{-\infty}^t e^{-(t-s)\fnorm{n}^2} dW(\dimind, n, s) =  \linear[N][r][i](t, x). \]
For $k \in [d]$, we write $\ptl_k f^i_{t, x, N}$ to be the derivative of $f^i_{t, x, N}$ in the $x_k$-variable. Explicitly, we have that
\[ \ptl_k f^i_{t, x, N}(\dimind, n, s) = \icomplex \delta^{i}_{\dimind} \ind(s \leq t) \rho_N(n) n_k \e_n(x) e^{-(t-s)\fnorm{n}^2} I_{\cfrkg}.  \]
Similar as for $\linear[N][r][i]$, we can then write 
\[ \ptl_k \linear[N][r][i](t, x) = \int_{\indexset} \ptl_k f^i_{t, x, N} dW. \]
\end{example}

\begin{example}[Duhamel integral of the linear object]\label{prelim:example-integrated-linear}
In this example, we express the Duhamel integrals $\Duhinf\big(\linear[N][r][i]\big)$ and $\Duhinf\big(\ptl_k \linear[N][r][i]\big)$ as stochastic integrals. Recall that 
\beq\label{eq:example-integrated-linear-intermediate} \Duhinf\big(\linear[N][r][i]\big)(t, x) = \int_{-\infty}^t \big(e^{(t-s)(-1 + \Delta)} \linear[N][r][i](s)\big)(x) ds .\eeq
From Example \ref{prelim:example-linear}, we have that
\[ e^{(t-s)(-1 + \Delta)} \linear[N][r][i](s)  = e^{(t-s)(-1 + \Delta)} \int_{\indexset} f^i_{s, \cdot, N} dW = \int_{\indexset} \big( e^{(t-s)(-1+\Delta)} f^i_{s, \cdot, N}\big) dW.  \]
From \eqref{prelim:eq-example-linear} above, it follows that 
\[ \big(e^{(t-s)(-1 + \Delta)} f^i_{s, \cdot, N}\big)(\dimind, n, u) =  \delta^{i}_{\dimind} \ind(u \leq s) \rho_N(n) \e_n(\cdot) e^{-(t-u)\fnorm{n}^2} I_{\cfrkg}. \]
For brevity, let $g_{t, s,  \cdot, N}^i(\dimind, n, u)$ be defined by the right hand side above. Using \eqref{eq:example-integrated-linear-intermediate}, we have that
\begin{equation}
\label{prelim:eq-example-integrated-linear}
\Duhinf\big(\linear[N][r][i]\big)(t, x)  =  \int_{-\infty}^t \bigg(\int_{\indexset} g^i_{t, s, x, N}(\dimind, n, u) dW(\dimind, n, u) \bigg) ds = \int_{\indexset} \bigg( \int_{-\infty}^t g^i_{t, s, x, N}(\dimind, n, u) ds \bigg) dW(\dimind, n, u). 
\end{equation}
After introducing
\[ F^i_{t,x, N}(\dimind, n, u) := \int_{-\infty}^t g^i_{t, s, x, N}(\dimind, n, u) ds = \delta^{i}_{\dimind} \ind(u \leq t) \rho_N(n) \e_n(x) (t-u)e^{-(t-u)\fnorm{n}^2} I_{\cfrkg},\] 
the identity in \eqref{prelim:eq-example-integrated-linear} can be written as 
\[ \Duhinf\big(\linear[N][r][i]\big)(t, x) = \int_{\indexset} F^i_{t, x, N} dW. \]
Proceeding similarly, we may also obtain
\[ \Duhinf\big(\ptl_k \linear[N][r][i]\big)(t, x) = \int_{\indexset} \ptl_k F^i_{t, x, N} dW, \]
where explicitly
\[ (\ptl_k F^i_{t, x, N})(\dimind, n, u) =  \icomplex \delta^{i}_{\dimind} \ind(u \leq t) \rho_N(n) n_k \e_n(x) (t-u)e^{-(t-u)\fnorm{n}^2} I_{\cfrkg}. \]
\end{example}

We now proceed to state one of the main technical results, which is the (tensor) product formula for vector-valued multiple stochastic integrals. In the following, recall the notion of pairing (Definition \ref{def:pairing}) and the definition of tensor contraction (Definition \ref{def:tensor-contraction}).

\begin{notation}
For $z = (\dimind, n, s)\in \indexset$, let $\bar{z} := (\dimind, -n, s)$.
\end{notation}

\begin{definition}[Paired function]\label{def:contracted-function}
Let $k,m \geq 1$ and let $f : \indexset^m \ra \frkg^{* \otimes m} \otimes \frkg^{\otimes k}$. Let $\mc{P}$ be a pairing of $[m]$. Let $\mc{U} = \mc{U}(\mc{P}) \sse [m]$ be the elements of $[m]$ which are left unpaired by $\mc{P}$. Then, we define the paired function
\beq \paired{\mc{P}}(f)((z_u)_{u \in \mc{U}}) :=  \int_{\indexset^m} \tensorcon{\mc{P}} (f(w_1, \ldots, w_m)) \prod_{u \in \mc{U}} \delta_{w_u z_u} \prod_{\{u, v\} \in \mc{P}} \delta_{w_u \bar{w}_v}  d\lebI^m(w_1, \ldots, w_m). \eeq
\end{definition}
In other words, the above definition integrates $\tensorcon{\mc{P}} f$ over $\indexset^m$, subject to the constraints that $w_u = z_u$ for $u \in \mc{U}$ and $w_u = \bar{w}_v$ for $\{u, v\} \in \mc{P}$.
We remark that if all elements of $[m]$ are paired, i.e., if $\mc{U}$ is empty, then $\paired{\mc{P}}(f)$ is a constant tensor in $\frkg^{* \otimes m} \otimes \frkg^{\otimes k}$. 

\begin{example}
Let $1\leq i,j \leq d$, let $\varphi \colon \Z^d \times \Z^d \rightarrow \C$, let $\psi\colon \R \times \R\rightarrow \C$, and let $\mathcal{T}\in \frkg^{* \otimes 2} \otimes \frkg^{\otimes k}$. Furthermore, let $f\colon \indexset^2 \ra \frkg^{* \otimes 2} \otimes \frkg^{\otimes k}$ be defined as 
\begin{equation*}
f(\ell_1,\ell_2,n_1,n_2,s_1,s_2) 
= \delta^{i}_{\ell_1} \delta^{j}_{\ell_2} \varphi(n_1,n_2) \psi(s_1,s_2) \mathcal{T}. 
\end{equation*}
Then, $\paired{\{1,2\}}(f) $ is a constant tensor given by
\begin{equation*}
    \paired{\{1,2\}}(f) = \delta^{ij} \Big( \sum_{n\in \Z^d} \varphi(n,-n) \Big) 
    \Big( \int_{\R} \psi(s,s) ds \Big) \tensorcon{\{1, 2\}}(\mathcal{T}). 
\end{equation*}
\end{example}

\begin{lemma}[Product formula]\label{lemma:multiple-stochastic-integral-tensor-product}
Let $m_1, m_2, k_1, k_2 \geq 1$, $f \in L^2(\indexset^{m_1}, L(\cfrkg^{\otimes m_1}, \cfrkg^{\otimes k_1}))$, $g \in L^2(\indexset^{m_2}, L(\cfrkg^{\otimes m_2}, \cfrkg^{\otimes k_2}))$. Then
\[ \int_{\indexset^{m_1}} f dW^{\otimes m_1} \otimes \int_{\indexset^{m_2}} g dW^{\otimes m_2} = \sum_{\mc{P} \in \mc{M}(m_1, m_2)} \int_{\indexset^{\mc{U}}} \paired{\mc{P}} (f \otimes g) dW^{\otimes \mc{U}}. \] 
\end{lemma}
\begin{proof}
We reduce to the product formula for scalar integrals. Recall that (Definition \ref{def:multiple-stochastic-integral})
\[ \begin{split}
\int_{\indexset^{m_1}} f dW^{\otimes m_1} &=   \big( E_{c_1} \otimes \cdots \otimes E_{c_{k_1}} \big)  \int_{\indexset^{m_1}}f_{a_1 \cdots a_{m_1}}^{c_1 \cdots c_{k_1}} dW^{a_1} \cdots dW^{a_{m_1}},
\\
\int_{\indexset^{m_2}} g dW^{\otimes m_2} &=  \big( E_{d_1} \otimes \cdots \otimes E_{d_{k_2}} \big) \int_{\indexset^{m_2}} g_{b_1 \cdots b_{m_2}}^{d_1 \cdots d_{k_2}}  dW^{b_1} \cdots dW^{b_{m_2}},
\end{split}\]
where $f_{a_1 \cdots a_{m_1}}^{c_1 \cdots c_{k_1}}$ and $g_{b_1 \cdots b_{m_2}}^{d_1 \cdots d_{k_2}}$ are scalar-valued. For brevity, we now write $f_a^c = f_{a_1 \cdots a_{m_1}}^{c_1 \cdots c_{k_1}}$ and $g_b^d=g_{b_1 \cdots b_{m_2}}^{d_1 \cdots d_{k_2}}$. 
By the product formula for scalar multiple stochastic integrals (see e.g. \cite[Proposition 1.1.3]{Nua06}), we have that 
\[\begin{split}
\bigg( \int_{\indexset^{m_1}} f_a^c dW^{a_1}& \cdots dW^{a_{m_1}} \bigg) \bigg( \int_{\indexset^{m_2}} g_b^d dW^{b_1} \cdots dW^{b_{m_2}} \bigg) = \sum_{\mc{P} \in \mc{M}(m_1, m_2)} \int_{\indexset^{\mc{U}}}  \big(f_a^c \star_\mc{P} g_b^d \big)((z_u)_{u \in \mc{U}}) \prod_{u \in \mc{U}} dW^{e(u)_u} ,
\end{split}\]
where $e(u) := a$ if $u \in [m_1]$, $e(u) := b$ if otherwise $u \in \{m_1 + 1, \ldots, m_1 + m_2\}$, and 
\[
\begin{split}
\big(f_a^c\star_\mc{P}& g_b^d \big)((z_u)_{u \in \mc{U}}) := \\
&\prod_{\{u, v\} \in \mc{P}} \delta^{e(u)_u e(v)_v}
\int_{\indexset^{m_1 + m_2}} f_a^c(w_1, \ldots, w_{m_1}) g_b^d(w_{m_1 + 1}, \ldots w_{m_1 + m_2}) \prod_{u \in \mc{U}} \delta_{z_u w_u} \prod_{\{u, v\} \in \mc{P}} \delta_{w_u \bar{w}_v} d\lebI^{m_1 + m_2}.
\end{split}\]
Combining the previous few displays, we obtain
\beq\label{eq:product-rule-proof-intermediate}\int_{\indexset^{m_1}} f dW^{\otimes m_1} \otimes \int_{\indexset^{m_2}} g dW^{\otimes m_2} = \sum_{\mc{P} \in \mc{M}(m_1, m_2)} \int_{\indexset^{\mc{U}}} (f_a^c \star_\mc{P} g_b^d) ((z_u)_{u \in \mc{U}}) \prod_{u \in \mc{U}} dW^{e(u)_u} \big( E_c \otimes E_d \big), \eeq
where $E_c$ is shorthand for $E_{c_1} \otimes \cdots \otimes E_{c_{k_1}}$ and $E_d$ is shorthand for $E_{d_1} \otimes \cdots \otimes E_{d_{k_2}}$. Observe that
\[\begin{split}
\prod_{\{u, v\} \in \mc{P}} \delta^{e(u)_u e(v)_v} f_a^c(w_1, \ldots, w_{m_1}) g_b^d(w_{m_1 + 1}, &\ldots, w_{m_1 + m_2}) \big( E_c \otimes E_d \big) = \\
&\tensorcon{\mc{P}}\big(f(w_1, \ldots, w_{m_1}) \otimes g(w_{m_1 + 1}, \ldots, w_{m_1 + m_2})\big).
\end{split}\]
Using the definition of $\paired{\mc{P}}(f \otimes g)$ (Definition \ref{def:contracted-function}), we thus further obtain
\[ \eqref{eq:product-rule-proof-intermediate} = \sum_{\mc{P} \in \mc{M}(m_1, m_2)} \int_{\indexset^\mc{U}} \paired{\mc{P}}(f \otimes g) ((z_u)_{u \in \mc{U}}) dW^{\otimes \mc{U}},  \]
as desired.
\end{proof}

The following corollary collects the two main applications of Lemma \ref{lemma:multiple-stochastic-integral-tensor-product}. 

\begin{corollary}[Product formula]\label{cor:multiple-stochastic-integral-tensor-product}
Let $f_1, f_2, f_3 \in L^2(\indexset, L(\cfrkg, \cfrkg))$. We have that 
\[\begin{split} \int_\indexset f_1 dW \otimes \int_\indexset f_2 dW &= \int_{\indexset^2} (f_1 \otimes f_2) dW^{\otimes 2} +\paired{\mc{P}}(f_1 \otimes f_2), \\
\int_\indexset f_1 dW \otimes \int_\indexset f_2 dW \otimes \int_\indexset f_3 dW &= \int_{\indexset^3} (f_1 \otimes f_2 \otimes f_3) ~ dW^{\otimes 3} + \sum_{\substack{\mc{P} \in \{\{1, 2\}, \{1, 3\}, \{2, 3\} \}}}\int_\indexset \paired{\mc{P}}(f_1 \otimes f_2 \otimes f_3) dW. 
\end{split}\]
\end{corollary}

The next lemma shows that the second moment of a multiple stochastic integral is controlled by the $L^2$ norm of the integrand.

\begin{lemma}\label{lemma:multiple-stochastic-integral-second-moment-bound-l2}
Let $k,m \geq 1$, and let $f \in L^2(\indexset^m, L(\cfrkg^{\otimes m}, \cfrkg^{\otimes k}))$. We have that
\[ \E\Bigg[ \bigg| \int_{\indexset^n} f dW^{\otimes m}\bigg|_{\cfrkg^{\otimes k}}^2\Bigg] \leq m! \|f\|_{L^2}^2. \]
\end{lemma}
\begin{proof}
By Lemma \ref{lemma:multiple-stochastic-integral-tensor-product}, we have that
\[ \int_{\indexset^m} f dW^{\otimes m} \otimes \int_{\indexset^m} f dW^{\otimes m} =
\sum_{\mc{P} \in \mc{M}(m, m)} \int_{\indexset^{\mc{U}}} (f \otimes f)_{\mc{P}} dW^{\otimes \mc{U}} =: X + Y, \]
where $X$ collects all pairings $\mc{P}$ which leave some elements unpaired, while $Y$ collects all pairings $\mc{P}$ which have no unpaired elements, i.e., complete pairings. By linearity, we have that 
\[ \E[\langle X \rangle_{\cfrkg^{\otimes k}}] = \langle \E [X] \rangle_{\cfrkg^{\otimes k}} = \langle 0 \rangle_{\cfrkg^{\otimes k}} = 0, \]
where $\langle X \rangle_{\cfrkg^{\otimes k}}$ is as in   \ref{notation:inner-product-as-linear-map}. 
Thus to finish, it remains to show
\begin{equation}\label{prelim:eq-Y-tensor} \langle Y \rangle_{\cfrkg^{\otimes k}} \leq m! \|f\|_{L^2}^2. 
\end{equation}
To avoid confusion, we note that $Y$ is deterministic as a result of the fact that it only consists of complete pairings. 
In order to prove \eqref{prelim:eq-Y-tensor}, we note that complete pairings $\mc{P}$ may be identified with permutations of $[m]$. It thus suffices to show that for any fixed complete pairing $\mc{P}$, we have that
\[ \big|\langle \paired{\mc{P}}(f \otimes f) \rangle_{\cfrkg^{\otimes k}}\big| \leq \|f\|_{L^2}^2. \]
By definition (recall Definition \ref{def:contracted-function}), we have that
\[ \paired{\mc{P}}(f \otimes f) = \int_{\indexset^{2m}} \tensorcon{\mc{P}} (f(w_1, \ldots, w_m) \otimes f(w_{m+1}, \ldots, f_{w_{2m}})) \prod_{\{u, v\} \in \mc{P}} \delta_{w_u \bar{w}_v} d\lebI^{2m}. \]
By Lemma \ref{lemma:contracted-tensor-norm-bound}, we have that
\[ \big|\langle \tensorcon{\mc{P}} (f(w_1, \ldots, w_m) \otimes f(w_{m+1}, \ldots, w_{2m})) \rangle_{\cfrkg^{\otimes k}}\big| \leq |f(w_1, \ldots, w_m)|_{\cfrkg^{* \otimes m} \otimes \cfrkg^{\otimes k}} |f(w_{m+1}, \ldots, w_{2m})|_{\cfrkg^{* \otimes m} \otimes \cfrkg^{\otimes k}}.\]
Applying this bound, we obtain
\[\begin{split}
\big|\langle \paired{\mc{P}}(f \otimes f) \rangle_{\cfrkg^{\otimes k}}\big| &\leq \int_{\indexset^{2n}} |f(w_1, \ldots, w_m)|_{\cfrkg^{* \otimes m} \otimes \cfrkg^{\otimes k}} |f(w_{m+1}, \ldots, w_{2m})|_{\cfrkg^{* \otimes m} \otimes \cfrkg^{\otimes k}}\prod_{\{u, v\} \in \mc{P}} \delta_{w_u \bar{w}_v} d\lebI^{2m} \\
&\leq \bigg(\int_{\indexset^m} |f|^2_{\cfrkg^{* \otimes m} \otimes \cfrkg^{\otimes k}} d\lebI^m\bigg)^{1/2}\bigg(\int_{\indexset^m} |f|^2_{\cfrkg^{* \otimes m}\otimes \cfrkg^{\otimes k}} d\lebI^m\bigg)^{1/2} = \|f\|_{L^2}^2,
\end{split}\]
where we applied Cauchy-Schwarz in the second inequality. The desired result now follows.
\end{proof}

\section{Explicit stochastic objects}\label{section:objects} 
In both our para-controlled Ansatz for the solution of \eqref{ansatz:eq-SYM-leqN} as well as the resulting para-controlled system (Definition \ref{ansatz:def-paracontrolled}), we encountered multiple explicit stochastic objects. The estimates of the explicit stochastic objects are the subject of this section. For the reader's convenience, we already list the stochastic objects and the corresponding regularities in Figure \ref{figure:objects-regularities}. \\ 

From an expository perspective, it is sub-optimal to state the estimates of all stochastic objects from Figure \ref{figure:objects-regularities} separately. In order to capture all of our estimates in Proposition \ref{objects:prop-enhanced} below, it is necessary to introduce the enhanced data set, which collects all explicit stochastic objects. 

\begin{definition}[Enhanced data set]\label{objects:def-enhanced}
For all $N\in \dyadic$, we define the enhanced data set $\Xi_{\leq N}$ by 
\begin{align}
\Xi_{\leq N} := \Bigg(& 
\Big( \linear[\leqN][r][i] \Big)_{i\in [2]}, 
\quad 
\Big( \Big[ \linear[\leqN][r][i], \linear[\leqN][r][j] \Big] \Big)_{i,j \in [2]}, 
\label{objects:eq-enhanced-1} \allowdisplaybreaks[4]\\
&\Big( E \in \frkg \mapsto \Big[ \Big[ E, \linear[\leqN][r][i] \Big], 
\linear[\leqN][r][j] \Big] -  \delta^{ij} \sigma_{\leq N}^2 \Kil\big( E \big) \Big)_{i,j \in [2]}, 
\label{objects:eq-enhanced-2} \allowdisplaybreaks[4]\\ 
&\Big( \Big[ \Big[ \linear[\leqN][r][i], \linear[\leqN][r][j] \Big], \linear[\leqN][r][k] \Big] + \delta_{ik} \sigma_{\leq N}^2 \Kil \big( \linear[\leqN][r][j] \big)  -  \delta_{jk} \sigma_{\leq N}^2 \Kil \big( \linear[\leqN][r][i] \big) \Big)_{i,j,k\in [2]}, 
\label{objects:eq-enhanced-3} \allowdisplaybreaks[4]\\
&\Big( \quadratic[\leqN][r][i] \Big)_{i \in [2]}, 
\label{objects:eq-enhanced-4} \allowdisplaybreaks[4]\\
& \Big( E \in \frkg \mapsto \Big[ \Big[ E , \Duh \big( \partial_{k_1} \linear[\leqN][r][j_1] \big) \Big] \parasim \partial_{k_2} \linear[\leqN][r][j_2] \Big] - \frac{1}{4} \delta^{j_1 j_2} \delta_{k_1 k_2} \sigma_{\leq N}^2 \Kil \big( E \big) \Big)_{j_1,j_2,k_1,k_2 \in [2]}, 
\label{objects:eq-enhanced-5}\allowdisplaybreaks[4]\\
& \Big( 2 \Big[  \quadratic[\leqN][r][j] 
 \parasim 
 \Big( \partial_j \linear[\leqN][r][i] -  \partial^i \linear[\leqN][l][j] \Big) \Big] + 2  \Big[\quadratic[\leqN][r][i]  \parasim \partial^j \linear[\leqN][l][j] \Big]
 + \sigma_{\leq N}^2 \Kil \big( \linear[\leqN][r][i] \big) \Big)_{i \in [2]} 
\label{objects:eq-enhanced-6}
\Bigg). 
\end{align}
\end{definition}

\begin{figure}
\centering
\begin{tabular}{|P{\bigcolwidth}|P{\colwidth}|P{\colwidth}|P{\colwidth}|P{\colwidth}|P{\colwidth}|P{\Hugecolwidth}|P{1.5cm}|}
\hline 
Object 
& $\linear$ 
& \scalebox{0.85}{$\big[ \linear, \linear \big]$}
&  \scalebox{0.85}{$\hspace{-1ex}\big[ \big[ E, \linear \big], \linear \big]$}
& \scalebox{0.85}{$\hspace{-1ex}\big[ \big[ \linear, \linear \big], \linear \big]$}
& \scalebox{0.85}{$\quadratic$}
& \scalebox{0.8}{$\hspace{-1ex}\big[ \big[ E , \Duh (\linear) \big] \parasim \partial \, \linear\big]$}
& \scalebox{0.85}{$\hspace{-1ex}\big[ \quadratic \parasim \partial\, \linear \big]$}
\\ \hline
Regularity 
& $-\kappa$ 
& $-2\kappa$ 
& $-2\kappa$ 
& $-3\kappa$ 
& $1-2\kappa$ 
& $-2\kappa$ 
& $-3\kappa$ \\ \hline
\end{tabular}
\centering
    \caption{In this figure, we display the objects in the  enhanced data set (from Definition \ref{objects:def-enhanced}) and the corresponding regularities.}
    \label{figure:objects-regularities}
\end{figure}

 In order to state our estimates of the enhanced data set $\Xi_{\leq N}$,  we need to introduce a corresponding metric space. This metric space is the so-called data space, which is the subject of the next definition.

\begin{definition}[Data space]\label{objects:def-data-space}
For all $T>0$, we define the set 
\begin{align*}
\Dc([0,T]) :=& \, 
C_t^0 \Cs_x^{-\kappa} \big( [0,T] \times \T^2 \rightarrow \frkg\big)^2
\times C_t^0 \Cs_x^{-2\kappa} \big( [0,T] \times \T^2 \rightarrow \frkg\big)^{2 \times 2}  \\
&\times C_t^0 \Cs_x^{-2\kappa} \big( [0,T] \times \T^2 \rightarrow \End(\frkg)\big)^{2 \times 2}  \\
&\times C_t^0 \Cs_x^{-3\kappa} \big( [0,T] \times \T^2 \rightarrow \frkg\big)^{2 \times 2 \times 2}  \\
&\times \Big( C_t^0 \Cs_x^{1-2\kappa} \big( [0,T] \times \T^2 \rightarrow \frkg \big)^{2} \medcap \Sc^{1-2\kappa}([0,T] \rightarrow \frkg)^2 \Big)\\
&\times \Wc^{-2\kappa,\kappa}\big( [0,T]  \rightarrow \End(\frkg)\big)^{2 \times 2 \times 2 \times 2} \\
& \times C_t^0 \Cs_x^{-3\kappa} \big( [0,T] \times \T^2 \rightarrow \frkg\big)^{2}.
\end{align*}
We further define the collection of degrees  
$\dg \in \mathbb{N}^6$ and a collection of regularities $\reg \in \mathbb{N}^6$ by
\begin{equation}
\dg := \big( 1,2,2,3,2,2,3\big) \qquad \text{and} \qquad 
\reg:= \big( -\kappa,-2\kappa,-2\kappa,-3\kappa,1-2\kappa,-2\kappa,-3\kappa\big). 
\end{equation}
We then define a metric $\dc_T$ on  $\Dc([0,T])$ by 
\begin{equation}\label{objects:eq-metric}
\begin{aligned}
\dc_T \Big( (S_j)_{j=1}^7, (\widetilde{S}_j)_{j=1}^7\Big)
:= \sum_{j=1}^7 \big( \big\| S_j - \widetilde{S}_j \big\|_{C_t^0 \Cs_x^{\reg_j}}\big)^{1/\dg_j} + \Big\| S_5 - \widetilde{S}_5 \Big\|_{\Sc^{1-2\kappa}}^{\frac{1}{2}} . 
\end{aligned}
\end{equation}
For all $R\geq  1$, we also define 
\begin{equation}\label{eq:Dc-R}
\Dc_R([0,T]):= \Big\{ S \in \Dc([0,T])\colon \dc_T (S,0) \leq R \Big\}. 
\end{equation}
\end{definition}

\begin{remark} The purpose of the degrees in the definition of $\dc_T$  is that the condition $\Xi_{\leq N} \in \Dc_R([0,T])$ should yield bounds on the stochastic objects which exhibit the natural dependence on their degrees. 

The $\Sc^{1-2\kappa}$-norm is as in Definition \ref{prelim:def-solution-space} and only included for the quadratic object $\scalebox{0.9}{$\quadratic[\leqN][r]$}$. The reason is that we need an $\Sc^{1-2\kappa}$-bound for $B_{\leq N}$ (see e.g. Lemma \ref{nonlinear:lem-dd}), and, as is clear from \eqref{ansatz:eq-BN}, we therefore also need an $\Sc^{1-2\kappa}$-bound for $\scalebox{0.9}{$\quadratic[\leqN][r]$}$. 
\end{remark}

Equipped with both the enhanced data set (Definition \ref{objects:def-enhanced}) and the data space (Definition \ref{objects:def-data-space}), we can now state the main result of this section. 

\begin{proposition}[Control of enhanced data set]\label{objects:prop-enhanced}
Let $0<c\ll 1$ be a sufficiently small absolute constant and let $R \geq 1$. Then, there exists an event $E_R$ satisfying
\begin{equation*}
\mathbb{P} \big( E_R \big) \geq 1 - c^{-1} \exp \Big( - c R^2\Big) 
\end{equation*}
and such that, on this event,
\begin{align}\label{objects:eq-enhanced-main-estimate}
\sup_{N \in \dyadic} \dc_1 \big( \Xi_{\leq N}, 0 \big) \leq R \qquad \text{and} \qquad 
\lim_{M,N\rightarrow \infty} 
\dc_1 \big( \Xi_{\leq M}, \Xi_{\leq N} \big) =0. 
\end{align}
\end{proposition}

The proof of Proposition \ref{objects:prop-enhanced} occupies the rest of this section and is split over five subsections. In Subsection \ref{section:objects-linear}, we control the linear stochastic object $\linear[][r][]$. In Subsection \ref{section:objects-power}, we control power-type nonlinearities in $\linear[][r][]$, i.e., the quadratic and cubic terms in \eqref{objects:eq-enhanced-1}, \eqref{objects:eq-enhanced-2}, and \eqref{objects:eq-enhanced-3}. 
In Subsection \ref{section:objects-quadratic}, we control the quadratic stochastic object $\quadratic[][r][]$. In Subsection \ref{section:objects-derivative}, we control derivative-nonlinearities, i.e., the stochastic objects from \eqref{objects:eq-enhanced-5} and \eqref{objects:eq-enhanced-6}. Finally, in Subsection \ref{section:objects-proof}, we collect all of our estimates and prove Proposition \ref{objects:prop-enhanced}. \\

In all of this section, we will follow the following notational convention. 

\begin{notation}
We denote generic integration variables in $\indexset$ by $z = (\dimind, n, s)$. When working with more than one generic integration variable, we denote them by $z_1 = (\dimind_1, n_1, s_1)$, $z_2 = (\dimind_2, n_2, s_2)$, and so on. 

In all following lemmas, we use $\sreg>0$ as a regularity parameter and $0<\sdecay<\sreg$ as a parameter to measure decay in the highest frequency scale. In the proof of Proposition \ref{objects:prop-enhanced}, we will eventually choose $(\sreg,\sdecay)=(\kappa,\eta)$. However, in the proof of Proposition \ref{gauged:prop-enhanced}, we need the freedom to make a different choice for $(\sreg,\sdecay)$, which is our reason for not just working with $(\kappa,\eta)$. 
\end{notation}

\subsection{Linear stochastic object}\label{section:objects-linear}

In this very short subsection, we estimate the linear stochastic object $\linear$. The estimate does not require the product formula for multiple stochastic integrals (Lemma \ref{lemma:multiple-stochastic-integral-tensor-product}), and is therefore much simpler than all other estimates in this section. 

\begin{lemma}[Linear object]\label{lemma:linear-object} Let $0<\sdecay <\epsilon\ll 1$, 
let $i \in [2]$, let $N\in \dyadic$, and let $p \geq 1$. Then, we have that
\[ \E\Big[ \big\|\linear[N][r][i]\big\|_{C_t^0 \Cs_x^{-\sreg}([0,1]\times \T^2)}^p\Big]^{1/p}  \lesssim p^{1/2} N^{-\sdecay}. \]
\end{lemma}

\begin{proof}
Recall from Example \ref{prelim:example-linear} that
\[ \linear[N][r][i](t, x) = \int_{\indexset} f^i_{t, x, N} dW,  \]
where 
\[ f^i_{t, x, N}(z) = \delta^{i}_{\dimind} \ind(s \leq t) \rho_N(n) \e_n(x) e^{-(t-s)\fnorm{n}^2} I_{\cfrkg} .\]
Due to standard reductions (Lemma \ref{lemma:standard-reduction}), it suffices to prove that
\[  \sup_{t\in [0,1]} \sup_{x\in \T^2}\E\Big[\big|\linear[N][r][i](t, x) \big|_{\cfrkg}^2\Big] \lesssim 1. \]
Using Lemma \ref{lemma:multiple-stochastic-integral-second-moment-bound-l2},  we have that
\begin{equation*}
 \E\Big[\big|\linear[N][r][i](t, x)  \big|_{\cfrkg}^2\Big] \leq \|f^i_{t, x, N}\|^2_{L^2} \lesssim \sum_{n\in \Z^2} \rho_N(n) \int_{-\infty}^t e^{-2(t-s)\fnorm{n}^2} ds \leq \sum_n \frac{\rho_N(n)}{\fnorm{n}^2} \lesssim 1. \qedhere
 \end{equation*}
\end{proof}

\subsection{Power-type nonlinearities}
\label{section:objects-power}

In this subsection, we estimate the power-type nonlinearities in the enhanced data set, i..e, the quadratic terms in \eqref{objects:eq-enhanced-1} and \eqref{objects:eq-enhanced-2} and the cubic terms in \eqref{objects:eq-enhanced-3}. In order to state the estimates of this subsection, we need to introduce a dyadic decomposition of the divergent counterterm from Definition \ref{ansatz:def-renormalization}. 

\begin{definition}[Divergent counterterms]\label{objects:def-counterterms}
For any $N_1,N_2\in \dyadic$, we define
\begin{equation}
 \sigma_{N_1,N_2}^2 := \sum_{n\in \Z^2} \frac{\rho_{N_1}(n) \rho_{N_2}(n)}{2 \langle n \rangle^2}. 
\end{equation}
\end{definition}
As a direct consequence of Definition \ref{ansatz:def-renormalization} and Definition \ref{objects:def-counterterms}, it follows that 
\begin{equation}
\sigma_{\leq N}^2 = \sum_{\substack{N_1,N_2\colon \\ N_1,N_2\leq N}} 
\sigma_{N_1,N_2}^2. 
\end{equation}

\begin{lemma}[Quadratic power-type nonlinearities]\label{objects:lem-without} Let $0<\sdecay <\epsilon\ll 1$, 
let $1\leq j_1,j_2\leq 2$ and let  $N_1,N_2\in \dyadic$. Furthermore, let $E\in \frkg$ satisfy $\| E \|_{\frkg}\leq 1$
and let $p\geq 1$. Then, we have the following estimates:
\begin{align}
\E \bigg[ \Big\| \Big[  \linear[N_1][r][j_1] , \linear[N_2][r][j_2] \Big] \Big\|_{C_t^0 \Cs_x^{-2\sreg}([0,1]\times \T^2)}^p \bigg]^{1/p}
&\lesssim p N_{\textup{max}}^{-2\sdecay}, \label{objects:eq-without-e1} \\ 
\E \bigg[ \Big\| \Big[ \Big[ E, \linear[N_1][r][j_1] \Big],  \linear[N_2][r][j_2] \Big]
- \delta^{j_1 j_2} \sigma_{N_1,N_2}^2 \Kil(E) \Big\|_{C_t^0 \Cs_x^{-2\sreg}([0,1]\times \T^2)}^p \bigg]^{1/p}
&\lesssim p N_{\textup{max}}^{-2\sdecay}.\label{objects:eq-without-e2}
\end{align}
\end{lemma}

Before we dive into the proof of Lemma \ref{objects:lem-without}, we first derive two trace identities. The trace identities will then be used to analyze the resonant products in the quadratic objects in \eqref{objects:eq-without-e1} and \eqref{objects:eq-without-e2}. 

\begin{lemma}[Trace identities]\label{objects:lem-trace-identities}
Let $E\in \cfrkg$. Then, it holds that 
\begin{align}
\Big[ \tensorcon{\{1,2\}} \big( I_{\cfrkg} \otimes I_{\cfrkg} \big) \Big]_{\cfrkg} &= 0, \label{objects:eq-trace-e1} \\ 
\Big[ \tensorcon{\{1,2\}} \big( \mrm{ad}(E) \otimes I_{\cfrkg} \big) \Big]_{\cfrkg} &= \Kil (E). \label{objects:eq-trace-e2} 
\end{align}
\end{lemma}

\begin{proof}[Proof of Lemma \ref{objects:lem-trace-identities}] By the definition of $\tensorcon{\{1,2\}}$ (or recall Example \ref{example:1-2-tensor-contraction}), it follows that 
\begin{equation*}
\tensorcon{\{1,2\}}\big( I_{\cfrkg} \otimes I_{\cfrkg} \big) = \delta^{ab}\, E_a \otimes E_b \quad \text{and}  \quad
\tensorcon{\{1,2\}}\big( \mrm{ad}(E) \otimes I_{\cfrkg} \big) 
= \delta^{ab} \,  (\mrm{ad}(E)E_a) \otimes E_b. 
\end{equation*}
Using the skew-symmetry of the Lie-bracket and the definition of $\Kil$ (Definition \ref{def:Kil}), we obtain the two identities 
\begin{align*}
\Big[ \tensorcon{\{1,2\}} \big( I_{\cfrkg} \otimes I_{\cfrkg} \big) \Big]_{\cfrkg}  &= \delta^{ab} \big[ E_a \otimes E_b \big]_{\cfrkg}  = 0, \\ 
\Big[ \tensorcon{\{1,2\}} \big( \mrm{ad}(E) \otimes I_{\cfrkg} \big) \Big]_{\cfrkg} 
&= \delta^{ab} \big[ (\mrm{ad}(E)E_a) \otimes E_b \big]_{\cfrkg} = \Kil(E). \qedhere
\end{align*}
\end{proof}

Equipped with Lemma \ref{objects:lem-trace-identities}, we now estimate the quadratic power-type nonlinearities.

\begin{proof}[Proof of Lemma \ref{objects:lem-without}]
In order to present a unified treatment of \eqref{objects:eq-without-e1} and \eqref{objects:eq-without-e2}, we primarily analyze the tensor product $\linear[N_1][r][j_1]\otimes \linear[N_2][r][j_2]$. We recall from Example \ref{prelim:example-linear} that, for all $j\in [2]$ and $M\in \dyadic$,  
\begin{align}
\linear[M][r][j](t,x) =& \,  \int_{\indexset} f^j_{t,x,M} \, dW, \label{objects:eq-without-p2} \\
\textup{where} \quad f^j_{t,x,M}(\dimind, m, s):=& \, 
\delta^j_\dimind \ind_{(-\infty,t]}(s) \rho_M(m) \e_{m}(x) e^{-(t-s) \langle m \rangle^2} I_{\cfrkg}. \label{objects:eq-without-p3}
\end{align}
Using \eqref{objects:eq-without-p2} and the product formula (Corollary \ref{cor:multiple-stochastic-integral-tensor-product}), it follows that 
\begin{align}
\linear[N_1][r][j_1]\otimes \linear[N_2][r][j_2] 
&= \int_{\indexset^2} \big( f^{j_1}_{t,x,N_1} \otimes f^{j_2}_{t,x,N_2} \big) d(W\otimes W) 
\label{objects:eq-without-p5} \\
&+  \paired{\{1,2\}} \big(  f^{j_1}_{t,x,N_1} \otimes f^{j_2}_{t,x,N_2} \big) \label{objects:eq-without-p6}. 
\end{align}
We now split the remainder of this proof into three steps.\\

\emph{Step 1: Estimate of the non-resonant part \eqref{objects:eq-without-p5}.} 
For the non-resonant part \eqref{objects:eq-without-p5}, we prove the estimate 
\begin{equation}\label{objects:eq-without-p7}
\E \bigg[ \Big\| \int_{\indexset^2} \big(  f^{j_1}_{t,x,N_1} \otimes f^{j_2}_{t,x,N_2} \big) d(W\otimes W) \Big\|_{C_t^0\Cs_x^{-2\sreg}([0,1]\times \T^2)}^p \bigg]^{1/p} \lesssim p N_{\textup{max}}^{-2\sdecay}. 
\end{equation}
Using standard reductions (Lemma \ref{lemma:standard-reduction}), it suffices to prove that 
\begin{equation}\label{objects:eq-without-p8}
\sup_{t\in [0,1]} \sup_{x\in \T^2} 
\E \Big[ \Big| P_{N_0} \int_{\indexset^2}  \big(  f^{j_1}_{t,x,N_1} \otimes f^{j_2}_{t,x,N_2} \big) d(W\otimes W) \Big|_{\cfrkg}^2 \Big] \lesssim N_0^{4\sreg} N_{\textup{max}}^{-4\sreg}. 
\end{equation}
To this end, we first use Lemma \ref{lemma:multiple-stochastic-integral-second-moment-bound-l2} and \eqref{objects:eq-without-p3}, which yield that 
\begin{align}
&\E \Big[ \Big| P_{N_0} \int_{\indexset^2}  \big(  f^{j_1}_{t,x,N_1} \otimes f^{j_2}_{t,x,N_2} \big) d(W\otimes W) \Big|_{\cfrkg}^2 \Big]  \notag \\ 
\lesssim \, &  \sum_{\substack{n_0,n_1,n_2 \in \Z^2 \colon \\ n_0 = n_1 + n_2}} 
\rho_{N_0}(n_0) \rho_{N_1}(n_1) \rho_{N_2}(n_2)  \int_{(-\infty,t]^2} e^{-2(t-s_1) \langle n_1 \rangle^2} e^{-2(t-s_2) \langle n_2 \rangle^2} \ds_1 \ds_2.
\label{objects:eq-without-p9}
\end{align}
By first calculating the $s_1$ and $s_2$-integrals and then  using Lemma \ref{prelim:lem-counting}, it follows that
\begin{align*}
\eqref{objects:eq-without-p9}
&\lesssim  \sum_{\substack{n_0,n_1,n_2 \in \Z^2 \colon \\ n_0 = n_1 + n_2}} 
\rho_{N_0}(n_0) \rho_{N_1}(n_1) \rho_{N_2}(n_2)  
\langle n_1 \rangle^{-2} \langle n_2 \rangle^{-2} 
\lesssim (N_1 N_2)^{-2} \times (N_0^{4\sreg} N_1^2 N_2^2 N_{\max}^{-4\sreg} ) \lesssim N_0^{4\sreg} N_{\textup{max}}^{-4\sreg}. 
\end{align*}
This yields \eqref{objects:eq-without-p8} and hence completes the proof of \eqref{objects:eq-without-p7}. \\ 

\emph{Step 2: Analysis of the resonant part \eqref{objects:eq-without-p6}.} 
In this step, we prove the algebraic identity
\begin{equation}\label{objects:eq-without-p12}
  \paired{\{1,2\}} \big(  f^{j_1}_{t,x,N_1} \otimes f^{j_2}_{t,x,N_2} \big)   = \delta^{j_1 j_2} \sigma_{N_1,N_2}^2 \tensorcon{\{1,2\}} \big( I_{\cfrkg} \otimes I_{\cfrkg} \big). 
\end{equation}
This identity can be obtained from the definition  of $\paired{\{1,2\}}$ and a direct computation. To be more precise, we first use the definition of $\paired{\{1,2\}}$, which yields that
\begin{align*}
&\paired{\{1,2\}} \big(  f^{j_1}_{t,x,N_1} \otimes f^{j_2}_{t,x,N_2} \big)  \\
=& \, \sum_{n_1,n_2 \in \Z^2} \int_{\R^2} \delta^{\dimind_1 \dimind_2} \delta_{n_1+n_2=0} \delta(s_1-s_2) \Big( \prod_{k=1}^2 \delta^{j_k}_{\dimind_k} \rho_{N_k}(n_k) \ind_{(-\infty,t]}(s) \e_{n_k}(x) e^{-(t-s_k) \langle n_k \rangle^2} \Big) \, \ds_1 \ds_2 \\
&\, \times \tensorcon{\{1,2\}} \big( I_{\cfrkg} \otimes I_{\cfrkg} \big). 
\end{align*}
By first calculating the sums and integrals in $(\dimind_2,n_2, s_2)$, then in $(\dimind_1,s_1)$, and finally in $n_1$, it follows that 
\begin{align*}
    & \sum_{n_1,n_2 \in \Z^2} \int_{\R^2} \delta^{\dimind_1 \dimind_2} \delta_{n_1+n_2=0} \delta(s_1-s_2) \Big( \prod_{k=1}^2 \delta^{j_k}_{\dimind_k} \rho_{N_k}(n_k) \ind_{(-\infty,t]}(s) \e_{n_k}(x) e^{-(t-s_k) \langle n_k \rangle^2} \Big) \, \ds_1 \ds_2 \\
    =&\,  \delta^{\dimind_1 j_2} \delta^{j_1}_{\dimind_1} \sum_{n_1\in \Z^2} \rho_{N_1}(n_1) \rho_{N_2}(n_1)  \int_{-\infty}^t e^{-2(t-s_1) \langle n_1 \rangle^2} \, \ds_1  \\
    =& \, \delta^{j_1 j_2} \sum_{n_1 \in \Z^2} \frac{\rho_{N_1}(n_1) \rho_{N_2}(n_1)}{2 \langle n_1 \rangle^2} 
    \, = \delta^{j_1 j_2} \sigma_{N_1,N_2}^2. 
\end{align*}
This yields the desired algebraic identity \eqref{objects:eq-without-p12}.\\

\emph{Step 3: Proof of \eqref{objects:eq-without-e1} and \eqref{objects:eq-without-e2}.} Using the product formula \eqref{objects:eq-without-p5}-\eqref{objects:eq-without-p6} and the algebraic identity \eqref{objects:eq-without-p12}, it follows that
\begin{align}
\Big[  \linear[N_1][r][j_1] , \linear[N_2][r][j_2] \Big]
 =   \Big[  \linear[N_1][r][j_1] \otimes \linear[N_2][r][j_2] \Big]_{\cfrkg} 
 &= \bigg[ \int_{\indexset^2} \big(  f^{j_1}_{t,x,N_1} \otimes f^{j_2}_{t,x,N_2} \big) d(W\otimes W) \bigg]_{{\cfrkg}}  \label{objects:eq-without-p13} \\ 
 &+ \delta^{j_1 j_2} \sigma_{N_1,N_2}^2 \Big[ \tensorcon{\{1,2\}} \big( I_{\cfrkg} \otimes I_{\cfrkg} \big) \Big]_{\cfrkg}.  \label{objects:eq-without-p14}
\end{align}
Due to our estimate \eqref{objects:eq-without-p7}, the non-resonant part \eqref{objects:eq-without-p13} yields an acceptable contribution to \eqref{objects:eq-without-e1}. Using our trace identities (Lemma \ref{objects:lem-trace-identities}), the resonant part \eqref{objects:eq-without-p14} is identically zero. \\

Similarly, using the  product formula  \eqref{objects:eq-without-p5}-\eqref{objects:eq-without-p6} and the algebraic identity \eqref{objects:eq-without-p12}, it follows that
\begin{align}
\Big[ \Big[ E, \linear[N_1][r][j_1] \Big],  \linear[N_2][r][j_2] \Big]
- \delta^{j_1 j_2} \sigma_{N_1,N_2}^2 \Kil(E) 
&= \Big[ \mrm{ad}(E) \,  \linear[N_1][r][j_1] \otimes  \linear[N_2][r][j_2] \Big]_{\cfrkg}
- \delta^{j_1 j_2} \sigma_{N_1,N_2}^2 \Kil(E) \notag \\
&= \bigg[ \int_{\indexset^2} \big( \mrm{ad}(E) f^{j_1}_{t,x,N_1} \otimes f^{j_2}_{t,x,N_2} \big) d(W\otimes W) \bigg]_{{\cfrkg}} \label{objects:eq-without-p15} \\
&+ \delta^{j_1j_2} \sigma_{N_1,N_2}^2 \bigg( \Big[ \tensorcon{\{1,2\}} \big( \mrm{ad}(E) \otimes I_{\cfrkg} \big) \Big]_{\cfrkg} - \Kil(E) \bigg). \label{objects:eq-without-p16}
\end{align}
Due to our estimate \eqref{objects:eq-without-p7}, the non-resonant part \eqref{objects:eq-without-p15} yields an acceptable contribution to \eqref{objects:eq-without-e2}. Using our trace identities (Lemma \ref{objects:lem-trace-identities}), the resonant part \eqref{objects:eq-without-p16} is identically zero. 
\end{proof}

We now present a corollary of the proof of Lemma \ref{objects:lem-without}. While this corollary will not be needed in the proof of Proposition \ref{objects:prop-enhanced}, i.e., in the bound of the enhanced data set, it will be used in Section \ref{section:gauge-covariance} in the proof of gauge-covariance. 

\begin{corollary}\label{objects:cor-quadratic-Q} Let $0<\sdecay <\epsilon\ll 1$, 
let $i,j,k,l\in [2]$, let $N\in \dyadic$, let $Q^\ell_{>N}$ be as in \eqref{prelim:eq-Q}, and let $\theta_{\leq N}$ be defined as 
\begin{equation}\label{objects:eq-theta}
\theta_{\leq N} := \frac{1}{4} \sum_{n\in \Z^2} \frac{\langle n, \nabla \rho_{>N}(n) \rangle}{\langle n \rangle^2} \rho_{\leq N}(n). 
\end{equation}
Then, it holds for all $E\in \frkg$ satisfying $\| E \|_{\frkg}\leq 1$ and all $p\geq 1$ that
\begin{equation}\label{objects:eq-quadratic-Q}
\E \bigg[ \Big\| \Big[ \linear[\leqN][r][i] \parasim \Big[ E , \partial^k Q^\ell_{>N} \linear[][r][j] \Big] \Big] 
+ \delta^{ij} \delta^{k\ell} \theta_{\leq N} \Kil (E) \Big\|_{C_t^0 \Cs_x^{-2\sreg}([0,1]\times \T^2)}^p \bigg]^{1/p} \lesssim p N^{-2\sdecay}. 
\end{equation}
\end{corollary}

\begin{proof}
Since $\partial^k Q^\ell_{>N}$ is a zeroth-order Fourier multiplier, the argument is similar to the proof of Lemma \ref{objects:lem-without} and we therefore only sketch it. We first let $\sreg^\prime>0$ be sufficiently small depending on $\sreg$ and $\sdecay$. Using paraproduct estimates, H\"{o}lder's inequality, and Lemma \ref{lemma:linear-object}, it holds that
\begin{align*}
&\E \bigg[ \Big\| \Big[ \linear[\leqN][r][i] \paransim \Big[ E , \partial^k Q^\ell_{>N} \linear[][r][j] \Big] \Big] \Big\|_{C_t^0 \Cs_x^{-2\sreg}([0,1]\times \T^2)}^p \bigg]^{1/p} \\ 
\lesssim\,& \E \Big[ \big\| \linear[\leqN][r][i] \big\|_{C_t^0 \Cs_x^{-\sreg^\prime}([0,1]\times \T^2)}^{2p}\Big]^{1/(2p)} 
\E \Big[ \big\| \partial^k Q^\ell_{>N} \linear[][r][j] \big\|_{C_t^0 \Cs_x^{-2(\sreg-\sreg^\prime)}([0,1]\times \T^2)}^{2p}\Big]^{1/(2p)}
\lesssim p N^{-2\sdecay}. 
\end{align*}
It therefore suffices to prove \eqref{objects:eq-quadratic-Q} with the product rather than the high$\times$high-paraproduct. To this end, we first examine the tensor product  $\linear[\leqN][r][i] \otimes \partial^k Q^\ell_{>N} \linear[][r][j]$. Similar as in the proof of Lemma \ref{objects:lem-without}, the non-resonant component of the tensor product can be controlled in $C_t^0 \Cs_x^{-2\sreg}$. Furthermore, using similar computations as in the proof of Lemma \ref{objects:lem-without}, the resonant component of the tensor product can be written as 
\begin{equation}\label{objects:eq-quadratic-Q-p1}
\begin{aligned}
&\delta^{ij} \bigg( \sum_{n\in \Z^2} \int_{-\infty}^t \ds \rho_{\leq N}(n) n^k (\partial^\ell \rho_{>N})(n) e^{-2(t-s) \langle n \rangle^2} \bigg) \tensorcon{\{1,2\}} \big( I_{\cfrkg} \otimes I_{\cfrkg} \big) \\
=& \, \delta^{ij} \frac{1}{2} \bigg( \sum_{n\in \Z^2} 
\rho_{\leq N}(n) \frac{n^k (\partial^\ell \rho_{>N})(n)}{\langle n \rangle^2} \bigg) 
\tensorcon{\{1,2\}} \big( I_{\cfrkg} \otimes I_{\cfrkg} \big). 
\end{aligned}
\end{equation}
Since $\partial^\ell \rho_{>N}$ is odd in the $\ell$-th component of $n$, \eqref{objects:eq-quadratic-Q-p1} is only non-zero in the case $k=\ell$. Due to symmetry in the coordinates, it then follows that
\begin{equation}
\begin{aligned}
\eqref{objects:eq-quadratic-Q-p1} 
&= \delta^{ij} \delta^{k\ell} \frac{1}{4} 
\bigg( \sum_{n\in \Z^2} 
\rho_{\leq N}(n) 
\frac{n^1 (\partial_1 \rho_{>N})(n)+n^2 (\partial_2 \rho_{>N})(n)}
{\langle n \rangle^2} \bigg)\tensorcon{\{1,2\}} \big( I_{\cfrkg} \otimes I_{\cfrkg} \big) \\ 
&=  \delta^{ij} \delta^{k\ell}  \theta_{\leq N} \tensorcon{\{1,2\}} \big( I_{\cfrkg} \otimes I_{\cfrkg} \big). 
\end{aligned}
\end{equation}
Using Lemma \ref{objects:lem-trace-identities}, it follows that the resonant part of $\big[ \linear[\leqN][r][i] \parasim \big[ E , \partial^k Q^\ell_{>N} \linear[][r][j] \big] \big]$ is given by
\begin{equation*}
\delta^{ij} \delta^{k\ell}  \theta_{\leq N} 
\Big[ \Big( I_{\cfrkg} \otimes \mrm{ad}(E) \Big) \tensorcon{\{1,2\}} \big( I_{\cfrkg} \otimes I_{\cfrkg} \big) \Big]_{\cfrkg} \hspace{-0.5ex}
= -\delta^{ij} \delta^{k\ell}  \theta_{\leq N}  
\Big[ \tensorcon{\{1,2\}} \big( \mrm{ad}(E) \otimes I_{\cfrkg} \big) \Big]_{\cfrkg} \hspace{-0.5ex}=  -\delta^{ij} \delta^{k\ell}  \theta_{\leq N}   \Kil(E). 
\end{equation*}
Thus, we obtain that the renormalization in \eqref{objects:eq-quadratic-Q} cancels the resonant part, which completes the proof.  
\end{proof}

In Lemma \ref{objects:lem-without}, we estimated the quadratic power-type nonlinearities from \eqref{objects:eq-enhanced-1} and \eqref{objects:eq-enhanced-2}. Now, we turn to the cubic power-type nonlinearity from \eqref{objects:eq-enhanced-3}.

\begin{lemma}[Cubic power-type nonlinearity]\label{lemma:cubic-stochastic-object} Let $0<\sdecay <\epsilon\ll 1$, 
let $i_1, i_2, i_3 \in [2]$ and let $N_1, N_2, N_3\in \dyadic$. For all $p\geq 1$, it then holds that 
\[ \E \bigg[\Big\| \Big[\Big[\linear[N_1][r][i_1], \linear[N_2][r][i_2]\Big], \linear[N_3][r][i_3]\Big](t) +  \delta^{i_1 i_3} \sigma^2_{N_1 N_3}\Kil\big(\linear[\leqN][r][i_2](t)\big) - \delta^{i_2 i_3} \sigma^2_{N_2 N_3} \Kil\big(\linear[\leqN][r][i_1](t)\big) \Big\|_{C_t^0 \Cs_x^{-3\sreg}([0,1]\times \T^2)}^p \bigg]^{1/p} \lesssim p^{3/2} N_{\mrm{max}}^{-3\sdecay}. \] 
\end{lemma}

In order to prove Lemma \ref{lemma:cubic-stochastic-object}, we first derive a representation of the cubic power-type nonlinearity, which is the subject of the next lemma. 

\begin{lemma}\label{lemma:cubic-product-integral}
Let $i_1, i_2, i_3 \in [2]$, $N_1, N_2, N_3\in \dyadic$, $t \in \R$, and $x \in \T^d$. Then, it holds that
\[\begin{split}
\Big[\Big[\linear[N_1][r][i_1], \linear[N_2][r][i_2]\Big], \linear[N_3][r][i_3]\Big](t, x) =& \int_{\indexset^3} \Big[\big[f^{i_1}_{t, x, N_1} \otimes f^{i_2}_{t, x, N_2}\big]_{\cfrkg} \otimes f^{i_3}_{t, x, N_3}\Big]_{\cfrkg} dW^{\otimes 3}  \\
& - \delta^{i_1 i_3} \sigma^2_{ N_1 N_3}\Kil\big(\linear[N_2][r][i_2](t, x)\big) 
 + \delta^{i_2 i_3} \sigma^2_{N_2 N_3} \Kil\big(\linear[N_1][r][i_1](t, x)\big).
\end{split}\]
\end{lemma}
\begin{proof}
Recall that
\[ \linear[N][r][i](t, x) = \int_{\indexset} f^i_{t, x, N} dW.\]
For brevity, let 
\[ h_r(z) := \delta^{i_r}_{\dimind} \ind(s \leq t) \rho_{N_r}(n) \e_n(x) e^{-(t-s) \fnorm{n}^2}, ~~ r = 1, 2, 3, \]
so that $f_r(z) := f^{i_r}_{t, x, N_r}(z) = h_r(z) I_{\cfrkg}$. For brevity, define $g := f_1 \otimes f_2 \otimes f_3$.
By the product formula (Corollary \ref{cor:multiple-stochastic-integral-tensor-product}), we have that
\beq
\begin{aligned} \label{eq:cubic-object-no-derivative-representation-intermediate}
&\linear[N_1][r][i_1](t, x) \otimes \linear[N_2][r][i_2](t, x) \otimes \linear[N_3][r][i_3](t, x) \\ 
=\, & \int_{\indexset^3} g dW^{\otimes 3} + \int_{\indexset} \paired{\{1,2\}}(g)(z_3)  dW(z_3) + \int_{\indexset} \paired{\{1,3\}}(g)(z_2)  dW(z_2) + \int_{\indexset} \paired{\{2,3\}}(g)(z_1) dW(z_1). 
\end{aligned}
\eeq
We have that
\[ 
g(z_1, z_2, z_3) = (f_1 \otimes f_2 \otimes f_3)(z_1, z_2, z_3) = h_1(z_1) h_2(z_2) h_3(z_3) I_{\cfrkg} \otimes I_{\cfrkg} \otimes I_{\cfrkg}. \]
Writing $I_{\cfrkg} = E_a^* \otimes E^a$ and $I_{\cfrkg} \otimes I_{\cfrkg} \otimes I_{\cfrkg} = E_a^* \otimes E_b^* \otimes E_c^* \otimes E^a \otimes E^b \otimes E^c$, we also have that
\[ \tensorcon{\{1, 2\}} \big(f_1(z_1) \otimes f_2(z_2)\otimes f_3(z_3)\big)= h_1(z_1) h_2(z_2) h_3(z_3)  E_c^* \otimes E^a \otimes E_a \otimes E^c, \]
and thus
\[ \paired{\{1,2\}}(f_1 \otimes f_2 \otimes f_3)(z_3) =  \Big( \int_\indexset h_1(z_1) h_2(\bar{z}_1) d\lebI(z_1) \Big) h_3(z_3)  (E_c^* \otimes E^a \otimes E_a \otimes E^c) . \] 
Similarly, we may obtain
\[ \paired{\{1,3\}}(f_1 \otimes f_2 \otimes f_3)(z_2) =   \Big( \int_\indexset h_1(z_1) h_3(\bar{z}_1) d\lebI(z_1) \Big)  h_2(z_2) (E_b^* \otimes E^a \otimes E^b \otimes E_a) \]
\[ \paired{\{2,3\}}(f_1 \otimes f_2 \otimes f_3) (z_1) = \Big(\int_\indexset h_2(z_2) h_3(\bar{z}_2) d\lebI(z_2) \Big) h_1(z_1) (E_a^* \otimes E^a \otimes E^b \otimes E_b). \]
Let $M_{\{1, 2\}}, M_{\{1, 3\}}, M_{\{2, 3\}}$ be the elements of $L(\cfrkg, \cfrkg^{\otimes 3})$ which are canonically identified with $E_c^* \otimes E^a \otimes E_a \otimes E^c$, $E_b^* \otimes E^a \otimes E^b \otimes E_a$, and $E_a^* \otimes E^a \otimes E^b \otimes E_b$, respectively. Note that for $E \in \cfrkg$, we have that 
\begin{equation*}
M_{\{1, 2\}}(E) = E^a \otimes E_a \otimes E, \qquad M_{\{1, 3\}}(E) = E^a \otimes E \otimes E_a, \qquad \text{and} \qquad M_{\{2, 3\}}(E) = E \otimes E^b \otimes E_b.  
\end{equation*}
Next, let $R : \cfrkg^{\otimes 3} \ra \cfrkg$ be the linear map given by $E \otimes F \otimes G \mapsto [[E, F], G]$. We then have that the compositions satisfy
\begin{equation}\label{objects:eq-cubic-no-derivative-compositions}
R (M_{\{1, 2\}}(E)) = 0, ~~ R (M_{\{1, 3\}}(E)) = -\Kil(E), ~~ \text{and} ~~ R (M_{\{2, 3\}}(E)) = \Kil(E). 
\end{equation}
Indeed, the first identity in \eqref{objects:eq-cubic-no-derivative-compositions} follows directly from the skew-symmetry of the Lie bracket and the second and third identity in \eqref{objects:eq-cubic-no-derivative-compositions} follow from the trace identity \eqref{objects:eq-trace-e2} (and the skew-symmetry of the Lie bracket). 
In summary, by applying the map $R$ to both sides of \eqref{eq:cubic-object-no-derivative-representation-intermediate} and then applying the previous few observations, we have that
\[\begin{split}
\Big[\Big[\linear[N_1][r][i_1], \linear[N_2][r][i_2]\Big], \linear[N_3][r][i_3]\Big](t, x) = \int_{\indexset^3} (R \circ g) dW^{\otimes 3} - ~& \int_\indexset h_1(z_1) h_3(\bar{z}_1) d\lebI(z_1) \int_{\indexset} h_2 \Kil dW ~ \\
+~& \int_\indexset h_2(z_2) h_3(\bar{z}_2) d\lebI(z_2) \int_{\indexset} h_1 \Kil dW .
\end{split}\]
Recalling that $f_2 = h_2 I_{\cfrkg}$, we have that $h_2 \Kil = \Kil \circ f_2$, and thus 
\[ \int_{\indexset} h_2 \Kil dW = \Kil \int_{\indexset} f_2 dW = \Kil\big(\linear[N_2][r][i_2](t, x)\big), \text{ and similarly } \int_{\indexset} h_1 \Kil dW = \Kil\big(\linear[N_1][r][i_1](t, x)\big).\]
To finish, note that
\[ \int_{\indexset} h_1(z_1) h_3(\bar{z}_1) d\lebI(z_1) = \delta^{\ell_1 \ell_3}  \delta^{i_1}_{\ell_1} \delta^{i_3}_{\ell_3} \sum_{n_1 \in \Z^2} \rho_{N_1}(n_1) \rho_{N_3}(n_1) \int_{-\infty}^t ds_1 e^{-2(t-s_1) \fnorm{n_1}^2} = \delta^{i_1 i_3} \sigma^2_{N_1 N_3}, \]
and similarly
\[\int_{\indexset} h_2(z_2) h_3(\bar{z}_2) d\lebI(z_2) = \delta^{i_2 i_3} \sigma^2_{N_2 N_3}.\]
The desired result now follows upon combining the previous few displays.
\end{proof}

Equipped with Lemma \ref{lemma:cubic-product-integral}, we can now estimate the cubic power-type nonlinearity. 

\begin{proof}[Proof of Lemma \ref{lemma:cubic-stochastic-object}]
Due to Lemma \ref{lemma:cubic-product-integral}, it holds that \begin{equation}
\begin{aligned}
&\Big[\Big[\linear[N_1][r][i_1], \linear[N_2][r][i_2]\Big], \linear[N_3][r][i_3]\Big] + \delta^{i_1 i_3} \sigma^2_{N_1 N_3}\Kil\big(\linear[\leqN][r][i_2]\big) - \delta^{i_2 i_3} \sigma^2_{N_2 N_3} \Kil\big(\linear[\leqN][r][i_1]\big) \\
=& \, \int_{\indexset^3} \Big[\big[f^{i_1}_{t, x, N_1} \otimes f^{i_2}_{t, x, N_2}\big]_{\cfrkg} \otimes f^{i_3}_{t, x, N_3}\Big]_{\cfrkg} dW^{\otimes 3}. 
\end{aligned}
\end{equation}
Using standard reductions (Lemma \ref{lemma:standard-reduction}), it then suffices to prove that for $N_0 \lesssim N_{\max}$,
\begin{equation}\label{objects:eq-cubic-power-p0}
\sup_{t\in [0,1]} \sup_{x\in \T^2} 
\E \bigg[ \Big| P_{N_0} 
 \int_{\indexset^3} \Big[\big[f^{i_1}_{t, x, N_1} \otimes f^{i_2}_{t, x, N_2}\big]_{\cfrkg} \otimes f^{i_3}_{t, x, N_3}\Big]_{\cfrkg} dW^{\otimes 3} \Big|_{\cfrkg}^2 
 \bigg]^{1/2}
 \lesssim N_0^{6\sreg} N_{\max}^{-6\sreg}. 
\end{equation}
Towards this end, note that by Lemma \ref{lemma:multiple-stochastic-integral-second-moment-bound-l2}, we have that
\begin{align}
&\E \bigg[ \Big| P_{N_0} 
 \int_{\indexset^3} \Big[\big[f^{i_1}_{t, x, N_1} \otimes f^{i_2}_{t, x, N_2}\big]_{\cfrkg} \otimes f^{i_3}_{t, x, N_3}\Big]_{\cfrkg} dW^{\otimes 3} \Big|_{\cfrkg}^2\bigg] \notag \\ 
 \lesssim\, &  \int_{\indexset^3} \Big|P_{N_0} \Big[\big[f^{i_1}_{t, x, N_1} \otimes f^{i_2}_{t, x, N_2}\big]_{\cfrkg} \otimes f^{i_3}_{t, x, N_3}\Big]_{\cfrkg} \Big|_{\cfrkg^{*\otimes 3} \otimes \cfrkg}^3 d\lebI^3.
 \label{objects:eq-cubic-power-p1}
 \end{align}
For any $z_1, z_2, z_3 \in \indexset$, we have that
\begin{equation*}
\begin{split}
\Big|P_{N_0} \Big[\big[f^{i_1}_{t, x, N_1}(z_1) \otimes f^{i_2}_{t, x, N_2}(z_2)\big]_{\cfrkg} \otimes f^{i_3}_{t, x, N_3}(z_3)\Big]&\Big|_{\cfrkg^{* \otimes 3} \otimes \cfrkg}^2 \lesssim  \ind(s_1, s_2, s_3 \leq t) \rho_{N_1}(n_1)^2 \rho_{N_2}(n_2)^2 \rho_{N_3}(n_3)^2 ~\times \\
&\rho_{N_0}(n_{123})^2 e^{-2(t-s_1)\fnorm{n_1}^2} e^{-2(t-s_2) \fnorm{n_2}^2} e^{-2(t-s_3) \fnorm{n_3}^2}.
\end{split}
\end{equation*}
By inserting this into \eqref{objects:eq-cubic-power-p1}, calculating the $s_1$, $s_2$, and $s_3$-integrals, and using Lemma \ref{prelim:lem-counting} (recall that $N_0 \lesssim N_{\max}$), it follows that
\begin{equation*}
   \eqref{objects:eq-cubic-power-p1} \lesssim \sum_{n_1, n_2, n_3\in \Z^2} \frac{\rho_{N_1}(n_1) \rho_{N_2}(n_2) \rho_{N_3}(n_3)}{\fnorm{n_1}^2 \fnorm{n_2}^2 \fnorm{n_3}^2} \rho_{N_0}(n_{123}) \lesssim N_0^{6\sreg} N_{\max}^{-6\sreg}. 
\end{equation*}
This completes the proof of \eqref{objects:eq-cubic-power-p0}, and hence yields the desired estimate. 
\end{proof}

\subsection{Quadratic stochastic object}
\label{section:objects-quadratic}

In this subsection we study the quadratic stochastic object $\quadratic[][r]$, which is part of our Ansatz for the solution of \eqref{ansatz:eq-SYM-leqN} and appears as part of the enhanced data set in \eqref{objects:eq-enhanced-4}. Our main estimate is recorded in the following lemma. 

\begin{lemma}[Quadratic stochastic object with derivative]\label{lemma:quadratic-stochastic-object-with-derivative} Let $0<\sdecay <\epsilon\ll 1$, 
let $j_1, j_2, k \in [2]$ and let $N_1, N_2\in \dyadic$. 
For all $p\geq 1$, we have that
\begin{equation}\label{objects:eq-quadratic-object-e1}
 \E \bigg[ \Big\| \quadratic[N_1, N_2][d][j_1][j_2][k]\Big\|_{C_t^0\Cs_x^{1-2\sreg}([0,1]\times \T^2)}^p\bigg]^{1/p} \lesssim p N_{\mrm{max}}^{-2\sdecay}.
\end{equation}
Furthermore, for all $N_0 \in \dyadic$ and $\delta>0$, it also holds that
\begin{equation}\label{objects:eq-quadratic-object-e2}
 \E \bigg[ \Big\| \ptl_t P_{N_0} \quadratic[N_1, N_2][d][j_1][j_2][k]\Big\|_{C_t^0 \Cs_x^{-1}([0,1]\times \T^2)}^p\bigg]^{1/p} \lesssim p  N_{\mrm{max}}^{2\delta}. 
\end{equation}
\end{lemma}

Before we dive into the proof of Lemma \ref{lemma:quadratic-stochastic-object-with-derivative}, we make several preparations. For notational convenience, we first make the following definition. 

\begin{definition}\label{def:quadratic-one-derivative-integrand}
For $s_1, s_2 \in \R$, $n_1, n_2 \in \Z^2$, we define 
\[ \alpha_t(s_1, s_2, n_1, n_2) := \int_{-\infty}^t \ind(s_1, s_2 \leq s) e^{-(t - s) \fnorm{n_{12}}^2} e^{-(s - s_1) \fnorm{n_1}^2} e^{-(s - s_2) \fnorm{n_2}^2} ds. \]
For $j_1, j_2, k \in [2]$, $t \in \R$, $x \in \T^2$, and $N_1, N_2 \in \dyadic$, we also define the function $q^{j_1 j_2 k}_{t, x, N_1, N_2} : \indexset^2 \ra L(\cfrkg^{\otimes 2}, \cfrkg)$ by
\[ q^{j_1 j_2 k}_{t, x, N_1, N_2}(z_1, z_2) := \icomplex 
\delta^{j_1}_{\dimind_1} 
\delta^{j_2}_{\dimind_2} 
n_{2}^{k}
\rho_{N_1}(n_1) \rho_{N_2}(n_2) e_{n_{12}}(x) \alpha_t(s_1, s_1, n_1, n_2) [I_{\cfrkg} \otimes I_{\cfrkg}]_{\cfrkg}. \] 
\end{definition}

Equipped with Definition \ref{def:quadratic-one-derivative-integrand}, we now obtain a representation of the quadratic object. 

\begin{lemma}\label{lemma:quadratic-one-derivative-multiple-stochastic-integral-representation}
Let $j_1, j_2, k \in [2]$, $t \in \R$, $x \in \T^2$, and $N_1, N_2\in \dyadic$. Then, we have that
\[ \quadratic[N_1, N_2][d][j_1][j_2][k](t, x) = \int_{\indexset^2} q^{j_1 j_2 k}_{t, x, N_1, N_2} dW^{\otimes 2} .\]
\end{lemma}
\begin{proof}[Proof of Lemma \ref{lemma:quadratic-one-derivative-multiple-stochastic-integral-representation}]
Recall that by definition,
\[ \quadratic[N_1, N_2][d][j_1][j_2][k](t, x) = \Duhinf\Big(\Big[\linear[N_1][r][j_1], \ptl^k \linear[N_2][r][j_2]\Big]\Big)(t, x) = \int_{-\infty}^t \Big(e^{(t-s)(-1 + \Delta)}\Big[\linear[N_1][r][j_1], \ptl^k \linear[N_2][r][j_2]\Big](s)\Big)(x) ds .\]
We first find the multiple stochastic integral representation for $\big[\, \linear[N_1][r][j_1], \ptl^k \linear[N_2][r][j_2] \big](s, x)$. By proceeding similarly as in the proof of Lemma \ref{objects:lem-without}, we may obtain
\begin{equation}\label{objects:eq-derivative-nonlinearity-representation}
\Big[\linear[N_1][r][j_1], \ptl^k \linear[N_2][r][j_2]\Big](s, x) = \int_{\indexset^2} \Big[ f^{j_1}_{s, x, N_1} \otimes \ptl^k f^{j_2}_{s, x, N_2}\Big]_{\cfrkg} dW^{\otimes 2} 
+ \Big[ \paired{\{1,2\}} \Big( f^{j_1}_{s, x, N_1} \otimes \ptl^k f^{j_2}_{s, x, N_2} \Big) \Big]_{\cfrkg}. 
\end{equation}
The resonant term \eqref{objects:eq-derivative-nonlinearity-representation} vanishes since, due to parity, 
\begin{equation*}
\sum_{\substack{n_1,n_2\in \Z^2 \colon \\ n_1 + n_2 =0}}
n_2^k \prod_{j=1}^2 \rho_{N_j}(n_j) e^{-(t-s)\langle n_j \rangle^2} =0. 
\end{equation*}
Note that
\[\begin{split}
\Big[ f^{j_1}_{s, x, N_1} \otimes \ptl^k f^{j_2}_{s, x, N_2}\Big]_{\cfrkg}(z_1, z_2) = \icomplex 
\delta^{j_1}_{\dimind_1} 
\delta^{j_2}_{\dimind_2} 
n_{2}^{k}
&\ind(s_1, s_2 \leq s) \rho_{N_1}(n_1) \rho_{N_2}(n_2) \e_{n_{12}}(x) ~\times \\
&e^{-(s-s_1) \fnorm{n_1}^2} e^{-(s-s_2)\fnorm{n_2}^2} [I_{\cfrkg} \otimes I_{\cfrkg}]_{\cfrkg}. 
\end{split}\]
From this, we have that
\[\begin{split}
e^{(t-s)(-1 + \Delta)} \Big[ f^{j_1}_{s, \cdot, N_1} \otimes \ptl^k f^{j_2}_{s, \cdot, N_2}\Big]_{\cfrkg}(z_1, z_2)  = \icomplex &\delta^{j_1}_{\dimind_1} 
\delta^{j_2}_{\dimind_2} 
n_{2}^{k} \ind(s_1, s_2 \leq s) \rho_{N_1}(n_1) \rho_{N_2}(n_2) \e_{n_{12}}(\cdot) ~\times \\
&e^{-(t-s)\fnorm{n_{12}}^2} e^{-(s-s_1) \fnorm{n_1}^2} e^{-(s-s_2)\fnorm{n_2}^2} [I_{\cfrkg} \otimes I_{\cfrkg}]_{\cfrkg}.
\end{split}\]
For brevity, define $h^{j_1 j_2 k}_{s,x, N_1, N_2}(z_1, z_2)$ by the right hand side above. We then obtain
\[\begin{split}
\int_{-\infty}^t\Big(e^{(t-s)(-1 + \Delta)}\Big[\linear[N_1][r][j_1], &\ptl^k \linear[N_2][r][j_2]\Big](s)\Big)(x) ds = \int_{-\infty}^t \int_{\indexset^2}  \Big(e^{(t-s)(-1 + \Delta)} \Big[ f^{j_1}_{s, \cdot, N_1} \otimes \ptl^k f^{j_2}_{s, \cdot, N_2}\Big]_{\cfrkg}\Big)(x) dW^{\otimes 2} ds \\
&= \int_{\indexset^2} \int_{-\infty}^t  \Big(e^{(t-s)(-1 + \Delta)} \Big[ f^{j_1}_{s, \cdot, N_1} \otimes \ptl^k f^{j_2}_{s, \cdot, N_2}\Big]_{\cfrkg}(z_1, z_2)\Big)(x) ds dW^{\otimes 2}(z_1, z_2) \\
&= \int_{\indexset^2} \int_{-\infty}^t h^{j_1 j_2 k}_{s, x, N_1, N_2}(z_1, z_2) ds dW^{\otimes 2}. 
\end{split} \] 
The desired result then follows upon noting that
\[ q^{j_1 j_2 k}_{t, x, N_1, N_2}(z_1, z_2) = \int_{-\infty}^t h^{j_1 j_2 k}_{s, x, N_1, N_2}(z_1, z_2) ds, \]
and combining the previous few observations. 
\end{proof}

In the next lemma, we state integral estimates which will be used to control the quadratic stochastic object. 

\begin{lemma}\label{objects:lem-quadratic-integral-estimate}
For all $t\in \R$ and $n_1,n_2\in \Z^2$, it holds that 
\beq\label{eq:alpha-t-squared-integral-bound} \int_{\R^2} \alpha_t(s_1, s_2, n_1, n_2)^2 ds_1 ds_2 \leq \frac{1}{\fnorm{n_1}^2 \fnorm{n_2}^2 \fnorm{n_{12}}^2} \max(\fnorm{n_1}, \fnorm{n_2}, \fnorm{n_{12}})^{-2}.  \eeq
\end{lemma}
\begin{proof}
By alternately bounding 
\[ e^{-(s-s_1)\fnorm{n_1}^2} e^{-(s-s_2)\fnorm{n_2}^2} \leq 1, ~~ e^{-(t-s)\fnorm{n_{12}}^2} e^{-(s-s_1)\fnorm{n_1}^2} \leq 1,~~ e^{-(t-s)\fnorm{n_{12}}^2} e^{-(s-s_2)\fnorm{n_2}^2}  \leq 1, \]
and then computing the $s$-integral in the definition of $\alpha_t$, it follows that 
\beq\label{eq:alpha-t-bound} \alpha_t(s_1, s_2, n_1, n_2) \leq \max(\fnorm{n_1}, \fnorm{n_2}, \fnorm{n_{12}})^{-2}. \eeq
Furthermore, it holds that
\begin{equation}\label{eq:alpha-t-double-integral}
\begin{aligned}
\int_{\R^2} \alpha_t(s_1, s_2, n_1, n_2) ds_1 ds_2 &= \int_{-\infty}^t ds e^{-(t-s) \fnorm{n_{12}}^2}  \int_{\R^2} ds_1 ds_2 \ind(s_1, s_2 \leq s) e^{-(s - s_1) \fnorm{n_1}^2} e^{-(s-s_2) \fnorm{n_2}^2} \\ 
&= \fnorm{n_1}^{-2} \fnorm{n_2}^{-2} \fnorm{n_{12}}^{-2}. 
\end{aligned}
\end{equation}
By combining \eqref{eq:alpha-t-bound} and \eqref{eq:alpha-t-double-integral}, it follows that 
\begin{align*}
 \int_{\R^2} \alpha_t(s_1, s_2, n_1, n_2)^2 ds_1 ds_2
 &\leq  \max(\fnorm{n_1}, \fnorm{n_2}, \fnorm{n_{12}})^{-2} \int_{\R^2} \alpha_t(s_1, s_2, n_1, n_2)^2 ds_1 ds_2 \\ 
 &\leq \max(\fnorm{n_1}, \fnorm{n_2}, \fnorm{n_{12}})^{-2}
 \fnorm{n_1}^{-2} \fnorm{n_2}^{-2} \fnorm{n_{12}}^{-2}, 
\end{align*}
which yields the desired bound \eqref{eq:alpha-t-squared-integral-bound}.
\end{proof}

Equipped with Lemma \ref{lemma:quadratic-one-derivative-multiple-stochastic-integral-representation} and Lemma \ref{objects:lem-quadratic-integral-estimate}, we can now estimate the quadratic stochastic object. 

\begin{proof}[Proof of Lemma \ref{lemma:quadratic-stochastic-object-with-derivative}]
By Lemma \ref{lemma:quadratic-one-derivative-multiple-stochastic-integral-representation}, we have that
\[\quadratic[N_1, N_2][d][j_1][j_2][k](t, x) = \int_{\indexset^2} q^{j_1 j_2 k}_{t, x, N_1, N_2} dW^{\otimes 2}, \]
where $q^{j_1 j_2 k}_{t, x, N_1, N_2}$ is as in Definition \ref{def:quadratic-one-derivative-integrand}. 
We first prove the $C_t^0$-estimate \eqref{objects:eq-quadratic-object-e1}. 
Due to standard reductions (Lemma \ref{lemma:standard-reduction}), it suffices to prove for all $N_0\in \dyadic$ that 
\begin{equation}\label{objects:eq-quadratic-reduced}
 \sup_{t\in [0,1]} \sup_{x\in \T^2} \E \bigg[ \bigg| \int_{\indexset^2} P_{N_0} q^{j_1 j_2 k}_{t, x, N_1, N_2} dW^{\otimes 2} \bigg|_{{\cfrkg}}^2 \bigg] \lesssim  N_{\max}^{-2}. 
\end{equation}
By Lemma \ref{lemma:multiple-stochastic-integral-second-moment-bound-l2}, we have that 
\begin{equation}\label{objects:eq-quadratic-q1}
\E \bigg[ \bigg| \int_{\indexset^2} P_{N_0} q^{j_1 j_2 k}_{t, x, N_1, N_2} dW^{\otimes 2} \bigg|_{{\cfrkg}}^2 \bigg] \lesssim \int_{\indexset^2} \big|P_{N_0} q^{j_1 j_2 k}_{t, x, N_1, N_2}\big|_{\cfrkg^{*\otimes 2} \otimes \cfrkg}^2 d\lebI^2. 
\end{equation} 
For $z_1, z_2 \in \indexset$, the integrand in \eqref{objects:eq-quadratic-q1} can be bounded by 
\begin{equation}\label{objects:eq-quadratic-q2}
    \big|P_{N_0} q^{j_1 j_2 k}_{t, x, N_1, N_2}\big|_{\cfrkg^{*\otimes 2} \otimes \cfrkg}^2 \lesssim |n_2|^2 \rho_{N_0}(n_{12}) \rho_{N_1}(n_1) \rho_{N_2}(n_2)  \alpha_t(s_1, s_2, n_1, n_2)^2. 
\end{equation}
By first using \eqref{objects:eq-quadratic-q2} and then Lemma \ref{objects:lem-quadratic-integral-estimate}, it follows that
\begin{align}
\int_{\indexset^2} \big|P_{N_0} q^{j_1 j_2 k}_{t, x, N_1, N_2}\big|_{\cfrkg^{*\otimes 2} \otimes \cfrkg}^2 d\lebI^2 
&\lesssim 
 \sum_{n_1, n_2\in \Z^2} |n_2
|^2 \rho_{N_0}(n_{12}) \rho_{N_1}(n_1) \rho_{N_2}(n_2)  \int_{\R^2}  \alpha_t(s_1, s_2, n_1, n_2)^2 ds_1 ds_2 \notag \\
&\lesssim N_{\textup{max}}^{-2}   N_0^{-2} N_1^{-2} 
\sum_{n_1,n_2\in \Z^2}\rho_{N_0}(n_{12}) \rho_{N_1}(n_1) \rho_{N_2}(n_2). \label{objects:eq-quadratic-q3} 
\end{align}
By changing variables to $(n_1,n_{12})$, it follows that
\begin{equation*}
 \eqref{objects:eq-quadratic-q3} \lesssim N_{\textup{max}}^{-2}    N_0^{-2} N_1^{-2} \times N_0^2 N_1^2 \lesssim N_{\textup{max}}^{-2}. 
\end{equation*}
This completes the proof of \eqref{objects:eq-quadratic-reduced} and hence the proof of the $C_t^0$-estimate \eqref{objects:eq-quadratic-object-e1}.
It now remains to prove the $C_t^1$-estimate \eqref{objects:eq-quadratic-object-e2}. 
Using our standard reduction (Lemma \ref{lemma:standard-reduction}), it suffices to show that 
\beq\label{eq:quadratic-one-derivative-time-derivative-intermediate} 
\sup_{t\in [0,1]} \sup_{x\in \T^2} \E\bigg[\bigg|\ptl_t P_{N_0} \quadratic[N_1, N_2][d][j_1][j_2][k](t, x)\bigg|_{\cfrkg}^2\bigg] \lesssim N_0^{2}. \eeq
Observe that
\[ \ptl_t P_{N_0}  \quadratic[N_1, N_2][d][j_1][j_2][k](t, x) = \int_{\indexset^2} \ptl_t P_{N_0} q^{j_1 j_2 k}_{t, x, N_1, N_2} dW^{\otimes 2}. \]
Noting that
\[ \ptl_t \alpha_t(s_1, s_2, n_1, n_2) = -\fnorm{n_{12}}^2 \alpha_t(s_1, s_2, n_1, n_2) + \ind(s_1, s_2 \leq t) e^{-(t - s_1) \fnorm{n_1}^2} e^{-(t-s_2)\fnorm{n_2}^2}, \]
we may obtain
\[\begin{split} 
\ptl_t P_{N_0}  &q^{j_1 j_2 k}_{t, x, N_1, N_2}(z_1, z_2) = -\fnorm{n_{12}}^2 q^{j_1 j_2 k}_{t, x, N_1, N_2}(z_1, z_2) ~+ \\
&\icomplex \ind(s_1, s_2 \leq t) \delta^{j_1}_{\dimind_1} \delta^{j_2}_{ \dimind_2} n_{2, k} \rho_{N_0}(n_{12}) \rho_{N_1}(n_1) \rho_{N_2}(n_2) e_{n_{12}}(x) e^{-(t - s_1) \fnorm{n_1}^2} e^{-(t - s_2) \fnorm{n_2}^2} [I_{\cfrkg} \otimes I_{\cfrkg}]_{\cfrkg}.
\end{split}\]
For brevity, denote the terms on the right hand side above by $h_1(z_1, z_2), h_2(z_1, z_2)$. By Lemma \ref{lemma:multiple-stochastic-integral-second-moment-bound-l2}, to show \eqref{eq:quadratic-one-derivative-time-derivative-intermediate}, it suffices to show
\[ \|h_1\|^2_{L^2}, \|h_2\|^2_{L^2} \lesssim N_0^{2}.\]
By arguing as in the proof of the $C_t^0$-estimate \eqref{objects:eq-quadratic-object-e1}, we may obtain
\[\begin{split}
\|h_1\|^2_{L^2} &\lesssim \sum_{n_1, n_2 \in \Z^2} \fnorm{n_{12}}^4 |n_2|^2 \rho_{N_0}(n_{12}) \rho_{N_1}(n_1) \rho_{N_2}(n_2)  \max(\fnorm{n_1}, \fnorm{n_2}, \fnorm{n_{12}})^{-2} \fnorm{n_1}^{-2} \fnorm{n_2}^{-2} \fnorm{n_{12}}^{-2} \\
&\lesssim N_0^4 \sum_{n_1, n_2 \in \Z^2} |n_2|^2 \rho_{N_0}(n_{12}) \rho_{N_1}(n_1) \rho_{N_2}(n_2)  \max(\fnorm{n_1}, \fnorm{n_2}, \fnorm{n_{12}})^{-2} \fnorm{n_1}^{-2} \fnorm{n_2}^{-2} \fnorm{n_{12}}^{-2} .
\end{split}\]
Then by applying inequality \eqref{eq:quadratic-stochastic-object-one-derivative-combintarial-estimate}, we further obtain
\[ \|h_1\|_{L^2}^2 \lesssim N_0^4 N_0^{-2} \lesssim N_0^2. \]
Moving on, we may bound
\[ \|h_2\|^2_{L^2} \lesssim \sum_{n_1, n_2} |n_2|^2 \rho_{N_0}(n_{12}) \rho_{N_1}(n_1) \rho_{N_2}(n_2) \frac{1}{\fnorm{n_1}^2 \fnorm{n_2}^2} \leq \sum_{n_1, n_2} \rho_{N_0}(n_{12}) \rho_{N_1}(n_1) \rho_{N_2}(n_2)\frac{1}{\fnorm{n_1}^2}.  \]
By viewing $(n_1,n_{12})$ as the free variables, the right hand side above is bounded by $N_0^{2}$, as desired.
\end{proof}

By combining the $C_t^0$ and $C_t^1$-estimates from Lemma \ref{lemma:quadratic-stochastic-object-with-derivative}, we obtain the following corollary.

\begin{corollary}\label{objects:cor-quadratic-object-time} 
Let $j_1, j_2, k \in [2]$ and let $N_1, N_2 \in \dyadic$. Then, it holds for all $p\geq 1$ that 
\begin{equation}\label{objects:eq-quadratic-object-time} 
\E\bigg[ \sup_{\substack{s,t\in [0,1]\colon\\ s\neq t}}  \bigg( |t-s|^{-\frac{1-\kappa}{2}} \bigg\|\quadratic[N_1, N_2][d][j_1][j_2][k](t) - \quadratic[N_1, N_2][d][j_1][j_2][k](s)\bigg\|_{\Cs_x^{-\kappa}} \bigg)^p\bigg]^{1/p} \lesssim p N_{\mrm{max}}^{-\eta}. 
\end{equation}
\end{corollary}
\begin{proof}
In order to prove \eqref{objects:eq-quadratic-object-time}, we use a dyadic decomposition in the spatial frequency scale and in the time difference, i.e., we estimate 
\begin{align}
&\E\bigg[ \sup_{\substack{s,t\in [0,1]\colon\\ s\neq t}}  \bigg( |t-s|^{-\frac{1-\kappa}{2}} \bigg\|\quadratic[N_1, N_2][d][j_1][j_2][k](t) - \quadratic[N_1, N_2][d][j_1][j_2][k](s)\bigg\|_{\Cs_x^{-\kappa}} \bigg)^p\bigg]^{1/p} \notag \\
\lesssim&\, \sum_{L\in \dyadic} \sum_{\substack{N_0 \in \dyadic\colon \\ N_0 \lesssim N_{\mrm{max}}}} \E\bigg[
\sup_{\substack{s,t\in [0,1]\colon\\ |s- t| \sim L^{-1}}}  \bigg( |t-s|^{-\frac{1-\kappa}{2}} \bigg\| P_{N_0} \bigg( \quadratic[N_1, N_2][d][j_1][j_2][k](t) - \quadratic[N_1, N_2][d][j_1][j_2][k](s) \bigg)\bigg\|_{\Cs_x^{-\kappa}} \bigg)^p\bigg]^{1/p}. \label{objects:eq-quadratic-object-time-p1} 
\end{align}
By using \eqref{objects:eq-quadratic-object-e1} and \eqref{objects:eq-quadratic-object-e2} from Lemma \ref{lemma:quadratic-stochastic-object-with-derivative} with $(\epsilon,\nu,\delta)=(\kappa,\kappa-\eta,\eta)$, it follows that 
\begin{align*}
&\E\bigg[
\sup_{\substack{s,t\in [0,1]\colon\\ |s- t| \sim L^{-1}}}  \bigg( |t-s|^{-\frac{1-\kappa}{2}} \bigg\| P_{N_0} \bigg( \quadratic[N_1, N_2][d][j_1][j_2][k](t) - \quadratic[N_1, N_2][d][j_1][j_2][k](s) \bigg)\bigg\|_{\Cs_x^{-\kappa}} \bigg)^p\bigg]^{1/p} \\
\lesssim&\, p \min\Big( L^{\frac{1-\kappa}{2}} N_0^{-1+\kappa} N_{\mrm{max}}^{-2(\kappa-\eta)}, L^{-\frac{1+\kappa}{2}} N_0^{1-\kappa} N_{\mrm{max}}^{2\eta} \Big) \\ 
\lesssim&\, p \Big( L^{\frac{1-\kappa}{2}} N_0^{-1+\kappa} N_{\mrm{max}}^{-2(\kappa-\eta)} \Big)^{\frac{1}{2}}
\Big(  L^{-\frac{1+\kappa}{2}} N_0^{1-\kappa} N_{\mrm{max}}^{2\eta} \Big)^{\frac{1}{2}} 
= p L^{-\frac{\kappa}{2}}  N_{\mrm{max}}^{-\kappa+2\eta}. 
\end{align*} 
As a result, it follows that
\begin{equation*}
\eqref{objects:eq-quadratic-object-time-p1}  \lesssim  p 
\sum_{L\in \dyadic} \sum_{\substack{N_0 \in \dyadic\colon \\ N_0 \lesssim N_{\mrm{max}}}} L^{-\frac{\kappa}{2}}  N_{\mrm{max}}^{-\kappa+2\eta}
\lesssim p  N_{\mrm{max}}^{-\kappa+3\eta}. 
\end{equation*}
Since $\kappa\gg \eta$, this yields an acceptable contribution. 
\end{proof}

\subsection{Derivative nonlinearities}
\label{section:objects-derivative}

In this subsection, we control the derivative nonlinearities in the enhanced data set, i.e., the terms in \eqref{objects:eq-enhanced-5} and \eqref{objects:eq-enhanced-6}. We start with the quadratic terms from \eqref{objects:eq-enhanced-5}.

\begin{lemma}[Quadratic nonlinearity with derivatives]\label{objects:lem-with-two} Let $0<\sdecay <\sreg,\delta\ll 1$, 
let $N_1,N_2 \in \dyadic$ satisfy $N_1 \sim N_2$, let 
 $j_1,j_2,k_1,k_2 \in [2]$, and let $E\in \frkg$ satisfy $\| E \|_{\frkg}\leq 1$. For all $p\geq 1$, it then holds that 
\begin{align}
\E \bigg[ \Big\| \Big[  \Big[ E, \Duh \Big( \partial_{k_1} \linear[N_1][r][j_1] \Big) \Big] , \partial_{k_2} \linear[N_2][r][j_2] \Big] 
- \frac{1}{4} \delta^{j_1 j_2} \delta_{k_1 k_2}  \sigma_{N_1,N_2}^2 \Kil\big(E \big) \Big\|_{\Wc^{-2\sreg, \delta}([0,1]\times \T^2)}^p \bigg]^{1/p}
&\lesssim p N_{\textup{max}}^{-2\sdecay},\label{objects:eq-with-two}
\end{align}
\end{lemma}

\begin{remark} 
The reason for using the weighted space $\Wc^{-2\sreg,\sdecay}$ in Lemma \ref{objects:lem-with-two}, rather than our usual $C_t^0 \Cs_x^{-2\epsilon}$-space, lies in the non-stationary of $\Duh \big( \partial_{k_1} \linear[N_1][r][j_1] \big)$, which produces a time-dependent resonant term in  \eqref{objects:eq-with-two-p11} below.
\end{remark}

\begin{remark}\label{objects:rem-with-two}
In $d=2$, Wick-ordering yields the counterterm $\frac{1}{4}\sigma_{N_1,N_2}^2 \Kil\big(E \big)$ in \eqref{objects:eq-with-two}. In a general dimension $d\geq 2$, we expect that Wick-ordering yields the counterterm  $\frac{1}{2d} \sigma_{N_1,N_2}^2  \Kil\big(E \big)$.  
\end{remark}

The following argument is similar to the proof of Lemma \ref{objects:lem-without}. 
\begin{proof}[Proof of Lemma \ref{objects:lem-with-two}:] 
We first examine the tensor product of $\Duh \big( \partial_{k_1} \linear[N_1][r][j_1] \big) $ and $\partial_{k_2} \linear[N_2][r][j_2]$.
To this end, we recall from Example \ref{prelim:example-linear} that
\begin{align}
 \partial_{k} \linear[M][r][j]
= \int_{\indexset} f^{j}_{t,x,M,k}  \, dW, \label{objects:eq-with-two-p1}
\end{align}
where
\begin{equation}
f^{j}_{t,x,M,k}(\dimind,m,s) = \delta^j_\dimind \ind_{(-\infty,t]}(s) \rho_M(m)  \e_{m}(x)  e^{-(t-s) \langle m \rangle^2} I_{\cfrkg}. \label{objects:eq-with-two-p3} 
\end{equation}
Arguing similarly as in Example \ref{prelim:example-integrated-linear}, we also obtain that
\begin{equation}
\Duh \Big( \partial_{k} \linear[M][r][j] \Big) = \int_0^t e^{-(t-t^\prime)(1-\Delta)} \partial_k \linear[M][r][j](t^\prime) \mathrm{d}t^\prime 
= \int_{\indexset} F^{j}_{t,x,M,k}  \, dW
\end{equation}
where
\begin{equation}
F^{j}_{t,x,M,k}(\dimind,m,s) = \delta^j_\dimind \ind_{(-\infty,t]}(s) \rho_M(m) \big( \icomplex m_k \big) \e_{m}(x) \big(t-\max(s,0)\big) e^{-(t-s) \langle m \rangle^2} I_{\cfrkg} \label{objects:eq-with-two-p2}. 
\end{equation}
Using \eqref{objects:eq-with-two-p1}-\eqref{objects:eq-with-two-p2} and the product formula (Corollary \ref{cor:multiple-stochastic-integral-tensor-product}), it follows that 
\begin{align}
\Duh \Big( \partial_{k_1} \linear[N_1][r][j_1] \Big) 
\otimes \partial_{k_2} \linear[N_2][r][j_2] 
&= \int_{\indexset^2} \big( F^{j_1}_{t,x,N_1,k_1}\otimes f^{j_2}_{t,x,N_2,k_2} \big) \, d(W\otimes W) \label{objects:eq-with-two-p4}\\
&+ \paired{\{1,2\}} \big( F^{j_1}_{t,x,N_1,k_1}\otimes f^{j_2}_{t,x,N_2,k_2} \big). \label{objects:eq-with-two-p5}
\end{align}
We now split the remainder of the proof into three steps.\\

\emph{Step 1: Estimate of the non-resonant part \eqref{objects:eq-with-two-p4}.}
For the non-resonant part \eqref{objects:eq-with-two-p4}, we prove the estimate 
\begin{equation}\label{objects:eq-with-two-p6}
\E \bigg[ \Big\| \int_{\indexset^2} \big(  F^{j_1}_{t,x,N_1,k_1} \otimes f^{j_2}_{t,x,N_2,k_2} \big) d(W\otimes W) \Big\|_{C_t^0 \Cs_x^{-2\sreg}([0,1]\times \T^2)}^p \bigg]^{1/p} \lesssim p \max(N_1,N_2)^{-2\sdecay}. 
\end{equation}
Using standard reductions (Lemma \ref{lemma:standard-reduction}), it suffices to prove that 
\begin{equation}\label{objects:eq-with-two-p7}
\sup_{t\in [0,1]} \sup_{x\in \T^2} \E \bigg[ \Big| P_{N_0} \int_{\indexset^2} \big(  F^{j_1}_{t,x,N_1,k_1} \otimes f^{j_2}_{t,x,N_2,k_2} \big) d(W\otimes W) \Big|_{\cfrkg}^2 \bigg]^{1/2} \lesssim  N_0^{4\sreg} N_{\max}^{-4\sreg}. 
\end{equation}
Using Lemma \ref{lemma:multiple-stochastic-integral-second-moment-bound-l2}, \eqref{objects:eq-with-two-p2}, and \eqref{objects:eq-with-two-p3}, the left-hand side of \eqref{objects:eq-with-two-p7} can be estimated using 
\begin{equation}\label{objects:eq-with-two-p9}
\begin{aligned}
&\E \bigg[ \Big| P_{N_0} \int_{\indexset^2} \big(  F^{j_1}_{t,x,N_1,k_1} \otimes f^{j_2}_{t,x,N_2,k_2} \big) d(W\otimes W) \Big|_{\cfrkg}^2 \bigg] \bigg]\\
\lesssim& \, \sum_{\substack{n_0,n_1,n_2 \in \Z^2 \colon \\ n_0 = n_1 + n_2}} 
\rho_{N_0}(n_0) \rho_{N_1}(n_1) \rho_{N_2}(n_2)  \langle n_1 \rangle^2 \langle n_2 \rangle^2 \int_{(-\infty,t]^2} (t-s_1)^2 e^{-2(t-s_1) \langle n_1 \rangle^2} e^{-2(t-s_2) \langle n_2 \rangle^2} \ds_1 \ds_2. 
\end{aligned}
\end{equation}
In \eqref{objects:eq-with-two-p9}, we also estimated the $(t-\max(s,0))$-factor in \eqref{objects:eq-with-two-p2} by $t-s$, which is possible since $t\geq 0$. 
By calculating the $s_1$ and $s_2$-integrals, then using Lemma \ref{prelim:lem-counting}, and finally using $N_1 \sim N_2$, it follows that 
\begin{equation}\label{objects:eq-with-two-p10}
\begin{aligned}
\eqref{objects:eq-with-two-p9} 
\lesssim&\, N_1^{-4} \sum_{\substack{n_0,n_1,n_2 \in \Z^2 \colon \\ n_0 = n_1 + n_2}} 
\rho_{N_0}(n_0) \rho_{N_1}(n_1) \rho_{N_2}(n_2)  
\lesssim  N_1^{-4} \times (N_0^{4\sreg} N_1^2 N_2^2 N_{\max}^{-4\sreg} )   \lesssim N_0^{4\sreg} N_{\textup{max}}^{-4\sreg}. 
\end{aligned}
\end{equation}
By combining \eqref{objects:eq-with-two-p9} and \eqref{objects:eq-with-two-p10}, we obtain the desired estimate \eqref{objects:eq-with-two-p7}. \\ 

\emph{Step 2: Analysis of the resonant part \eqref{objects:eq-with-two-p5}.} In this step, we prove that 
\begin{equation}\label{objects:eq-with-two-p11}
\paired{\{1,2\}} \big( F^{j_1}_{t,x,N_1,k_1}\otimes f^{j_2}_{t,x,N_2,k_2} \big) 
= \delta^{j_1 j_2} \delta_{k_1k_2} \Big( \frac{1}{4} \sigma_{N_1,N_2}^2 + c_{N_1,N_2}(t) \Big) \tensorcon{\{1,2\}} \big( I_{\cfrkg} \otimes I_{\cfrkg} \big),
\end{equation}
where $c_{N_1,N_2}\colon [0,1] \rightarrow \R$ satisfies the estimate
\begin{equation}\label{objects:eq-c-estimate} 
t^\delta \Big| c_{N_1,N_2}(t) \Big| \lesssim N_{\mrm{max}}^{-2\sdecay}. 
\end{equation} 
Indeed, \eqref{objects:eq-with-two-p11} can be obtained from the definition of $\paired{\{1,2\}}$ and a direct computation. To be more precise, it follows from the definition of $\paired{\{1,2\}}$ that 
\begin{equation}\label{objects:eq-with-two-p12}
\begin{aligned}
&\int_{\indexset} \tensorcon{\{1,2\}} \big( F^{j_1}_{t,x,N_1,k_1}\otimes f^{j_2}_{t,x,N_2,k_2} \big) \lebI \\
=& \sum_{n_1, n_2 \in \Z^2} \int_{\R^2} \bigg( 
\delta^{\dimind_1 \dimind_2} \delta_{n_1+n_2=0} \delta(s_1-s_2)   \delta^{j_1}_{\dimind_1} \delta^{j_2}_{\dimind_2} 
\ind_{(-\infty,t]}(s_1) \ind_{(-\infty,t]}(s_2) 
\rho_{N_1}(n_1) \rho_{N_2}(n_2)  \\
&\times \big( \icomplex (n_1)_{k_1} \big) \big( \icomplex (n_2)_{k_2} \big)
\big(t-\max(s_1,0)\big) e^{-(t-s_1) \langle n_1 \rangle^2} 
e^{-(t-s_2) \langle n_2 \rangle^2}
\tensorcon{ \{ 1 , 2 \}}\big( I_{\cfrkg} \otimes I_{\cfrkg} \big) \bigg) \ds_1 \ds_2. 
\end{aligned}
\end{equation}
In order to simplify the notation, we set $n:=n_1$ and $s:=s_1$. By summing over $i_1,i_2\in [2]$, summing over $n_2 \in \Z^2$, and integrating over $s_2\in \R$, it follows that 
\begin{equation}\label{objects:eq-with-two-p13}
\begin{aligned}
\eqref{objects:eq-with-two-p12}
= \delta^{j_1 j_2} 
\sum_{n\in \Z^2} \int_{\R} \ind_{(-\infty,t]}(s) \rho_{N_1}(n) \rho_{N_2}(n) n_{k_1} n_{k_2} \big( t-\max(s,0)\big) e^{-2(t-s) \langle n \rangle^2} \tensorcon{ \{ 1 , 2 \}}\big( I_{\cfrkg} \otimes I_{\cfrkg} \big) \ds.
\end{aligned}
\end{equation}
By integrating over $s\in \R$, it follows that
\begin{equation}\label{objects:eq-with-two-p14}
\eqref{objects:eq-with-two-p13} = 
\delta^{j_1 j_2} \bigg( \sum_{n\in \Z^2} \rho_{N_1}(n) \rho_{N_2}(n) \frac{n_{k_1} n_{k_2}}{4\langle n \rangle^4} \big(1-e^{-2t\langle n \rangle^2}\big)\bigg) \tensorcon{ \{ 1 , 2 \}}\big( I_{\cfrkg} \otimes I_{\cfrkg} \big).
\end{equation}
We now note that, due to parity, the sum in \eqref{objects:eq-with-two-p14} vanishes if $k_1 \neq k_2$. Using symmetry in the first and second component of $n\in \Z^2$, it follows that\footnote{In a general dimension $d\geq 2$, symmetrization yields the pre-factor $1/d$ instead of $1/2$.}
\begin{align}
 &\sum_{n\in \Z^2} \rho_{N_1}(n) \rho_{N_2}(n) \frac{n_{k_1} n_{k_2}}{4\langle n \rangle^4} \big(1-e^{-2t\langle n \rangle^2}\big) \notag \\ 
=&\,  \delta_{k_1 k_2} \frac{1}{2}  \sum_{n\in \Z^2} \rho_{N_1}(n) \rho_{N_2}(n) \frac{|n|^2}{4\langle n \rangle^4} \big(1-e^{-2t\langle n \rangle^2}\big) \notag \\
=&\,  \frac{1}{4}  \delta_{k_1 k_2} \sum_{n\in \Z^2} \rho_{N_1}(n) \rho_{N_2}(n) \frac{1}{2\langle n \rangle^2} + c_{N_1,N_2}(t) \label{objects:eq-with-two-p15}, 
\end{align}
where 
\begin{equation}\label{objects:eq-c}
c_{N_1,N_2}(t):=-\, \frac{1}{8} \delta_{k_1 k_2}  \sum_{n\in \Z^2} \rho_{N_1}(n) \rho_{N_2}(n) \frac{1}{\langle n \rangle^4} 
- \frac{1}{8}  \delta_{k_1 k_2} \sum_{n\in \Z^2} \rho_{N_1}(n) \rho_{N_2}(n) \frac{1}{\langle n \rangle^2} e^{-2t\langle n \rangle^2}. 
\end{equation}
Since the first summand in \eqref{objects:eq-with-two-p15} coincides with $\frac{1}{4} \delta_{k_1 k_2} \sigma_{N_1,N_2}^2$, the algebraic identity \eqref{objects:eq-with-two-p11} is satisfied, and it remains to prove that $c_{N_1,N_2}(t)$ satisfies the desired estimate \eqref{objects:eq-c-estimate}. The first summand in \eqref{objects:eq-c} is clearly bounded by $N_{\mrm{max}}^{-2}$, which is acceptable. For the second summand in \eqref{objects:eq-c}, we use that $e^{-z}\lesssim z^{-\delta}$ for all $z\geq 0$, which yields that
\begin{equation*}
\Big| \sum_{n\in \Z^2} \rho_{N_1}(n) \rho_{N_2}(n) \frac{1}{\langle n \rangle^2} e^{-2t\langle n \rangle^2}\Big| 
\lesssim t^{-\delta}  \sum_{n\in \Z^2} \rho_{N_1}(n) \rho_{N_2}(n) \frac{1}{\langle n \rangle^{2+2\delta}}
\lesssim t^{-\delta} N_{\mrm{max}}^{-2\delta}.
\end{equation*}
Since $\delta>\nu$, this is acceptable. \\ 

\emph{Step 3: Proof of \eqref{objects:eq-with-two}.}
Using the product formula \eqref{objects:eq-with-two-p4}-\eqref{objects:eq-with-two-p5} and the algebraic identity \eqref{objects:eq-with-two-p11}, it holds that 
\begin{align}
&\Big[\Big[ E, \Duh \Big( \partial_{k_1} \linear[N_1][r][j_1] \Big) \Big]_{\cfrkg} , \partial_{k_2} \linear[N_2][r][j_2] \Big]_{\cfrkg} 
- \frac{1}{4} \delta^{j_1 j_2} \delta_{k_1 k_2} \sigma_{N_1,N_2}^2 \Kil\big(E \big) \notag \\
=& \Big[  \mrm{ad}(E) \Duh \Big( \partial_{k_1} \linear[N_1][r][j_1] \Big)  \otimes  \partial_{k_2} \linear[N_2][r][j_2] \Big]_{\cfrkg} 
- \frac{1}{4}\delta^{j_1 j_2} \delta_{k_1 k_2} \sigma_{N_1,N_2}^2 \Kil\big(E \big) \notag \\
=& \Big[ \int_{\indexset^2} \big( \mrm{ad}(E) F^{j_1}_{t,x,N_1,k_1} \otimes f^{j_2}_{t,x,N_2,k_2} \big) d(W\otimes W) \Big]_{\cfrkg} \label{objects:eq-with-two-p16} \\
+&  \frac{1}{4} \delta^{j_1j_2} \delta_{k_1 k_2} \sigma_{N_1,N_2}^2 
\bigg( \Big[ \tensorcon{ \{1,2\}} \big( \mrm{ad}(E) \otimes I_{\cfrkg} \big) \Big]_{\cfrkg} - \Kil(E) \bigg) \label{objects:eq-with-two-p17} \\
+&  \delta^{j_1 j_2} \delta_{k_1 k_2} c_{N_1,N_2}  \Big[ \tensorcon{ \{1,2\}} \big( \mrm{ad}(E) \otimes I_{\cfrkg} \big) \Big]_{\cfrkg}. \label{objects:eq-with-two-p18}
\end{align}
Due to our estimate \eqref{objects:eq-with-two-p6} from Step 1, the contribution of the non-resonant part \eqref{objects:eq-with-two-p16} is acceptable. Using the trace identity from Lemma \ref{objects:lem-trace-identities}, \eqref{objects:eq-with-two-p17} is identically zero. Finally, since $c_{N_1,N_2}$ satisfies \eqref{objects:eq-c-estimate}, the contribution of \eqref{objects:eq-with-two-p18} is also acceptable. 
\end{proof}

It now only remains to treat the cubic nonlinearity with derivatives, i.e., the terms in \eqref{objects:eq-enhanced-6}. 

\begin{lemma}[Cubic stochastic object with two derivatives]\label{lemma:cubic-stochastic-object-with-two-derivatives} Let $0<\sdecay <\epsilon\ll 1$, 
let $N_1, N_2, N_{12}, N_3 \in \dyadic$ satisfy $N_{12}\sim N_3$ and let $j_1, j_2, j_3, k_2, k_3 \in [2]$. For all $p\geq 1$, it then holds that
\[ 
\label{eq:cubic-stochastic-object-with-two-derivatives} \E\bigg[\Big\| \Big[  P_{N_{12}} \quadratic[N_1, N_2][d][j_1][j_2][k_2], \ptl_{k_3} \linear[N_3][r][j_3]\Big] - \frac{1}{4} \delta_{j_2 j_3} \delta_{k_2 k_3} \sigma^2_{N_{12}, N_2, N_3}  \Kil\big(\linear[N_1][r][j_1]\big)  \Big\|_{C_t^0\Cs_x^{-3\sreg}([0,1]\times \T^2)}^p\bigg]^{1/p} \lesssim p^{3/2} N_{\mrm{max}}^{-3\sdecay}. \]
\end{lemma}

\begin{remark}\label{objects:rem-Wick-d-two-derivatives}
In a general dimension $d\geq 2$, Wick-ordering yields a  prefactor $\frac{1}{2d}$ for the $\Kil (\linear[][r] \hspace{-0.2ex})$-term. 
\end{remark}

The proof of Lemma \ref{lemma:cubic-stochastic-object-with-two-derivatives} occupies the remainder of this subsection. For expository purposes, we first collect one integral identity and one integral estimate in a separate lemma. 

\begin{lemma}\label{objects:lem-derivative-integral}
Let $t, s_1, s_2 \in \R$, $n_1, n_2 \in \Z^2$. Then, we have the integral identity
\begin{equation}\label{objects:eq-derivative-integral-e1}
 \int_\R \alpha_t(s_1, s_2, n_1, n_2) \ind(s_1 \leq t) e^{-(t - s_1) \fnorm{n_1}^2} ds_1 = \frac{\ind(s_2 \leq t)}{2\fnorm{n_1}^2} \int_{s_2}^t ds e^{-(t-s) (\fnorm{n_{12}}^2 + \fnorm{n_1}^2)} e^{-(s-s_2) \fnorm{n_2}^2}.
\end{equation}
Furthermore, if $t\geq s_1$ and $|n_2| >|n_1|$, then we also have the integral estimate 
\begin{equation}\label{objects:eq-derivative-integral-e2} 
\int_{s_1}^t ds (t-s) e^{-(t - s) \fnorm{n_2}^2} e^{-(s - s_1) \fnorm{n_1}^2} \leq  \frac{1}{(\fnorm{n_2}^2 - \fnorm{n_1}^2)^2 } e^{-(t - s_1) \fnorm{n_1}^2}. 
\end{equation}
\end{lemma}

\begin{proof} We first prove the identity \eqref{objects:eq-derivative-integral-e1}. 
From the definition of $\alpha_t$ (Definition \ref{def:quadratic-one-derivative-integrand}), we obtain that
\[\begin{split}
\int_\R & \alpha_t(s_1, s_2, n_1, n_2) \ind(s_1 \leq t) e^{-(t-s_1) \fnorm{n_1}^2} ds_1 = \\
&\int_\R ds_1 \int_\R ds \ind(s_1 \leq t) \ind(s \leq t) \ind(s_1, s_2 \leq s) e^{-(t-s_1) \fnorm{n_1}^2} e^{-(t-s) \fnorm{n_{12}}^2} e^{-(s - s_1) \fnorm{n_1}^2} e^{-(s - s_2) \fnorm{n_2}^2}.
\end{split}\]
This can be further rewritten as 
\[ \int_{s_2}^t ds e^{-(t-s) \fnorm{n_{12}}^2} e^{-(s-s_2) \fnorm{n_2}^2} \int_{-\infty}^s ds_1 e^{-(t + s -2s_1) \fnorm{n_1}^2} = \frac{\ind(s_2 \leq t)}{2\fnorm{n_1}^2} \int_{s_2}^t ds e^{-(t-s) (\fnorm{n_{12}}^2 + \fnorm{n_1}^2)} e^{-(s-s_2) \fnorm{n_2}^2}, \]
which yields the desired identity. It now remains to prove the integral estimate \eqref{objects:eq-derivative-integral-e2}. We first write 
\begin{equation*}
\int_{s_1}^t ds (t-s) e^{-(t - s) \fnorm{n_2}^2} e^{-(s - s_1) \fnorm{n_1}^2} 
= e^{-(t-s_1)\langle n_1 \rangle^2} \int_{s_1}^t \ds \, (t-s) e^{-(t-s)\big( \langle n_2 \rangle^2 - \langle n_1 \rangle^2\big)}. 
\end{equation*}
Since $|n_2|>|n_1|$, it holds that $\langle n_2 \rangle^2 > \langle n_1 \rangle^2$, and thus 
\begin{equation*}
\int_{s_1}^t \ds \, (t-s) e^{-(t-s)\big( \langle n_2 \rangle^2 - \langle n_1 \rangle^2\big)} 
\leq \int_{-\infty}^t \ds \, (t-s) e^{-(t-s)\big( \langle n_2 \rangle^2 - \langle n_1 \rangle^2\big)} = \frac{1}{(\fnorm{n_2}^2 - \fnorm{n_1}^2)^2 }, 
\end{equation*}
which yields the desired estimate. 
\end{proof}

Equipped with Lemma \ref{objects:lem-derivative-integral}, we are now ready to control the cubic nonlinearity with derivatives.

\begin{proof}[Proof of Lemma \ref{lemma:cubic-stochastic-object-with-two-derivatives}:]
Recall that (by Lemma \ref{lemma:quadratic-one-derivative-multiple-stochastic-integral-representation} and Definition \ref{def:quadratic-one-derivative-integrand})
\begin{equation*}
    \quadratic[N_1, N_2][d][j_1][j_2][k_2](t, x) = \int_{\indexset^2} q^{j_1 j_2 k_2}_{t, x, N_1, N_2} dW^{\otimes 2} \qquad \text{and} \qquad 
\ptl^{k_3} \linear[N_3][r][j_3](t, x) = \int_\indexset \ptl^{k_3} f^{j_3}_{t, x, N_3} dW,
\end{equation*} 
where
\begin{align*}
 q^{j_1 j_2 k_2}_{t, x, N_1, N_2}(z_1, z_2) &= \icomplex \delta^{j_1}_{\ell_1} \delta^{j_2}_{\ell_2} n_{2}^{k_2} \rho_{N_1}(n_1) \rho_{N_2}(n_2) \e_{n_1 + n_2}(x) \alpha_t(s_1, s_1, n_1, n_2) [I_{\cfrkg} \otimes I_{\cfrkg}]_{\cfrkg}, \\
 \alpha_t(s_1, s_2, n_1, n_2) &= \int_{-\infty}^t \ind(s_1, s_2 \leq s) e^{-(t - s) \fnorm{n_{12}}^2} e^{-(s - s_1) \fnorm{n_1}^2} e^{-(s - s_2) \fnorm{n_2}^2} ds, 
 \end{align*}
 and 
 \begin{align*}
 \ptl^{k_3} f^{j_3}_{t, x, N_3}(z_3) &= \icomplex \delta^{j_3}_{\ell_3} n_{3}^{k_3} \e_{n_3}(x) \rho_{N_3}(n_3) \ind(s_3 \leq t) e^{-(t-s_3)\fnorm{n_3}^2} I_{\cfrkg}.
\end{align*}
For brevity, define the function $\tilde{h}_{t, x} : \indexset^3 \ra \cfrkg^{*\otimes 3} \otimes \cfrkg^2$ by $\tilde{h}_{t, x}(z_1, z_2, z_3) := (P_{N_{12}} q^{j_1 j_2 k_2}_{t, x, N_1, N_2}) \otimes \ptl^{k_3} f^{j_3}_{t, x, N_3} $. Explicitly, we have that
\[ \begin{split}
\tilde{h}_{t, x}(z_1, z_2, z_3) =- \big([I_{\cfrkg} \otimes I_{\cfrkg}] \otimes I_{\cfrkg} \big)\delta^{j_1}_{\ell_1} &\delta^{j_2}_{\ell_2} \delta^{j_3}_{\ell_3} n_{2}^{k_2} n_{3}^{k_3} \rho_{N_1}(n_1) \rho_{N_2}(n_2) \rho_{N_{12}}(n_{12}) \rho_{N_3}(n_3) ~\times \\
&\e_{n_{123}}(x) \alpha_t(s_1, s_2, n_1, n_2) \ind(s_3 \leq t) e^{-(t-s_3) \fnorm{n_3}^2}.
\end{split} \]
By the product formula (Lemma \ref{lemma:multiple-stochastic-integral-tensor-product}), we have that
\[\begin{split}\label{objects:eq-cubic-derivative-p0}
\bigg(P_{N_{12}} \quadratic[N_1, N_2][d][j_1][j_2][k_2](t, x) \bigg) \otimes \ptl^{k_3} \linear[N_3][r][j_3](t, x) = \int_{\indexset^3} \tilde{h}_{t, x} dW^{\otimes 3} + \int_{\indexset} \paired{\{1,3\}} (\tilde{h}_{t, x})(z_2) dW(z_2) + \int_{\indexset} \paired{\{2,3\}}(\tilde{h}_{t, x})(z_1) dW(z_1), 
\end{split}\]
where $\paired{\{1,3\}}$ and $\paired{\{2,3\}}$ are as in Definition \ref{def:contracted-function}.  
Define $h_{t, x} : \indexset^3 \ra \cfrkg^{*\otimes 3} \otimes \cfrkg$ by $h_{t, x} := [\tilde{h}_{t, x}]_{\cfrkg}$.
Observe that $[\paired{\{1,3\}}(\tilde{h}_{t, x})]_{\cfrkg} = \paired{\{1,3\}}(h_{t, x})$ and similarly for the $\{2, 3\}$-pairing (this essentially follows by linearity). We then have that
\[ \Big[  P_{N_{12}} \quadratic[N_1, N_2][d][j_1][j_2][k_2], \ptl^{k_3} \linear[N_3][r][j_3]\Big](t, x) = \int_{\indexset^3} h_{t, x} dW^{\otimes 3} + \int_{\indexset} \paired{\{1,3\}}(h_{t,x})(z_2) dW(z_2) + \int_{\indexset} \paired{\{2,3\}}(h_{t,x})(z_1) dW(z_1). \]

\emph{Step 1: Estimate of the non-resonant part.}
In this step, we show that
\[ \E\bigg[ \bigg\| \int_{\indexset^3} h_{t, x} dW^{\otimes 3} \bigg\|_{C_t^0\Cs_x^{-3\sreg}([0,1]\times \T^2)}^p\bigg]^{1/p} \lesssim p^{3/2} N_{\max}^{-3\sdecay}.  \]
Using our standard reduction (Lemma \ref{lemma:standard-reduction}),  it suffices to show for all $N_0 \in \dyadic$ that 
\[ \sup_{t\in [0,1]} \sup_{x\in \T^2} \E \bigg[ \bigg|P_{N_0}\int_{\indexset^3} h_{t, x} dW^{\otimes 3} \bigg|_{\cfrkg}^2\bigg] \lesssim N_0^{6\sreg} N_{\max}^{-6\sreg}.\]
Using Lemma \ref{lemma:multiple-stochastic-integral-second-moment-bound-l2} and \eqref{objects:eq-cubic-derivative-p0}, we have that
\begin{align}
 \E \bigg[ \bigg|P_{N_0}\int_{\indexset^3} h_{t, x} dW^{\otimes 3} \bigg|_{\cfrkg}^2\bigg] &\lesssim \int_{\indexset^3} \big| P_{N_0} h_{t, x} \big|_{\cfrkg^{*\otimes 3} \otimes \cfrkg}^2 d\lebI^3 \notag \\ 
 &\lesssim  \sum_{n_1, n_2, n_3\in \Z^2} \bigg( |n_2|^2 |n_3|^2 \rho_{N_0}(n_{123}) \rho_{N_1}(n_1) \rho_{N_2}(n_2) \rho_{N_{12}}(n_{12}) \rho_{N_3}(n_3) 
 \label{objects:eq-cubic-derivative-p1} \\
&~~\times \int_{\R^3} \alpha_t(s_1, s_2, n_1, n_2)^2 \ind(s_3 \leq t) e^{-2(t - s_3) \fnorm{n_3}^2} ds_1 ds_2 ds_3 \bigg). \notag 
\end{align}
Applying inequality \eqref{eq:alpha-t-squared-integral-bound}, we further obtain
\[ \eqref{objects:eq-cubic-derivative-p1} \lesssim \sum_{n_1, n_2, n_3\in \Z^2} \frac{1}{\fnorm{n_1}^2 \fnorm{n_{12}}^2} \min\big(\fnorm{n_1}^{-2}, \fnorm{n_2}^{-2}, \fnorm{n_{12}}^{-2}\big) \rho_{N_0}(n_{123}) \rho_{N_1}(n_1) \rho_{N_2}(n_2) \rho_{N_{12}}(n_{12}) \rho_{N_3}(n_3) . \] 
By using inequality \eqref{eq:cubic-two-derivative-combinatorial-estimate} and $N_{12} \sim N_3$, the right hand side above is bounded by $N_0^{6\sreg} N_{\max}^{-6\sreg}$, as desired. \\

\emph{Step 2: Estimate of the $\{1, 3\}$ resonance.}
In this step, we show that the $\{1, 3\}$ resonance can be controlled (and therefore does not contribute to the renormalization). For brevity, define $F(t, x)$ by
\[ F(t, x) := \int_{\indexset} \paired{\{1,3\}}(h_{t,x})(z_2) dW(z_2). \]
We will show that
\beq\label{eq:cubic-two-derivative-1-3-resonance-desired-estimate} \E\big[ \|F(t)\|_{C_t^0 \Cs_x^{-3\sreg}([0,1]\times \T^2)}^p\big]^{1/p} \lesssim p^{1/2} N_{\max}^{-3\sdecay}. \eeq
We compute (note that $\tensorcon{\{1, 3\}}([I_{\cfrkg} \otimes I_{\cfrkg}]_{\cfrkg} \otimes I_{\cfrkg}) = -\Kil$)
\[\begin{split}
\paired{\{1,3\}}(h_{t,x})(z_2) = - \delta^{j_2}_{\ell_2} n_{2}^{k_2}  \e_{n_2}(x) \rho_{N_2}(n_2)  \sum_{n_1\in \Z^2} \bigg(  & \delta^{\ell_1 \ell_3} \delta^{j_1}_{\ell_1} \delta^{j_3}_{\ell_3}  n_{1}^{k_3} \rho_{N_1}(n_1) \rho_{N_{12}}(n_{12}) \rho_{N_3}(n_1) ~\times \\
&\int_\R ds_1 \alpha_t(s_1, s_2, n_1, n_2) \ind(s_1 \leq t) e^{-(t - s_1) \fnorm{n_1}^2}  \bigg) \Kil.
\end{split}\]
Since $n_2$ is the only Fourier mode which appears in the above, to show \eqref{eq:cubic-two-derivative-1-3-resonance-desired-estimate} it suffices to show  (after using Lemma \ref{lemma:standard-reduction}) that 
\[ \sup_{t\in [0,1]} \sup_{x\in \T^2} \int_\indexset |\paired{\{1,3\}}(h_{t,x})(z_2)|_{\cfrkg^* \otimes \cfrkg}^2 d\lebI(z_2) \lesssim N_{\max}^{-2} N_{\min}^{2}.\] 
Towards this end, first recall from Lemma \ref{objects:lem-derivative-integral} that 
\[ \int_\R \alpha_t(s_1, s_2, n_1, n_2) \ind(s_1 \leq t) e^{-(t - s_1) \fnorm{n_1}^2} ds_1 = \frac{\ind(s_2 \leq t)}{2\fnorm{n_1}^2} \int_{s_2}^t ds e^{-(t-s) (\fnorm{n_{12}}^2 + \fnorm{n_1}^2)} e^{-(s-s_2) \fnorm{n_2}^2}.\]
We now split into two cases: $N_2 \ll N_1$ and $N_2 \gtrsim N_1$. In the first case, for $n_1, n_2\in \Z^2$ satisfying $|n_1|\sim N_1$ and $|n_2|\sim N_2$, we may bound
\[\begin{split} \int_{s_2}^t ds e^{-(t-s) (\fnorm{n_{12}}^2 + \fnorm{n_1}^2)} e^{-(s-s_2) \fnorm{n_2}^2} &\leq \int_{s_2}^t ds e^{-(t-s) \fnorm{n_1}^2} e^{-(s - s_2) \fnorm{n_2}^2} \\
&= \frac{1}{\fnorm{n_1}^2 - \fnorm{n_2}^2} \Big(e^{-(t- s_2) \fnorm{n_2}^2} - e^{-(t-s_2) \fnorm{n_1}^2}\Big) \\
&\leq \frac{1}{\fnorm{n_1}^2 - \fnorm{n_2}^2} e^{-(t-s_2) \fnorm{n_2}^2} \lesssim \frac{1}{\fnorm{n_1}^2}e^{-(t-s_2) \fnorm{n_2}^2}  .
\end{split}\]
From this, we obtain
\[\begin{split} 
\big| \paired{\{1,3\}}(h_{t,x})(z_2) \big|_{\cfrkg^* \otimes \cfrkg} &\lesssim \ind(s_2 \leq t) e^{-(t - s_2) \fnorm{n_2}^2}  \sum_{n_1\in \Z^2} \frac{|n_1|}{\fnorm{n_1}^4} |n_2|  \rho_{N_1}(n_1) 
\rho_{N_2}(n_2) \\
&\leq \ind(s_2 \leq t) e^{-(t-s_2) \fnorm{n_2}^2}  N_1^{-1} N_2  \rho_{N_2}(n_2).
\end{split}\]
We thus obtain the upper bound
\begin{align*}
\int_{\indexset} 
\big|\paired{\{1,3\}}(h_{t,x})(z_2)\big|_{\cfrkg^* \otimes\cfrkg}^2 d\lebI(z_2) 
&\lesssim  N_1^{-2} N_2^{2} \sum_{n_2\in \Z^2} \rho_{N_2}(n_2) \int_{-\infty}^t \ds_2 e^{-(t-s_2) \langle n_2 \rangle^2} \\ 
&\lesssim N_1^{-2} N_2^2 \sum_{n_2 \in \Z^2} \frac{1}{\fnorm{n_2}^2}\rho_{N_2}(n_2)   \\
&\lesssim N_1^{-2} N_2^2 = N_{\max}^{-2} N_{\min}^{2}, 
\end{align*}
where the final equality follows because we are in the case $N_2 \ll N_1$. It now remains to treat the case $N_2 \gtrsim N_1$. To this end, let $c>0$ be a sufficiently small absolute constant. For all $n_1,n_2\in \Z^2$ satisfying $|n_1|\sim N_1$ and $|n_2|\sim N_2$, we then bound  
\[\begin{split}
\int_{s_2}^t ds e^{-(t-s) (\fnorm{n_{12}}^2 + \fnorm{n_1}^2)} e^{-(s-s_2) \fnorm{n_2}^2} &\leq \int_{s_2}^t ds e^{-c(t-s) \fnorm{n_2}^2} e^{-c(s - s_2) \fnorm{n_2}^2} = (t - s_2) e^{-c(t - s_2) \fnorm{n_2}^2}.
\end{split}\]
We then obtain the bound
\[\begin{split} \big| \paired{\{1,3\}}(h_{t,x})(z_2) \big|_{\cfrkg^* \otimes \cfrkg} &\lesssim \ind(s_2 \leq t) (t - s_2) e^{-c(t-s_2) \fnorm{n_2}^2} |n_2| \rho_{N_2}(n_2) \sum_{n_1 \in \Z^2} \frac{|n_1|}{\fnorm{n_1}^2} \rho_{N_1}(n_1) \\
&\leq \ind(s_2 \leq t) (t-s_2) e^{-c(t-s_2) \fnorm{n_2}^2} N_1 N_2 \rho_{N_2}(n_2).
\end{split}\]
We thus obtain
\begin{align*}
 \int_\indexset \big|\paired{\{1,3\}}(h_{t,x})(z_2)\big|_{\cfrkg^* \otimes \cfrkg}^2 d\lebI(z_2) 
 &\lesssim N_1^2 N_2^2 \sum_{n_2 \in \Z^2} \rho_{N_2}(n_2) \int_{-\infty}^t \ds_2 (t-s_2)^2 e^{-2c(t-s_2)\langle n_2 \rangle^2} \\ 
 &\lesssim N_1^2 N_2^2 \sum_{n_2\in \Z^2} \rho_{N_2}(n_2) \frac{1}{\fnorm{n_2}^6} \\ &\lesssim N_1^2 N_2^{-2}\lesssim  N_{\min}^{2} N_{\max}^{-2},  
\end{align*}
where recall $N_2 \gtrsim N_1$ for the final inequality. \\ 

\emph{Step 3: Analysis of the $\{2, 3\}$ resonance.} 
In this step, we show that the $\{2, 3\}$ resonance contributes the renormalization in \eqref{eq:cubic-stochastic-object-with-two-derivatives}.  For brevity, define $G(t, x)$ by
\beq\label{eq:cubic-two-derivative-G-t-x-def} G(t, x) := \int_\indexset \paired{\{2,3\}}(h_{t,x})(z_1) dW(z_1). \eeq
We will show that
\beq\label{eq:cubic-two-derivative-2-3-resonance-desired-estimate} \E\Big[\Big\|G(t) - \frac{1}{4} \delta^{j_2 j_3} \delta^{k_2 k_3} \sigma^2_{N_{12}, N_2, N_3} \Kil\big(\linear[N_1][r][j_1](t)\big)  \Big\|_{C_t^0\Cs_x^{-3\sreg}([0,1]\times \T^2)}^p\Big]^{1/p} \lesssim p^{1/2} N_{\max}^{-3\sdecay}. \eeq 
Note that $n_1$ is the only Fourier mode which appears in the expression for $\paired{\{2,3\}}(h_{t,x})(z_1)$. From this observation, combined with essentially the same argument as for the $\{1, 3\}$ resonance, we have that when $N_2 \lesssim N_1$, 
\[ \E\big[ \|G(t)\|_{C_t^0\Cs_x^{-3\sreg}([0,1]\times \T^2)}^p\big]^{1/p} \lesssim p^{1/2} N_{\max}^{-3\sdecay}.\]
Noting that $\sigma^2_{N_{12}, N_2, N_3} \lesssim \log N_2$ and applying Lemma \ref{lemma:linear-object} with $(\sreg,\sdecay)$ replaced by $(3\sreg,3\sdecay+(\sreg-\sdecay))$, we further have that when $N_2 \lesssim N_1 $, 
\[ \E\Big[\Big\|\sigma^2_{N_{12}, N_2, N_3}  \Kil\big(\linear[N_1][r][j_1](t)\big) \Big\|_{C_t^0 \Cs_x^{-3\sreg}([0,1]\times \T^2)}^p\Big]^{1/p} \lesssim p^{1/2} \log(N_{\max})  N_{\max}^{\nu-\epsilon} N_{\max}^{-3\nu} \lesssim p^{\frac{1}{2}} N_{\max}^{-3\nu} . \]
Thus when $N_2 \lesssim N_1$, we have that \eqref{eq:cubic-two-derivative-2-3-resonance-desired-estimate} holds. For the rest of this proof, we therefore assume that $N_2 \gg N_1$. Using our standard reductions (Lemma \ref{lemma:standard-reduction}), to show \eqref{eq:cubic-two-derivative-2-3-resonance-desired-estimate}, it suffices to show that \begin{equation}\label{objects:eq-two-derivatives-G-bound}
 \sup_{t\in [0,1]} \sup_{x\in \T^2} \E\Big[ \Big| G(t, x) - \frac{1}{4} \delta^{j_2 j_3} \delta^{k_2 k_3} \sigma^2_{N_{12}, N_2, N_3} \Kil\big(\linear[N_1][r][j_1](t, x)\big)  \Big|_{\cfrkg}^2\Big] \lesssim N_{\min}^2 N_{\max}^{-2}.    
\end{equation}  
We compute
\[\begin{split}
\paired{\{2,3\}}(h_{t,x})(z_1) = \delta^{j_1}_{\ell_1} \delta^{\ell_2 \ell_3} \delta^{j_2}_{\ell_2} \delta^{j_3}_{\ell_3}  \e_{n_1}(x) \rho_{N_1}(n_1)  &  \sum_{n_2\in \Z^2} \bigg( n_{2}^{k_2} n_{2}^{k_3} \rho_{N_2}(n_2) \rho_{N_{12}}(n_{12}) \rho_{N_3}(n_2) ~\times \\
&\int_{-\infty}^t ds_2 \alpha_t(s_1, s_2, n_1, n_2) e^{-(t-s_2) \fnorm{n_2}^2} \big) \Kil.
\end{split}\]
Note that $  \delta^{\ell_2 \ell_3} \delta^{j_2}_{\ell_2} \delta^{j_3}_{\ell_3}  = \delta^{j_2 j_3}$. Using Lemma \ref{objects:lem-derivative-integral} and the symmetry of $\alpha_t$ (see Definition \ref{def:quadratic-one-derivative-integrand}), we also have that 
\[ \int_{-\infty}^t ds_2 \alpha_t(s_1, s_2, n_1, n_2) e^{-(t-s_2)\fnorm{n_2}^2}  = \frac{\ind(s_1 \leq t)}{2\fnorm{n_2}^2} \int_{s_1}^t ds e^{-(t-s)(\fnorm{n_{12}}^2 + \fnorm{n_2}^2)} e^{-(s - s_1) \fnorm{n_1}^2}.\]
We may thus write
\beq\label{eq:h-t-x-2-3-expression} \paired{\{2,3\}}(h_{t,x})(z_1) = \frac{1}{4} \delta^{j_1}_{\ell_1} \delta^{j_2 j_3} \e_{n_1}(x)  \ind(s_1 \leq t) g(s_1, n_1) \Kil, \eeq
where
\begin{equation}\label{objects:eq-two-derivatives-g}
g(s_1, n_1) := 2 \rho_{N_1}(n_1) \sum_{n_2 \in \Z^2} n_{2}^{k_2} n_{2}^{k_3} \rho_{N_2}(n_2) \rho_{N_{12}}(n_{12}) \rho_{N_3}(n_2) \frac{1}{\fnorm{n_2}^2} \int_{s_1}^t ds e^{-(t-s)(\fnorm{n_{12}}^2 + \fnorm{n_2}^2)} e^{-(s - s_1) \fnorm{n_1}^2}.   
\end{equation}
We also recall from Example \ref{prelim:example-linear} that 
\begin{equation}\label{objects:eq-two-derivatives-Ric}
\delta^{j_2 j_3} \delta^{k_2k_3} \sigma_{N_{12},N_2,N_2}^2 \Kil \big( \linear[N_1][r][j_1](t) \big) 
=  \int_{\indexset} \big( \delta^{j_1}_{\ell_1} \delta^{j_2 j_3} \e_{n_1}(x)  \ind(s_1 \leq t) \widetilde{g}(s_1,n_1) \Kil \big) \mathrm{d}W(z_1),
\end{equation}
where 
\begin{equation}\label{objects:eq-two-derivatives-gtilde}
\widetilde{g}(s_1,n_1) = \rho_{N_1}(n_1) \delta^{k_2 k_3} \sigma_{N_{12},N_2,N_3}^2 e^{-(t-s_1)\langle n_1 \rangle^2}. 
\end{equation}
By combining \eqref{eq:cubic-two-derivative-G-t-x-def}, \eqref{eq:h-t-x-2-3-expression}, and \eqref{objects:eq-two-derivatives-Ric}, 
it then follows that 
\begin{equation*}
 \E\Big[ \Big| G(t, x) -\frac{1}{4} \delta^{j_2 j_3} \delta^{k_2 k_3} \sigma^2_{N_{12}, N_2, N_3} \Kil\big(\linear[N_1][r][j_1](t, x)\big)  \Big|_{\cfrkg}^2\Big] 
 \lesssim \sum_{n_1\in \Z^2} \int_{-\infty}^t \ds_1 \big| g(s_1,n_1) - \widetilde{g}(s_1,n_1)\big|^2. 
\end{equation*}
In order to prove \eqref{objects:eq-two-derivatives-G-bound}, it therefore suffices to prove that
\begin{equation}\label{objects:eq-two-derivatives-g-bound}
\sum_{n_1\in \Z^2} \int_{-\infty}^t \ds_1 \big| g(s_1,n_1) - \widetilde{g}(s_1,n_1)\big|^2 \lesssim N_{\textup{min}}^2 N_{\textup{max}}^{-2}. 
\end{equation}
To this end, we introduce two different approximations of $g$, which are given by 
\begin{align}
g_1(s_1, n_1) &:= 2 \rho_{N_1}(n_1) \sum_{n_2\in \Z^2} n_{2}^{k_2} n_{2}^{k_3} \rho_{N_2}(n_2) \rho_{N_{12}}(n_{12}) \rho_{N_3}(n_2) \frac{1}{\fnorm{n_2}^2} \int_{s_1}^t ds e^{-2(t-s) \fnorm{n_2}^2} e^{-(s - s_1) \fnorm{n_1}^2}, 
\label{objects:eq-two-derivatives-g1} \\ 
g_2(s_1, n_1) &:= 2 \rho_{N_1}(n_1) e^{-(t - s_1) \fnorm{n_1}^2}\sum_{n_2 \in \Z^2} n_{2}^{k_2} n_{2}^{k_3} \rho_{N_2}(n_2) \rho_{N_{12}}(n_{12}) \rho_{N_3}(n_2) \frac{1}{2\fnorm{n_2}^4} .
\label{objects:eq-two-derivatives-g2} 
\end{align}
We went from $g$ to $g_1$ by replacing $\langle n_{12} \rangle^2$ in the exponent with $\langle n_2 \rangle^2$. As will be more clear from the estimates below, we went from $g_1$ to $g_2$ via integration by parts and keeping only a boundary term. In order to prove \eqref{objects:eq-two-derivatives-g-bound}, it now suffices to prove the following three estimates: 
\begin{align}
\big| g(s_1,n_1) - g_1(s_1,n_1) \big| &\lesssim N_1 N_2^{-1} \rho_{N_1}(n_1) e^{-(t-s_1)\langle n_1 \rangle^2}, 
\label{objects:eq-two-derivatives-approx-1}\\ 
\big| g_1(s_1,n_1) - g_2(s_1,n_1) \big| &\lesssim \rho_{N_1}(n_1) 
\Big( N_1^2 N_2^{-2} e^{-(t-s_1)\langle n_1 \rangle^2} + e^{-c(t-s_1) N_2^2} \Big), 
\label{objects:eq-two-derivatives-approx-2} \\ 
\big| g_2(s_1,n_1) - \widetilde{g}(s_1,n_1) \big| &\lesssim N_1 N_2^{-1} \rho_{N_1}(n_1) e^{-(t-s_1)\langle n_1 \rangle^2}. 
\label{objects:eq-two-derivatives-approx-3}
\end{align}
\emph{Proof of \eqref{objects:eq-two-derivatives-approx-1}:} From the definitions of $g$ and $g_1$, it follows that
\begin{align*}
(g-g_1)(s_1,n_1) 
&= 2 \rho_{N_1}(n_1) \sum_{n_2\in \Z^2} \bigg( n_{2}^{k_2} n_{2}^{k_3} \rho_{N_2}(n_2) \rho_{N_{12}}(n_{12}) \rho_{N_3}(n_2) \frac{1}{\fnorm{n_2}^2} \\
&\quad  \times  \int_{s_1}^t \ds \Big( e^{-(t-s) \fnorm{n_2}^2} 
\big( e^{-(t-s)\langle n_2 \rangle^2} - e^{-(t-s)\langle n_{12} \rangle^2} \big) 
e^{-(s - s_1) \fnorm{n_1}^2} \Big) \bigg). 
\end{align*}
By first using the elementary inequality $|e^{-u}-e^{-v}|\leq |u-v|$ for all $u,v\geq 0$ and using that, since $N_2 \gg N_1$, $|\langle n_{12} \rangle^2 - \langle n_2 \rangle^2|\lesssim N_1N_2$, it follows that
\begin{equation}\label{objects:eq-two-derivatives-approx-1-p1}
\big| (g-g_1)(s_1,n_1)  \big| 
\lesssim N_1 N_2 \rho_{N_1}(n_1) \sum_{n_2 \in \Z^2} \rho_{N_2}(n_2) \int_{s_1}^t \ds (t-s) e^{-(t-s) \fnorm{n_2}^2} e^{-(s - s_1) \fnorm{n_1}^2}. 
\end{equation}
Using $N_2 \gg N_1$ and Lemma \ref{objects:lem-derivative-integral}, we obtain that
\begin{equation*}
\eqref{objects:eq-two-derivatives-approx-1-p1} \lesssim N_1 N_2^{-3} \rho_{N_1}(n_1) \Big( \sum_{n_2 \in \Z^2} \rho_{N_2}(n_2) \Big) e^{-(t-s_1)\langle n_1 \rangle^2} 
\lesssim N_1 N_2^{-1} \rho_{N_1}(n_1)e^{-(t-s_1)\langle n_1 \rangle^2}, 
\end{equation*}
which yields the desired estimate. \\

\emph{Proof of \eqref{objects:eq-two-derivatives-approx-2}:} Using integration by parts, it holds that 
\begin{align}
&\qquad (g_1-g_2)(s_1,n_1) \notag \\ 
&= 2 \rho_{N_1}(n_1) \sum_{n_2\in \Z^2} \bigg( n_{2}^{k_2} n_{2}^{k_3} \rho_{N_2}(n_2) \rho_{N_{12}}(n_{12}) \rho_{N_3}(n_2) \frac{1}{\fnorm{n_2}^2} \notag \\
&\quad \times \Big( \int_{s_1}^t \ds e^{-2(t-s) \fnorm{n_2}^2} e^{-(s - s_1) \fnorm{n_1}^2} - \frac{1}{2\langle n_2\rangle^2} e^{-(t-s_1) \langle n_1 \rangle^2} \Big) \bigg) \notag \\
&= - 2 \rho_{N_1}(n_1) \sum_{n_2\in \Z^2} \bigg( n_{2}^{k_2} n_{2}^{k_3} \rho_{N_2}(n_2) \rho_{N_{12}}(n_{12}) \rho_{N_3}(n_2) \frac{1}{2\langle n_2 \rangle^4} e^{-2(t-s_1)\langle n_2 \rangle^2} \bigg) \label{objects:eq-two-derivatives-approx-2-p1} \\ 
&+ 2  \rho_{N_1}(n_1)  \sum_{n_2\in \Z^2} \bigg( n_{2}^{k_2} n_{2}^{k_3} \rho_{N_2}(n_2) \rho_{N_{12}}(n_{12}) \rho_{N_3}(n_2) \frac{\langle n_1 \rangle^2}{\langle n_2 \rangle^4} \int_{s_1}^t \ds e^{-2(t-s)\langle n_2 \rangle^2} e^{-(s-s_1)\langle n_1 \rangle^2} \bigg). \label{objects:eq-two-derivatives-approx-2-p2} 
\end{align}
We estimate \eqref{objects:eq-two-derivatives-approx-2-p1} and \eqref{objects:eq-two-derivatives-approx-2-p2} separately. For the first term, we first estimate 
\begin{equation*}
e^{-2(t-s)\langle n_2 \rangle^2} \leq e^{-c(t-s_1)N_2^2}
\end{equation*}
and then sum in $n_2 \in \Z^2$, which yields 
\begin{equation*}
\big| \eqref{objects:eq-two-derivatives-approx-2-p1} \big| \lesssim \rho_{N_1}(n_1) e^{-c(t-s_1)N_2^2}. 
\end{equation*}
For the second term, we first use Lemma \ref{objects:lem-derivative-integral} and then sum in $n_2\in \Z^2$, which yields 
\begin{equation*}
\big| \eqref{objects:eq-two-derivatives-approx-2-p2} \big| 
\lesssim N_1^2 N_2^{-2} \rho_{N_1}(n_1) e^{-(t-s_1)\langle n_1 \rangle^2}. 
\end{equation*}

\emph{Proof of \eqref{objects:eq-two-derivatives-approx-3}:} 
Using the definition of $\sigma_{N_{12},N_2,N_3}^2$, writing $\langle n_2 \rangle^{-2} = (|n_2|^2+1) \langle n_2 \rangle^{-4}$, and using the symmetry in the first and second component of $n_2\in \Z^2$,  it follows that 
\begin{align*}
\widetilde{g}(s_1,n_1) 
&= \delta^{k_2 k_3} \rho_{N_1}(n_1) e^{-(t-s_1)\langle n_1\rangle^2} 
\sum_{n_2\in \Z^2}  \rho_{N_{12}}(n_2) \rho_{N_2}(n_2) \rho_{N_3}(n_3) \frac{1}{2\langle n_2 \rangle^2} \\ 
&=  \rho_{N_1}(n_1) e^{-(t-s_1)\langle n_1\rangle^2} 
\sum_{n_2\in \Z^2} n_2^{k_2} n_2^{k_3} \rho_{N_{12}}(n_2) \rho_{N_2}(n_2) \rho_{N_3}(n_3) \frac{1}{\langle n_2 \rangle^4} \\ 
&+ \delta^{k_2 k_3} \rho_{N_1}(n_1) e^{-(t-s_1)\langle n_1\rangle^2} 
\sum_{n_2\in \Z^2} \rho_{N_{12}}(n_2)  \rho_{N_3}(n_3) \rho_{N_2}(n_2) \frac{1}{2\langle n_2 \rangle^4}. 
\end{align*}
Using the definition of $g_2$ from \ref{objects:eq-two-derivatives-g2}, it then follows that
\begin{align}
(g_2-\widetilde{g})(s_1,n_1) 
&= \rho_{N_1}(n_1) e^{-(t-s_1)\langle n_1\rangle^2} 
\sum_{n_2\in \Z^2} n_2^{k_2} n_2^{k_3} 
\big( \rho_{N_{12}}(n_{12})-\rho_{N_{12}}(n_2) \big)
\rho_{N_2}(n_2) \rho_{N_3}(n_3) \frac{1}{\langle n_2 \rangle^4} 
\label{objects:eq-two-derivatives-approx-3-p1} \\
&- \delta^{k_2 k_3} \rho_{N_1}(n_1) e^{-(t-s_1)\langle n_1\rangle^2} 
\sum_{n_2\in \Z^2} \rho_{N_{12}}(n_2)  \rho_{N_3}(n_3) \rho_{N_2}(n_2) \frac{1}{2\langle n_2 \rangle^4} \label{objects:eq-two-derivatives-approx-3-p2}.
\end{align}
Since $N_2\gg N_1$, we have the Lipschitz estimate
\begin{equation*}
\big|  \rho_{N_{12}}(n_{12})-\rho_{N_{12}}(n_2) \big| \lesssim N_{12}^{-1} \big| n_{12}-n_2 \big| \lesssim N_2^{-1} N_1. 
\end{equation*}
Using direct estimates of the sums in \eqref{objects:eq-two-derivatives-approx-3-p1} and \eqref{objects:eq-two-derivatives-approx-3-p2}, it then follows that
\begin{equation*}
\big| \eqref{objects:eq-two-derivatives-approx-3-p1} \big| 
\lesssim N_1 N_2^{-1} \rho_{N_1}(n_1) e^{-(t-s_1)\langle n_1\rangle^2}  
\quad \text{and} \quad 
\big| \eqref{objects:eq-two-derivatives-approx-3-p2} \big| 
\lesssim N_2^{-2} \rho_{N_1}(n_1) e^{-(t-s_1)\langle n_1\rangle^2}, 
\end{equation*}
which are acceptable.
\end{proof}

\subsection{Proof of Proposition \ref{objects:prop-enhanced}}
\label{section:objects-proof}

\begin{proof}[Proof of Proposition \ref{objects:prop-enhanced}:]
As mentioned above, this follows directly by choosing $(\epsilon,\nu)=(\kappa,\eta)$ and then using the estimates of this section. To be more precise, the terms in \eqref{objects:eq-enhanced-1}-\eqref{objects:eq-enhanced-5} are estimated using the following lemmas and corollaries: 
\begin{enumerate}[label=(\alph*)]
    \item $\linear[\leqN][r]$ is estimated using Lemma \ref{lemma:linear-object}, 
    \item $\big[ \linear[\leqN][r], \linear[\leqN][r]\big]$ is estimated using Lemma \ref{objects:lem-without}, 
    \item the renormalized version of $E\in \frkg \mapsto \big[ \big[ E , \linear[\leqN][r] \big], \linear[\leqN][r] \big]$ is estimated using Lemma \ref{objects:lem-without}, 
    \item the renormalized version of $ \big[ \big[ \linear[\leqN][r] , \linear[\leqN][r] \big], \linear[\leqN][r] \big]$ is estimated using Lemma \ref{lemma:cubic-stochastic-object}, 
    \item $\quadratic[\leqN][r]$ is estimated using Lemma \ref{lemma:quadratic-stochastic-object-with-derivative} and Corollary \ref{objects:cor-quadratic-object-time}, 
    \item and the renormalized version of $E \mapsto \big[ \big[ E, \Duh (\partial \hspace{0.1ex} \linear[\leqN][r])\big]\parasim \partial \hspace{0.1ex} \linear[\leqN][r]\big]$ is estimated using Lemma \ref{objects:lem-with-two}.
\end{enumerate}
Thus, it remains to treat the terms in \eqref{objects:eq-enhanced-6}, which are given by
\begin{equation}\label{objects:eq-enhanced-p0}
2 \Big[  \quadratic[\leqN][r][j] 
 \parasim 
 \Big( \partial_j \linear[\leqN][r][i] -  \partial^i \linear[\leqN][l][j] \Big) \Big] + 2  \Big[\quadratic[\leqN][r][i]  \parasim \partial^j \linear[\leqN][l][j] \Big]
 + \sigma_{\leq N}^2 \Kil \big( \linear[\leqN][r][i] \big). 
\end{equation}
Using the definition of the combined quadratic object \eqref{ansatz:eq-quadratic-combined}, we obtain that
\begin{equation}\label{objects:eq-enhanced-p1}
\begin{aligned}
&2 \Big[  \quadratic[\leqN][r][j] 
 \parasim 
 \Big( \partial_j \linear[\leqN][r][i] -  \partial^i \, \linear[\leqN][l][j] \Big) \Big] + 2  \Big[\quadratic[\leqN][r][i]  \parasim \partial^j \linear[\leqN][l][j] \Big] \\
 =&\, 4 \Big[ \quadratic[\leqN][d][k][j][k] \parasim \partial_j \linear[\leqN][r][i] \Big] 
 - 4 \Big[ \quadratic[\leqN][d][k][j][k] \parasim \partial^i \, \linear[\leqN][l][j] \Big]
 -2  \Big[ \quadratic[\leqN][d][k][k][j] \parasim \partial_j \linear[\leqN][r][i] \Big] + 2 \Big[ \quadratic[\leqN][d][k][k][j] \parasim \partial^i \,  \linear[\leqN][l][j] \Big] \\
 +& \,4  \Big[ \quadratic[\leqN][d][k][i][k] \parasim \partial^j \,  \linear[\leqN][l][j] \Big]
 -2  \Big[ \quadratic[\leqN][d][k][k][i] \parasim \partial^j \,  \linear[\leqN][l][j] \Big]. 
\end{aligned}
\end{equation}
Using Lemma \ref{lemma:cubic-stochastic-object-with-two-derivatives}, \eqref{objects:eq-enhanced-p1} can be written as the sum of an acceptable contribution and the renormalization term
\begin{align*}
&\frac{1}{4} \sigma_{\leq N}^2 \Big( 4 \delta^{ij}\delta_{jk} - 4 \delta^j_j \delta^i_k - 2 \delta^i_k \delta^j_j + 2 \delta^j_k \delta^i_j + 4 \delta^i_j \delta^j_k - 2 \delta^{ij} \delta_{jk} \Big) \Kil \big( \linear[\leqN][r][k] \big)\\
=\, & \frac{1}{4} \sigma_{\leq N}^2 \Big( 4 - 8 -4 + 2 +4 -2 \Big) \delta^{i}_k  \Kil \big( \linear[\leqN][r][k] \big) \\ 
=\, & - \sigma_{\leq N}^2 \Kil \big( \linear[\leqN][r][i] \big). 
\end{align*}
Since this is cancelled by the renormalization in \eqref{objects:eq-enhanced-p0}, it follows that \eqref{objects:eq-enhanced-p0} is under control. 
\end{proof}
\section{Nonlinear estimates}\label{section:nonlinear-estimates}
In this section, we perform the contraction mapping argument for the para-controlled stochastic Yang-Mills heat equation (Definition \ref{ansatz:def-paracontrolled}). To simplify the notation in the following proposition, we introduce
\begin{equation}\label{eq:BR}
\BS := \Big\{ A^{(0)} \in \Cs_x^{-\kappa}(\T_x^2 \rightarrow \frkg)^2 \colon \big\| A^{(0)} \big\|_{\Cs_x^{-\kappa}} \leq S \Big\}. 
\end{equation}

\begin{proposition}[Well-posedness of the para-controlled stochastic Yang-Mills heat flow]\label{nonlinear:prop-wellposedness-para}
Let $C\geq 1$ be a sufficiently large absolute constant, let $R,S\geq 1$, and let $0<\tau\leq C^{-1} (RS)^{-C}$. Furthermore, let $A^{(0)} \in \BS$ and let  $\Xi_{\leq N}\in \Dc_R([0,\tau])$. Then, there exists a unique solution 
\begin{equation*}
(X_{\leq N},Y_{\leq N})\in (\Sc^{1-2\kappa}\times \Sc^{2-5\kappa})([0,\tau])
\end{equation*}
of the para-controlled stochastic Yang-Mills heat equation (Definition \ref{ansatz:def-paracontrolled}). Furthermore, it can be written as
\begin{equation*}
(X_{\leq N},Y_{\leq N}) = \Big( \bX_{R,S,\tau}(\Xi_{\leq N},A^{(0)}),
\bY_{R,S,\tau}(\Xi_{\leq N},A^{(0)}) \Big),
\end{equation*}
where 
\begin{equation*}
\bX_{R,S,\tau}\colon \Dc_R([0,\tau])\times \BS \rightarrow \Sc^{1-2\kappa}([0,\tau])
\qquad \text{and}\qquad 
\bY_{R,S,\tau}\colon \Dc_R([0,\tau])\times \BS \rightarrow \Sc^{2-5\kappa}([0,\tau])
\end{equation*}
are Lipschitz continuous. 
\end{proposition}

The solutions $(X_{\leq N}, Y_{\leq N})$ appearing in Proposition \ref{nonlinear:prop-wellposedness-para} will be realized as fixed points of the following map.

\begin{definition}\label{def:Gamma-contraction-map}
We now define a map $\Gamma=(\Gamma^X,\Gamma^Y)$ which incorporates the integral formulation of \eqref{ansatz:eq-new-X} and \eqref{ansatz:eq-Y-new-1}-\eqref{ansatz:eq-Y-new-9}. That is, we define
\begin{equation}\label{nonlinear:eq-Gamma-X}
\Gamma^X(X_{\leq N},Y_{\leq N})^i 
:= \Duh \Big( \eqref{ansatz:eq-new-X} \Big) 
\end{equation}
and 
\begin{equation}\label{nonlinear:eq-Gamma-Y}
\Gamma^Y(X_{\leq N},Y_{\leq N})^i 
:= e^{-t(1-\Delta)} \Big( A^{(0),i} - \linear[\leqN][r][i] (0) -\quadratic[\leqN][r][i] (0) \Big) 
+ \Duh \Big( \eqref{ansatz:eq-Y-new-1}
+ \eqref{ansatz:eq-Y-new-2}
+ \hdots 
+ \eqref{ansatz:eq-Y-new-9} \Big). 
\end{equation}
\end{definition}

Before turning to the proof of Proposition \ref{nonlinear:prop-wellposedness-para}, we first need several nonlinear estimates. In the first lemma, we control the right-hand side of \eqref{ansatz:eq-new-X}, i.e., of the evolution equation for $X_{\leq N}$. 

\begin{lemma}\label{nonlinear:lem-X}
Let $T\in [0,1]$, let $R\geq 1$, and assume that the enhanced data set satisfies $\Xi_{\leq N}\in \Dc_R([0,T])$. Then, it holds for all $i\in [2]$ and $B_{\leq N}\colon [0,T] \times \T^2 \rightarrow \frkg^2$ that
\begin{equation}\label{nonlinear:eq-X}
\Big\| \Big[ B_{\leq N,j} \parall \big( 2 \partial^j \linear[\leqN][r][i] - \partial^i \linear[\leqN][r][j] \big) \Big] \Big\|_{\Wc^{-1-\kappa,\kappa+\theta}}
\lesssim R \big\| B_{\leq N} \big\|_{\Sc^{2\kappa}}. 
\end{equation}
\end{lemma}

\begin{proof}
Using the low$\times$high-estimate from Lemma \ref{prelim:lem-para-product} and $\Xi_{\leq N}\in \Dc_R([0,T])$, it follows for all $t\in [0,T]$ that
\begin{align*}
t^{\kappa+\theta} 
\Big\|  \Big[ B_{\leq N,j} \parall \big( 2 \partial^j \linear[\leqN][r][i] - \partial^i \linear[\leqN][r][j] \big) \Big](t) \Big\|_{\Cs^{-1-\kappa}} 
\lesssim t^{\kappa+\theta} 
\big\| B_{\leq N}(t) \big\|_{\Cs_x^{2\kappa}}
\big\| \partial\hspace{0.1ex} \linear[\leqN][r][](t) \big\|_{\Cs_x^{-1-\kappa}}
\lesssim R \big\| B_{\leq N}\big\|_{\Sc^{2\kappa}([0,T])}. 
\end{align*}
After taking the supremum over $t\in [0,T]$, this yields \eqref{nonlinear:eq-X}. 
\end{proof}

We now turn to the different terms in the evolution equation for $Y_{\leq N}$, i.e., \eqref{ansatz:eq-Y-new-1}-\eqref{ansatz:eq-Y-new-9}. In the next lemma, we control the double Duhamel-terms from \eqref{ansatz:eq-Y-new-2} and \eqref{ansatz:eq-Y-new-3}. 

\begin{lemma}[Double Duhamel-terms]\label{nonlinear:lem-dd}
Let $T\in [0,1]$, let $R\geq 1$, and assume that the enhanced data set satisfies $\Xi_{\leq N} \in \Dc_R([0,T])$. Furthermore, let $\nu \geq 3\kappa+\theta$. Then, it holds for all $i\in [2]$ and $B_{\leq N}\colon [0,T] \times \T^2 \rightarrow \frkg$ that 
\begin{align}
&\bigg\| 
2 \bigg[ \Duh \Big( \Big[ B_{\leq N,k} \parall \big( 2 \partial^k \linear[\leqN][r][j] - \partial^j \linear[\leqN][r][k] \big) \Big] \Big) \parasim \Big( \partial_j \linear[\leqN][r][i] - \partial^i \linear[\leqN][l][j] \Big) \bigg] 
+ \frac{3}{2} \sigma_{\leq N}^2 \Kil(B_{\leq N}^i)
\bigg\|_{\Wc^{-4\kappa,\nu}} \label{nonlinear:eq-DD-1} \\
\lesssim&\, R^2  \, \big\| B_{\leq N} \big\|_{\Sc^{1-2\kappa}}, \notag \\[1ex]
&\bigg\| 
2 \bigg[ \Duh \Big( \Big[ B_{\leq N,k} \parall \big( 2 \partial^k \linear[\leqN][r][i] - \partial^i \linear[\leqN][r][k] \big) \Big] \Big) \parasim  \partial_j \linear[\leqN][r][j]  \bigg] - \frac{1}{2} \sigma_{\leq N}^2 \Kil(B_{\leq N}^i)\bigg\|_{\Wc^{-4\kappa,\nu}}  \label{nonlinear:eq-DD-2} \\
\lesssim&\, R^2 \, \big\| B_{\leq N} \big\|_{\Sc^{1-2\kappa}}. \notag 
\end{align}
\end{lemma}

\begin{remark}\label{nonlinear:rem-wick-para}
In a general dimension $d \geq 2$, we expect that Wick-ordering yields the counterterms
\begin{equation}
\frac{3(d-1)}{d} \sigma_{\leq N}^2 \Kil(B_{\leq N}^i) 
\qquad \text{and} \qquad -\frac{1}{d} \sigma_{\leq N}^2 \Kil(B_{\leq N}^i)
\end{equation}
in \eqref{nonlinear:eq-DD-1} and \eqref{nonlinear:eq-DD-2}, respectively. The more general counterterms can be derived from Remark \ref{objects:rem-with-two} and a modification of the combinatorial identities \eqref{nonlinear:eq-DD-p4} and \eqref{nonlinear:eq-DD-p5}. 
\end{remark}
\begin{proof}
For expository purposes, we separate the proof into two steps.\\

\emph{Step 1:} In the first step, we prove for all $j_1,j_2,j_3 \in [2]$ and all $k_2,k_3\in [2]$ that 
\begin{equation}\label{nonlinear:eq-DD-p1}
\begin{aligned}
\bigg\| \bigg[ \Duh \Big( \Big[ B_{\leq N}^{j_1} \parall \partial^{k_2} \linear[\leqN][r][j_2]  \Big] \Big) \parasim  \partial^{k_3} \linear[\leqN][r][j_3]  \bigg]
- \frac{1}{4} \delta^{k_2k_3} \delta^{j_2j_3} \sigma_{\leq N}^2
\Kil\big( B_{\leq N}^{j_1} \big) \bigg\|_{\Wc^{-4\kappa,\nu}} 
\lesssim R^2 \big\| B_{\leq N}\big\|_{\Sc^{1-2\kappa}}. 
\end{aligned}
\end{equation}
We first write\footnote{It is slightly inconvenient that we momentarily work in coordinates of $\frkg$. Our reason for switching to coordinates is that the para-product in \eqref{nonlinear:eq-DD-q1} acts on the second and third argument. Since the inner Lie bracket acts on the first and second argument, this makes it hard to combine the para-product and Lie bracket notation.}
\begin{align*}
\bigg[ \Duh \Big( \Big[ B_{\leq N}^{j_1} \parall \partial^{k_2} \linear[\leqN][r][j_2]  \Big] \Big) \parasim  \partial^{k_3} \linear[\leqN][r][j_3]  \bigg]
= \Big( \Duh \big( B^{j_1,a_1}_{\leq N} \parall \partial^{k_2} \linear[\leqN][r][j_2,a_2] \big) \parasim \partial^{k_3} \linear[\leqN][r][j_3,a_3] \Big)
\Big[ \big[ E_{a_1}, E_{a_2} \big], E_{a_3} \Big]. 
\end{align*}
Using Lemma \ref{prelim:lem-lh-hh} and  Lemma \ref{prelim:lem-trilinear-integral-commutator}, it follows that
\begin{equation}\label{nonlinear:eq-DD-p2}
\begin{aligned}
&\Big\| \Big( \Duh \big( B^{j_1,a_1}_{\leq N} \parall \partial^{k_2} \linear[\leqN][r][j_2,a_2] \big) \parasim \partial^{k_3} \linear[\leqN][r][j_3,a_3] \Big) 
- B^{j_1,a_1}_{\leq N} 
\Big( \Duh \big( \partial^{k_2} \linear[\leqN][r][j_2,a_2] \big) 
\parasim \partial^{k_3} \linear[\leqN][r][j_3,a_3] \Big) 
\Big\|_{\Wc^{-4\kappa,\nu}} \\
\lesssim&\, \big\| B_{\leq N} \big\|_{\Sc^{1-2\kappa}}
\big\| \partial\hspace{0.1ex} \linear[\leqN][r] \big\|_{\Wc^{-1-\kappa,0}}^2 
\lesssim R^2 \big\| B_{\leq N} \big\|_{\Sc^{1-2\kappa}}. 
\end{aligned}
\end{equation}
We can therefore replace the first summand in \eqref{nonlinear:eq-DD-p1} with 
\begin{equation}\label{nonlinear:eq-DD-q1}
B^{j_1,a_1}_{\leq N} 
\Big( \Duh \big( \partial^{k_2} \linear[\leqN][r][j_2,a_2] \big) 
\parasim \partial^{k_3} \linear[\leqN][r][j_3,a_3] \Big)  \Big[ \big[ E_{a_1}, E_{a_2} \big], E_{a_3} \Big] 
= B_{\leq N}^{j_1,a_1} \Big[ \big[ E_{a_1}, 
\Duh \big( \partial^{k_2} \linear[\leqN][r][j_2] \big) \big] 
\parasim \partial^{k_3} \linear[\leqN][r][j_3] \Big]. 
\end{equation}
Using our product estimate (Lemma \ref{prelim:lem-para-product}) and our assumption $\Xi_{\leq N} \in \Dc_R([0,T])$, we obtain that
\begin{equation}\label{nonlinear:eq-DD-p3}
\begin{aligned}
&\bigg\|  B_{\leq N}^{j_1,a_1} \Big[ \big[ E_{a_1}, 
\Duh \big( \partial^{k_2} \linear[\leqN][r][j_2] \big) \big] 
\parasim \partial^{k_3} \linear[\leqN][r][j_3] \Big] 
- \frac{1}{4}  B_{\leq N}^{j_1,a_1} \delta^{k_2k_3} \delta^{j_2j_3} \sigma_{\leq N}^2 \Kil(E_{a_1}) 
\bigg\|_{\Wc^{-4\kappa,\nu}} \\
\lesssim&\, 
\big\| B_{\leq N} \big\|_{\Wc^{3\kappa,\nu-\kappa}} 
\max_{a_1\in [\dim\frkg]} 
\Big\|  \Big[ \big[ E_{a_1}, 
\Duh \big( \partial^{k_2} \linear[\leqN][r][j_2] \big) \big] 
\parasim \partial^{k_3} \linear[\leqN][r][j_3] \Big] 
- \frac{1}{4}  \delta^{k_2k_3} \delta^{j_2j_3} \sigma_{\leq N}^2 \Kil(E_{a_1})\Big\|_{\Wc^{-2\kappa,\kappa}} \\
\lesssim&\, R^2 \big\| B_{\leq N} \big\|_{\Wc^{3\kappa,\nu-\kappa}} . 
\end{aligned}
\end{equation}
Since $\nu -\kappa\geq 3\kappa/2+\theta$, the right-hand side of \eqref{nonlinear:eq-DD-p3} is acceptable. Since 
\begin{equation}
B_{\leq N}^{j_1,a_1} \delta^{k_2k_3} \delta^{j_2j_3} \sigma_{\leq N}^2 \Kil(E_{a_1}) =  \delta^{k_2k_3} \delta^{j_2j_3} \sigma_{\leq N}^2  \Kil(B_{\leq N}^{j_1}), 
\end{equation}
this completes the proof of \eqref{nonlinear:eq-DD-p1}. \\

\emph{Step 2:} In the second step, we complete the proof of \eqref{nonlinear:eq-DD-1} and \eqref{nonlinear:eq-DD-2}. After using our estimate \eqref{nonlinear:eq-DD-p1} from the first step, it only remains to prove the two combinatorial identities 
\begin{align}
-\frac{1}{4} 2 \big( 2 \delta^k_j \delta^{ji} - 2 \delta^{ki} \delta^{j}_j 
- \delta^j_j \delta^{ki} + \delta^{ji} \delta^{k}_j \big) &= 
\frac{3}{2} \delta^{ki}, \label{nonlinear:eq-DD-p4} \\ 
-\frac{1}{4} 2 \big( 2 \delta^{kj} \delta^{i}_j - \delta^{ij} \delta^{k}_j \big) =  -\frac{1}{2} \delta^{ki}. \label{nonlinear:eq-DD-p5}
\end{align}
Using that $\delta^j_j = 2$, the first identity \eqref{nonlinear:eq-DD-p4} follows from
\begin{equation}
-\frac{1}{4}2\big( 2 \delta^k_j \delta^{ji} - 2 \delta^{ki} \delta^{j}_j 
- \delta^j_j \delta^{ki} + \delta^{ji} \delta^{k}_j \big)
= -\frac{1}{4}2 \big( 2 -4 - 2 +1 \big) \delta^{ki} 
= \frac{6}{4}  \delta^{ki}. 
\end{equation}
Similarly, the second identity \eqref{nonlinear:eq-DD-p5} follows from 
\begin{equation}
-\frac{1}{4} 2 \big( 2 \delta^{kj} \delta^{i}_j - \delta^{ij} \delta^{k}_j \big) 
= -\frac{1}{4} 2 (2-1) \delta^{ki}  = -\frac{1}{2} \delta^{ki}. 
\end{equation}
\end{proof}

In the next lemma, we control the cubic terms from \eqref{ansatz:eq-Y-new-8}. 

\begin{lemma}[Cubic terms]\label{nonlinear:lem-cubic}
Let $N\in \dyadic$, $T\in [0,1]$, $i\in [2]$, and $\nu \geq 3 (\kappa+\theta)$.   Furthermore, let $R\geq 1$ and assume that $\Xi_{\leq N}\in \Dc_R([0,T])$. 
Then, we have for all $B_{\leq N}\colon [0,T] \times \T_x^2 \rightarrow \frkg^2$ that
\begin{equation}\label{nonlinear:eq-cubic}
\begin{aligned}
&\Big\| \Big[ \Big[ \Big(\linear[\leqN][r][i] + B_{\leq N}^i\Big), 
 \Big(\linear[\leqN][r][j] + B_{\leq N}^j\Big) \Big],
 \Big(\linear[\leqN][l][j] + B_{\leq N,j}\Big) \Big] 
 -  \sigma_{\leq N}^2 \Kil \big( \linear[\leqN][r][i] + B_{\leq N}^i \big) \Big\|_{\Wc^{-4\kappa,\nu}([0,T])} \\
 \lesssim& \, \Big( R + \big\| B_{\leq N} \big\|_{\Sc^{3\kappa}([0,T])} \Big)^3. 
 \end{aligned}
\end{equation}
\end{lemma}

\begin{remark} \label{nonlinear:rem-Wick-cubic}
In dimension $d\geq 2$, Wick-ordering yields the counterterm $- (d-1) \sigma_{\leq N}^2 \Kil \big( \linear[\leqN][r][i] + B_{\leq N}^i \big)$, i.e., the general form of the pre-factor is $-(d-1)$. This follows from slight modifications of \eqref{nonlinear:eq-counterterm-1} and \eqref{nonlinear:eq-counterterm-2} below.  
\end{remark}

\begin{proof} 
We first decompose 
\begin{align}
&\Big[ \Big[ \Big(\linear[\leqN][r][i] + B_{\leq N}^i\Big), 
 \Big(\linear[\leqN][r][j] + B_{\leq N}^j\Big) \Big],
 \Big(\linear[\leqN][l][j] + B_{\leq N,j}\Big) \Big]  \notag \\
 =& \Big[ \Big[ 
\linear[\leqN][r][i] , 
\linear[\leqN][r][j] \Big],
\linear[\leqN][l][j]  \Big] \label{nonlinear:eq-cubic-deg3} \\
+& \Big[ \Big[ 
B_{\leq N}^i, 
\linear[\leqN][r][j] \Big],
\linear[\leqN][l][j]  \Big] 
+\Big[ \Big[ 
\linear[\leqN][r][i] , 
B_{\leq N}^j \Big],
\linear[\leqN][l][j]  \Big] 
+\Big[ \Big[ 
\linear[\leqN][r][i] , 
\linear[\leqN][r][j] \Big],
B_{\leq N,j}  \Big] 
\label{nonlinear:eq-cubic-deg2} \\
+&\Big[ \Big[ 
\linear[\leqN][r][i] , 
B_{\leq N}^j \Big],
B_{\leq N,j}  \Big] 
+\Big[ \Big[ 
B_{\leq N}^i  , 
\linear[\leqN][r][j] \Big],
B_{\leq N,j}  \Big] 
+\Big[ \Big[ 
B_{\leq N}^i  , 
B_{\leq N}^j \Big],
\linear[\leqN][l][j]  \Big] 
\label{nonlinear:eq-cubic-deg1} \\
+&\Big[ \Big[ 
B_{\leq N}^i  , 
B_{\leq N}^j \Big],
B_{\leq N,j}  \Big].  \label{nonlinear:eq-cubic-deg0}
\end{align}
We now separately estimate the contributions of \eqref{nonlinear:eq-cubic-deg3}-\eqref{nonlinear:eq-cubic-deg0}. \\

\emph{Contribution of \eqref{nonlinear:eq-cubic-deg3}:} We first write 
\begin{equation*}
\Big[ \Big[ 
\linear[\leqN][r][i] , 
\linear[\leqN][r][j] \Big],
\linear[\leqN][l][j]  \Big] = \delta_{jk} 
\Big[ \Big[ 
\linear[\leqN][r][i] , 
\linear[\leqN][r][j] \Big],
\linear[\leqN][r][k]  \Big]. 
\end{equation*}
Using  $\Xi_{\leq N} \in \Dc_R([0,T])$,  it follows for all $t\in [0,T]$ that
\begin{equation*}
\Big\| \Big[ \Big[ 
\linear[\leqN][r][i] , 
\linear[\leqN][r][j] \Big],
\linear[\leqN][r][k]  \Big](t) 
- \sigma_{\leq N}^2 \delta^{jk} \Kil\big( \linear[\leqN][r][i](t) \big)
+  \sigma_{\leq N}^2 \delta^{ik} \Kil\big( \linear[\leqN][r][j](t) \big)
\Big\|_{\Cs^{-4\kappa}}
\lesssim R^3. 
\end{equation*}
Since 
\begin{align}\label{nonlinear:eq-counterterm-1}
 -\sigma_{\leq N}^2 \Big( 
 \delta_{jk}  \delta^{jk} \Kil\big( \linear[\leqN][r][i] \big)
  - \delta_{jk}   \delta^{ik} \Kil\big( \linear[\leqN][r][j] \big) \Big)
= -\sigma_{\leq N}^2 (2-1) \Kil\big( \linear[\leqN][r][i] \big) 
= - \sigma_{\leq N}^2 \Kil\big( \linear[\leqN][r][i] \big), 
\end{align}
it follows that 
\begin{equation}
\Big\| \Big[ \Big[ 
\linear[\leqN][r][i] , 
\linear[\leqN][r][j] \Big],
\linear[\leqN][l][j]  \Big](t) - \sigma_{\leq N}^2 \Kil\big( \linear[\leqN][r][i](t) \big) \Big\|_{\Cs^{-4\kappa}} \lesssim R^3.
\end{equation}
Since $\nu \geq 0$, this is acceptable. \\ 

\emph{Contribution of \eqref{nonlinear:eq-cubic-deg2}:} Using our bound on the enhanced data set and the product estimate (Lemma \ref{prelim:lem-para-product}), it follows for all $t\in [0,T]$ that 
\begin{align*}
\Big\|  \Big[ \Big[ 
B_{\leq N}^i, 
\linear[\leqN][r][j] \Big],
\linear[\leqN][l][j]  \Big] (t)
- \sigma_{\leq N}^2 \delta^{j}_{j} \Kil\big( B_{\leq N}^i(t) \big) 
\Big\|_{\Cs^{-4\kappa}} &\lesssim R^2 \big\| B_{\leq N}(t) \big\|_{\Cs^{3\kappa}}, \\
\Big\| \Big[ \Big[ 
\linear[\leqN][r][i] , 
B_{\leq N}^j \Big],
\linear[\leqN][l][j]  \Big] (t)
+ \sigma_{\leq N}^2 \delta^{i}_{j}  \Kil\big( B_{\leq N}^j(t) \big) 
\Big\|_{\Cs^{-4\kappa}} &\lesssim R^2 \big\| B_{\leq N}(t) \big\|_{\Cs^{3\kappa}}, \\
\Big\| \Big[ \Big[ 
\linear[\leqN][r][i] , 
\linear[\leqN][r][j] \Big],
B_{\leq N,j}  \Big] (t)
\Big\|_{\Cs^{-4\kappa}} &\lesssim  R^2 \big\| B_{\leq N}(t) \big\|_{\Cs^{3\kappa}}. 
\end{align*}
Since 
\begin{equation}\label{nonlinear:eq-counterterm-2}
-\sigma_{\leq N}^2 \big( \delta^{j}_{j}  \Kil\big( B_{\leq N}^i \big) - \delta^{i}_{j}  \Kil\big( B_{\leq N}^j \big)  \big) = - \sigma_{\leq N}^2 (2-1)  \Kil\big( B_{\leq N}^i \big) = -\sigma_{\leq N}^2 \Kil\big( B_{\leq N}^i \big),
\end{equation}
it follows that
\begin{equation*}
\big\| \eqref{nonlinear:eq-cubic-deg2}(t) - \sigma_{\leq N}^2 \Kil\big( B_{\leq N}^i(t) \big) \big\|_{\Cs^{-4\kappa}} \lesssim R^2  \big\| B_{\leq N}(t) \big\|_{\Cs^{3\kappa}} 
\lesssim  t^{-\frac{3\kappa}{2}-\theta} R^2 \big\| B_{\leq N} \big\|_{\Sc^{3\kappa}([0,T])}. 
\end{equation*}
Since $\nu \geq \frac{3}{2} \kappa +\theta$, this contribution is acceptable. \\

\emph{Contributions of \eqref{nonlinear:eq-cubic-deg1} and \eqref{nonlinear:eq-cubic-deg0}:} 
Using our bound on the enhanced data set and the product estimate (Lemma \ref{prelim:lem-para-product}), it follows that
\begin{align*}
\big\| \eqref{nonlinear:eq-cubic-deg1}(t) \big\|_{\Cs^{-4\kappa}}  + 
\big\| \eqref{nonlinear:eq-cubic-deg0}(t) \big\|_{\Cs^{-4\kappa}}  
&\lesssim \big\| \linear[\leqN][r][](t) \big\|_{\Cs^{-\kappa}} 
\big\| B_{\leq N}(t) \big\|_{\Cs^{2\kappa}}^2 
+ \big\| B_{\leq N}(t) \big\|_{\Cs^{2\kappa}}^3  \\
&\lesssim  \big( t^{-\frac{2\kappa}{2}-\theta} \big)^3  \Big( R+\big\| B_{\leq N} \big\|_{\Sc^{3\kappa}([0,T])} \Big)^3. 
\end{align*}
Since $\nu \geq 3(\kappa+\theta)$, this contribution is acceptable. 
\end{proof}

In the following lemma, we control the contributions of all remaining terms in the evolution equation for $Y_{\leq N}$, i.e.,  \eqref{ansatz:eq-Y-new-4}, \eqref{ansatz:eq-Y-new-5}, \eqref{ansatz:eq-Y-new-6}, \eqref{ansatz:eq-Y-new-7}, and \eqref{ansatz:eq-Y-new-9}.

\begin{lemma}\label{nonlinear:lem-remaining}
Let $T\in [0,1]$, let $i,j,k\in [2]$, and 
\begin{equation}\label{nonlinear:eq-remaining-nu}
\nu \geq \frac{1}{2} + 2 \kappa + 2 \theta. 
\end{equation}
Furthermore, let $N\in \dyadic$, let $R\geq 1$,  and assume that $\Xi_{\leq N} \in \Dc_R([0,T])$. For all $B_{\leq N}\colon [0,T] \times \T^2_x \rightarrow \frkg^2$ and $Y_{\leq N}\colon [0,T] \times \T^2_x \rightarrow \frkg^2$, we then have the following estimates:
\begin{align}
  \Big\| \Big[  Y_{\leq N}^i  \parasim  \partial^j \, \linear[\leqN][r][k]  \Big] \Big\|_{\Wc^{-4\kappa,\nu}([0,T])} 
 &\lesssim R \big\| Y_{\leq N} \big\|_{\Sc^{1+2\kappa}([0,T])}, 
 \label{nonlinear:eq-remaining-e1} \\ 
  \Big\| \partial^i \Big[ B^j_{\leq N} \parasim   \linear[\leqN][r][k] \Big]  \Big\|_{\Wc^{-4\kappa,\nu}([0,T])} 
 &\lesssim  R \big\| B_{\leq N} \big\|_{\Sc^{1-2\kappa}([0,T])}, \label{nonlinear:eq-remaining-e2} \\ 
 \Big\| \Big[ B^j_{\leq N} , 2 \partial_j B^i_{\leq N} - \partial^i B_{\leq N,j} \Big] \Big\|_{\Wc^{-4\kappa,\nu}([0,T])} 
 &\lesssim  \big\| B_{\leq N} \big\|_{\Sc^{1-2\kappa}([0,T])}^2, \label{nonlinear:eq-remaining-e3} \\ 
 \Big\| \Big[ B_{\leq N}^i\paragg \partial^j \linear[\leqN][r][k]  \Big] \Big\|_{\Wc^{-4\kappa,\nu}([0,T])}  
 + \Big\| \Big[ \linear[\leqN][r][i] \paransim  \, \partial^j B_{\leq N}^k \Big]  \Big\|_{\Wc^{-4\kappa,\nu}([0,T])}   
 &\lesssim R \big\| B_{\leq N} \big\|_{\Sc^{1-2\kappa}([0,T])},  \label{nonlinear:eq-remaining-e4}\\
 \Big\| \, \linear[\leqN][r][i] + B_{\leq N}^i\Big\|_{\Wc^{-4\kappa,\nu}([0,T])}  
 &\lesssim R + \big\| B_{\leq N} \big\|_{\Sc^{1-2\kappa}([0,T])}. \label{nonlinear:eq-remaining-e5}
\end{align}
\end{lemma}

\begin{proof}
We prove \eqref{nonlinear:eq-remaining-e1}-\eqref{nonlinear:eq-remaining-e5} separately.  In the following arguments, we often use para-product estimates (as in Lemma \ref{prelim:lem-para-product}) and $\Xi_{\leq N} \in \Dc_R([0,T])$ without explicit reference. \\

\emph{Proof of \eqref{nonlinear:eq-remaining-e1}:} For any $t\in [0,T]$, it holds that 
\begin{align*}
\Big\| \Big[  Y_{\leq N}^i  \parasim  \partial^j \, \linear[\leqN][r][k]  \Big](t) \Big\|_{\Cs^{-4\kappa}}
\lesssim \, \big\|  Y_{\leq N}^i(t) \big\|_{\Cs^{1+2\kappa}} 
\big\|  \partial^j \, \linear[\leqN][r][k] (t) \big\|_{\Cs^{-1-\kappa}} 
\lesssim  t^{-1/2-\kappa-\theta} R  \big\| Y_{\leq N} \big\|_{\Sc^{1+2\kappa}}.
\end{align*}
Due to our assumption \eqref{nonlinear:eq-remaining-nu}, this yields an acceptable contribution. \\

\emph{Proof of \eqref{nonlinear:eq-remaining-e2}:} For any $t\in [0,T]$, it holds that 
\begin{align*}
 &\Big\| \partial^i \Big[ B^j_{\leq N} \parasim   \linear[\leqN][r][k] \Big](t)  \Big\|_{\Cs^{-4\kappa}} 
 \lesssim  \Big\|  \Big[ B^j_{\leq N} \parasim   \linear[\leqN][r][k] \Big](t)  \Big\|_{\Cs^{1-4\kappa}}
 \lesssim  \big\| B^j_{\leq N}(t) \big\|_{\Cs^{1-2\kappa}} \Big\|   \, \linear[\leqN][r][k](t)  \Big\|_{\Cs^{-\kappa}} \\ 
 \lesssim& \,  t^{-1/2+\kappa-\theta} R \big\| B_{\leq N} \big\|_{\Sc^{1-2\kappa}}. 
\end{align*}
Due to our assumption \eqref{nonlinear:eq-remaining-nu}, this yields an acceptable contribution. \\

\emph{Proof of \eqref{nonlinear:eq-remaining-e3}:}
For any $t\in [0,T]$, we have that
\begin{align*}
 \Big\| \Big[ B^j_{\leq N} , 2 \partial_j B^i_{\leq N} - \partial^i B_{\leq N,j} \Big](t) \Big\|_{\Cs_x^{-4\kappa}} 
 \lesssim \big\| B_{\leq N}(t) \big\|_{\Cs_x^{3\kappa}} 
 \big\| \partial B_{\leq N}(t) \big\|_{\Cs_x^{-2\kappa}} 
 \lesssim t^{-\frac{3\kappa}{2}-\theta} t^{-\frac{1-2\kappa}{2}-\theta} \big\| B_{\leq N} \big\|_{\Sc^{1-2\kappa}([0,T])}^2. 
\end{align*}
Due to our assumption \eqref{nonlinear:eq-remaining-nu}, this yields an acceptable contribution.  \\ 

\emph{Proof of \eqref{nonlinear:eq-remaining-e4}:} 
We treat the two summands on the left-hand side of \eqref{nonlinear:eq-remaining-e4} separately. For any $t\in [0,T]$, the first summand can be estimated by 
\begin{align*}
\Big\| \Big[ B_{\leq N}^i\paragg \partial^j \linear[\leqN][r][k]  \Big](t) \Big\|_{\Cs^{-4\kappa}}
\lesssim \big\| B_{\leq N}^i(t) \big\|_{\Cs^{1-2\kappa}} \Big\| \,  \partial^j \linear[\leqN][r][k](t) \Big\|_{\Cs^{-1-\kappa}} 
\lesssim t^{-1/2+\kappa-\theta}  R \big\| B_{\leq N} \big\|_{\Sc^{1-2\kappa}}.  
\end{align*}
For the second summand, we estimate
\begin{align*}
\Big\| \Big[ \linear[\leqN][r][i] \paransim  \, \partial^j B_{\leq N}^k \Big](t) \Big\|_{\Cs^{-4\kappa}}
\lesssim \Big\| \, \linear[\leqN][r][i](t) \Big\|_{\Cs^{-\kappa}} \Big\| \partial^j B_{\leq N}^k(t) \Big\|_{\Cs^{-2\kappa}} 
\lesssim t^{-1/2+\kappa-\theta} R \big\|B_{\leq N}^k \big\|_{\Sc^{1-2\kappa}}.  
\end{align*}
Due to our assumption \eqref{nonlinear:eq-remaining-nu}, this yields an acceptable contribution. \\ 

\emph{Proof of \eqref{nonlinear:eq-remaining-e5}:}
For any $t\in [0,T]$, we have that 
\begin{align*}
 \Big\| \, \linear[\leqN][r][i](t) + B_{\leq N}^i(t) \Big\|_{\Cs_x^{-4\kappa}}  
 \lesssim \Big\| \, \linear[\leqN][r][i](t) \Big\|_{\Cs_x^{-2\kappa}}  + \Big\| B_{\leq N}^i(t)\Big\|_{\Cs_x^{-2\kappa}}
 \lesssim R + \big\| B_{\leq N} \big\|_{\Sc^{1-2\kappa}([0,T])}. 
\end{align*}
Since $\nu \geq 0$, this is acceptable. 
\end{proof}

Equipped with Lemma \ref{nonlinear:lem-X}, Lemma \ref{nonlinear:lem-dd}, Lemma \ref{nonlinear:lem-cubic}, and Lemma \ref{nonlinear:lem-remaining}, we now turn to the proof of the main proposition. 

\begin{proof}[Proof of Proposition \ref{nonlinear:prop-wellposedness-para}:]
The argument relies on the contraction-mapping principle and the estimates from Lemma \ref{nonlinear:lem-X}, Lemma \ref{nonlinear:lem-dd}, Lemma \ref{nonlinear:lem-cubic}, Lemma \ref{nonlinear:lem-remaining}. In the following argument, all norms will be implicitly restricted to the time-interval $[0,\tau]$. For expository purposes, we separate the proof into three steps. \\

\emph{Step 1: Setup.}
Let $\widetilde{C}\geq 1$ remain to be chosen. We then define
\begin{equation*}
\Xc_{R,S} := \Big\{ X \in \Sc^{1-2\kappa}\colon \big\| X \big\|_{\Sc^{1-2\kappa}} \leq \widetilde{C} (R+S)\Big\} 
\qquad \text{and} \qquad 
\Yc_{R,S} := \Big\{ Y \in \Sc^{2-5\kappa}\colon \big\| Y \big\|_{\Sc^{2-5\kappa}} \leq \widetilde{C} (R+S) \Big\}.
\end{equation*}
For any $(X_{\leq N},Y_{\leq N})\in \Xc_{R,S} \times \Yc_{R,S}$, we define
\begin{equation}
B_{\leq N} = \quadratic[\leqN][r] 
+ X_{\leq N} + Y_{\leq N}.
\end{equation}
 Due to $\Xi_{\leq N}\in \Dc_R$ and $(X_{\leq N},Y_{\leq N})\in \Xc_{R,S} \times \Yc_{R,S}$, it holds that 
\begin{equation}\label{nonlinear:eq-contract-p0}
\big\| B_{\leq N} \big\|_{\Sc^{1-2\kappa}}
\leq 
\Big\| \quadratic[\leqN][r]\Big\|_{\Sc^{1-2\kappa}}
+ \big\| X_{\leq N} \big\|_{\Sc^{1-2\kappa}}
+ \big\| Y_{\leq N} \big\|_{\Sc^{2-5\kappa}}
\lesssim \widetilde{C} (R+S). 
\end{equation}

In the following, recall the map $\Gamma = (\Gamma^X, \Gamma^Y)$ from Definition \ref{def:Gamma-contraction-map}.

\emph{Step 2: $\Gamma$ is a self-map.} 
In this step, we show that $\Gamma$ is a self-map, i.e., we prove that $\Gamma$ maps  $\Xc_{R,S} \times \Yc_{R,S}$  back into  $\Xc_{R,S} \times \Yc_{R,S}$. To this end, we define
\begin{equation}
\nu:= \frac{1}{2} + 2\kappa + 2\theta 
\qquad \text{and} \qquad
\delta:= \frac{1}{2} - 4 \kappa - 2 \theta
\end{equation}
and let $(X_{\leq N},Y_{\leq N})\in \Xc_{R,S}\times \Yc_{R,S}$ be arbitrary. 
In order to control $\Gamma^X(X_{\leq N},Y_{\leq N})$, we use Proposition \ref{prelim:prop-Duhamel-Salpha}, which yields that
\begin{equation}\label{nonlinear:eq-contract-p5}
\Big\| \, \Gamma^X(X_{\leq N},Y_{\leq N}) \Big\|_{\Sc^{1-2\kappa}}
=  \Big\| \Duh \Big( \eqref{ansatz:eq-new-X} \Big)\Big\|_{\Sc^{1-2\kappa}}
\lesssim \tau^\delta \big\| \eqref{ansatz:eq-new-X} \big\|_{\Wc^{-1-\kappa,\kappa+\theta}}. 
\end{equation}
Using Lemma \ref{nonlinear:lem-X} and \eqref{nonlinear:eq-contract-p0}, it further holds that
\begin{equation}\label{nonlinear:eq-contract-p6}
 \tau^\delta \big\| \eqref{ansatz:eq-new-X} \big\|_{\Wc^{-1-\kappa,\kappa+\theta}} \lesssim \tau^\delta (R+S). 
\end{equation}
In order to control $\Gamma^Y(X_{\leq N},Y_{\leq N})$, we use 
Lemma \ref{prelim:lem-heat-Salpha} and Proposition \ref{prelim:prop-Duhamel-Salpha}, which yield that 
\begin{align}
& \Big\| \,\Gamma^Y(X_{\leq N},Y_{\leq N}) \,  \Big\|_{\Sc^{2-5\kappa}} \notag \\
\leq\, & \Big\|  e^{-t(1-\Delta)} \Big( A^{(0)} - \linear[\leqN][r][] (0)  -\quadratic[\leqN][r][] (0)  \Big) 
\Big\|_{\Sc^{2-5\kappa}} 
+ \Big\|  \Duh \Big( \eqref{ansatz:eq-Y-new-1}
+ \eqref{ansatz:eq-Y-new-2}
+ \hdots 
+ \eqref{ansatz:eq-Y-new-9} \Big) \Big\|_{\Sc^{2-5\kappa}} \notag \\
\lesssim\, & \Big\| A^{(0)} - \linear[\leqN][r][](0)  -\quadratic[\leqN][r][] (0)   \Big\|_{\Cs_x^{-\kappa}} 
+ \tau^\delta \, \Big\| \eqref{ansatz:eq-Y-new-1}
+ \eqref{ansatz:eq-Y-new-2}
+ \hdots 
+ \eqref{ansatz:eq-Y-new-9} \Big\|_{\Wc^{-4\kappa,\nu}}. 
\label{nonlinear:eq-contract-p1}
\end{align}
Using the assumptions $\| A^{(0)} \|_{\Cs_x^{-\kappa}}\leq S$ and $\Xi_{\leq N}\in \Dc_R$, it follows that
\begin{equation}\label{nonlinear:eq-contract-p2}
\Big\| A^{(0)} - \linear[\leqN][r][](0)  -\quadratic[\leqN][r][] (0)   \Big\|_{\Cs_x^{-\kappa}} \lesssim R+S. 
\end{equation}
Using $\Xi_{\leq N}\in \Dc_R$, Lemma \ref{nonlinear:lem-dd}, Lemma \ref{nonlinear:lem-cubic}, and Lemma \ref{nonlinear:lem-remaining}, it holds that 
\begin{equation}\label{nonlinear:eq-contract-p3}
\Big\| \eqref{ansatz:eq-Y-new-1}
+ \eqref{ansatz:eq-Y-new-2}
+ \hdots 
+ \eqref{ansatz:eq-Y-new-9} \Big\|_{\Wc^{-4\kappa,\nu}}
\lesssim (R+ \big\| X_{\leq N} \big\|_{S^{1-2\kappa}} 
+ \big\| Y_{\leq N} \big\|_{S^{2-5\kappa}} )^3 
\lesssim \widetilde{C}^3 (R+S)^3. 
\end{equation}
By choosing $\widetilde{C}=\widetilde{C}(\kappa,\theta)\geq 1$ as sufficiently large, it follows from \eqref{nonlinear:eq-contract-p5}-\eqref{nonlinear:eq-contract-p3} that
\begin{equation}\label{nonlinear:eq-contract-p4}
\big\| \, \Gamma(X_{\leq N},Y_{\leq N}) \,  \big\|_{\Sc^{1-2\kappa}\times \Sc^{2-5\kappa}}   
\leq 
\frac{\widetilde{C}}{2} (R+S) + \tau^\delta \widetilde{C}^4 (R+S)^3. 
\end{equation}
By then choosing $C=C(\widetilde{C},\kappa,\theta)\geq 1$ as sufficiently large, it then follows from $\tau\leq C^{-1} (R+S)^{-C}$ that
\begin{equation}\label{nonlinear:eq-contract-p7}
\big\| \, \Gamma(X_{\leq N},Y_{\leq N}) \,  \big\|_{\Sc^{1-2\kappa}
\times \Sc^{2-5\kappa}} 
\leq \widetilde{C} (R+S). 
\end{equation}
As a result, 
\begin{equation*}
    \big( \Gamma^X(X_{\leq N},Y_{\leq N}), \Gamma^Y(X_{\leq N},Y_{\leq N})\big) \in \Xc_{R,S} \times \Yc_{R,S},
\end{equation*}
i.e., $\Gamma$ is a self-map on $\Xc_{R,S}\times \Yc_{R,S}$. \\

\emph{Step 3: Conclusion.} 
In Step 2, we proved the self-mapping property of $\Gamma$, i.e, \eqref{nonlinear:eq-contract-p7}. Using similar estimates, i.e., using minor modifications of Lemma \ref{nonlinear:lem-X}, Lemma \ref{nonlinear:lem-dd}, Lemma \ref{nonlinear:lem-cubic}, and Lemma \ref{nonlinear:lem-remaining}, we obtain for all $(X_{\leq N}, Y_{\leq N}),(\widetilde{X}_{\leq N}, \widetilde{Y}_{\leq N}) \in \Xc_{R,S} \times \Yc_{R,S}$ that 
\begin{equation}\label{nonlinear:eq-contract-p8}
\Big\| \Gamma(X_{\leq N}, Y_{\leq N}) - \Gamma(\widetilde{X}_{\leq N}, \widetilde{Y}_{\leq N}) \Big\|_{\Sc^{1-2\kappa} \times \Sc^{2-5\kappa}} 
\leq \frac{1}{2} \, 
\Big\| (X_{\leq N}, Y_{\leq N}) - (\widetilde{X}_{\leq N}, \widetilde{Y}_{\leq N}) \Big\|_{\Sc^{1-2\kappa} \times \Sc^{2-5\kappa}}. 
\end{equation}
Due to \eqref{nonlinear:eq-contract-p8}, $\Gamma$ is a contraction on $\Xc_{R,S} \times \Yc_{R,S}$. As a result, there exists a unique fixed point of $\Gamma$ on $\Xc_{R,S} \times \Yc_{R,S}$, which is the desired solution of the para-controlled system. \\ 

As is clear from \eqref{nonlinear:eq-Gamma-X} and \eqref{nonlinear:eq-Gamma-Y}, the definition of $\Gamma$ depends on the initial data $A^{(0)}\in \BS$ and the enhanced data set $\Xi_{\leq N}\in \Dc_R$. In order to emphasize this dependence, we now write 
\begin{equation}
\Gamma(X_{\leq N}, Y_{\leq N}) 
= \Gamma(X_{\leq N}, Y_{\leq N}; A^{(0)}, \Xi_{\leq N}). 
\end{equation}
By again using similar estimates, i.e., using minor modifications of Lemma \ref{nonlinear:lem-X}, Lemma \ref{nonlinear:lem-dd}, Lemma \ref{nonlinear:lem-cubic}, and Lemma \ref{nonlinear:lem-remaining}, we obtain for all $(X_{\leq N}, Y_{\leq N}),(\widetilde{X}_{\leq N}, \widetilde{Y}_{\leq N}) \in \Xc_{R,S} \times \Yc_{R,S}$, $A^{(0)},  \widetilde{A}^{(0)} \in \BS$, and $\Xi_{\leq N}, \widetilde{\Xi}_{\leq N}\in \Dc_R$ that
\begin{equation}\label{nonlinear:eq-contract-p9}
\begin{aligned}
&\Big\| \Gamma(X_{\leq N}, Y_{\leq N};A^{(0)}, \Xi_{\leq N})) - \Gamma(\widetilde{X}_{\leq N}, \widetilde{Y}_{\leq N};\widetilde{A}^{(0)},\widetilde{\Xi}_{\leq N}) \Big\|_{\Sc^{1-2\kappa} \times \Sc^{2-5\kappa}} \\ 
\leq\,&  \frac{1}{2} \, 
\Big\| (X_{\leq N}, Y_{\leq N}) - (\widetilde{X}_{\leq N}, \widetilde{Y}_{\leq N}) \Big\|_{\Sc^{1-2\kappa} \times \Sc^{2-5\kappa}}
+ \widetilde{C}(R,S) \Big( \big\| A^{(0)} - \widetilde{A}^{(0)} \big\|_{\Cs_x^{-\kappa}} + \dc \big( \Xi_{\leq N}, \widetilde{\Xi}_{\leq N}) \Big),
\end{aligned}
\end{equation}
where $\widetilde{C}(R,S) \geq 1$ is a constant depending on $R$ and $S$. From \eqref{nonlinear:eq-contract-p9}, we then directly obtain the desired Lipschitz-dependence on the initial data $A^{(0)}$ and enhanced data set $\Xi_{\leq N}$. 
\end{proof}

\section{Gauge-covariance}\label{section:gauge-covariance}

This section is devoted to the proof of our second main theorem, Theorem \ref{thm:gauge-covariance}. Recall the definition of space-time Besov spaces $\Cs_{tx}^\alpha$ from Section \ref{section:function-spaces}. These spaces will feature heavily in the technical arguments of this section. We begin with some formal discussion of gauge covariance. 
We note that much of the following discussion may already be found in \cite[Section 2]{CCHS22}. Nevertheless, we still include it here because it serves as important motivation for how the proof of Proposition \ref{prop:main-gauge-covariance} will proceed, and also there are some key points where our argument differs from that of \cite{CCHS22} -- see Remark \ref{remark:comparison-cchs-gauge-covariance}. \\

Suppose we have a solution $A$ to the SYM equation
\[ \ptl_t A = - D_A^* F_A - D_A D_A^* A + \xi. \]
Now let $g\colon [0,\infty) \times \T^2 \rightarrow G$ be a time-dependent gauge transformation. By a direct computation (recall the covariance properties \eqref{eq:covariance-formulas}), we then obtain that 
\begin{equation}\label{gauged:eq-gtgauged-A}
\ptl_t (A^g) = -D^*_{A^g} F_{A^g} - D_{A^g} D^*_{A^g} (\mrm{Ad}_g A) + \mrm{Ad}_g \xi - D_{A^g}((\ptl_t g) g^{-1}).
\end{equation}
In order to obtain gauge-covariance (Theorem \ref{thm:gauge-covariance}), we want that \eqref{gauged:eq-gtgauged-A} coincides with the stochastic Yang-Mills heat equation driven by the adjointed noise $\mrm{Ad}_g \xi$. To achieve this, we need to 
choose the gauge-transformation $g$ as a solution of 
\begin{equation}\label{gauged:eq-g-evolution}
(\ptl_t g)g^{-1} = - D_{A^g}^* ((dg) g^{-1}),
\end{equation}
which then implies the desired identity 
\[ \ptl_t (A^g) = -D^*_{A^g} F_{A^g} - D_{A^g} D_{A^g}^* A^g + \mrm{Ad}_g \xi.\]
We thus see that formally, solutions to the SYM equation are gauge covariant in the sense of Theorem \ref{thm:gauge-covariance}. \\

However, there is a major caveat ignored by the preceding discussion, which is that solutions to the SYM equation are defined by first mollifying the noise $\xi$ and then taking limits. The problem here is that mollification does not commute with $\Ad_g$, and therefore $P_{\leq N} (\Ad_g \xi)$ and  $\Ad_g (P_{\leq N} \xi)$ are different.  
This means that $\bA\big(A_0, P_{\leq N} (\Ad_g \xi)\big)$ and $\bA\big(A_0, \Ad_g(P_{\leq N} \xi)\big)$  are not necessarily equal. On the other hand, since mollification by $P_{\leq N}$ looks more and more like the identity as $N$ gets large, we should expect that in the $N \toinf$ limit, $P_{\leq N}$ and $\mrm{Ad}_g$ {\it do commute}. It would perhaps not be too hard to show some qualitative form of commutation (e.g. $P_{\leq N} (\Ad_g \xi) - \Ad_g (P_{\leq N} \xi) \ra 0$), however for our purposes we will need a more detailed quantitative version of this approximate commutation at finite values of $N$ (Proposition \ref{prop:linear-objects-close-gauge-covariance}). This is the most technical part of the paper, as we will need to understand the precise nature of the difference $\Ad_g^{-1} P_{\leq N} (\Ad_g \xi) - P_{\leq N} \xi$ and its regularity. It will turn out that this difference may be decomposed into a rougher part which later causes additional resonances (but fortunately these resonances cancel out in the end -- see Section \ref{section:gauged-transformed-enhanced-data-set}), and a smoother part which may be later handled deterministically.
\\

We now works towards stating the main result of Section \ref{section:gauge-covariance}, from which Theorem \ref{thm:gauge-covariance} will follow.
Let $A = \bA(A_0, \xi)$ be the solution to the SYM equation driven by $\xi$ with initial data $A_0$. Let $g_0 \in \Cs_x^{1-\kappa}(\T^2 \ra G)$ (recall from Remark \ref{remark:intro-cchs-comparision-gauge-covariance} that $\Cs_x^{1-\kappa}(\T^2 \ra G)$ is the natural space of gauge transformations associated to the space $\Cs_x^{-\kappa}(\T^2 \ra \frkg^2)$ of connections),
and let $g$ be the solution to
\beq\label{eq:gauge-transformation} (\ptl_t g) g^{-1} = - D^*_{A^g} ((dg) g^{-1}) = \ptl^j ((\ptl_j g)g^{-1}) + [\mrm{Ad}_g A^j, (\ptl_j g) g^{-1}], ~~ g(0) = g_0. \eeq
Note that the above is a nonlinear parabolic equation, and by classical local existence theory (and since $A \in C_t^0 \Cs_x^{-\kappa}$), we may obtain local solutions $g \in C_t^0 \Cs_x^{1-\kappa} \cap \Wc^{3/2-\kappa, 1/4}$.  (We note that this regularity is dictated by the fact that the initial data $g_0 \in \Cs_x^{1-\kappa}(\T^2 \ra G)$; if instead $g_0$ is smooth (say) then we could obtain $g \in C_t^0 \Cs_x^{2-}$, where now the regularity is dictated by the fact that $A \in C_t^0 \Cs_x^{-\kappa}$.)
Our first result, which is part of Theorem \ref{thm:gauge-covariance}, is that $\Ad_g \xi$ is well-defined as a space-time distribution.

\begin{lemma}\label{lemma:U-xi-defined}
Almost surely, the following holds for any $g_0 \in \Cs^{1-\kappa}(\T^2, G)$. Let $g$ be the solution to \eqref{eq:gauge-transformation} with initial data $g_0$, and let $\maxtime$ be the maximum simultaneous existence time of $A$ and $g$. Then for any $T \in [0, \maxtime) \cap [0, 1]$, we have that
\[ \Ad_g \xi := \lim_{N \toinf} \Ad_g P_{\leq N} \xi \text{ exists in $\Cs_{tx}^{(-1/2-100\kappa, -1-100\kappa)}((0, T))$}. \]
\end{lemma}

Given $A, g$, let $\gauged{A} = \gauged{A}_{\leq N}$ be the solution to
\beq\label{eq:gauged-A} \ptl_t \gauged{A} = - D^*_{\gauged{A}} F_{\gauged{A}} - D_{\gauged{A}} D^*_{\gauged{A}} \gauged{A} + D_{\gauged{A}}([A^j - \gauged{A}^j, g^{-1} \ptl_j g]) + \mrm{Ad}_{g}^{-1} P_{\leq N}(\mrm{Ad}_{g} \xi) , ~~ \gauged{A}(0) = A_0. \eeq

\begin{remark}
It may not be evident that $\Ad_g^{-1}$ has enough time regularity to multiply $P_{\leq N} (\Ad_g \xi)$, but this will become clear in Section \ref{section:glinear}, where we show that in addition to $\Ad_g \in C_t^0 \Cs_x^{1-\kappa} \cap \Wc^{3/2-\kappa, 1/4}$, we will have that $\ptl_t \Ad_g \in \Wc^{-1/2 - \kappa, 1/4}$, and this will be enough by our multiplication result Lemma \ref{lemma:paraproduct-Wc-space-time-besov-space}.
\end{remark}

(We omit the subscript ``$\leq N$" from $\gauged{A}$ in the above for notational brevity.)
Note that this equation is a modified version of the SYM equation, where the extra term involves $A$, which we have already solved for. The motivation for this equation is so that the gauge transformation of $\gauged{A}$ by $g$ satisfies the (mollified) SYM equation driven by $\mrm{Ad}_g \xi$, as the next lemma shows. The proof is deferred\footnote{If $\xi$ were smooth then the proof would be a computation similar to the one at the beginning of this section.} to Section \ref{section:contraction-mapping-gauge-covariance}.

\begin{lemma}\label{lemma:gtgauged-equation}
Let $\maxtime$ be the maximal simultaneous existence time of $A, g$. We have that $\bA(A_0^{g_0}, P_{\leq N} (\Ad_g \xi))$ exists on $[0, \maxtime) \cap [0, 1]$ and moreover is equal to $(\gauged{A}_{\leq N})^g$ on this interval. I.e., $(\gauged{A}_{\leq N})^g$ is the mild solution to
\[ \ptl_t ((\gauged{A}_{\leq N})^g)= -D^*_{(\gauged{A}_{\leq N})^g} F_{(\gauged{A}_{\leq N})^g} - D_{(\gauged{A}_{\leq N})^g} D^*_{(\gauged{A}_{\leq N})^g} \big((\gauged{A}_{\leq N})^g\big) + P_{\leq N} (\mrm{Ad}_g \xi), ~~ (\gauged{A}_{\leq N})^g(0) = \gauged{A}_0^{g_0}.\]
\end{lemma}

We now state the main result of Section \ref{section:gauge-covariance}. Our main gauge-covariance theorem, Theorem \ref{thm:gauge-covariance}, will easily follow from this proposition.

\begin{proposition}\label{prop:main-gauge-covariance}
Almost surely, the following holds for all $g_0 \in \Cs^{1-\kappa}(\T^2, G)$. Let $g$ be the solution to \eqref{eq:gauge-transformation} with initial data $g_0$, and let $\maxtime$ be the maximal simultaneous existence time of $A$ and $g$. Let $\gauged{A}_{\leq N}$ be defined as in \eqref{eq:gauged-A}. Then for any $T \in [0, \maxtime) \cap [0, 1]$, we have that
\[\begin{split}
\lim_{N \toinf} \|A_{\leq N} - \gauged{A}_{\leq N}\|_{C_t^0 \Cs_x^{-\kappa}([0, T] \times \T^2)} &= 0.
\end{split}\]
\end{proposition}

\begin{notation}
In the rest of Section \ref{section:gauge-covariance}, let $\maxtime$ be the maximum simultaneous existence time of $A$, $g$.
\end{notation}

\begin{remark}
Since we are only concerned with the local theory in this paper, we restrict the time $T$ to be at most 1 in Lemmas \ref{lemma:U-xi-defined}, \ref{lemma:gtgauged-equation} and Proposition \ref{prop:main-gauge-covariance}. This is mainly for convenience; the restriction may be removed with additional arguments. 
\end{remark}

\begin{remark}[Comparison with \cite{CCHS22}]\label{remark:comparison-cchs-gauge-covariance}
Our approach towards gauge covariance is no doubt inspired by the approach of \cite{CCHS22}. However, there is one big difference in our strategy: whereas \cite{CCHS22} analyzes a system involving the connection and the gauge transformation, we avoid having to consider such as system, and thus are able to rely more heavily on the existing local existence theory that we established in previous sections.

As mentioned in Remark \ref{remark:intro-cchs-comparision-gauge-covariance}, the difference between the two strategies lies in comparing $P_{\leq N}(\mrm{Ad}_g \xi)$ with $\mrm{Ad}_g (P_{\leq N} \xi)$ (which very roughly is what \cite{CCHS22} does), or comparing $\mrm{Ad}_g^{-1} P_{\leq N} (\mrm{Ad}_g \xi)$ with $P_{\leq N} \xi$ (which is what we do). We find it more convenient to make this latter comparison, because (as just mentioned) by doing so we do not need to solve a system involving the connection and the gauge transformation. Instead, we proceed in the following multi-step process:
\begin{enumerate}
    \item First obtain $A$ (by Theorem \ref{intro:thm-lwp}, our local existence theorem),
    \item then use this $A$ to obtain $g$ (by equation \eqref{eq:gauge-transformation})
    \item then use $A, g$ to define the equation \eqref{eq:gauged-A} for $\gauged{A}$.
\end{enumerate}
By construction, $\gauged{A}$ is driven by the modified noise $\mrm{Ad}_g^{-1} P_{\leq N}(\mrm{Ad}_g \xi)$, and by Lemma \ref{lemma:gtgauged-equation}, the gauge transformation $\gauged{A}^g$ exactly solves the SYM equation with driving noise $P_{\leq N}(\mrm{Ad}_g \xi)$. Thus, provided that $\gauged{A} = A$ (in the $N \toinf$ limit, which is the content of Proposition \ref{prop:main-gauge-covariance}), we obtain that $(\bA(A_0, \xi)^g) = A^g = \gauged{A}^g = \bA(A_0, \Ad_g \xi)$, which is exactly the content of Theorem \ref{thm:gauge-covariance}.
\end{remark}

Next, we show how Theorem \ref{thm:gauge-covariance} follows from Proposition \ref{prop:main-gauge-covariance}.

\begin{proof}[Proof of Theorem \ref{thm:gauge-covariance}]
Fix $T \in (0, \maxtime) \cap [0, 1]$. The following discussion is implicitly on $[0, T]$.
First, the fact that $\Ad_g \xi = \lim_{N \toinf} \Ad_g P_{\leq N} \xi$ exists follows by Lemma \ref{lemma:U-xi-defined}. Next, note by Proposition \ref{prop:main-gauge-covariance} and Lemma \ref{lemma:gtgauged-equation}, we have that
\[ A^g = \lim_{N \toinf} A_{\leq N}^g = \lim_{N \toinf} \gauged{A}_{\leq N}^g = \lim_{N \toinf} \bA(A_0^{g_0}, P_{\leq N}(\Ad_g \xi)).   \]
This shows that (1) $\bA(A_0^{g_0}, \Ad_g \xi)$ exists and (2) it is equal to $A^g = (\bA(A_0, \xi))^g$. 

Finally, to lower-bound $\maxtime$, first observe that by Proposition \ref{nonlinear:prop-wellposedness-para}, there is some time $\tau_S > 0$ such that if $\|A_0\|_{\Cs_x^{-\kappa}} \leq S$, then $A$ exists on $[0, \tau_S]$, and moreover $\|A\|_{C_t^0 \Cs_x^{-\kappa}([0, \tau_S])} \leq f(R, S)$ for some function $f$, where $R$ (the bound on the enhanced dataset) is a.s. finite. Now from standard parabolic local existence theory, $g$ (which is the solution to \eqref{eq:gauge-transformation}) exists on some interval $[0, \tau_0] \sse [0, \tau_S]$, where $\tau_{0}$ may be lower-bounded in terms of $\|g_0\|_{\Cs_x^{1-\kappa}}$ and $\|A\|_{C_t^0 \Cs_x^{-\kappa}([0, \tau_S])} \leq f(R, S)$. This completes the proof of the theorem.
\end{proof}

We now give an outline of the rest of Section \ref{section:gauge-covariance}, which is devoted to the proof of Proposition \ref{prop:main-gauge-covariance}.
In Section \ref{section:space-time-besov-space}, we collect many of the technical results about space-time Besov spaces that will be needed in Section \ref{section:glinear}.
In Section \ref{section:stochastic-estimate-gauge-covariance}, we give estimates on the new stochastic objects which appear. 
In Section \ref{section:glinear}, we take the first step towards showing that $\gauged{A}$ is close to $A$, by showing that the linear object obtained from the modified noise $\mrm{Ad}_g^{-1} P_{\leq N} (\mrm{Ad}_g \xi)$ is close to our original linear object $\linear[\leqN][r][]$ (which is obtained from $\xi$). Using this result, we then further show in Section \ref{section:gauged-transformed-enhanced-data-set} that the enhanced dataset obtained from the modified linear object is close to the original enhanced dataset $\Xi$. Once we know this, then Proposition \ref{prop:main-gauge-covariance} will follow by a simple contraction mapping argument, which we give in Section \ref{section:contraction-mapping-gauge-covariance}. 

\begin{remark}
Besides Section \ref{section:stochastic-estimate-gauge-covariance} and an application of Corollary \ref{objects:cor-quadratic-Q} in Section \ref{section:gauged-transformed-enhanced-data-set}, the arguments in Section \ref{section:gauge-covariance} are entirely deterministic.
\end{remark}

\subsection{Space-time Besov space estimates}\label{section:space-time-besov-space}

In this section, we collect the various technical results about space-time Besov spaces $\mc{B}_{p, q}^\alpha$ (which recall were defined in Section \ref{section:function-spaces}) and related function spaces that will be needed in Section \ref{section:glinear}. Recall also the spaces $\Wc^{\alpha, \nu}$ from Definition \ref{prelim:def-function-space}. We first make the following general definition, which is needed because at many points later on, when we want to multiply two distributions $fg$, we will need to decompose into paraproducts such as $fg = f \paransim_t g + f \parasim_t g$, and then handle each part separately.

\begin{definition}
Given two spaces $(X, \|\cdot\|_X)$, $(Y, \|\cdot\|_Y)$, with $X, Y \sse \mc{D}'(\R \times \T^2)$, we define 
\[ \|f\|_{X + Y} := \inf\{\|g\|_X + \|h\|_Y : f = g + h\}. \]
We also define
\[ \|f\|_{X \cap Y} := \|f\|_X + \|f\|_Y.\]
\end{definition}

We also find it essential to work with the following space-time Besov spaces with singular time weights.

\begin{definition}[Space-time Besov space with singular time weights]
Let $\alpha \in \R^2$, $\theta > 0$. Define the space $\Cs_{tx}^{\alpha, \theta}$ to be the set of $f \in \mc{D}'(\R \times \T^2)$ such that
\[ \|f\|_{\Cs_{tx}^{\alpha, \theta}} :=  \sup_{t > 0} t^\theta \big\| f ~|_{(t, \infty)} \big\|_{\Cs_{tx}^\alpha((t, \infty))} < \infty. \]
\end{definition}

\begin{example}[Paraproduct estimate with singular time weights]
Paraproduct estimates for space-time Besov spaces with singular time weights follow from the definition of these spaces combined with paraproduct estimates for space-time Besov spaces. We give some examples of such estimates here. Let $\alpha = (\alpha_t, \alpha_x) \in \R^2$, $\beta = (\beta_t, \beta_x) \in \R^2$, $\gamma = (\gamma_t, \gamma_x) \in \R^2$, $\theta > 0$. 
\begin{enumerate}[label=(\roman*)]
\item If $\alpha_x + \beta_x > 0$, $\gamma_t \leq \min(\alpha_t, 0) + \beta_t$, and $\gamma_x \leq \alpha_x + \beta_x$, then 
\[ \|f \parall_t \parasim_x g\|_{\Cs_{tx}^{\gamma, \theta}} \lesssim \|f\|_{\Cs_{tx}^\alpha} \|g\|_{\Cs_{tx}^{\beta, \theta}}.\]
\item If $\gamma_t \leq \min(\alpha_t, 0) + \beta_t$ and $\gamma_x \leq \min(\alpha_x, 0) + \beta_x$, then
\[ \|f \parall_t \paragg_x g\|_{\Cs_{tx}^{\gamma, \theta}} \lesssim  \|f\|_{\Cs_{tx}^\alpha} \|g\|_{\Cs_{tx}^{\beta, \theta}}.\]
\end{enumerate}
\end{example}

The proof of the following lemma is omitted, as it is entirely analogous to the proof of the analogous continuous embedding result for the usual Besov spaces.

\begin{lemma}[Continuous embedding]\label{lemma:space-time-besov-space-continuous-embedding}
Let $\alpha \in \R^2$, $1 \leq p_1 \leq p_2 \leq \infty$, $1 \leq q_1 \leq q_2 \leq \infty$. Let $\beta = \alpha - \big(\frac{1}{p_1} - \frac{1}{p_2}\big)(1, 2)$. We have that
\[ \|f\|_{\mc{B}^{\beta}_{p_2, q_2}} \lesssim \|f\|_{\mc{B}^{\alpha}_{p_1, q_1}}. \]
In other words, $\mc{B}^{\alpha}_{p_1, q_1}$ embeds continuously into $\mc{B}^{\beta}_{p_2, q_2}$.
\end{lemma}

The following result is the space-time analogue of Lemma \ref{prelim:lem-Q}. The proof is omitted as it is essentially the same as the proof of the earlier result.

\begin{lemma}\label{lemma:Q-gains-one-spatial-derivative-space-time-besov-space}
Let $\alpha \in \R^2$, $N$ be a dyadic scale, and $\ell \in [2]$. We have that
\[ \|Q^\ell_{> N} f\|_{\Cs_{tx}^{\alpha + (0, 1)}} \lesssim \|f\|_{\Cs_{tx}^{\alpha}}. \]
\end{lemma}

The proofs of the following four results may be found in Appendix \ref{appendix:space-time-besov-space}. 

\begin{lemma}[Schauder estimate for space-time Besov spaces]\label{lemma:space-time-Besov-space-Schauder} 
Let $\alpha = (\alpha_t, \alpha_x), \beta_x \in \R$ satisfy the following conditions:
\[ \alpha_t \in (-1, 0), ~~ \alpha_x \leq \beta_x < 2 + \alpha_x + 2\alpha_t. \]
We have that
\[ \|\Duh(f)\|_{C_t^0 \Cs_x^{\beta_x}([0, \infty))} \lesssim \|f\|_{\Cs^\alpha_{tx}}. \]
\end{lemma}

\begin{lemma}[Schauder estimate for space-time Besov spaces with singular time weights]\label{lemma:schauder-space-time-besov-space-singular-time-weights}
Let $\alpha = (\alpha_t, \alpha_x) \in \R^2$, $\beta = (\beta_t, \beta_x) \in \R^2$, $\gamma \in \R$, $\theta \geq \theta' > 0$ satisfy:
\[ \alpha_t, \beta_t \in (-1, 0), ~~ \alpha_x \leq \gamma < 2 + \alpha_x + 2\alpha_t, ~~ \beta_x \leq \gamma + 2\theta' < 2 + \beta_x + 2\beta_t. \]
We have that
\[ \|\Duh(f)\|_{\Wc^{\gamma+2\theta', \theta}([0, \infty))} \lesssim \|f\|_{\Cs_{tx}^\alpha \cap \Cs_{tx}^{\beta, \theta}}. \]
\end{lemma}

\begin{lemma}[Integral commutator estimate]\label{lemma:space-time-integral-commutator-estimate}
Let $T \geq 0$. Let $\alpha < 1$, $\nu \in [0, 1)$, $\beta = (\beta_t, \beta_x) \in \R^2$, $\beta_t \in (-1, 0)$. Let $f \in \Wc^{\alpha, \nu}([0, T])$, $\ptl_t f \in \Wc^{\alpha - 2, \nu}([0, T])$, $g \in \Cs_{tx}^{\beta}$. Let $\alpha', \nu'$ be such that
\[ \alpha + \beta_x + 2\beta_t \leq \alpha' < \alpha + \beta_x + 2\beta_t + 2, ~~ 1 + \nu' - \nu - \frac{\alpha' - \alpha - \beta_x - 2\beta_t}{2} > 0. \]
We have that
\[ \|\Duh(f \parall g) - f \parall \Duh(g)\|_{\Wc^{\alpha', \nu'}([0, T])} \lesssim \big(\|f\|_{\Wc^{\alpha - 2\nu, 0}([0, T])} + \|f\|_{\Wc^{\alpha, \nu}([0, T])} + \|\ptl_t f\|_{\Wc^{\alpha -2 , \nu}([0, T])}\big)\|g\|_{\Cs_{tx}^\beta}.  \]
\end{lemma}

\begin{lemma}\label{lemma:f-Wc-ptl-t-f-Wc}
Let $\alpha \in \R, \theta \in [0, 1)$. Let $f \in \Wc^{\alpha, \theta}(\R)$. For dyadic scales $L, N \in \dyadic$, $t \in \R$, we have that
\[ \|P_{L, N} f(t)\|_{L_x^\infty} \lesssim \min\big(L^\theta, \min(|t|, 1)^{-\theta}\big) N^{-\alpha} \|f\|_{\Wc^{\alpha, \theta}(\R)}. \]
If additionally $\ptl_t f \in \Wc^{\alpha - 2, \theta}(\R)$, then
\[ \|P_{L, N} f(t)\|_{L_x^\infty} \lesssim \min\big(L^\theta, \min(|t|, 1)^{-\theta}\big) N^{-\alpha} \min(1, L^{-1} N^2) \big(\|f\|_{\Wc^{\alpha, \theta}(\R)} + \|\ptl_t f\|_{\Wc^{\alpha-2, \theta}(\R)}\big). \]
\end{lemma}

Lemma \ref{lemma:f-Wc-ptl-t-f-Wc} has the following consequences.

\begin{corollary}\label{cor:Wc-implies-besov}
Let $\alpha \in \R, \theta \in [0, 1)$. We have that
\[ \begin{split}
\|f\|_{\Cs_{tx}^{(-\theta, \alpha)} \cap \Cs_{tx}^{(0, \alpha), \theta}} &\lesssim \|f\|_{\Wc^{\alpha, \theta}(\R)} \\
\|f\|_{\Cs_{tx}^{(-\theta, \alpha)} \cap \Cs_{tx}^{(1-\theta, \alpha-2)}} &\lesssim \|f\|_{\Wc^{\alpha, \theta}(\R)} + \|\ptl_t f\|_{\Wc^{\alpha -2, \theta}(\R)},  \\
\|f\|_{\Cs_{tx}^{(0, \alpha), \theta} \cap \Cs_{tx}^{(1, \alpha - 2), \theta}} &\lesssim \|f\|_{\Wc^{\alpha, \theta}(\R)} + \|\ptl_t f\|_{\Wc^{\alpha -2 , \theta}(\R)}.
\end{split}\]
\end{corollary}
\begin{proof}
By Lemma \ref{lemma:f-Wc-ptl-t-f-Wc}, we have that
\[ \|P_{L, N} f\|_{L_{tx}^\infty} \lesssim L^\theta N^{-\alpha} \|f\|_{\Wc^{\alpha, \theta}}. \]
This shows that $\|f\|_{\Cs_{tx}^{(-\theta, \alpha)}} \lesssim \|f\|_{\Wc^{\alpha, \theta}}$. Next, given $t_0 > 0$, define the extension $h(t) := f(\max(t, t_0))$ of $f|_{(t_0, \infty)}$ to $\R$. Observe that 
\beq\label{eq:cor-Wc-implies-besov-intermediate} \|h\|_{\Wc^{\alpha, 0}} \lesssim t_0^{-\theta} \|f\|_{\Wc^{\alpha, \theta}}.  \eeq
From this and our previously proved estimate, we obtain 
\[\|h\|_{\Cs_{tx}^\alpha} \lesssim t_0^{-\theta} \|f\|_{\Wc^{\alpha, \theta}}. \]
This implies that $\|f\|_{\Cs_{tx}^{(0, \alpha), \theta}} \lesssim \|f\|_{\Wc^{\alpha, \theta}}$, which completes the proof of the first estimate.

For the second estimate, similarly note that by Lemma \ref{lemma:f-Wc-ptl-t-f-Wc}, we have that
\[ \|P_{L, N} f\|_{L_{tx}^\infty} \lesssim \min\big(L^\theta N^{-\alpha}, L^{-(1-\theta)} N^{-(\alpha - 2)}\big) \big(\|f\|_{\Wc^{\alpha, \theta}} + \|\ptl_t f\|_{\Wc^{\alpha -2, \theta}}\big).   \]
For the last estimate, let $t_0 > 0$, and let $h$ be defined as before. Observe that 
\[  \|\ptl_t h\|_{\Wc^{\alpha-2, 0}} \lesssim t_0^{-\theta} \|f\|_{\Wc^{\alpha-2, \theta}}. \]
By \eqref{eq:cor-Wc-implies-besov-intermediate}, the above, and the second estimate, we then obtain
\[ \|h\|_{\Cs_{tx}^{(0, \alpha)} \cap \Cs_{tx}^{(1, \alpha-2)}} \lesssim t_0^{-\theta} \big(\|f\|_{\Wc^{\alpha, \theta}} + \|\ptl_t f\|_{\Wc^{\alpha -2, \theta}}\big). \]
The last estimate now follows.
\end{proof}

The following lemma allows us to multiply distributions $f, g$ in the case where the regularities of $f$ and $g$ sum to a positive number only after introducing singular weights for $f$.

\begin{lemma}\label{lemma:paraproduct-Wc-space-time-besov-space}
Let $\alpha > 0$, $\theta \in [0, 1)$, $\beta = (\beta_t, \beta_x) \in \R^2$, $\gamma \in \R$ satisfy
\[ \beta_t \in (-1, 0), ~~ 0 < \gamma < \alpha + \beta_x + 2\beta_t. \]
Then
\[\begin{split}
\|f \parasim_t \parall_x g\|_{\Wc^{\min(\gamma, \beta_x), \theta}(\R)} , ~~ \|f\paragg_t \parall_x g\|_{\Wc^{\min(\gamma, \beta_x), \theta}(\R)} &\lesssim \big(\|f\|_{\Wc^{\alpha, \theta}(\R)} + \|\ptl_t f\|_{\Wc^{\alpha-2, \theta}(\R)}\big) \|g\|_{\Cs_{tx}^\beta} \\
\|f \parasim_{tx} g\|_{\Wc^{\gamma, \theta}(\R)}, ~~ \|f \paragg_t \parasim_x g\|_{\Wc^{\gamma, \theta}(\R)} &\lesssim \big(\|f\|_{\Wc^{\alpha, \theta}(\R)} + \|\ptl_t f\|_{\Wc^{\alpha-2, \theta}(\R)}\big) \|g\|_{\Cs_{tx}^\beta}, \\
\|f \parasim_{t} \paragg_x g\|_{\Wc^{\min(\gamma, \gamma-\beta_x), \theta}(\R)}, ~~ \|f \paragg_{tx} g\|_{\Wc^{\min(\gamma, \gamma-\beta_x), \theta}(\R)} &\lesssim \big(\|f\|_{\Wc^{\alpha, \theta}(\R)} + \|\ptl_t f\|_{\Wc^{\alpha-2, \theta}(\R)}\big) \|g\|_{\Cs_{tx}^\beta} \\
\|f \parasim_t g\|_{\Wc^{\min(\gamma, \gamma - \beta_x, \beta_x), \theta}(\R)}, ~~\|f \paragg_t g\|_{\Wc^{\min(\gamma, \gamma - \beta_x, \beta_x), \theta}(\R)} &\lesssim  \big(\|f\|_{\Wc^{\alpha, \theta}(\R)} + \|\ptl_t f\|_{\Wc^{\alpha-2, \theta}(\R)}\big) \|g\|_{\Cs_{tx}^\beta}.
\end{split}\]
\end{lemma}
\begin{remark}
It may be helpful to keep in mind that one of the main uses cases of this lemma is when $f \in \Wc^{2+, 1/2+}$, $\ptl_t f \in \Wc^{0+, 1/2+}$, and $g \in \Cs_{tx}^{(-1/2-, -1-)}$. The lemma then gives $f \parasim_{tx} g \in \Wc^{0+, 1/2+}$ and similarly for the other combinations of high$\times$high and high$\times$low space-time paraproducts.
\end{remark}
\begin{proof}
Let $L, L' \in \dyadic$ be such that $L \gtrsim L'$. Let $N, N' \in \dyadic$. Let $\varep := \alpha + \beta_x + 2\beta_t - \gamma > 0$. Without loss of generality, take $t > 0$. Due to our assumptions on the various parameters, all the stated claims are consequences of the following estimate:
\[ \|(P_{L, N} f)(t) (P_{L', N'} g)(t)\|_{L_x^\infty} \lesssim L^{-\varep/4} N^{-\varep/2} \min(t, 1)^{-\theta} N^{-(\gamma - \beta_x)} (N')^{-\beta_x} \big(\|f\|_{\Wc^{\alpha, \theta}} + \|\ptl_t f\|_{\Wc^{\alpha-2, \theta}}\big) \|g\|_{\Cs_{tx}^\beta}.  \]
Towards this end, first note that
\[ \|(P_{L, N} f)(t) (P_{L', N'} g)(t)\|_{L_x^\infty}  \lesssim L^{-\beta_t} (N')^{-\beta_x} \|(P_{L, N} f)(t)\|_{L_x^\infty} \|g\|_{\Cs_{tx}^{\beta}}. \]
Next, by Lemma \ref{lemma:f-Wc-ptl-t-f-Wc}, we have that
\[ \|(P_{L, N} f)(t)\|_{L_x^\infty} \lesssim \min(t, 1)^{-\theta} \min(1, L^{-1} N^2) N^{-\alpha} \big(\|f\|_{\Wc^{\alpha, \theta}} + \|\ptl_t f\|_{\Wc^{\alpha-2, \theta}}\big).\]
Combining the previous two estimates, we obtain the further upper bound (recall $\beta_t < 0$)
\[ \min(t, 1)^{-\theta} (L^{-1} N^2)^{\varep/4} L^{-\beta_t} (N')^{-\beta_x}  (L^{-1} N^2)^{-\beta_t} N^{-\alpha} \big(\|f\|_{\Wc^{\alpha, \theta}} + \|\ptl_t f\|_{\Wc^{\alpha-2, \theta}}\big) \|g\|_{\Cs_{tx}^{\beta}}.\]
To finish, recall that $\varep = \alpha + \beta_x + 2\beta_t - \gamma$, so that
\[ (L^{-1} N^2)^{\varep/4} L^{-\beta_t}  (L^{-1} N^2)^{-\beta_t} N^{-\alpha} = L^{-\varep/4} N^{-(\alpha + 2\beta_t - \varep/2)}  = L^{-\varep/4} N^{-\varep/2} N^{-(\gamma-\beta_x)} . \qedhere\]
\end{proof}

The following commutator lemma is needed in Section \ref{section:lo-hi}.

\begin{lemma}\label{lemma:P-N-commutator-space-time}
Let $\alpha = (\alpha_t, \alpha_x) \in \R^2$, $\alpha' = (\alpha_t', \alpha_x') \in \R^2$, $\gamma = (\gamma_t, \gamma_x) \in \R^2$, $\gamma' = (\gamma_t', \gamma_x') \in \R^2$, $\beta = (\beta_t, \beta_x) \in (-\infty, 0)^2$, $\theta, \theta' > 0$ be such that 
\[ \alpha_x, \alpha_x', \gamma_x, \gamma_x' < 2,~~ \alpha_t \geq 0, ~~\gamma_t \geq 0, ~~\alpha_t' + \beta_t > 0, ~~\gamma_t' + \beta_t > 0. \]
Let $f \in \Cs_{tx}^\alpha \cap \Cs_{tx}^{\alpha'} \cap \Cs_{tx}^{\gamma, \theta} \cap \Cs_{tx}^{\gamma', \theta'}$, $g \in \Cs_{tx}^\beta$. For $N \in \dyadic$, $\delta \in [0, 2)$, we have that
\[ \big\|P_{\leq N} (f \parall g) - f \parall P_{\leq N} g - \ptl_j f \parall Q^j_{> N} g\big\|_{X + Y} \lesssim N^{-\delta} \|f\|_{\Cs_{tx}^\alpha \cap \Cs_{tx}^{\alpha'} \cap \Cs_{tx}^{\gamma, \theta} \cap \Cs_{tx}^{\gamma', \theta'}} \|g\|_{\Cs_{tx}^\beta},\]
where $X = \Cs_{tx}^{\beta + (0, \alpha_x-\delta)} \cap \Cs_{tx}^{\beta + (0, \gamma_x - \delta), \theta}$ and $Y = \Cs_{tx}^{\alpha' + \beta - (0, \delta)} \cap \Cs_{tx}^{\gamma' + \beta - (0, \delta), \theta'}$.
\end{lemma}
\begin{remark}
It may be helpful to keep in mind that we will apply this lemma with $f \in L_t^\infty \Cs_x^{1 -\kappa} \cap \Cs_{tx}^{(1/2+, 0-)} \cap \Wc^{1+, 0+} \cap \Cs_{tx}^{(1/2+, 0+), 0+}$, $g \in \Cs_{tx}^{(-1/2-, -1-)}$. In this case (taking $\delta$ sufficiently small) $X = \Cs_{tx}^{(-1/2-, 0-)} \cap \Cs_{tx}^{(-1/2-, 0+), 0+}$, $Y = \Cs_{tx}^{(0+, -1-)} \cap \Cs_{tx}^{(0+, -1+), 0+}$. The point of having the additional singular time weights, instead of just using $\Cs_{tx}^{(-1/2-, 0-)}$ and $\Cs_{tx}^{(0+, -1-)}$, is so that upon applying the Schauder estimate (Lemma \ref{lemma:schauder-space-time-besov-space-singular-time-weights}), we will obtain a bound in $C_t^0 \Cs_x^{1-} \cap \Wc^{1+, 0+}$, as opposed to a bound just on $C_t^0 \Cs_x^{1-}$. The extra control on $\Wc^{1+, 0+}$ is crucial for the arguments in Section \ref{section:gauged-transformed-enhanced-data-set} -- see Remark \ref{remark:why-singular-weights-necessary} for more detail on why.
\end{remark}
\begin{proof}
For brevity, define the bilinear operator
\[ \mc{B}(f, g) := P_{\leq N}(f \parall g) - f \parall P_{\leq N} g - \ptl_j f \parall Q^j_{> N} g. \]
Define
\[ F_{\ll} := \sum_{L} \mc{B}(P_{\ll L}^t f, P_{L}^t g), ~~ F_{\gtrsim} :=  \sum_{L \gtrsim L'} \mc{B}(P_L^t f, P_{L'}^t g) = \sum_{L} \mc{B}(P_L^t f, P_{\lesssim L}^t g). \]
It suffices to show that
\[\begin{split}
\|F_{\ll}\|_{\Cs_{tx}^{\beta + (0, \alpha_x-\delta)} \cap \Cs_{tx}^{\beta + (0, \gamma_x - \delta), \theta}} &\lesssim N^{-\delta} \big(\|f\|_{\Cs_{tx}^\alpha} + \|f\|_{\Cs_{tx}^{\gamma, \theta}}\big) \|g\|_{\Cs_{tx}^\beta}, \\
\|F_{\gtrsim}\|_{\Cs_{tx}^{\alpha' + \beta - (0, \delta)} \cap \Cs_{tx}^{\gamma' + \beta - (0, \delta), \theta'}} &\lesssim N^{-\delta} \big(\|f\|_{\Cs_{tx}^{\alpha'}} + \|f\|_{\Cs_{tx}^{\gamma', \theta'}}\big) \|g\|_{\Cs_{tx}^\beta} .  
\end{split}\]
Towards this end, by Lemma \ref{lemma:P-N-commutator-space}, we have for fixed $L$ that
\[ \|\mc{B}(P^t_{\ll L} f, P^t_L g)\|_{L_t^\infty \Cs_x^{\alpha_x + \beta_x - \delta}} \lesssim N^{-\delta} \|P^t_{\ll L}  f\|_{L_t^\infty \Cs_x^{\alpha_x}} \|P^t_L g\|_{L_t^\infty \Cs_x^{\beta_x}} \lesssim N^{-\delta} L^{-\beta_t} \|f\|_{\Cs_{tx}^{\alpha}} \|g\|_{\Cs_{tx}^\beta},   \]
where we used that $\alpha_t \geq 0$. Since $\mc{B}(P^t_{\ll L} f, P_L^t g)$ is localized to temporal frequency $L$, this shows that
\[ \|F_{\ll}\|_{\Cs_{tx}^{\beta + (0, \alpha_x - \delta)}} \lesssim N^{-\delta} \|f\|_{\Cs_{tx}^{\alpha}} \|g\|_{\Cs_{tx}^\beta}.\]
For the other bound on $F_{\ll}$, note that by definition of $\Cs_{tx}^{\gamma, \theta}$, for any $t_0 > 0$ there exists an extension $h$ of $f |_{(t_0, \infty)}$ to $\R \times \T^2$ such that $\|h\|_{\Cs_{tx}^{\gamma}} \lesssim t_0^{-\theta} \|f\|_{\Cs_{tx}^{\gamma, \theta}}$. Note that for each $L$, $\mc{B}(P^t_{\ll L} h, P_L^t g)$ is an extension of $\mc{B}(P_{\ll L}^t f, P_L^t g) |_{(t_0, \infty)}$ to $\R \times \T^2$. Moreover, by the previous argument, we have that
\[ \|\mc{B}(P_{\ll L}^t h, P_L^t g)\|_{\Cs_{tx}^{\beta + (0, \gamma_x - \delta)}} \lesssim N^{-\delta} \|h\|_{\Cs_{tx}^{\gamma}} \|g\|_{\Cs_{tx}^\beta} \lesssim t_0^{-\theta} N^{-\delta} \|f\|_{\Cs_{tx}^{\gamma, \theta}} \|g\|_{\Cs_{tx}^\beta}.\]
This shows that
\[ \|F_{\ll}\|_{\Cs_{tx}^{\beta + (0, \gamma_x - \delta), \theta}} \lesssim  N^{-\delta} \|f\|_{\Cs_{tx}^{\gamma, \theta}} \|g\|_{\Cs_{tx}^\beta}.\]

Next, we bound $F_{\gtrsim}$. By Lemma \ref{lemma:P-N-commutator-space}, we have for fixed $L$ that
\[ \|\mc{B}(P_{L}^t f, P_{\lesssim L}^t g)\|_{L_t^\infty \Cs_x^{\alpha_x' + \beta_x - \delta}} \lesssim N^{-\delta} \|P_L^t f\|_{L_t^\infty \Cs_x^{\alpha_x'}} \|P_{\lesssim L}^t g\|_{L_t^\infty \Cs_x^{\beta_x}} \lesssim N^{-\delta} L^{-(\alpha_t' + \beta_t)} \|f\|_{\Cs_x^{\alpha'}} \|g\|_{\Cs_tx^{\beta}}, \]
where we used that $\beta_t \leq 0$ in the second inequality. Observe that $\mc{B}(P_L^t f, P_{\lesssim L}^t g)$ is supported on temporal frequencies $\lesssim L$, and thus since $\alpha_t' + \beta_t > 0$ by assumption, we obtain that
\[ \|F_{\gtrsim} \|_{\Cs_{tx}^{\alpha' + \beta - (0, \delta)}} \lesssim N^{-\delta} \|f\|_{\Cs_x^{\alpha'}} \|g\|_{\Cs_{tx}^{\beta}}.\]
The bound on $F_{\gtrsim}$ involving singular time weights may be argued similarly to before.
\end{proof}

\begin{remark}
We have stated the various technical results in terms of $\Cs_{tx}^\alpha$ spaces. These results also hold for  $\Cs_{tx}^\alpha((0, T))$ spaces. This follows directly from the definition (Definition \ref{def:space-time-besov-space-finite-interval}) of these latter spaces.
\end{remark}

\subsection{Stochastic estimate}\label{section:stochastic-estimate-gauge-covariance}

In this section, we prove the additional stochastic estimates which are needed for Proposition \ref{prop:main-gauge-covariance}.

\begin{lemma}\label{lemma:white-noise-space-time-besov-space-bound}
For all $p \geq 1$, we have that
\begin{equation}\label{eq:white-noise-space-time-besov-space-bound}
\E\Big[ \|\xi \|_{\Cs_{tx}^{(-\frac{1}{2}-\frac{\kappa}{2},-1-\kappa)}((0, 1))}^p\Big]^{1/p} \lesssim p^{1/2}.
\end{equation}
\end{lemma}
\begin{proof}
It suffices to show the following estimate for the projection of $\xi$ to the time interval $(0, 1)$:
\[ \E\Big[ \|\xi \ind_{(0, 1)} \|_{\Cs_{tx}^{(-\frac{1}{2}-\frac{\kappa}{2},-1-\kappa)}}^p\Big]^{1/p} \lesssim p^{1/2} .\]
For brevity, let $\eta = \xi \ind_{(0, 1)}$. Let $L,N\in \dyadic$ and let $\chi\colon \R \times \T^2 \rightarrow [0,1]$ be a smooth cutoff function satisfying $\chi(t)=1$ for all $t\in [0,1]$. Using the decay properties of $\widecheck{\rho}_L, \widecheck{\rho}_N$, and $\chi$, we obtain for all $t\in \R$ and $x\in \T^2$ that 
\begin{equation}\label{eq:white-noise-space-time-besov-space-bound-p1}
\E\Big[ \big|P_{L, N} \big(\chi \eta\big) (t, x)\big|^2_{\frkg}\Big]
= \int_0^1 ds \int_{\T^2} dy \widecheck{\rho}_L(t-s)^2 \chi(s)^2 \widecheck{\rho}_N(x-y)^2  \lesssim L N^2 \langle t \rangle^{-4}. 
\end{equation}
We now let $q=q_\kappa:= 4\kappa^{-1}$, i.e., we let $q$ be sufficiently large depending on $\kappa$. 
For brevity, we now also let $\alpha := (-\frac{1}{2}-\frac{\kappa}{2}, 1-\kappa)$ 
and let $\alpha_q:= (-\frac{1}{2}-\frac{\kappa}{2}+\frac{1}{q}, 1-\kappa+\frac{2}{q})$. Using the embedding properties of Besov spaces (Lemma \ref{lemma:space-time-besov-space-continuous-embedding}), it follows that
\begin{equation*}
\|\eta \|_{\Cs_{tx}^{\alpha}} \lesssim \|\eta \|_{\mc{B}_{q, q}^{\alpha_q}}.
\end{equation*}
For all $p\geq q$, it then follows from Minkowski's integral inequality that
\begin{align}
\E\Big[ \|\eta \|_{\mc{B}_{q, q}^{\alpha_q}}^p\Big]^{1/p} 
&= \E \Big[ \big\| L^{-\frac{1}{2}-\frac{\kappa}{2}+\frac{1}{q}} N^{-1-\kappa+\frac{2}{q}}
    \big\| P_L^t P_N^x (\chi \eta) \big\|_{L_t^q L_x^q} \big\|_{\ell_L^q \ell_N^q}^p \Big]^{\frac{1}{p}} \notag \\ 
&\leq \Big\| L^{-\frac{1}{2}-\frac{\kappa}{2}+\frac{1}{q}} N^{-1-\kappa+\frac{2}{q}}
    \Big\| \E \big[ \big| P_L^t P_N^x (\chi \eta) (t,x) \big|^p \big]^{\frac{1}{p}} \Big\|_{L_t^q L_x^q} \Big\|_{\ell_L^q \ell_N^q}.
\label{eq:white-noise-space-time-besov-space-bound-p2}
\end{align}
Using Gaussian hypercontractivity and \eqref{eq:white-noise-space-time-besov-space-bound-p1}, we obtain that
\begin{equation*}
 \E \big[ \big| P_L^t P_N^x (\chi \eta) (t,x) \big|^p \big]^{\frac{1}{p}}
 \lesssim \sqrt{p} \,  \E \big[ \big| P_L^t P_N^x (\chi \eta) (t,x) \big|^2 \big]^{\frac{1}{2}}
 \lesssim \sqrt{p} N L^{\frac{1}{2}} \langle t \rangle^{-2}. 
\end{equation*}
Inserting this into \eqref{eq:white-noise-space-time-besov-space-bound-p2}, we obtain that 
\begin{equation*}
\eqref{eq:white-noise-space-time-besov-space-bound-p2} 
\lesssim \sqrt{p}\, \Big\| L^{-\frac{\kappa}{2}+\frac{1}{q}} N^{-\kappa+\frac{2}{q}} 
\big\| \langle t \rangle^{-2} \big\|_{L_t^q} \Big\|_{\ell_L^2 \ell_N^2} 
\lesssim \sqrt{p},
\end{equation*}
where we used that $-\frac{\kappa}{2}+\frac{1}{q}<0$. In total, this completes the proof of \eqref{eq:white-noise-space-time-besov-space-bound} for all $p\geq q$. The case $1\leq p <q$ can be deduced from the case $p=q$ and H\"{o}lder's inequality (w.r.t. the probability measure), and thus we obtain the desired estimate \eqref{eq:white-noise-space-time-besov-space-bound} for all $1\leq p <\infty$. 
\end{proof}

Later on, we will need to assume that $\big(P_{N_1}P_{N_2} \Duh\big(\linear\big)\big) \parasim \xi$ may be defined as an element of $\Cs_{tx}^{(-1/2-\kappa, 1-\kappa)}$ (see Lemma \ref{lemma:U-r-hi-hi-noise-estimate}), and thus we will need the following lemma.

\begin{lemma}\label{lemma:Duhamel-linear-times-noise-besov-space-bound}
Let $N_1, N_2, N_3, N_4 \in \dyadic$ be dyadic scales with $N_1 \sim N_2 \sim N_3 \sim N_4$. For all $p \geq 1$, we have that
\[ \E\Big[\Big\|P_{N_1}^x P_{N_2}^x P_{N_3}^x\Duh\big(\linear[][r][]\big) \otimes P_{N_4}^x \xi \Big\|^p_{\Cs_{tx}^{(-1/2-\kappa, 1-\kappa)}((0, 1))} \Big]^{1/p} \lesssim p N_1^{-\kappa/4}.\]
\end{lemma}

Before we prove this lemma, we first prove the following intermediate result.

\begin{lemma}\label{lemma:duhamel-linear-times-noise-l2-bound}
Let $N_1, N_2, N_3, N_4, L, M$ be dyadic scales. We have that for all $t \in \R$, $x \in \T^2$,
\[ \E\Big[ \Big|P_{L, M}\Big( P_{N_1}^x P_{N_2}^x P_{N_3}^x \Duh\big(\linear[][r][]\big) \otimes P_{N_4}^x (\xi \ind_{(0, 1)})\Big)(t, x)\Big|^2_{\cfrkg^{\otimes 2}}\Big] \lesssim \frac{K_1^2 K_2^2}{N_1^6} \int_0^1 ds \widecheck{\rho}_L(t-s)^2, \]
where $K_1 = \min(M, N_i, i \in [4])$, $K_2 = \min_{(2)}(M, N_i, i \in [4])$ (here, $\min_{(2)}$ denotes the second smallest number).
\end{lemma}

\begin{remark}\label{remark:no-extra-resonace}
Here, it is crucial that we only frequency localize in time for the combined object $\Duh\big(\linear\big) \otimes \xi$, as opposed to individually localizing $\Duh\big(\linear\big)$ and $\xi$ (i.e. taking $P_{L_1, N_1} \Duh\big(\linear\big)$ and $P_{L_2, N_2} \xi$). This is because the latter approach would result in a resonance which we would then later need to account for. Roughly, this is due to the fact that $P_{L_1, N_1} \Duh\big(\linear\big) \otimes P_{L_2, N_2} \xi$ is like the product of It\^{o} integrals (which in general may have nonzero mean), whereas $P_{L, M} \big(\Duh\big(\linear\big) \otimes \xi\big)$ is like a single It\^{o} integral (which always has zero mean). This is one of the differences between our work and \cite{CCHS22}, in which an extra resonance does occur when arguing gauge-covariance, essentially because $\xi$ is mollified in both space and time. This extra resonance is responsible for the additional counterterm proportional to $(dg) g^{-1}$ in the gauge-covariance arguments of \cite{CCHS22}.
\end{remark}

\begin{proof}
Arguing similarly as in Example \ref{prelim:example-integrated-linear}, we have that
\[ \Duh\big(\linear[N_1][r][]\big)(t, x) = \ind(t \geq 0) \sum_n \rho_{N_1}(n) \e_n(x) \int_{-\infty}^t (t- \max(s, 0)) e^{-(t-s) \fnorm{n}^2} \hat{\xi}(s, n) ds.\]
Let 
\[ Z_n(t) := \ind(t \geq 0) \int_{-\infty}^{\min(t, 1)} (t-\max(s, 0)) e^{-(t-s) \fnorm{n}^2} \hat{\xi}(s, n) ds, ~~ X_n(t) := \int_{-\infty}^{\min(t, 1)} e^{-(t-s) \fnorm{n}^2} \hat{\xi}(s, n) ds. \]
We may compute by It\^{o} isometry that
\[ \E\Big[|Z_n(t)|_{\cfrkg}^2\Big] = \dim(\frkg)^2 \int_{-\infty}^{\min(t, 1)} (t-\max(s, 0))^2 e^{-2(t-s)\fnorm{n}^2} ds \lesssim \fnorm{n}^{-6}. \]
For brevity, let $\eta = \xi \ind_{(0, 1)}$, $P_{N_1 \cdots N_3}^x = P_{N_1}^x P_{N_2}^x P_{N_3}^3$, and $\rho_{N_1 \cdots N_3}(n) = \rho_{N_1}(n) \rho_{N_2}(n) \rho_{N_3}(n)$. Now, observe that
\[ P_{L, M}\Big( P_{N_1 \cdots N_3}^x \Duh\big(\linear[][r][]\big) \otimes P_{N_4}^x \eta\Big)(t, x) = \sum_{n_1, n_2} \rho_M(n_{12}) \rho_{N_1 \cdots N_3}(n_1) \rho_{N_4}(n_2) \e_{n_{12}}(x) \int_0^1 ds \widecheck{\rho}_L(t-s) Z_{n_1}(s) \otimes \hat{\eta}(s, n_2) .\]
In the following, we now view $t$ as fixed and $s$ as the new time variable.
For $n_1, n_2 \in \Z^2$, define the process $M_{n_1, n_2} : [0, 1] \ra \frkg \otimes \frkg$ via the SDE
\[ dM_{n_1, n_2}(s) = \widecheck{\rho}_L(t-s) Z_{n_1}(s) \otimes \hat{\eta}(s, n_2), ~~ M_{n_1, n_2}(0) := 0.\]
Observe that $M_{n_1, n_2}$ is a martingale, and that
\[  M_{n_1, n_2}(1) = \int_\R ds \widecheck{\rho}_L(t-s) Z_{n_1}(s) \otimes \hat{\eta}(s, n_2). \]
Given additionally $m_1, m_2 \in \Z^2$, we may compute
\[\begin{split}
d\langle M_{n_1, n_2}, M_{m_1, m_2} \rangle_{\cfrkg^{\otimes 2}} &=  \langle dM_{n_1, n_2}, M_{m_1, m_2}\rangle_{\cfrkg^{\otimes 2}} + \langle M_{n_1, n_2}, dM_{m_1, m_2}\rangle_{\cfrkg^{\otimes 2}}  + \langle dM_{n_1, n_2}, dM_{m_1, m_2} \rangle_{\cfrkg^{\otimes 2}} \\
&= \langle dM_{n_1, n_2}, M_{m_1, m_2}\rangle_{\cfrkg^{\otimes 2}} + \langle M_{n_1, n_2}, dM_{m_1, m_2}\rangle_{\cfrkg^{\otimes 2}}  + \delta_{n_2 m_2} \widecheck{\rho}_L(t-s)^2 \langle Z_{n_1}, Z_{m_1} \rangle_{\cfrkg} dt. 
\end{split}\]
Noting that the first two terms above are derivatives of martingales, we obtain
\[\begin{split}
\E\Big[\langle M_{n_1, n_2}(1), M_{m_1, m_2}(1) \rangle_{\cfrkg^{\otimes ^2}} \Big] &= \delta_{n_2 m_2} \int_0^1 ds  \widecheck{\rho}_L(t-s)^2 \E\big[\langle Z_{n_1}(s), Z_{m_1}(s) \rangle_{\cfrkg}\big] ds \\
&\lesssim  \delta_{n_1 m_1}  \delta_{n_2 m_2} \fnorm{n_1}^{-6}\int_{0}^1 ds \widecheck{\rho}_L(t-s)^2. \end{split}\]
We may thus obtain
\[\begin{split}
\E\Big[&\Big|P_{L, M}\Big( P_{N_1 \cdots N_3}^x \Duh\big(\linear[][r][]\big) \otimes P_{N_4}^x \xi\Big)(t, x)\Big|_{\cfrkg^{\otimes 2}}^2 \Big] \lesssim \\
&\sum_{n_1, n_2} \rho_M(n_{12})^2 \rho_{N_1 \cdots N_3}(n_1)^2 \rho_{N_4}(n_2)^2 \fnorm{n_1}^{-6} \int_0^1 ds \widecheck{\rho}_L(t-s)^2
\lesssim \frac{K_1^2 K_2^2}{N_1^6} \int_0^1 ds \widecheck{\rho}_L(t-s)^2. \qedhere
\end{split}\]
\end{proof}

\begin{proof}[Proof of Lemma \ref{lemma:Duhamel-linear-times-noise-besov-space-bound}]
Given Lemma \ref{lemma:duhamel-linear-times-noise-l2-bound}, the proof of Lemma \ref{lemma:Duhamel-linear-times-noise-besov-space-bound} becomes very similar to the proof of Lemma \ref{lemma:white-noise-space-time-besov-space-bound}. The details are omitted.
\end{proof}

\subsection{Modified linear object}\label{section:glinear}

The first step towards showing Proposition \ref{prop:main-gauge-covariance} is to understand the difference between the original linear object $\linear$ and the modified linear object arising from the equation \eqref{eq:gauged-A} for $\gauged{A}$. In the end, we will show that this difference is going to $0$ with $N$. Recall from Section \ref{section:parameters} the parameter $\kfactors = 100\kappa$. Recall also $g$ which is the solution of \eqref{eq:gauge-transformation}.

\begin{definition}\label{def:U-and-h} 
Define $U : [0, \maxtime) \ra \mrm{End}(\frkg)$ by $U := \mrm{Ad}_g$. Let $h := g^{-1} dg$, so that $h_i = g^{-1} \ptl_i g$, $i \in [2]$. 

Let $S \geq 0$. Define $\tau_S$ to be the supremum over $T \in [0, \maxtime)$ such that
\[\begin{split}
\|U\|_{C_t^0 \Cs_x^{1-\kappa}([0, T] \times \T^2)},~~ \|U\|_{\Wc^{3/2-\kappa, 1/2}([0, T])}, ~~ \|\ptl_t U\|_{\Wc^{-1/2-\kappa, 1/2}([0, T])} &\leq S, \\
\|h\|_{C_t^0 \Cs_x^{-\kappa}([0, T])}, ~~ \|h\|_{\Wc^{10\kappa, 10\kappa}([0, T])}, ~~  \|B \|_{\mc{S}^{1-2\kappa}([0, T])} &\leq S. 
\end{split} \]
\end{definition}

\begin{remark}
As we will see in the following, we may work entirely in terms of $U = \Ad_g$ and $h = g^{-1} dg$, rather than $g$.
\end{remark}

\begin{definition}
Define the modified linear object $\glinear$ by
\beq\label{gauged:eq-glinear}  \glinear[\leqN][r][i](t)
:= \Ht \Big(\linear[\leqN][r][i](0) + \Big[ h_j(0) \parall Q^j_{> N} \linear[][r][i](0) \Big] \Big) + \int_{0}^{t} \ds \, \Hts \Big( U^{-1} P_{\leq N} \big( U \xi^i \big)\Big), ~~ t \in [0, \maxtime), i \in [2]. \eeq
\end{definition}

The main result of Section \ref{section:glinear} is the following.

\begin{proposition}\label{prop:linear-objects-close-gauge-covariance}
We have that $\lim_{S \toinf} \tau_S = \maxtime$. Moreover,  suppose there exists $1 \leq R < \infty$ such that
\[ \big\|\linear\big\|_{C_t^0 \Cs_x^{-\kappa}([0, 1] \times \T^2)}, ~~ \|\xi\|_{\Cs_{tx}^{(-1/2-\kappa, -1-\kappa)}((0, 1))}, ~~ \sum_{N_1, N_2} \Big\|\big(P_{N_1} P_{N_2} \Duh\big(\linear\big)\big) \parasim \xi \Big\|_{\Cs_{tx}^{(-1/2-\kappa, 1-\kappa)}((0, 1))}^{1/2} \leq R. \]
Then on $[0, \maxtime) \cap [0, 1]$, we have that
\begin{equation}\label{gauged:eq-glinear-decomposition}
\glinear[\leqN][r][i] = \linear[\leqN][r][i] + \Phi^i_{\leq N} = \linear[\leqN][r][i] + \Big[h_j \parall Q^j_{> N} \linear[][r][i]\Big] + \Psi^i_{\leq N}, ~~ i \in [2],    
\end{equation}
where $\Phi^i, \Psi^i$ satisfy the following bounds. For any $S \geq 1$ and $T \in [0, \tau_S \wedge 1]$, we have that
\[\begin{split}\big\|\Phi^{i}_{\leq N} \big\|_{C_t^0 \Cs_x^{1-\rho}([0, T] \times \T^2)} &\lesssim N^{-\kappa} R^2 S^4 , ~~
\big\| \Psi^i_{\leq N} \big\|_{C_t^0 \Cs_x^{1-\rho}([0, T] \times \T^2) \cap \Wc^{1+5\kappa, \rho}([0, T])} \lesssim N^{-\kappa} R^2 S^4. 
\end{split}\]  
\end{proposition}

\begin{remark}\label{remark:why-singular-weights-necessary}
In \eqref{gauged:eq-glinear-decomposition}, we provide two different representation the difference between $\glinear[\leqN][r]$ and $\linear[\leqN][r]$: We either write it as $\Phi_{\leq N}$, which has spatial regularity $1-\rho$, or write it as $\big[h_j \parall Q^j_{> N} \linear\big]+ \Psi_{\leq N}$, where $\Psi_{\leq N}$ has time-weighted regularity $1+5\kappa$. The more detailed representation will be necessary in our analysis of $\gquadratic[\leqN][r][][][]$ (see Lemma \ref{gauged:lem-quadratic}), whose definitions contains terms of the schematic form $\big[ \linear[][r], \partial \big( \glinear[][r]-\linear[][r]\big) \big]$. 
\end{remark}

The proof of Proposition \ref{prop:linear-objects-close-gauge-covariance} will be given in Section \ref{section:proof-linear-objects-close-gauge-covariance}, after we have developed all the necessary estimates in Sections \ref{section:lo-hi}-\ref{section:hi-lo}.
Before we move on, we recall a few facts about $U$, which has been defined as $U=\Ad_g$ and therefore takes values in $\End(\frkg)$. Since $g$ takes values in the compact Lie group $G$, it holds that $\| U(t) \|_{L_x^\infty}\lesssim 1$ for all $t\in [0,\maxtime)$. Furthermore, since the inner product on $\frkg$ is $\ad$-invariant, it follows that $U$ is unitary, i.e., $U^{-1}=U^\ast$. In order to derive the evolution equation for $U$ from the evolution equation \eqref{gauged:eq-g-evolution} for $g$, we first need the following lemma. In this lemma, $L_{g_0} : G \ra G$ denotes the left multiplication by $g_0\in G$.

\begin{lemma}\label{lemma:Ad-derivative-properties} 
We have the following three properties: 
\begin{enumerate}
    \item[(i)] For all $g_0 \in G$, the derivative $d \mrm{Ad}_{g_0} : T_{g_0} G \ra \End(\frkg)$ is given by $d \mrm{Ad}_{g_0} = \mrm{Ad}_{g_0} \circ \mrm{ad} \circ L_{g_0}^{-1}$.
    \item[(ii)] For all $i\in [2]$, it holds that $\ptl_i U = U \mrm{ad}(h_i)$.
    \item[(iii)] Let $g_0 \in G$, let $U_0 = \mrm{Ad}_{g_0} \in \mrm{End}(\frkg)$, and let $X \in \frkg$. Then, we have that $U_0 \mrm{ad}(X) = \mrm{ad}(U_0 X) U_0$. 
\end{enumerate}
\end{lemma}
\begin{proof}
For the first claim, let $X \in \frkg$. We have that
\[ d \mrm{Ad}_{g_0} (g_0 X) = \frac{d}{dt} \mrm{Ad}_{g_0 e^{tX} } \bigg|_{t = 0} = \mrm{Ad}_{g_0} \frac{d}{dt} \mrm{Ad}_{e^{tX}} \bigg|_{t = 0} = \mrm{Ad}_{g_0} \mrm{ad}(X) = \big(\mrm{Ad}_{g_0} \circ \mrm{ad} \circ L_{g_0}^{-1}\big)(g_0X) . \]
Next, for the second claim, since $U = \mrm{Ad}_g$, by the chain rule we have that 
\[ \ptl_i U = (d \mrm{Ad}_g) \ptl_i g = (\mrm{Ad}_g \circ \mrm{ad} \circ L_g^{-1}) \ptl_i g = U \mrm{ad}(g^{-1} \ptl_i g) .\]
Finally, for the last claim, note
\[ U_0 \mrm{ad}(X) = \mrm{Ad}_{g_0} \frac{d}{dt} \mrm{Ad}_{e^{tX}} \bigg|_{t = 0} = \frac{d}{dt} \mrm{Ad}_{g_0 e^{tX}} \bigg|_{t = 0} = \frac{d}{dt} \mrm{Ad}_{\mrm{Ad}_{g_0} e^{tX}} \bigg|_{ t= 0} \mrm{Ad}_{g_0} = \mrm{ad}(U_0X) U_0. \qedhere\]
\end{proof}

Equipped with Lemma \ref{lemma:Ad-derivative-properties}, we can now derive the evolution equation for $U$. 

\begin{lemma}\label{lemma:U-equation} If $g$ satisfies the evolution equation \eqref{gauged:eq-g-evolution} and $U=\Ad_g$, then $U$ satisfies 
\[ \ptl_t U = (-1 + \Delta) U - (\ptl_j U) U^{-1} \ptl_j U - \mrm{ad}(\ptl_j U A^j) U + U. \]
\end{lemma}
\begin{proof}
By the chain rule, we have that
\[ \ptl_t U = d \mrm{Ad}_g \ptl_t g = U \mrm{ad}(g^{-1} \ptl_t g). \]
Using that
\[\begin{split}
g^{-1} \ptl_t g &= -\mrm{Ad}_g^{-1} ((\ptl_t g) g^{-1})= - d^*_A (g^{-1} dg) = \ptl_j h_j + [A_j, h_j],
\end{split}\]
we further obtain
\[ \ptl_t U = U \mrm{ad}(\ptl_j h_j + [A_j, h_j]). \]
Note also that
\[ \ptl_j U = U \mrm{ad}(h_j), \text{ and so } \Delta U = \ptl_{jj} U = U \mrm{ad}(h_j)^2 + U \mrm{ad}(\ptl_j h_j). \]
We thus obtain
\[ \ptl_t U = \Delta U - U \mrm{ad}(h_j)^2 + U \mrm{ad}([A_j, h_j]). \]
To finish, use that $\mrm{ad}(h_j) = U^{-1} \ptl_j U$ and $U \mrm{ad}([A^j, h_j]) = - U \mrm{ad}(U^{-1} \ptl_j U A^j) = - \mrm{ad}(\ptl_j U A^j) U$.
\end{proof}

From standard parabolic local existence theory, we obtain the following regularity estimates for $U$. The proof is omitted.

\begin{lemma}
For any $T \in [0, \maxtime)$ and $\delta \in [0, 1)$, we have that $U \in C_t^0 \Cs_x^{1-\kappa}([0, T] \times \T^2) \cap \Wc^{1-\kappa + \delta, \delta/2}([0, T])$.
\end{lemma}

With the preliminary results out of the way, we give a high level overview of the proof of Proposition \ref{prop:linear-objects-close-gauge-covariance} (along the way, Lemma \ref{lemma:U-xi-defined} will follow). We write $U^{-1} P_{\leq N}(U \xi) - P_{\leq N} \xi = U^{-1}(P_{\leq N}(U \xi) - U P_{\leq N} \xi)$, and then split
\beq\label{eq:P-N-U-xi-split}\begin{split}
P_{\leq N}(U \xi)& - U P_{\leq N} \xi = \\
&\big(P_{\leq N}(U \parall \xi) - U \parall P_{\leq N} \xi\big) + \big(P_{\leq N} (U \parasim \xi) - U \parasim P_{\leq N} \xi\big) + \big(P_{\leq N}(U \paragg \xi) - U \paragg P_{\leq N} \xi\big). 
\end{split}\eeq
We will handle each of the above terms on the right hand side separately. The first and third terms may be treated deterministically. For the first term,  we will apply a Taylor expansion and then commutator estimates to see that this term essentially contributes $\linear[][r][i] + (U^{-1} \ptl_j U) \parall Q^j_{> N} \linear[][r][i]$ plus a smoother remainder which is absorbed into the $\Psi$ term of Proposition \ref{prop:linear-objects-close-gauge-covariance}. Note by Lemma \ref{lemma:Ad-derivative-properties} that $U^{-1} \ptl_j U = \ad(g^{-1} \ptl_j g) = \ad(h_j)$, and so $(U^{-1} \ptl_j U) \parall Q^j_{> N} \linear = \big[ h_j \parall Q^j_{> N} \linear\big]$ is precisely one of the terms which appears in Proposition \ref{prop:linear-objects-close-gauge-covariance}.
By standard paraproduct and Schauder estimates, we will see that the third term only contributes a smoother remainder which is absorbed into $\Psi$. \\

We will also show that the high$\times$high product $U \parasim \xi$ only contributes to $\Psi$. However, in order to do so, we will need to split $U = U^r + U^s$, where $U^r$ is the rougher part of $U$ and $U^s$ is the smoother part of $U$. To control $U^r \parasim \xi$, we need to assume that the stochastic object $\Duh\big(\linear\big) \parasim \xi$ may be defined (and this is the reason for the additional stochastic estimates in Section \ref{section:stochastic-estimate-gauge-covariance}), whereas $U^s$ is smooth enough so that $U^s \parasim \xi$ may just be bounded deterministically. To determine how to define $U^r$, note that in the equation for $U$ (given by Lemma \ref{lemma:U-equation}), the nonlinear term $\ad(\ptl_j U A^j) U$ is the main barrier against higher regularity of $U$, due to the fact that $A$ only has spatial regularity $-\kappa$. More precisely, since $A = \linear + B$ where $B$ is has spatial regularity $1-2\kappa$ at positive times, we see that the main barrier is $\ad\big(\ptl_j U \parall \linear[][r][j]\big) \paragg U$. We thus isolate out this term to define $U^r$.

\begin{definition}
We define $U^r, U^s : [0, \maxtime) \ra \mrm{End}(\frkg)$ 
by 
\[ U^r := - \Duh\Big(\mrm{ad}\big(\ptl_j U \parall \linear[][r][j]\big) \paragg U \Big), ~~ U^s := U - U^r. \]
Explicitly, we have that
\beq\label{eq:Us}\begin{split}
U^s = -&\Duh \Big( \mrm{ad}\big(\ptl_j U \paragtrsim \linear[][r][j]\big) \paragg U \Big) - \Duh \Big( \mrm{ad}\big(\ptl_j U \linear[][r][j]\big) \paralesssim U \Big) ~- \\
& \Duh\Big(\mrm{ad}(\ptl_j U B^j) U\Big) - \Duh\Big((\ptl_j U) U^{-1} \ptl_j U\Big) + \Duh(U) + \Ht U_0.
\end{split} \eeq
Define also $U^{\mrm{lin}} := \Ht U_0$.
\end{definition}

The next lemma gives estimates for $U^r$ and $U^s$.

\begin{lemma}\label{lemma:U-r-U-s-regularity}
For any $T \in [0, \maxtime)$, $\alpha \in (0, 1/4)$, we have that
\begin{align*}
\|U^r\|_{C_t^0 \Cs_x^{2-4\kappa}([0, T] \times \T^2) \cap C_t^1 \Cs_x^{-4\kappa}([0, T] \times \T^2)} &\lesssim \|U\|_{C_t^0 \Cs_x^{1-\kappa}} \big\| \linear\big\|_{C_t^0 \Cs_x^{-\kappa}([0, T] \times \T^2)}, \\
\|U^s - U^{\mrm{lin}}\|_{\Wc^{2+2\alpha, 2\alpha + 4\kappa}([0, T])} &\lesssim   \|U\|_{C_t^0 \Cs_x^{1-\kappa}([0, T] \times \T^2)} \|U\|_{\Wc^{1 + 2\alpha + 3\kappa, \alpha + 2\kappa}([0, T])} \\ &~~~~~\times \big(\big\| \linear\big\|_{C_t^0 \Cs_x^{-\kappa}([0, T] \times \T^2)} + \|B\|_{\Sc^{1-2\kappa}([0, T])}\big), \\
\|\ptl_t (U^s - U^{\mrm{lin}})\|_{\Wc^{2\alpha, 2\alpha + 4\kappa}([0, T])} &\lesssim \|U\|_{C_t^0 \Cs_x^{1-\kappa}([0, T] \times \T^2)} \|U\|_{\Wc^{1 + 2\alpha + 3\kappa, \alpha + 2\kappa}([0, T])} \\
&~~~~~\times\big(\big\| \linear\big\|_{C_t^0 \Cs_x^{-\kappa}([0, T] \times \T^2)} + \|B\|_{\Sc^{1-2\kappa}([0, T])}\big). \end{align*}
For any $\beta \geq 0$, we have that
\[\begin{split}
\|U^{\mrm{lin}}\|_{\Wc^{1+ \beta, (\beta + \kappa)/2}([0, T])} &\lesssim \|U_0\|_{\Cs_x^{1-\kappa}}, \\
\|\ptl_t U^{\mrm{lin}}\|_{\Wc^{-1 + \beta, (\beta + \kappa)/2}([0, T])} &\lesssim \|U_0\|_{\Cs_x^{1-\kappa}}.
\end{split}\]
\end{lemma}
\begin{proof}
\emph{Bound on $U^r$.} To bound $U^r$, apply Schauder and paraproduct estimates to obtain (recall that $\|U(t)\|_{L_x^\infty} \lesssim 1$ since $U(t) = \Ad_{g(t)}$ is an isometry of $\frkg$)
\begin{equation}\label{gauged:eq-Ur-estimate-1} 
\|U^r\|_{C_t^0 \Cs_x^{2-4\kappa}} \lesssim \big\|\mrm{ad}\big(\ptl_j U \parall \linear[][r][j]\big) \paragg U\big\|_{C_t^0 \Cs_x^{-3\kappa}} \lesssim \|\ptl_j U\|_{C_t^0 \Cs_x^{-2\kappa}} \big\|\linear\big\|_{C_t^0 \Cs_x^{-\kappa}} \|U\|_{L_{tx}^\infty} \lesssim \|U\|_{C_t^0 \Cs_x^{1-\kappa}}\big\|\linear\big\|_{C_t^0 \Cs_x^{-\kappa}} . 
\end{equation}
The bound for $\|U^r\|_{C_t^1 \Cs_x^{-4\kappa}}$ follows directly from the evolution equation
\[ \ptl_t U^r = \Delta U^r - \mrm{ad}(\ptl_j U \parall \linear[][r][j]) \paragg U, \]
the estimate $\|\Delta U^r\|_{C_t^0 \Cs_x^{-4\kappa}} \lesssim \|U^r\|_{C_t^0 \Cs_x^{2-4\kappa}}$, and \eqref{gauged:eq-Ur-estimate-1}. 

\emph{Bound on $U^s - U^{\mrm{lin}}$.} By applying Schauder and paraproduct estimates, we bound each of the terms in the formula \eqref{eq:Us} for $U^s$. The first term may be bounded
\[\begin{split}
\Big\|\Duh \Big( \mrm{ad}\big(\ptl_j U \paragtrsim \linear[][r][j]\big) \paragg U \Big) (t)\Big\|_{\Cs_x^{2+2\alpha}} &\lesssim \int_0^t (t-s)^{-(1-\kappa)} \Big\|\big(\mrm{ad}\big(\ptl_j U \paragtrsim \linear[][r][j]\big) \paragg U\big)(s) \Big\|_{\Cs_x^{2\alpha + 2\kappa}} ds \\
&\lesssim \Big( \int_0^t (t-s)^{-(1-\kappa)} s^{-(\alpha + 2\kappa)} ds \Big) \|\ptl_j U\|_{\Wc^{2\alpha + 3\kappa, \alpha + 2\kappa}} \|\linear\|_{C_t^0 \Cs_x^{-\kappa}} \|U\|_{L_{tx}^\infty} \\
&\lesssim t^{-(\alpha + \kappa)} \|U\|_{\Wc^{1 + 2\alpha + 3\kappa, \alpha + 2\kappa}} \|\linear\|_{C_t^0 \Cs_x^{-\kappa}}.
\end{split}\]
Next, the second term in \eqref{eq:Us} may be bounded
\[\begin{split}
\Big\|\Duh \Big( \mrm{ad}\big(\ptl_j U \linear[][r][j]\big) \paralesssim U \Big)(t)\Big\|_{\Cs_x^{2+2\alpha}} 
&\lesssim \int_0^t (t-s)^{-(1/2 + \alpha + \kappa)} \big\| \mrm{ad}\big(\ptl_j U \linear[][r][j]\big) \paralesssim U (s)\big\|_{\Cs_x^{1-2\kappa}} ds \\ &\lesssim \Big( \int_0^t (t-s)^{-(1/2 + \alpha + \kappa)} s^{-2\kappa} ds \Big) \|\ptl_j U\|_{\Wc^{2\kappa, 2\kappa}} \|\linear\|_{C_t^0 \Cs_x^{-\kappa}} \|U\|_{C_t^0 \Cs_x^{1-\kappa}} \\
&\lesssim t^{1/2 -\alpha - 3\kappa} \|U\|_{C_t^0 \Cs_x^{1-\kappa}} \|U\|_{\Wc^{1+2\kappa, 2\kappa}} \|\linear\|_{C_t^0 \Cs_x^{-\kappa}}.
\end{split}\]
The third term in \eqref{eq:Us} may be bounded 
\[\begin{split}
&\hspace{2ex}\Big\|\Duh\Big(\mrm{ad}(\ptl_j U B^j) U\Big)(t)\Big\|_{ \Cs_x^{2+2\alpha}} \\
&\lesssim \int_0^t (t-s)^{-(1-\kappa)} \big\|\big(\mrm{ad}(\ptl_j U B^j) U\big)(s)\big\|_{\Cs_x^{2\alpha + 2\kappa}} ds \\
&\lesssim \Big( \int_0^t (t-s)^{-(1-\kappa)} s^{-(2\alpha + 3\kappa)} ds \Big) \|\ptl_j U\|_{\Wc^{2\alpha + 2\kappa, (2\alpha + 3\kappa)/2}} \|B\|_{\Wc^{2\alpha +2\kappa, (2\alpha + 3\kappa)/2}} \|U\|_{C_t^0 \Cs_x^{1-\kappa}} \\
&\lesssim t^{-2(\alpha + \kappa)} \|U\|_{C_t^0 \Cs_x^{1-\kappa}} \|U\|_{\Wc^{1+2\alpha+2\kappa, (2\alpha + 3\kappa)/2}} \|B\|_{\Sc^{1-2\kappa}} .
\end{split}\]
The fourth term in \eqref{eq:Us} may be bounded similarly to the third term. The fifth and final term in \eqref{eq:Us} may be bounded
\[ \|\Duh(U)\|_{C_t^0 \Cs_x^{2+\alpha}} \lesssim \|U\|_{C_t^0 \Cs_x^{1-\kappa}}. \]
\emph{Bound on $\ptl_t (U^s - U^{\mrm{lin}})$.} To bound $\ptl_t (U^s - U^{\mrm{lin}})$, we may use the fact that this term solves an explicit equation, and argue similarly to the bound for $U^s - U^{\mrm{lin}}$. Thus, this part is omitted.

\emph{Bounds on $U^{\mrm{lin}}$.} These bounds just follow by standard heat flow estimates (Lemma \ref{prelim:lem-heat-flow}).
\end{proof}

Lemma \ref{lemma:U-r-U-s-regularity} has the following consequence.

\begin{corollary}\label{cor:U-Us-regularity}
We have that $U \in \Wc^{3/2-\kappa, 1/2}([0, T])$, $\ptl_t U \in \Wc^{-1/2 -\kappa, 1/2}([0, T])$. We have that $U^s \in \Wc^{2 + 1/4 -\kappa, 5/8}([0, T])$, $\ptl_t U^s \in \Wc^{1/4-\kappa, 5/8}([0, T])$.
\end{corollary}

\begin{notation}
Next, we fix a number of conventions that will hold in the remainder of Section \ref{section:glinear}:
\begin{enumerate}
    \item[(i)] Let $\xireg := (-1/2-\kappa, -1-\kappa)$.
    \item[(ii)] Fix a time $T \in [0, \maxtime) \cap [0, 1]$.
    \item[(iii)] For norms like $C_t^0 \Cs_x^\alpha([0, T] \times \T^2)$, we will include $[0, T] \times \T^2$ in all result statements, however in proofs, we will just write $C_t^0 \Cs_x^\alpha$ for brevity. Similarly, we will write $\Cs_{tx}^{\xireg}((0, 1))$ or more generally $\Cs_{tx}^{\alpha}((0, T))$ in all result statements, but in the proofs we will typically just write $\Cs_{tx}^{\xireg}$ and $\Cs_{tx}^{\alpha}$.
\end{enumerate}
\end{notation}

The fact that $U^{-1} = U^*$ implies that $\|U^{-1}\| = \|U\|$ for any norm that we will use. We will use this latter fact without explicit mention in what follows. Additionally, throughout the remainder of Section \ref{section:glinear}, we will need to assume that $U, U^r, U^s$ are defined on all of $\R$, in order to apply the space-time paraproduct estimates. We extend $U, U^r, U^s$ in the following way. We extend $V = U, U^r$ by:
\[ V(t) := \begin{cases} V(T) & t > T \\ V(0) & t < 0 \end{cases}, \]
and we extend $U^s$ by:
\[ U^s(t) := \begin{cases} U^s(-t) & t \in [-T, 0) \\ U^s(T) & |t| > T \end{cases}. \]

\begin{remark} Due to our above choice of the extensions, the earlier identity $U = U^r + U^s$ breaks down outside of $[0,T]$, but this does not cause any problems in our arguments. 
\end{remark}

These choices of extension ensure that
\beq\label{eq:extension-bound} \|V\|_{\Wc^{\eta, \theta}(\R)} \lesssim \|V\|_{\Wc^{\eta, \theta}([0, T])}, ~~  V = U, U^r, U^s. \eeq
On the other hand, we introduced a discontinuity in $\ptl_t U$, $\ptl_t U^s$, but this does not matter for what we will need. For $t \notin [0, T]$, we have that $\ptl_t U(t) = 0$. For $t \in (-T, 0)$, we have that $\ptl_t U^s(t) = -\ptl_t U^s(-t)$, and for $|t| > T$, we have that $\ptl_t U^s(t) = 0$. We thus have that 
\beq\label{eq:extension-derivative-bound} \|\ptl_t U \|_{\Wc^{\eta, \theta}(\R)} \lesssim \|\ptl_t U\|_{\Wc^{\eta, \theta}([0, T])}, ~~ \|\ptl_t U^s\|_{\Wc^{\eta, \theta}(\R)} \lesssim \|\ptl_t U^s\|_{\Wc^{\eta, \theta}([0, T])}. \eeq
In particular, Corollary \ref{cor:U-Us-regularity}, it follows that 
\[ U \in \Wc^{3/2-\kappa, 1/2}(\R), ~~ \ptl_t U \in \Wc^{-1/2-\kappa, 1/2}(\R), ~~ U^s \in \Wc^{2+1/4-\kappa, 5/8}(\R), ~~ \ptl_t U \in \Wc^{1/4-\kappa, 5/8}(\R).\]

\begin{definition}
Define the norms $\|\cdot\|_{\Unorm}, \|\cdot\|_{\Usnorm}$ by
\[\begin{split}
\|V\|_{\Unorm} &:= \|V\|_{L_t^\infty \Cs_x^{1-\kappa}(\R \times \T^2)} + \|V\|_{\Wc^{3/2-\kappa, 1/2}(\R)} + \|\ptl_t V\|_{\Wc^{-1/2-\kappa, 1/2}(\R)},   \\
\|V\|_{\Usnorm} &:= \|V\|_{L_t^\infty \Cs_x^{1-\kappa}(\R \times \T^2)} +  \|V\|_{\Wc^{2+1/4-\kappa, 5/8}(\R)} + \|\ptl_t V\|_{\Wc^{1/4-\kappa, 5/8}(\R)}. \end{split}\]
\end{definition}

\begin{remark}
By Corollary \ref{cor:Wc-implies-besov} and interpolation, we have that
\[ \|V\|_{\Wc^{1+10\kappa, 10\kappa}(\R)}, ~~ \|V\|_{\Cs_{tx}^{(1/2+2\kappa, -5\kappa)}}, ~~ \|V\|_{\Cs_{tx}^{(1/2+2\kappa, 20\kappa), 30\kappa}}, ~~ \|V\|_{\Cs_{tx}^{(0, 20\kappa), 20\kappa}} \lesssim \|V\|_{\Unorm}, \]
\[ \|V\|_{\Wc^{2+10\kappa, 10\kappa}(\R)}, ~~  \|V\|_{\Cs_{tx}^{(-50\kappa, 1+50\kappa)}}, \|V\|_{\Cs_{tx}^{(0, 1+50\kappa), 50\kappa}} \lesssim \|V\|_{\Usnorm}.\]
(To be clear, we are not optimizing in the precise factors of $\kappa$ in most terms of the above.) We will use these and related facts in what follows without explicit reference.
\end{remark}

\begin{remark}
The proof of Proposition \ref{prop:linear-objects-close-gauge-covariance} ultimately rests on an understanding of the regularity of the product $U \xi$ on $(0, T)$. It is thus somewhat ironic that we need to extend $U, U^r, U^s$ to all of $\R$ (which is required in order to Littlewood-Paley decompose), since the question of regularity is fundamentally a local issue. This extension to $\R$ can perhaps be avoided if we use other tools for understanding the product of distributions. On the other hand, Littlewood-Paley theory already underlies many of the technical arguments of this article, and so we prefer to continue using it here. We note that in the end, the extension only appears in intermediate steps, since our main estimates (Lemmas \ref{lemma:lo-hi}, \ref{lemma:Us-hi-hi-xi}, \ref{lemma:hi-hi-rougher-part}, \ref{lemma:hi-lo}) are stated on the interval $[0, T]$, so that the quantities which appear in these main estimates will not depend on the particular choices of extension of $U, U^r, U^s$. 
\end{remark}

We state and prove the following Schauder estimate which we will frequently use in the coming sections.

\begin{lemma}\label{lemma:schauder-U-V}
Let $\alpha, \varep_1, \varep_2, \varep_3, \varep_4, \varep_5 \geq 0$. Let
$\theta_1 := (2\varep_1 + \varep_2 + 2\kappa)/2$, $\theta_2 := (2\varep_3 + \varep_4 + 2\kappa)/2$.  Suppose that
\[ 1 + 2\varep_3 + \varep_4 + \kappa < \frac{3}{2}, ~~ 1-\kappa - \varep_2 > 0, ~~ 1 + 2\varep_1 + \varep_2 + \kappa < \frac{3}{2} - \kappa, ~~ \theta_1 < \frac{1}{2} + \varep_1,\]
\beq\label{eq:varep-conditions} \varep_4 - 2\varep_3 \geq 10\kappa, ~~ \varep_5 + \theta_2 \leq \alpha, ~~ \frac{2\varep_1 + \varep_2 + 7\kappa}{2} \leq \theta_2 + \varep_5.\eeq
Then for $V \in \Cs_{tx}^{(-1/2-\varep_1, -\varep_2)}((0, 1)) \cap \Cs_{tx}^{(-1/2-\varep_3, \varep_4), \varep_5}$((0, 1)), we have that
\[ \|\Duh(U^{-1} V)\|_{C_t^0 \Cs_x^{1-\alpha}([0, T] \times \T^2) \cap \Wc^{1+5\kappa, \alpha}([0, T])} \lesssim \|U\|_{\Unorm} \|V\|_{\Cs_{tx}^{(-1/2-\varep_1, -\varep_2)}((0, 1)) \cap \Cs_{tx}^{(-1/2-\varep_3, \varep_4), \varep_5}((0, 1))}. \]
Let $\delta_1, \delta_2 \geq 0$ be such that 
\[ \delta_1 \leq 1-\kappa, ~~ \delta_2 < \frac{1}{2}, ~~ \alpha < 1 - 2\delta_2 - \delta_1, ~~ \frac{\delta_1}{2} + \delta_2 + \frac{5\kappa}{2} < \alpha.\]
Then for $W \in \Wc^{\delta_1, 1/2 + \delta_2}([0, T])$, we have that
\[ \|\Duh(U^{-1} W)\|_{C_t^0 \Cs_x^{1-\alpha}([0, T] \times \T^2) \cap \Wc^{1+5\kappa, \alpha}([0, T])} \lesssim \|U\|_{\Unorm} \|W\|_{\Wc^{\delta_1, 1/2+\delta_2}([0, T])}.\]
\end{lemma}
\begin{remark}
For our use cases, one should think of the $\varep$ and $\delta$ parameters as different sizes of ``$0+$", and $\alpha$ as a very large size of ``$0+$" which dominates all the rest. In this case, the conditions on $\delta_1, \delta_2$ are always going to be satisfied, and the only conditions on the $\varep_j$ that one needs to check are the first and third ones in \eqref{eq:varep-conditions}. Moreover, note that the third condition in \eqref{eq:varep-conditions} may be ensured by increasing $\varep_5$, so really the only condition to check is $\varep_4 - 2\varep_3 \geq 10\kappa$.
\end{remark}
\begin{proof}
By the assumption that $1 + 2\varep_3 + \varep_4 + \kappa < 3/2-\kappa$ combined with interpolation, we have that
\[ \|U\|_{\Cs_{tx}^{(1/2+\varep_3 + \kappa/2, \varep_4), \theta_2}} \lesssim \|U\|_{\Unorm}. \]
Thus by a product estimate, we obtain
\[ \|U^{-1} V\|_{\Cs_{tx}^{(-1/2-\varep_3, \varep_4), \varep_5 + \theta_2}} \lesssim \|U\|_{\Unorm} \|V\|_{\Cs_{tx}^{(-1/2-\varep_3, \varep_4), \varep_5}}. \]
Next, by a paraproduct estimate and the assumption $1-\kappa - \varep_2 > 0$, we have that
\[ \|U^{-1} \paransim_t V\|_{\Cs_{tx}^{(-1/2-\varep_1, -\varep_2)}} \lesssim \|U\|_{L_t^0 \Cs_x^{1-\kappa}} \|V\|_{\Cs_{tx}^{(-1/2-\varep_1, -\varep_2)}} \]
Next, recall that $\theta_1 = (2\varep_1 + \varep_2 + 2\kappa)/2$. By the assumption that $1 + 2\varep_1 + \varep_2 + \kappa < 3/2 - \kappa$ and interpolation, we have that
\[ \|U\|_{\Wc^{1+2\varep_1 + \varep_2 + \kappa, \theta_1}} + \|\ptl_t U\|_{\Wc^{-1 + 2\varep_2 + \varep_2 + \kappa, \theta_1}} \lesssim \|U\|_{\Unorm}.  \]
Thus by the assumption that $\theta_1  < 1/2 + \varep_1$ combined with Corollary \ref{cor:Wc-implies-besov} and Lemma \ref{lemma:paraproduct-Wc-space-time-besov-space}, we have that
\[ \|U^{-1} \parasim_t V\|_{\Cs_{tx}^{(-1/2-\varep_1, -\varep_2)}} \leq \|U^{-1} \parasim_t V\|_{\Cs_{tx}^{(-\theta_1, -\varep_2)}} \lesssim \|U^{-1} \parasim_t V \|_{\Wc^{-\varep_1, \theta_1}} \lesssim \|U\|_{\Unorm}\|V\|_{\Cs_{tx}^{(-1/2-\varep_1, -\varep_2)}}. \]
Combining the previous few estimates, we obtain
\[ \|U^{-1} V\|_{\Cs_{tx}^{(-1/2-\varep_1, -\varep_2)} \cap \Cs_{tx}^{(-1/2-\varep_3, \varep_4), \varep_5 + \theta_2}} \lesssim \|U\|_{\Unorm} \|V\|_{\Cs_{tx}^{(-1/2-\varep_1, -\varep_2)} \cap \Cs_{tx}^{(-1/2-\varep_3, \varep_4), \varep_5}}. \]
Now by Schauder (i.e. Lemma \ref{lemma:space-time-Besov-space-Schauder} applied with $\beta_x = 1 - 2\varep_1 - \varep_2 -\kappa$ and Lemma \ref{lemma:schauder-space-time-besov-space-singular-time-weights} applied with $\gamma = 1 - 2\varep_1 - \varep_2 - 2\kappa$, $\theta = \varep_5 + \theta_2$, $\theta' = (2\varep_1 + \varep_2 + 7\kappa)/2$) and the assumptions $2\varep_1 + \varep_2 + \kappa < \alpha$, $\theta' \leq \theta$, $\varep_4 - 2\varep_3 \geq 10\kappa$, $\varep_5 + \theta_2 \leq \alpha$, we have that
\[\begin{split}
\|\Duh(U^{-1} V)\|_{C_t^0 \Cs_x^{1-\alpha} \cap \Wc^{1+5\kappa, \alpha}}  &\lesssim \|U^{-1} V\|_{\Cs_{tx}^{(-1/2-\varep_1, -\varep_2)} + \Cs_{tx}^{(-1/2-\varep_3, \varep_4), \varep_5 + \theta_1}}\\
&\lesssim \|U\|_{\Unorm} \|V\|_{\Cs_{tx}^{(-1/2-\varep_1, -\varep_2)} \cap \Cs_{tx}^{(-1/2-\varep_3, \varep_4), \varep_5}}.
\end{split}\]
This shows the first estimate.
The second estimate follows by standard product and Schauder estimates.
\end{proof}

\subsubsection{Low$\times$high estimate}\label{section:lo-hi}

We first look at the low$\times$high product in the decomposition \eqref{eq:P-N-U-xi-split}. The main result for the low$\times$high product is the following.

\begin{lemma}[Low$\times$high estimate]\label{lemma:lo-hi}
We have that
\[\begin{split}
\Big\|\Duh\Big(U^{-1}\big( P_{\leq N}(U \parall \xi) - U \parall P_{\leq N}\xi\big)\Big) - (U^{-1} \ptl_j U) \parall Q^j_{> N}\big(\linear(t) - &\Ht \linear(0)\big)\Big\|_{C_t^0 \Cs_x^{1-\rho}([0, T] \times \T^2) \cap \Wc^{1+5\kappa, \rho}([0, T])} \lesssim \\
&N^{-\kappa} \|U\|_{\Unorm}^2 \|\xi\|_{\Cs_{tx}^{\xireg}((0, 1))}. 
\end{split}\]
\end{lemma}

The proof of this result is given later, after we have proven several intermediate results.

\begin{lemma}\label{lemma:U-noise-lo-hi-estimate}
We have that
\[ \big\|P_{\leq N}(U \parall \xi) - U \parall P_{\leq N} \xi - \ptl_j U \parall Q^j_{> N} \xi \big\|_{X + Y} \lesssim N^{-\kappa} \|U\|_{\Unorm} \|\xi\|_{\Cs_{tx}^{\xireg}((0, 1))},  \]
where $X = \Cs_{tx}^{(-1/2-\kappa, -20\kappa)}((0, 1)) \cap \Cs_{tx}^{(-1/2-\kappa, 15\kappa), 20\kappa}((0, 1))$ and $Y = \Cs_{tx}^{(\kappa, -1- 20\kappa)}((0, 1)) \cap \Cs_{tx}^{(\kappa, -1+15\kappa), 20\kappa}((0, 1))$.
\end{lemma}
\begin{proof}
The first estimate follows by Lemma \ref{lemma:P-N-commutator-space-time} with $\alpha = (0, 1-\kappa)$, $\alpha' = (1/2 + 2\kappa, -5\kappa)$, $\gamma = (0, 1+20\kappa)$, $\gamma' = (1/2 + 2\kappa, 20\kappa)$, $\theta = \theta' = 20\kappa$, $\beta = (-1/2-\kappa, -1-\kappa)$. 
\end{proof}

We can now obtain the following result.

\begin{lemma}\label{lemma:lo-hi-intermediate}
We have that
\[\begin{split}
\Big\|\Duh\Big(U^{-1}\Big( P_{\leq N}(U \parall \xi) - U \parall P_{\leq N}\xi - \ptl_j U \parall Q^j_{> N}\xi\Big)\Big)&\Big\|_{C_t^0 \Cs_x^{1-\kfactors}([0, T] \times \T^2) \cap \Wc^{1+5\kappa, \kfactors}([0, T])} \lesssim \\
&N^{-\kappa}  \|U\|_{\Unorm}^2 \|\xi\|_{\Cs_{tx}^{\xireg}((0, 1))}. 
\end{split}\]
\end{lemma}
\begin{proof}
This follows by combining a Schauder estimate (Lemma \ref{lemma:schauder-U-V}) with Lemma \ref{lemma:U-noise-lo-hi-estimate}.
\end{proof}

\begin{lemma}\label{lemma:U-hi-lo-ptl-j-U-lo-hi-Q-xi}
We have that 
\[ \Big\| \Duh\big( U^{-1} \paragtrsim (\ptl_j U \parall Q^j_{> N} \xi) \big) \Big\|_{C_t^0 \Cs_x^{1-\kfactors}([0, T] \times \T^2) \cap \Wc^{1+5\kappa, \kfactors}([0, T])} \lesssim N^{-\kappa} \|U\|_{\Unorm}^2 \|\xi\|_{\Cs_{tx}^{\xireg}((0, 1))}.\]
\end{lemma}
\begin{proof}
We have that
\[ \ptl_j U \parall Q^j_{> N} \xi = \ptl_j U \paransim_t \parall_x Q^j_{> N} \xi + \ptl_j U \parasim_t \parall_x Q^j_{> N} \xi. \]
By a paraproduct estimate and Lemma \ref{lemma:Q-gains-one-spatial-derivative-space-time-besov-space}, we have that
\[\begin{split}
\| \ptl_j U \paransim_t \parall_x Q^j_{> N} \xi \|_{\Cs_{tx}^{(-1/2-\kappa, -20\kappa)}} &\lesssim \|\ptl_j U\|_{L_t^\infty \Cs_x^{-\kappa}} \|Q^j_{> N} \xi\|_{\Cs_{tx}^{(-1/2-\kappa, -2 \kappa)}} \\
&\lesssim N^{-\kappa} \|U\|_{L_t^\infty \Cs_x^{1-\kappa}}  \|Q^j_{> N} \xi\|_{\Cs_{tx}^{(-1/2-\kappa, -\kappa)}} \\
&\lesssim N^{-\kappa} \|U\|_{L_t^\infty \Cs_x^{1-\kappa}}  \|\xi\|_{\Cs_{tx}^{\xireg}}. 
\end{split}\]
We also have that
\[\begin{split}
\| \ptl_j U \parasim_t \parall_x Q^j_{> N} \xi\|_{\Cs_{tx}^{(\kappa, -1-20\kappa)}} &\lesssim \|\ptl_j U\|_{\Cs_{tx}^{(1/2+2\kappa, -1-5\kappa)}} \|Q^j_{> N} \xi\|_{\Cs_{tx}^{(-1/2-\kappa, -2\kappa)}}  \\
&\lesssim N^{-\kappa}  \|U\|_{\Cs_{tx}^{(1/2+2\kappa, -5\kappa)}} \|\xi\|_{\Cs_{tx}^{\xireg}}.
\end{split}\]
Next, we may split
\[ U^{-1} \paragtrsim (\ptl_j U \parall Q^j_{> N} \xi) = U^{-1} \parall_t \paragtrsim_x (\ptl_j U \parall Q^j_{> N} \xi) + U^{-1} \paragtrsim_t \paragtrsim_x (\ptl_j U \parall Q^j_{> N} \xi). \]
By applying a paraproduct estimate to estimate the first term and Lemma \ref{lemma:paraproduct-Wc-space-time-besov-space} to estimate the second term, we obtain
\[
\begin{split}
\| U^{-1} \paragtrsim (\ptl_j &U \parall Q^j_{> N} \xi)\|_{\Cs_{tx}^{(-1/2-\kappa, 1 -40\kappa)} + \Wc^{\kappa, 40\kappa} +  \Cs_{tx}^{(0, -40\kappa)}} \lesssim \\
& \big(\|U\|_{L_t^0 \Cs_x^{1-\kappa}} + \|U\|_{\Wc^{1+40\kappa, 40\kappa}} + \|\ptl_t U\|_{\Wc^{-1+40\kappa, 40\kappa}} \big) \|\ptl_j U \parall Q^j_{> N} \xi\|_{\Cs_{tx}^{(-1/2-\kappa, -20\kappa)} + \Cs_{tx}^{(\kappa, -1-20\kappa)}} . 
\end{split}\]
The desired result now follows by combining the previous few estimates with various Schauder estimates. (note we can in fact control the $C_t^0 \Cs_x^{2-}$ norm, but this is not needed).
\end{proof}

The final intermediate result we will need before proving Lemma \ref{lemma:lo-hi} is the following low$\times$high commutator estimate.

\begin{lemma}\label{lemma:duhamel-lo-hi-commutator}
We have that
\[\begin{split}
\Big\| \Duh\big(U^{-1} \parall (\ptl_j U \parall Q^j_{> N} \xi)\big) - (U^{-1} \ptl_j U) &\parall Q^j_{> N} \big(\linear(t) - \Ht \linear(0)\big)\Big\|_{C_t^0 \Cs_x^{1-\kfactors}([0, T] \times \T^2) \cap \Wc^{1+5\kappa, \kfactors}([0, T])} \lesssim \\
&N^{-\kappa}  \|U\|_{\Unorm}^2 \big(\big\|\linear\big\|_{C_t^0 \Cs_x^{-\kappa}([0, 1] \times \T^2)} + \|\xi\|_{\Cs_{tx}^{\xireg}((0, 1))}\big).
\end{split}\]
\end{lemma}
\begin{proof}
For brevity, let $F := \ptl_j U \parall Q^j_{> N} \xi$. From the proof of Lemma \ref{lemma:U-hi-lo-ptl-j-U-lo-hi-Q-xi}, we have that
\[ \|F\|_{\Cs_{tx}^{(-1/2-\kappa, -20\kappa)} + \Cs_{tx}^{(\kappa, -1-20\kappa)}} \lesssim N^{-\kappa} \big(\|U\|_{L_t^\infty \Cs_x^{1-\kappa}} + \|U\|_{\Cs_{tx}^{(1/2+2\kappa, -5\kappa)}}\big) \|\xi\|_{\Cs_{tx}^{\xireg}}. \]
Now, by several applications of Lemma \ref{lemma:space-time-integral-commutator-estimate} and the previous estimate, we have that
\[\begin{split}
\big\| \Duh(U^{-1} \parall F) &- U^{-1} \parall \Duh(F)\|_{C_t^0 \Cs_x^{2-40\kappa}} \lesssim \\
&\big(\|U\|_{\Wc^{1-\kappa, 0}} + \|U\|_{\Wc^{10\kappa, 10\kappa}} + \|\ptl_t U \|_{\Wc^{-1+10\kappa, 10\kappa}}\big) \|F\|_{\Cs_{tx}^{(-1/2-\kappa, -20\kappa)} + \Cs_{tx}^{(\kappa, -1-20\kappa)}} \lesssim \\
& N^{-\kappa}\big(\|U\|_{\Wc^{1-\kappa, 0}} + \|U\|_{\Wc^{10\kappa, 10\kappa}} + \|\ptl_t U \|_{\Wc^{-1+10\kappa, 10\kappa}}\big) \big(\|U\|_{L_t^\infty \Cs_x^{1-\kappa}} + \|U\|_{\Cs_{tx}^{(1/2+2\kappa, -5\kappa)}}\big) \|\xi\|_{\Cs_{tx}^{\xireg}}.
\end{split}\]
Next, by two applications of Lemma \ref{lemma:space-time-integral-commutator-estimate} (one which does not give a singular weight and one which does), we have that
\[\begin{split}
\big\|\Duh(\ptl_j U \parall &Q^j_{> N} \xi) - \ptl_j U \parall \Duh(Q^j_{> N} \xi) \big\|_{C_t^0 \Cs_x^{1-40\kappa} \cap \Wc^{1+5\kappa, 40\kappa}} \lesssim \\
&\big(\|\ptl_j U\|_{\Wc^{-\kappa, 0}} + \|\ptl_j U\|_{\Wc^{20\kappa, 20\kappa}} + \|\ptl_t \ptl_j U \|_{\Wc^{-2 + 20\kappa, 20\kappa}}\big) \|Q^j_{> N} \xi \|_{\Cs_{tx}^{(-1/2-\kappa, -2\kappa)}}. 
\end{split}\]
By Lemma \ref{lemma:Q-gains-one-spatial-derivative-space-time-besov-space}, we may further bound the right hand side above by
\[ N^{-\kappa} \big(\|U\|_{\Wc^{1-\kappa, 0}} + \|U\|_{\Wc^{20\kappa, 20\kappa}} + \|\ptl_t U \|_{\Wc^{-1+20\kappa, 20\kappa}}\big) \|\xi\|_{\Cs_{tx}^{\xireg}}.  \]
We thus have
\[\begin{split}
\big\| U^{-1} \parall \Duh(\ptl_j &U \parall Q^j_{> N} \xi) - U^{-1} \parall \big( \ptl_j U \parall \Duh(Q^j_{> N} \xi)\big) \big\|_{C_t^0 \Cs_x^{1-40\kappa} \cap \Wc^{1+5\kappa, 40\kappa}} \lesssim \\
&N^{-\kappa} \|U\|_{L_t^\infty \Cs_x^{1-\kappa}} \big(\|U\|_{\Wc^{1-\kappa, 0}} + \|U\|_{\Wc^{20\kappa, 20\kappa}} + \|\ptl_t U \|_{\Wc^{-1+20\kappa, 20\kappa}}\big) \|\xi\|_{\Cs_{tx}^{\xireg}}. 
\end{split}\]
Now, observe that $\Duh(Q^j_{> N} \xi) = Q^j_{> N} \Duh(\xi) = Q^j_{> N}\big(\linear(t) - \Ht \linear(0)\big)$. By a standard commutator estimate and Lemma \ref{prelim:lem-Q}, we have that
\[\begin{split} \Big\|&U^{-1} \parall \big(\ptl_j U \parall Q^j_{> N}\big(\linear(t) - \Ht \linear(0)\big)\big) - (U^{-1} \ptl_j U) \parall Q^j_{> N} \big(\linear(t) - \Ht \linear(0)\big)\Big\|_{C_t^0 \Cs_x^{1-\kfactors} \cap \Wc^{1+5\kappa, \kfactors}} \lesssim \\
& \|U\|_{L_t^\infty \Cs_x^{1-\kappa}} \big( \|U\|_{L_t^\infty \Cs_x^{1-\kappa}} + \|U\|_{\Wc^{1+20\kappa, 20\kappa}} \big) \|Q^j_{> N} \linear\|_{C_t^0 \Cs_x^{1-2\kappa}} \lesssim  N^{-\kappa} \|U\|_{\Unorm}^2 \|\linear\|_{C_t^0 \Cs_x^{-\kappa}}. \qedhere
\end{split}\]
The desired result now follows upon combining the previous few estimates.
\end{proof}

\begin{proof}[Proof of Lemma \ref{lemma:lo-hi}]
This now follows as a direct consequence of Lemmas \ref{lemma:lo-hi-intermediate}, \ref{lemma:U-hi-lo-ptl-j-U-lo-hi-Q-xi}, and \ref{lemma:duhamel-lo-hi-commutator}.
\end{proof}

\subsubsection{High$\times$high estimates}\label{section:hi-hi}

We next move on to the high$\times$high estimates. We first prove a general estimate which will apply to the term $U^s \parasim \xi$ (as well as other terms at the same regularity as $U^s$).

\begin{lemma}\label{lemma:Us-hi-hi-xi}
We have that
\[ \|V \parasim \xi\|_{\Cs_{tx}^{(-1/2-20\kappa, 5\kappa)}((0, 1)) \cap \Cs_{tx}^{(-1/2-\kappa, 15\kappa), 20\kappa}((0, 1)) + \Wc^{5\kappa, 1/2 + 20\kappa}([0, 1])} \lesssim \|V\|_{\Usnorm} \|\xi\|_{\Cs_{tx}^{\xireg}((0, 1))}.\]
Additionally, given a dyadic scale $N \in \dyadic$, we have that
\[ \|V \parasim P_{> N} \xi\|_{\Cs_{tx}^{(-1/2-20\kappa, 5\kappa)}((0, 1)) \cap \Cs_{tx}^{(-1/2-\kappa, 14\kappa), 20\kappa}((0, 1)) + \Wc^{4\kappa, 1/2 + 20\kappa}([0, 1])} \lesssim N^{-\kappa} \|V\|_{\Usnorm} \|\xi\|_{\Cs_{tx}^{\xireg}((0, 1))}. \]
\end{lemma}
\begin{proof}
We decompose
\[ V \parasim \xi = V \parall_t \parasim_x \xi + V \paragtrsim_t \parasim_x \xi.  \]
For the first term, we apply paraproduct estimates to obtain
\[ \|V \parall_t \parasim_x \xi \|_{\Cs_{tx}^{(-1/2-20\kappa, 5\kappa)} \cap \Cs_{tx}^{(-1/2-\kappa, 15\kappa), 20\kappa}} \lesssim \|V\|_{\Cs_{tx}^{(-10\kappa, 1+10\kappa)} \cap \Cs_{tx}^{(0, 1+20\kappa), 20\kappa}} \|\xi\|_{\Cs_{tx}^{\xireg}}.\]
For the second term, we combine Lemma \ref{lemma:paraproduct-Wc-space-time-besov-space} with the fact that $V \in \Wc^{2+10\kappa, 1/2 + 10\kappa}$ and $\ptl_t V \in \Wc^{10\kappa, 1/2 + 10\kappa}$ to obtain
\[ \|V \paragtrsim_t \parasim_x \xi\|_{\Wc^{5\kappa, 1/2 + 20\kappa}} \lesssim \big( \|U^s\|_{\Wc^{2+10\kappa, 1/2 + 10\kappa}} + \|\ptl_t V\|_{\Wc^{10\kappa, 1/2 + 10\kappa}} \big) \|\xi\|_{\Cs_{tx}^{\xireg}}. \]
This shows the first estimate. The second estimate follows similarly, except we give up a single $\kappa$ in spatial regularity to gain the $N^{-\kappa}$ term.
\end{proof}

From Lemma \ref{lemma:Us-hi-hi-xi}, we may obtain the following result. 

\begin{lemma}[High$\times$high estimates with smoother part]\label{lemma:hi-hi-smooth}
We have that
\[\begin{split}
\Big\| \Duh\Big(U^{-1} P_{> N} (V \parasim \xi)\Big)\Big\|_{C_t^0 \Cs_x^{1-\kfactors}([0, T] \times \T^2) \cap \Wc^{1+5\kappa, \kfactors}([0, T])} &\lesssim N^{-\kappa} \|U\|_{\Unorm} \|V\|_{\Usnorm} \|\xi\|_{\Cs_{tx}^{\xireg}((0, 1))}, \\
\Big\| \Duh\Big(U^{-1} (V \parasim P_{> N} \xi)\Big)\Big\|_{C_t^0 \Cs_x^{1-\kfactors}([0, T] \times \T^2) \cap \Wc^{1+5\kappa, \kfactors}([0, T])} &\lesssim N^{-\kappa} \|U\|_{\Unorm} \|V\|_{\Usnorm} \|\xi\|_{\Cs_{tx}^{\xireg}((0, 1))}.
\end{split}\]
\end{lemma}
\begin{proof}
By Lemma \ref{lemma:Us-hi-hi-xi}, we have that
\[ \|P_{> N} (V \parasim \xi)\|_{\Cs_{tx}^{(-1/2-20\kappa, 4\kappa)} \cap \Cs_{tx}^{(-1/2-\kappa, 14\kappa), 20\kappa} + \Wc^{4\kappa, 1/2+20\kappa}} \lesssim N^{-\kappa} \|V\|_{\Usnorm}.  \]
The first estimate now follows by Lemma \ref{lemma:schauder-U-V}. The second estimate follows similarly.
\end{proof}

The last estimate we need is the high$\times$high estimate with $U^r$. Let 
\[ \bar{U}^r := \mrm{ad}\big(\ptl_j U \parall \Duh\big(\linear[][r][j]\big)\big) \paragg U. \]
By several applications of the Duhamel low$\times$high commutator estimates Lemma \ref{prelim:lem-general-integral-commutator} and Corollary \ref{cor:general-integral-commutator}, we obtain the following result. The proof is omitted.

\begin{lemma}\label{lemma:U-r-bar-U-r-difference}
We have that
\[ \|U^r - \bar{U}^r\|_{\Wc^{2-20\kappa, 0}([0, T]) \cap \Wc^{2+2\kappa, 20\kappa}([0, T])} \lesssim \|U\|_{\Unorm} \big\|\linear\big\|_{C_t^0 \Cs_x^{-\kappa}([0, T] \times \T^2)},\]
\[ \|\ptl_t U^r - \ptl_t \bar{U}^r \|_{\Wc^{-20\kappa, 0}([0, T]) \cap \Wc^{2\kappa, 20\kappa}([0, T])} \lesssim \|U\|_{\Unorm} \big\|\linear\big\|_{C_t^0 \Cs_x^{-\kappa}([0, T] \times \T^2)}. \]
\end{lemma}

Recall that we assumed that $U$ has been extended to all of $\R$. We also extend $\bar{U}^r$ to $\R$ as follows. First, let $V_j$ be the extension of $\ptl_j U$ to $\R$ defined by $V_j(t) := (\ptl_j U)(T)$ for $t > T$ and $V_j(0) := (\ptl_j U)(0)$ for $t < 0$. Extend $\Duh(\linear)$ to $\R$ by setting $\Duh(\linear)(t) := \Duh(\linear)(1)$ for $t > 1$ and $\Duh(\linear)(t) := \Duh(\linear)(0) = 0$ for $t < 0$. Then, define
\[ \bar{U}^r(t) := \ad\big(V_j(t) \parall \Duh\big(\linear[][r][j]\big)(t)\big) \paragg U(t), ~~ t \in \R.\]
(Observe that since $\Duh(\linear)(0) = 0$, we in fact have that $\bar{U}^r(t) = 0$ for $t < 0$.)
Note that with this extension, we have that
\[ \|\Duh\big(\linear\big)\|_{C_t^0 \Cs_x^{2-2\kappa}(\R \times \T^2)} = \|\Duh(\linear)\|_{C_t^0 \Cs_x^{2-2\kappa}([0, 1] \times \T^2)}. \]
The reason for working with $\bar{U}^r$ instead of $U^r$ is that we may prove the following estimate for $\bar{U}^r \parasim \xi$.

\begin{lemma}\label{lemma:U-r-hi-hi-noise-estimate}
We have that
\[\begin{split}
\|\bar{U}^r \parasim \xi&\|_{\Cs_{tx}^{(-1/2-\kappa,-20\kappa)}((0, 1)) \cap \Cs_{tx}^{(-1/2-\kappa, 15\kappa), 20\kappa}((0, 1))+ \Wc^{\kappa, 1/2+20\kappa}([0, T])} \lesssim \\
&\|U\|_{\Unorm}^2 \sum_{M, M'} \Big\|\big(P_{M'} P_M \Duh\big(\linear\big)\big) \parasim \xi\Big\|_{\Cs_{tx}^{(-1/2-\kappa, 1-\kappa)}((0, 1))} + \|U\|_{\Unorm}^2 \big\|\linear\big\|_{C_t^0 \Cs_x^{-\kappa}} \|\xi\|_{\Cs_{tx}^{\xireg}((0, 1))}. 
\end{split}\]
Similarly, given a dyadic scale $N$, we have that
\[\begin{split}
\|\bar{U}^r \parasim P_{> N} \xi&\|_{\Cs_{tx}^{(-1/2-\kappa,-20\kappa)}((0, 1)) \cap \Cs_{tx}^{(-1/2-\kappa, 15\kappa), 20\kappa}((0, 1))+ \Wc^{\kappa, 1/2+20\kappa}([0, T])} \lesssim \\
& N^{-\kappa}\|U\|_{\Unorm}^2 \sum_{M, M'} \Big\|\big(P_{M'} P_M \Duh\big(\linear\big)\big) \parasim \xi\Big\|_{\Cs_{tx}^{(-1/2-\kappa, 1-\kappa)}((0, 1))} + N^{-\kappa}\|U\|_{\Unorm}^2 \big\|\linear\big\|_{C_t^0 \Cs_x^{-\kappa}} \|\xi\|_{\Cs_{tx}^{\xireg}((0, 1))} . 
\end{split}\]
\end{lemma}
\begin{proof}
We first show that
\beq\label{eq:first-commutator-estimate}\begin{split} \sum_M \Big\| \ad\Big(P_M&\big(\ptl_j U \parall \Duh(\linear)\big)\Big) \parasim \xi \Big\|_{\Cs_{tx}^{(-1/2-\kappa, -20\kappa)} \cap \Cs_{tx}^{(-1/2-\kappa, 20\kappa), 20\kappa} + \Wc^{\kappa, 1/2+ 20\kappa}} \lesssim \\
&\|U\|_{\Unorm} \sum_{M, M'} \Big\|\big(P_{M'} P_M \Duh\big(\linear\big)\big) \parasim \xi\Big\|_{\Cs_{tx}^{(-1/2-\kappa, 1-\kappa)}} + \|U\|_{\Unorm} \big\|\linear\big\|_{C_t^0 \Cs_x^{-\kappa}} \|\xi\|_{\Cs_{tx}^{\xireg}}. 
\end{split}\eeq
For brevity, let $f = \ptl_j U$, $g = \Duh(\linear)$, $h = \xi$. Then schematically, we want to bound $(P_M(f \parall g)) \parasim h$ (we can ignore the $\ad$ by expanding $f, g, h$ into bases and working with scalar-valued functions). First, we may write
\[\begin{split}
(P_M(f \parall g)) \parasim h &= \sum_{N \sim N'} P_N P_M (f \parall g) P_{N'} h \\
&= \sum_{N \sim N'} (f \parall P_{N} P_M g) P_{N'} h + R_{N, M} (f, g) P_{N'} h \\
&= \sum_{N \sim N'} \sum_{M'} (P_{\ll M'} f) (P_{M'} P_N P_M g) (P_{N'} h) + \sum_{N \sim N'} R_{N, M}(f, g) P_{N'} h \\
&=: I_1(M) + I_2(M),
\end{split}\]
where $R_{N, M} (f, g) := P_N P_M ( f \parall g) - f \parall (P_N P_M g)$. We have that
\[ I_1(M) = \sum_{M'} P_{\ll M'} f \sum_{N \sim N'} (P_{M'} P_N P_M g) P_{N'} h = \sum_{M'} (P_{\ll M'} f) ((P_{M'} P_M g) \parasim h). \]
Using that $f \in L_t^\infty \Cs_x^{-\kappa} \cap \Wc^{20\kappa, 20\kappa} \cap \Wc^{1+20\kappa, 1/2+20\kappa}$, $\ptl_t f \in \Wc^{-1+20\kappa, 1/2+20\kappa}$, and $(P_M P_{M'} g) \parasim h \in \Cs_{tx}^{(-1/2-\kappa, 1-\kappa)}$, we obtain (applying Corollary \ref{cor:Wc-implies-besov} to obtain $\|f\|_{\Cs_{tx}^{(0, 20\kappa), 20\kappa}} \lesssim \|f\|_{\Wc^{20\kappa, 20\kappa}}$ for the first estimate and applying Lemma \ref{lemma:paraproduct-Wc-space-time-besov-space} for the second estimate)
\[ \|(P_{\ll M'} f) \paransim_t ((P_{M'} P_M g) \parasim h)\|_{\Cs_{tx}^{(-1/2-\kappa, -\kappa)} \cap \Cs_{tx}^{(-1/2-\kappa, 20\kappa), 20\kappa}} \lesssim  \|f\|_{L_t^\infty \Cs_x^{-\kappa} \cap \Wc^{20\kappa, 20\kappa}} \|(P_{M'} P_M g) \parasim h\|_{\Cs_{tx}^{(-1/2-\kappa, 1-\kappa)}},  \]
\[  \|(P_{\ll M'} f) \parasim_t ((P_{M'} P_M g) \parasim h)\|_{\Wc^{\kappa, 1/2 + 20\kappa}} \lesssim \big(\|f\|_{\Wc^{1+20\kappa, 1/2+20\kappa}} + \|\ptl_t f\|_{\Wc^{1+20\kappa, 1/2+20\kappa}}\big) \|(P_{M'} P_M g) \parasim h\|_{\Cs_{tx}^{(-1/2-\kappa, 1-\kappa)}}. \]
Combining, we obtain
\[ \sum_M \|I_1(M)\|_{\Cs_{tx}^{(-1/2-\kappa, -20\kappa)} \cap \Cs_{tx}^{(-1/2-\kappa, 20\kappa), 20\kappa} + \Wc^{\kappa, 1/2+ 20\kappa}} \lesssim \|U\|_{\Unorm} \sum_{M, M'} \|(P_{M'} P_M g) \parasim h\|_{\Cs_{tx}^{(-1/2-\kappa, 1-\kappa)}}.\]
To bound the second term, fix $M \sim N \sim N'$. Note that
\[ R_{N, M}(f, g) = P_N P_M (f \parall g) - f \parall (P_N P_M g) = P_N \big(P_M(f \parall g) - f \parall (P_M g) \big) + P_N(f \parall (P_M g)) - f \parall (P_N P_M g). \]
Since $f \in L_t^\infty \Cs_x^{-\kappa} \cap \Wc^{20\kappa, 20\kappa}$, $\ptl_t f \in \Wc^{-2+20\kappa, 20\kappa}$, $g \in L_t^\infty \Cs_x^{2-2\kappa}$, $\ptl_t g \in L_t^\infty \Cs_x^{-2\kappa}$, we may obtain by Lemma \ref{lemma:P-N-lo-hi-commutator}
\[ \|R_{N, M} (f, g)\|_{L_t^\infty \Cs_x^{2-5\kappa} \cap \Wc^{2+10\kappa, 20\kappa}} \lesssim M^{-\kappa} \|f\|_{\Unorm} \|g\|_{L_t^\infty \Cs_x^{2-2\kappa}}, \]
\[ \|\ptl_t R_{N, M} (f, g)\|_{\Wc^{10\kappa, 20\kappa}} \lesssim M^{-\kappa} \|f\|_{\Unorm} \big(\|g\|_{L_t^\infty \Cs_x^{2-2\kappa}} + \|\ptl_t g\|_{L_t^\infty \Cs_x^{-2\kappa}}\big). \]
By a paraproduct estimate, we have that (note that $R_{N, M}(f, g)$ is localized to spatial frequency $M$)
\[\begin{split}
\|R_{N, M}(f, g) \paransim_t P_{N'} h\|_{\Cs_{tx}^{(-1/2-\kappa, 1/4)}} &\lesssim \|R_{N, M}(f, g)\|_{L_t^\infty \Cs_x^{1/4}} \|P_{N'} h\|_{\Cs_{tx}^{(-1/2-\kappa, 1/4)}} \\
&\lesssim  M^{-(2-5\kappa - 1/4)}  (N')^{1/4 + 1 + \kappa}  \|R_{N, M}(f, g)\|_{L_t^\infty \Cs_x^{2-5\kappa}} \|P_{N'} h\|_{\Cs_{tx}^{(-1/2-\kappa, -1-\kappa)}} \\
&\lesssim M^{-\kappa}  \|R_{N, M}(f, g)\|_{L_t^\infty \Cs_x^{2-5\kappa}} \|P_{N'} h\|_{\Cs_{tx}^{(-1/2-\kappa, -1-\kappa)}} .
\end{split}\]
Next, for the high$\times$high time paraproduct, by Lemma \ref{lemma:paraproduct-Wc-space-time-besov-space} we have that
\[\begin{split}
\|R_{N, M}(f, g) \parasim_t P_{N'} h\|_{\Wc^{\kappa, 20\kappa}} &\lesssim \big(\|R_{N, M}\|_{\Wc^{1+6\kappa, 20\kappa}} + \|\ptl_t R_{N, M}\|_{\Wc^{-1+ 6\kappa, 20\kappa}} \big) \|P_{N'} h\|_{\Cs_{tx}^{(-1/2-\kappa, 0)}} \\
&\lesssim M^{-(1+4\kappa)} (N')^{1+\kappa}  \big(\|R_{N, M}\|_{\Wc^{2+10\kappa, 20\kappa}} + \|\ptl_t R_{N, M}\|_{\Wc^{10\kappa, 20\kappa}} \big) \|P_{N'} h\|_{\Cs_{tx}^{(-1/2-\kappa, -1-\kappa)}} \\
&\lesssim M^{-\kappa} \big(\|R_{N, M}\|_{\Wc^{2+10\kappa, 20\kappa}} + \|\ptl_t R_{N, M}\|_{\Wc^{10\kappa, 20\kappa}} \big) \|P_{N'} h\|_{\Cs_{tx}^{(-1/2-\kappa, -1-\kappa)}} .
\end{split}\]
Combining the previous estimates, we obtain (note that $\ptl_t \Duh(\linear) = \Delta \Duh(\linear) + \linear$)
\[ \sum_M \|I_2(M)\|_{\Cs_{tx}^{(-1/2-\kappa, 1/4)} + \Wc^{\kappa, 20\kappa}} \lesssim \|U\|_{\Unorm} \big(\|\Duh(\linear)\|_{L_t^\infty \Cs_x^{2-2\kappa}} + \|\linear\|_{L_t^\infty \Cs_x^{-2\kappa}}\big) \|\xi\|_{\Cs_{tx}^{\xireg}}. \]
By Schauder, we have that $\|\Duh(\linear)\|_{L_t^\infty \Cs_x^{2-2\kappa}} \lesssim \|\linear\|_{C_t^0 \Cs_x^{-\kappa}}$. This shows the estimate \eqref{eq:first-commutator-estimate}. Now, to control 
\[ \bar{U}^r \parasim \xi = \Big(\ad \big(\ptl_j U \parall \Duh(\linear)\big) \paragg U \Big) \parasim \xi, \]
we can argue as before, where we now set $f = U$, $g = \ad \big(\ptl_j U \parall \Duh(\linear)\big)$, and $h = \xi$, and use our estimate \eqref{eq:first-commutator-estimate} to control $(P_M g) \parasim h$. The steps are omitted (note that $f$ has more regularity than before, which makes the estimates easier).
\end{proof}

Lemma \ref{lemma:U-r-hi-hi-noise-estimate} implies the following result. The proof is omitted as it uses Schauder and paraproduct estimates in a manner very similar to previous proofs.

\begin{lemma}[High$\times$high estimate]\label{lemma:hi-hi-rougher-part}
We have that
\[\begin{split}
\Big\|\Duh\Big(U^{-1}\big( \bar{U}^r \parasim P_{> N} \xi\big)\Big) \Big\|_{C_t^0 \Cs_x^{1-\rho}([0, T] \times \T^2) \cap \Wc^{1+5\kappa, \rho}([0, T])} \lesssim ~&N^{-\kappa} \|U\|_{\Unorm}^3 \|\Duh\big(\linear\big) {\parasim} \xi\|_{\Cs_{tx}^{(-1/2-\kappa,1-\kappa)}((0, 1))} ~+ \\
&N^{-\kappa} \|U\|_{\Unorm}^3 \big\|\linear\big\|_{C_t^0 \Cs_x^{-\kappa}([0, 1] \times \T^2)} \|\xi\|_{\Cs_{tx}^{\xireg}((0, 1))} ,
\end{split}\]
\[
\begin{split}
\Big\|\Duh\Big(U^{-1} P_{> N}\big(\bar{U}^r \parasim \xi\big)\Big) \Big\|_{C_t^0 \Cs_x^{1-\rho}([0, T] \times \T^2) \cap \Wc^{1+5\kappa, \rho}([0, T])} \lesssim ~& N^{-\kappa} \|U\|_{\Unorm}^3 \|\Duh\big(\linear\big) {\parasim} \xi\|_{\Cs_{tx}^{(-1/2-\kappa,1-\kappa)}((0, 1))} ~+ \\
& N^{-\kappa} \|U\|_{\Unorm}^3 \big\|\linear\big\|_{C_t^0 \Cs_x^{-\kappa}([0, 1] \times \T^2)} \|\xi\|_{\Cs_{tx}^{\xireg}((0, 1))}. 
\end{split}\]
\end{lemma}

\subsubsection{High$\times$low estimates}\label{section:hi-lo}

We state the following high$\times$low estimate, whose proof is omitted as it consists of applications of Schauder and paraproduct estimates similar to what we've already seen.

\begin{lemma}[High$\times$low estimates]\label{lemma:hi-lo}
We have that
\[\begin{split}
\Big\| \Duh\Big(U^{-1} P_{> N} (U \paragg \xi\big)\Big) \Big\|_{C_t^0 \Cs_x^{1-\rho}([0, T] \times \T^2) \cap \Wc^{1+5\kappa, \rho}([0, T])} &\lesssim N^{-\kappa} \|U\|_{\Unorm}^2 \|\xi\|_{\Cs_{tx}^{\xireg}((0, 1))}   , \\
\Big\| \Duh\Big(U^{-1} (U \paragg P_{> N}\xi\big)\Big) \Big\|_{C_t^0 \Cs_x^{1-\rho}([0, T] \times \T^2) \cap \Wc^{1+5\kappa, \rho}([0, T])} &\lesssim N^{-\kappa} \|U\|_{\Unorm}^2 \|\xi\|_{\Cs_{tx}^{\xireg}((0, 1))}.  \end{split}\]
\end{lemma}

\subsubsection{Proof of Lemma \ref{lemma:U-xi-defined} and Proposition \ref{prop:linear-objects-close-gauge-covariance}}\label{section:proof-linear-objects-close-gauge-covariance}

By combining the estimates of Section \ref{section:lo-hi}, \ref{section:hi-hi} and \ref{section:hi-lo}, we can now prove Lemma \ref{lemma:U-xi-defined} and Proposition \ref{prop:linear-objects-close-gauge-covariance}.

\begin{proof}[Proof of Lemma \ref{lemma:U-xi-defined}]
First note that by Lemmas \ref{lemma:white-noise-space-time-besov-space-bound} and \ref{lemma:Duhamel-linear-times-noise-besov-space-bound}, we have that a.s., 
\[\|\xi\|_{\Cs_{tx}^{\xireg}((0, 1))}, ~~ \sum_{M, M'} \Big\|\big(P_{M'} P_M \Duh\big(\linear\big)\big) \parasim \xi\Big\|_{\Cs_{tx}^{(-1/2-\kappa, 1-\kappa)}((0, 1))} < \infty.\]
We restrict to this a.s. event in what follows.
We split
\[ UP_{\leq N} \xi = U \paransim P_{\leq N} \xi + U \parasim P_{\leq N} \xi .\]
We further split
\[ U \paransim P_{\leq N} \xi = U \paransim_{tx} P_{\leq N} \xi + U \parasim_t \paransim_x P_{\leq N} \xi. \]
By a paraproduct estimate, we have that 
\[ \|U \paransim_{tx} P_{> N} \xi\|_{\Cs_{tx}^{\xireg - (0, \kappa)}}  \lesssim \|U\|_{L_t^\infty \Cs_x^{1-\kappa}((0, 1))} \|P_{> N} \xi\|_{\Cs_{tx}^{\xireg - (0, \kappa)}((0, 1))} \ra 0. \]
By Corollary \ref{cor:Wc-implies-besov} and Lemma \ref{lemma:paraproduct-Wc-space-time-besov-space}, we have that
\[ \|U \parasim_t \paransim_x P_{> N} \xi\|_{\Cs_{tx}^{(-10\kappa, -1-10\kappa)}((0, 1))} \lesssim  \|U \parasim_t \paransim_x P_{> N} \xi\|_{\Wc^{-1-10\kappa, 10\kappa}([0, 1])} \lesssim \|U\|_{\Unorm} \|P_{> N} \xi\|_{\Cs_{tx}^{\xireg - (0, \kappa)}((0, 1))} \ra 0. \]
For the high$\times$high paraproduct in space, we further split
\[ U \parasim P_{\leq N} \xi = U^r \parasim P_{\leq N} \xi + U^s \parasim P_{\leq N} \xi = \bar{U}^r \parasim P_{\leq N} \xi + (U^r - \bar{U}^r) \parasim P_{\leq N} \xi + U^s \parasim P_{\leq N} \xi.\]
By Corollary \ref{cor:Wc-implies-besov} and Lemma \ref{lemma:U-r-hi-hi-noise-estimate}, we have that 
\[\begin{split}
\|&\bar{U}^r \parasim P_{> N} \xi\|_{\Cs_{tx}^{(-1/2-20\kappa, -20\kappa)}((0, 1))} \lesssim \|\bar{U}^r \parasim P_{> N} \xi\|_{\Cs_{tx}^{(-1/2-\kappa, -20\kappa)}((0, 1)) + \Wc^{\kappa, 1/2+20\kappa}([0, 1])} \\
&\lesssim N^{-\kappa}\|U\|_{\Unorm}^2 \sum_{M, M'} \Big\|\big(P_{M'} P_M \Duh\big(\linear\big)\big) \parasim \xi\Big\|_{\Cs_{tx}^{(-1/2-\kappa, 1-\kappa)}((0, 1))} + N^{-\kappa}\|U\|_{\Unorm}^2 \big\|\linear\big\|_{C_t^0 \Cs_x^{-\kappa}} \|\xi\|_{\Cs_{tx}^{\xireg}((0, 1))}  \ra 0.  
\end{split}\]
By Corollary \ref{cor:Wc-implies-besov} and Lemma \ref{lemma:Us-hi-hi-xi}, we have that
\[\begin{split}
\|U^s \parasim P_{> N} \xi\|_{\Cs_{tx}^{(-1/2-20\kappa, 4\kappa)}((0, 1))} &\lesssim \|U^s \parasim P_{> N} \xi\|_{\Cs_{tx}^{(-1/2-20\kappa, 5\kappa)}((0, 1)) + \Wc^{4\kappa, 1/2+20\kappa}([0, 1])} \\
&\lesssim N^{-\kappa} \|U^s\|_{\Usnorm} \|\xi\|_{\Cs_{tx}^{\xireg}((0, 1))} \ra 0. 
\end{split}\]
Finally, the term with $U^r - \bar{U}^r$ may be similarly controlled (where we apply Lemma \ref{lemma:U-r-bar-U-r-difference} to obtain sufficient regularity for $U^r - \bar{U}^r$).
\end{proof}

\begin{proof}[Proof of Proposition \ref{prop:linear-objects-close-gauge-covariance}]
First of all, the fact that $\lim_{S \toinf} \tau_S = \maxtime$ follows from Lemma \ref{lemma:U-r-U-s-regularity} and our local existence theory for the stochastic Yang-Mills equation, which imply that  $\|B\|_{\Sc^{1-2\kappa}([0, T])}$ as well as the various norms on $U$ and $h$ which appear in the definition of $\tau_S$ cannot blow up before $\maxtime$.

Next, let $R, S, \geq 1$ $T \in [0, \tau_S)$ be as in the statement of Proposition \ref{prop:linear-objects-close-gauge-covariance}. Observe that by the definition of $\tau_S$ and the bounds on extensions \eqref{eq:extension-bound}, \eqref{eq:extension-derivative-bound}, we have that $\|U\|_{\Unorm} \leq S$. By the same bounds on extensions and Lemma \ref{lemma:U-r-U-s-regularity}, we also have the following estimate:
\[ \|U^s\|_{\Usnorm} \lesssim \|U\|_{\Unorm}^2 \big(\|\linear\|_{C_t^0 \Cs_x^{-\kappa}} + \|B\|_{\Sc^{1-2\kappa}}\big) + \|U\|_{\Unorm} \lesssim S^2(R + S). \]
These estimates will be used without reference in what follows. We may write for $t \in [0, T]$
\[\begin{split}
\glinear[\leq N][r][] - \linear[\leq N][r][] &= \Duh\Big(U^{-1}\big(P_{\leq N}(U \xi) - U P_{\leq N} \xi\big)\Big) + \Ht\big(\glinear[\leqN][r][](0) - \linear[\leqN][r][](0)\big) \\
&= D_{\ll} + D_{\sim} + D_{\gg} + \Ht \big((U^{-1} \ptl_j U)(0) \parall Q^j_{> N} \linear[][r][](0)\big),
\end{split}\]
where
\[ D_{\ll} := \Duh\Big(U^{-1}\big(P_{\leq N}(U \parall \xi) - U \parall P_{\leq N} \xi\big)\Big), \]
and $D_{\sim}, D_{\gg}$ are analogously defined. Recall from Lemma \ref{lemma:Ad-derivative-properties} that $U^{-1} \ptl_j U = \ad(h_j)$, and thus
\[ (U^{-1} \ptl_j U) \parall Q^j_{> N} \linear = \Big[h_j \parall Q^j_{> N} \linear\Big]. \]
We thus define
\[ \Psi_{\leq N} := \Big(D_{\ll} - (U^{-1} \ptl_j U) \parall Q^j_{> N} \linear + \Ht \big((U^{-1} \ptl_j U)(0) \parall Q^j_{> N} \linear(0)\big)\Big) + D_{\sim} + D_{\gg}. \]
To estimate $\Psi$, we estimate each of the terms $D_{\ll}, D_\sim, D_\gg$ separately.

By Lemma \ref{lemma:lo-hi}, we have that
\[\begin{split}
\Big\| D_{\ll} - (U^{-1} \ptl_j U) \parall Q^j_{> N} \linear + \Ht \big((U^{-1} \ptl_j U)(0) &\parall Q^j_{> N} \linear[][][](0)\big) \Big\|_{C_t^0 \Cs_x^{1-\rho} \cap \Wc^{1+5\kappa, \rho}} \lesssim N^{-\kappa} R S^2 + \|F\|_{C_t^0 \Cs_x^{1-\rho} \cap \Wc^{1+5\kappa, \rho}}, 
\end{split}\]
where
\[ F := \Ht \big((U^{-1} \ptl_j U)(0) \parall Q^j_{> N} \linear[][r][](0)\big)  - (U^{-1} \ptl_j U)(t) \parall Q^j_{> N}\Ht \linear[][r][](0).\]
For brevity, let $f := U^{-1} \ptl_j U$. By writing
\[\begin{split}
F = ~&\Big(\Ht \big(f(0) \parall Q^j_{> N} \linear[][r][](0)\big) - f(0) \parall Q^j_{> N} \Ht \linear[][r][](0)\Big) + \Big(\big(f(0)- f(t)\big) \parall Q^j_{> N} \Ht \linear[][r][](0)\Big),
\end{split}\]
we may separately bound the two terms by applying paraproduct estimates (along with Lemma \ref{prelim:lem-heat-commutator} for the first term) to obtain
\[ \|F\|_{C_t^0 \Cs_x^{1-\rho} \cap \Wc^{1 + 5\kappa, \rho}} \lesssim \|f\|_{C_t^0 \Cs_x^{-\kappa} \cap \Wc^{20\kappa, 20\kappa}} \|Q^j_{> N} \linear[][r][i](0)\|_{\Cs_x^{1-2\kappa}} + \|f\|_{C_t^0 \Cs_x^{-\kappa}} \|\Ht Q^j_{> N} \linear[][r][i](0)\|_{C_t^0 \Cs_x^{1-2\kappa} \cap \Wc^{1+10\kappa, 10\kappa}} .\]
Since $\|f\|_{C_t^0 \Cs_x^{-\kappa} \cap \Wc^{20\kappa, 20\kappa}} \lesssim \|U\|_{\Unorm}^2$, $\|Q^j_{> N} \linear\|_{\Cs_x^{1-2\kappa}} \lesssim N^{-\kappa} \|\linear\|_{C_t^0 \Cs_x^{-\kappa}}$ (by Lemma \ref{prelim:lem-Q}), and 
\[ \|\Ht Q^j_{> N} \linear(0)\|_{C_t^0 \Cs_x^{1-2\kappa} \cap \Wc^{1+10\kappa, 10\kappa}} \lesssim \|Q^j_{> N} \linear(0)\|_{\Cs_x^{1-2\kappa}}, \]
we further obtain
\[ \|F\|_{C_t^0 \Cs_x^{1-\rho} \cap \Wc^{1 + 5\kappa, \rho}} \lesssim N^{-\kappa} R S^2, \]
which is acceptable. 

Next, we bound $D_{\sim}$. By writing $P_{\leq N} = 1 - P_{> N}$, we have that
\[ D_{\sim} = \Duh\Big(U^{-1}\big(U \parasim P_{> N} \xi - P_{> N} (U \parasim \xi)\big)\Big). \]
By splitting $U = \bar{U}^r + (U^r - \bar{U}^r) + U^s$ and then applying Lemmas \ref{lemma:hi-hi-smooth}, \ref{lemma:U-r-bar-U-r-difference}, and \ref{lemma:hi-hi-rougher-part}, we obtain an acceptable bound on $D_{\sim}$. Finally, to bound $D_{\gg}$, we again use $P_{\leq N} = 1 - P_{> N}$ to write
\[ D_{\gg} = \Duh\Big(U^{-1}\big(U \paragg P_{> N} \xi - P_{> N} (U \paragg \xi)\big)\Big).\]
An acceptable bound for $D_{\gg}$ now follows directly by Lemma \ref{lemma:hi-lo}.

Combining the bounds involving $D_{\ll}, D_{\sim}, D_\gg$, we obtain 
\[ \|\Psi_{\leq N} \|_{C_t^0 \Cs_x^{1-\rho} \cap \Wc^{1+5\kappa, \rho}} \lesssim N^{-\kappa} R^2 S^4. \]
The bound on $\Phi_{\leq N} = (U^{-1} \ptl_j U) \parall Q^j_{>N} \linear + \Psi_{\leq N}$ follows from this bound, as well as the estimate
\[ \big\| (U^{-1} \ptl_j U) \parall Q^j_{> N} \linear \big\|_{C_t^0 \Cs_x^{1-\rho}} \lesssim \|U^{-1} \ptl_j U\|_{L_t^\infty \Cs_x^{-\kappa}} \|Q^j_{> N} \linear\|_{C_t^0 \Cs_x^{1-2\kappa}} \lesssim N^{-\kappa} R S^2. \qedhere \]
\end{proof}

\subsection{Modified enhanced data set}\label{section:gauged-transformed-enhanced-data-set}

We first introduce the modified enhanced data set. To this end, we recall that the gauged-transformed linear stochastic object $\glinear[][r][]$ was introduced in \eqref{gauged:eq-glinear} and studied in Proposition \ref{prop:linear-objects-close-gauge-covariance}. In particular, we previously obtained the decomposition
\begin{equation*}
\glinear[\leqN][r][i] = \linear[\leqN][r][i] + \Phi^i_{\leq N} = \linear[\leqN][r][i] + \Big[h_j \parall Q^j_{> N} \linear[][r][i]\Big] + \Psi^i_{\leq N},
\end{equation*}
where
\begin{align}
\| \Phi^{i}_{\leq N} \|_{C_t^0 \Cs_x^{1-\rho}} &\lesssim N^{-\kappa} R^2 S^4, \label{gauged:eq-Phi} \\ 
\| \Psi^i_{\leq N} \|_{C_t^0 \Cs_x^{1-\rho} \cap \Wc^{1+5\kappa, \rho}} &\lesssim N^{-\kappa} R^2 S^4. \label{gauged:eq-Psi} 
\end{align} 
Furthermore, similar as in \eqref{ansatz:eq-quadratic} and \eqref{ansatz:eq-quadratic-combined}, we introduce the modified quadratic stochastic object as the solution of 
\begin{equation}\label{gauged:eq-gquadratic}
(\partial_t - \Delta +1 ) \gquadratic[\leqN][d][i][j][k] = \Big[ \glinear[\leqN][r][i], \partial^k \glinear[\leqN][r][j] \Big], ~~ \gquadratic[\leqN][d][i][j][k](0) = \quadratic[\leqN][d][i][j][k](0),
\end{equation}
and define the combined object
\begin{equation}
 \gquadratic[\leqN][r][i] = 2  \gquadratic[\leqN][d][j][i][j] - \gquadratic[\leqN][d][j][j][i],
\end{equation}
where we sum over repeated indices on the right-hand side. 

\begin{definition}[Modified enhanced data set] 
For any $N\in \dyadic$, we define the modified enhanced data set as 
\begin{align*}
\Xi_{\leq N}^g := \Bigg(& 
\Big( \glinear[\leqN][r][i] \Big)_{i\in [2]}, 
\quad 
\Big( \Big[ \glinear[\leqN][r][i], \glinear[\leqN][r][j] \Big] \Big)_{i,j \in [2]}, 
\\
&\Big( E \in \frkg \mapsto \Big[ \Big[ E, \glinear[\leqN][r][i] \Big], 
\glinear[\leqN][r][j] \Big] - \delta^{ij} \sigma_{\leq N}^2 \Kil\big( E \big) \Big)_{i,j \in [2]}, 
\\ 
&\Big( \Big[ \Big[ \glinear[\leqN][r][i], \glinear[\leqN][r][j] \Big], \glinear[\leqN][r][k] \Big] +  \delta_{ik} \sigma_{\leq N}^2 \Kil \big( \glinear[\leqN][r][j] \big)  -  \delta_{jk} \sigma_{\leq N}^2 \Kil \big( \glinear[\leqN][r][i] \big) \Big)_{i,j,k\in [2]}, \\
&\Big( \gquadratic[\leqN][r][i] \Big)_{i \in [2]}, 
\\ 
& \Big( E \in \frkg \mapsto \Big[ \Big[ E , \Duh \big( \partial_{k_1} \glinear[\leqN][r][j_1] \big) \Big] \parasim \partial_{k_2} \glinear[\leqN][r][j_2] \Big] - \frac{1}{4} \delta^{j_1 j_2} \delta_{k_1 k_2} \sigma_{\leq N}^2 \Kil \big( E \big) \Big)_{j_1,j_2,k_1,k_2 \in [2]}, \\ 
&\Big( 2 \Big[  \gquadratic[\leqN][r][j] 
 \parasim 
 \Big( \partial_j \glinear[\leqN][r][i] -  \partial^i \glinear[\leqN][l][j] \Big) \Big] + 2  \Big[\gquadratic[\leqN][r][i]  \parasim \partial^j \glinear[\leqN][l][j] \Big]
 + \sigma_{\leq N}^2 \Kil \big( \glinear[\leqN][r][i] \big) \Big)_{i \in [2]}  \Bigg). 
\end{align*}
\end{definition}

The main result of this subsection, which is contained in the next proposition, shows that the enhanced data set $\Xi_{\leq N}$ and the modified enhanced data set $\Xi^g_{\leq N}$ converge to the same limit as $N\rightarrow \infty$.

\begin{proposition}\label{gauged:prop-enhanced}
Let $0<c\ll 1$ be a sufficiently small absolute constant and let $R \geq 1$. Then, there exists an event $E_R^g$ satisfying
\begin{equation}\label{gauged:eq-high-probability}
\mathbb{P} \big( E_R^g \big) \geq 1 - c^{-1} \exp \Big( - c R^2\Big) 
\end{equation}
and such that, on this event, the estimate
\begin{equation}\label{gauged:eq-enhanced-main-estimate}
\dc_T \Big( \Xi_{\leq N}, \Xi_{\leq N}^g \Big) \lesssim N^{-\eta} R^8 S^{16}
\end{equation} 
is satisfied for all $N\in \dyadic$, all $S\geq 1$, and all $0<T\leq \tau_S$.
\end{proposition}

\begin{remark}
We note that, in contrast to \eqref{objects:eq-enhanced-main-estimate}, \eqref{gauged:eq-enhanced-main-estimate} no longer scales linearly in $R$ (and also depends on $S$). The reason is that $\glinear[][r][]$ depends on $g$, which is a nonlinear object. 
\end{remark}

The proof of Proposition \ref{gauged:prop-enhanced} is rather delicate. In particular, it will require the detailed expansion of $\glinear[\leqN][r][]$ from Proposition \ref{prop:linear-objects-close-gauge-covariance}. 
For expository purposes, we first analyze the modified quadratic stochastic object from \eqref{gauged:eq-gquadratic}. This will be an important ingredient in the proof of Proposition \ref{gauged:prop-enhanced}.  

\begin{lemma}[Difference estimate for the quadratic stochastic object]
\label{gauged:lem-quadratic}
Let $R,S\geq 1$ and let $0<T\leq \tau_{S}$. For all $N\in \dyadic$, we assume that $\Xi_{\leq N}\in \Dc_R([0,1])$, where $\Dc_R([0,1])$ is as in Definition \ref{objects:def-data-space}. Let $\varep:= \frac{\kappa}{10}> 0$ and assume that
\begin{equation}\label{eq:gauged-lem-quadratic-extra-data-bound}
\begin{aligned} 
\big\|\linear\big\|_{C_t^0 \Cs_x^{-\varep}([0, 1] \times \T^2)} &\leq R, \\ 
\sup_{N \in \dyadic} N^\kappa \max_{\substack{E\in \frkg\colon \\ \| E \|_\frkg \leq 1 }} 
\Big\|  \Big[ \linear[\leqN][r][i] \parasim \Big[ E, \partial^k Q^\ell_{>N} \linear[][r][j] \Big] \Big] +\delta^{ij} \delta^{\ell k} \theta_{\leq N} \Kil E 
\Big\|_{C_t^0 \Cs_x^{-5\kappa}(([0, 1] \times \T^2)}^{1/2} &\leq R. 
\end{aligned} 
\end{equation}
Let $i,j,k\in [2]$,  let $h$ be as in  Definition \ref{def:U-and-h}, and let $\theta_{\leq N}$ be as in \eqref{objects:eq-theta}. Then, it holds for all $N\in \dyadic$ that  
\begin{align}
\bigg\| \gquadratic[\leqN][d][i][j][k] - \Big( \quadratic[\leqN][d][i][j][k] - 2 \delta^{ij} \theta_{\leq N} \Duh \Kil  h^k  \Big) 
\bigg\|_{C_t^0 \Cs_x^{1}([0,T]\times \T^2)} &\lesssim N^{-(\kappa-2\varep)} R^4 S^{8}, \label{gauged:eq-quadratic} \\ 
\bigg\| \gquadratic[\leqN][d][i][j][k] - \Big( \quadratic[\leqN][d][i][j][k] - 2 \delta^{ij} \theta_{\leq N} \Duh \Kil  h^k  \Big) 
\bigg\|_{\Sc^{1-2\kappa}([0,T])} &\lesssim N^{-(\kappa-2\varep)} R^4 S^{8}. \label{gauged:eq-quadratic-time}
\end{align}
\end{lemma}
\begin{remark}
While $\theta_{\leq N}$ is uniformly bounded in $N$, it does not decay as $N$ tends to infinity. In particular, it follows that the difference \begin{equation*}
\gquadratic[\leqN][d][i][j][k]-\quadratic[\leqN][d][i][j][k]  
\end{equation*}
does not converge to zero as $N$ tends to infinity. As we will see in the proof of Proposition \ref{gauged:prop-enhanced}, the $\Duh \Kil  h$-terms cancel once we consider the difference between the combined objects $\gquadratic[\leqN][r][j]$ and $\quadratic[\leqN][r][j]$. 
\end{remark}

\begin{proof}[Proof of Lemma \ref{gauged:lem-quadratic}:] Throughout the proof, all space-time norms are implicitly restricted to $[0,T]\times \T^2$. Recall that by definition of $\tau_S$ (Definition \ref{def:U-and-h}), we have that
\beq\label{eq:gauged-lem-quadratic-h-estimate} \|h\|_{C_t^0 \Cs_x^{-\kappa}}, ~~ \|h\|_{\Wc^{10\kappa, 10\kappa}} \leq S. \eeq
Using the definitions of $\gquadratic[\leqN][d][i][j][k]$ and $\quadratic[\leqN][d][i][j][k]$, it holds that
\begin{equation}\label{gauged:eq-quadratic-p1}
\begin{aligned}
&\gquadratic[\leqN][d][i][j][k] - \Big( \quadratic[\leqN][d][i][j][k] - 2 \delta^{ij} \theta_{\leq N} \Duh \Kil  h^k  \Big) \\ 
= \, &\Duh \bigg( 
\Big[ \glinear[\leqN][r][i], \partial^k \glinear[\leqN][r][j] \Big] - 
\Big( \Big[ \linear[\leqN][r][i], \partial^k \linear[\leqN][r][j] \Big] 
- 2 \delta^{ij} \theta_{\leq N} \Kil h^k \Big) \bigg). 
\end{aligned}
\end{equation}
Due to our estimates for the Duhamel integral (Proposition \ref{prelim:prop-Duhamel-Salpha} and Lemma \ref{prelim:lem-Schauder}), the proof of \eqref{gauged:eq-quadratic} and \eqref{gauged:eq-quadratic-time} can be reduced to 
\begin{equation}\label{gauged:eq-quadratic-p2}
\bigg\| \Big[ \glinear[\leqN][r][i], \partial^k \glinear[\leqN][r][j] \Big] - 
\Big( \Big[ \linear[\leqN][r][i], \partial^k \linear[\leqN][r][j] \Big] 
- 2 \delta^{ij} \theta_{\leq N} \Kil h^k \Big) \bigg\|_{C_t^0 \Cs_x^{-1+\varep} + \Wc^{-10\kappa, \rho}} \lesssim N^{-(\kappa - 2\varep)} R^4 S^{8}. 
\end{equation}
To this end, we first use the two decompositions from \eqref{gauged:eq-glinear-decomposition}, from which it follows that 
\begin{align}
& \Big[ \glinear[\leqN][r][i], \partial^k \glinear[\leqN][r][j] \Big] 
- \Big[ \linear[\leqN][r][i], \partial^k \linear[\leqN][r][j] \Big]
\notag \\
=\, & \Big[ \linear[\leqN][r][i], \partial^k 
\Big[ h_{\ell} \parall Q^\ell_{>N} \linear[][r][j] \Big] \Big] 
+ \Big[ \Big[ h_{\ell} \parall Q^{\ell}_{>N} \linear[][r][i] \Big] , \partial^k \linear[\leqN][r][j] \Big] 
\label{gauged:eq-quadratic-p3} \\ 
+\,& \Big[ \linear[\leqN][r][i], \partial^k \Psi^j_{\leq N} \Big]
+ \Big[ \Psi^i_{\leq N}, \partial^k \linear[\leqN][r][j] \Big] 
+ \Big[ \Phi^i_{\leq N}, \partial^k \Phi^j_{\leq N} \Big]. 
\label{gauged:eq-quadratic-p4}
\end{align}
Thus, the desired estimate \eqref{gauged:eq-quadratic-p2} can be further reduced to the three estimates
\begin{align}
\Big\| 
\Big[ \linear[\leqN][r][i], \partial^k 
\Big[ h_{\ell} \parall Q^\ell_{>N} \linear[][r][j] \Big] \Big]
+ \delta^{ij} \theta_{\leq N} \Kil h^k
\Big\|_{\Wc^{-10\kappa, 10\kappa}} 
&\lesssim N^{-\kappa} R^2 S
\label{gauged:eq-quadratic-p5}, \\ 
\Big\| 
 \Big[ \Big[ h_{\ell} \parall Q^{\ell}_{>N} \linear[][r][i] \Big] , \partial^k \linear[\leqN][r][j] \Big] 
+ \delta^{ij} \theta_{\leq N} \Kil h^k
\Big\|_{\Wc^{-10\kappa, 10\kappa}} 
&\lesssim N^{-\kappa} R^2 S
\label{gauged:eq-quadratic-p6}, \\
\Big\| 
\Big[ \linear[\leqN][r][i], \partial^k \Psi^j_{\leq N} \Big]
+ \Big[ \Phi^i_{\leq N}, \partial^k \Phi^j_{\leq N} \Big]
\Big\|_{\Wc^{-10\kappa, \rho}} 
&\lesssim N^{-\kappa} R^4 S^{8}
\label{gauged:eq-quadratic-p7} \\
\Big\|\Big[ \Psi^i_{\leq N}, \partial^k \linear[\leqN][r][j] \Big] \Big\|_{C_t^0 \Cs_x^{-1+\varep} + \Wc^{-10\kappa, \rho}} &\lesssim N^{-(\kappa - 2\varep)} R^3 S^4.\label{gauged:eq-quadratic-p20}
\end{align}
We prove \eqref{gauged:eq-quadratic-p5} and \eqref{gauged:eq-quadratic-p20} in full detail, but omit the arguments for both \eqref{gauged:eq-quadratic-p6} and \eqref{gauged:eq-quadratic-p7}. The reason is that the proof of \eqref{gauged:eq-quadratic-p6} is similar to the proof of \eqref{gauged:eq-quadratic-p5} and that \eqref{gauged:eq-quadratic-p7} follows directly from para-product estimates (Lemma \ref{prelim:lem-para-product}), \eqref{gauged:eq-Phi}, and \eqref{gauged:eq-Psi}. In order to prove \eqref{gauged:eq-quadratic-p5}, we further decompose 
\begin{align}
&\Big[ \linear[\leqN][r][i], \partial^k 
\Big[ h_{,\ell} \parall Q^\ell_{>N} \linear[][r][j] \Big] \Big]
+ \delta^{ij} \theta_{\leq N} \Kil h^k 
\notag \\
=\, & \Big[ \linear[\leqN][r][i], \Big[ \partial^k h_{\ell} \parall Q^\ell_{>N} \linear[][r][j] \Big] \Big] 
\label{gauged:eq-quadratic-p8} \\ 
+\, &  \Big[ \linear[\leqN][r][i] \paransim \Big[ h_{\ell} \parall \partial^k Q^\ell_{>N} \linear[][r][j] \Big] \Big] 
\label{gauged:eq-quadratic-p9} \\ 
+\, &  \Big[ \linear[\leqN][r][i] \parasim \Big[ h_{\ell} \parall \partial^k Q^\ell_{>N} \linear[][r][j] \Big] \Big]  
- h^a_{\ell} \Big[ \linear[\leqN][r][i] \parasim \Big[ E_a, \partial^k Q^\ell_{>N} \linear[][r][j] \Big] \Big] 
\label{gauged:eq-quadratic-p10} \\
+\, & h^a_{\ell} \bigg( \Big[ \linear[\leqN][r][i] \parasim \Big[ E_a, \partial^k Q^\ell_{>N} \linear[][r][j] \Big] \Big] +\delta^{ij} \delta^{\ell k} \theta_{\leq N} \Kil E_a \bigg). 
\label{gauged:eq-quadratic-p11} 
\end{align}
We now estimate \eqref{gauged:eq-quadratic-p8}-\eqref{gauged:eq-quadratic-p11} separately. Using Lemma \ref{prelim:lem-para-product}, Lemma \ref{prelim:lem-Q}, and \eqref{eq:gauged-lem-quadratic-h-estimate}, it holds that 
\begin{align*}
\Big\| \eqref{gauged:eq-quadratic-p8} \Big\|_{\Wc^{-10\kappa, 10\kappa}} 
&\lesssim \big\| \, \linear[\leqN][r] \big\|_{C_t^0 \Cs_x^{-\kappa}} 
\big\| \partial h \big\|_{\Wc^{-1+10\kappa, 10\kappa}} 
\big\| Q_{>N} \linear[][r][] \big\|_{C_t^0 \Cs_x^{1-2\kappa}} \\
&\lesssim N^{-\kappa}
\big\| h \big\|_{\Wc^{10\kappa, 10\kappa}} 
\big\| \, \linear[][r] \big\|_{C_t^0 \Cs_x^{-\kappa}}^2 \\ 
&\lesssim N^{-\kappa} R^2 S. 
\end{align*}
Using Lemma \ref{prelim:lem-para-product}, Lemma \ref{prelim:lem-Q}, and \eqref{eq:gauged-lem-quadratic-h-estimate}, it follows that 
\begin{align*}
\Big\| \eqref{gauged:eq-quadratic-p9} \Big\|_{C_t^0 \Cs_x^{-10\kappa}} 
&\lesssim 
\big\| \, \linear[\leqN][r][i] \big\|_{C_t^0 \Cs_x^{-\kappa}} 
\Big\| \Big[ h_{\ell} \parall \partial^k Q^\ell_{>N} \linear[][r][j] \Big] \Big\|_{C_t^0 \Cs_x^{-8\kappa}} \\ 
&\lesssim \big\| \, \linear[\leqN][r][] \big\|_{C_t^0 \Cs_x^{-\kappa}}
\big\| h_{\ell} \big\|_{C_t^0 \Cs_x^{-\kappa}} 
\big\| \partial^k Q^\ell_{>N} \linear[][r][j] \big\|_{C_t^0 \Cs_x^{-5\kappa}} \\ 
&\lesssim N^{-\kappa} 
\big\| h \big\|_{C_t^0 \Cs_x^{-\kappa}} \big\| \, \linear[][r][] \big\|_{C_t^0 \Cs_x^{-\kappa}}^2 
\lesssim N^{-\kappa} R^2 S. 
\end{align*}
Using \cite[Lemma 2.4]{GIP15}, Lemma \ref{prelim:lem-Q}, and \eqref{eq:gauged-lem-quadratic-h-estimate}, it follows that
\begin{align*}
\Big\| \eqref{gauged:eq-quadratic-p10} \Big\|_{\Wc^{-10\kappa, 10\kappa}}
&\lesssim \big\| \, \linear[\leqN][r][i] \big\|_{C_t^0 \Cs_x^{-\kappa}} 
\big\| h_{\ell} \big\|_{\Wc^{10\kappa, 10\kappa}} 
\big\| \partial^k Q^{\ell}_{>N} \linear[][r][j] \big\|_{C_t^0 \Cs_x^{-5\kappa}}  \\
&\lesssim N^{-\kappa} R^2 S. 
\end{align*}
Finally, using a product estimate (Lemma \ref{prelim:lem-para-product}) and \eqref{eq:gauged-lem-quadratic-extra-data-bound}, it holds that 
\begin{align*}
\Big\|  \eqref{gauged:eq-quadratic-p11} \Big\|_{\Wc^{-10\kappa, 10\kappa}}
&\lesssim \big\| h \big\|_{\Wc^{10\kappa, 10\kappa}} 
\max_{\substack{E\in \frkg\colon \\ \| E \|_\frkg \leq 1 }} 
\Big\|  \Big[ \linear[\leqN][r][i] \parasim \Big[ E, \partial^k Q^\ell_{>N} \linear[][r][j] \Big] \Big] +\delta^{ij} \delta^{\ell k} \theta_{\leq N} \Kil E 
\Big\|_{C_t^0 \Cs_x^{-5\kappa}} \\ 
&\lesssim N^{-\kappa} R^2 S. 
\end{align*}
Thus, the $\Wc^{-10\kappa, 10\kappa}$ norm of all four terms \eqref{gauged:eq-quadratic-p8}, \eqref{gauged:eq-quadratic-p9}, \eqref{gauged:eq-quadratic-p10}, and \eqref{gauged:eq-quadratic-p11} are under control, which completes the proof of \eqref{gauged:eq-quadratic-p5}.

Next, we show \eqref{gauged:eq-quadratic-p20}. We split
\[ \Big[ \Psi^i_{\leq N}, \partial^k \linear[\leqN][r][j] \Big] = \Big[ \Psi^i_{\leq N} \parall  \partial^k \linear[\leqN][r][j] \Big] + \Big[ \Psi^i_{\leq N} \paragtrsim \partial^k \linear[\leqN][r][j] \Big]. \]
By a paraproduct estimate, \eqref{gauged:eq-Psi}, and \eqref{eq:gauged-lem-quadratic-extra-data-bound}, we have that
\[ \Big\|\Big[ \Psi^i_{\leq N} \parall  \partial^k \linear[\leqN][r][j] \Big]\Big\|_{C_t^0 \Cs_x^{-1+\varep}} \lesssim \|\Psi_{\leq N}\|_{C_t^0 \Cs_x^{1-\rho}} \big\|\partial \linear[\leqN][r][] \big\|_{C_t^0 \Cs_x^{-1+\varep}} \lesssim N^{2\varep} \|\Psi_{\leq N}\|_{C_t^0 \Cs_x^{1-\rho}} \big\|\ptl\linear\big\|_{C_t^0 \Cs_x^{-1-\varep}} \lesssim N^{-(\kappa - 2\varep)} R^3 S^4.  \]
By a paraproduct estimate and \eqref{gauged:eq-Psi}, we have that
\[ \Big\|\Big[ \Psi^i_{\leq N} \paragtrsim \partial^k \linear[\leqN][r][j] \Big]\Big\|_{\Wc^{\kappa, \rho}} \lesssim \|\Psi_{\leq N}\|_{\Wc^{1+5\kappa, \rho}} \big\|\ptl \linear\big\|_{C_t^0 \Cs_x^{-1-\kappa}} \lesssim N^{-\kappa} R^3 S^4.  \]
This completes the proof of \eqref{gauged:eq-quadratic-p20} and thus the proof of the lemma.
\end{proof}

\begin{proof}[Proof of Proposition \ref{gauged:prop-enhanced}:]
Let $\varep:= \frac{\kappa}{10}$ be as in Lemma \ref{gauged:lem-quadratic}. 
Using Proposition \ref{objects:prop-enhanced}, Lemma \ref{lemma:linear-object}, and Corollary \ref{objects:cor-quadratic-Q}, we can choose an event $E_R^g$ which satisfies \eqref{gauged:eq-high-probability} and such that, on this event, 
$\Xi_{\leq N} \in \Dc_R([0,1])$ for all $N\in \dyadic$ and \eqref{eq:gauged-lem-quadratic-extra-data-bound} is satisfied. In particular, this implies that, on this event, all assumptions in Lemma \ref{gauged:lem-quadratic} are satisfied. It remains to show that, on this event, the differences of the stochastic objects in  \eqref{objects:eq-enhanced-1}-\eqref{objects:eq-enhanced-6} are controlled. Since the argument for \eqref{objects:eq-enhanced-2} and  \eqref{objects:eq-enhanced-5} is similar to the argument for \eqref{objects:eq-enhanced-1}, and the argument for \eqref{objects:eq-enhanced-3} is similar to the argument for \eqref{objects:eq-enhanced-6}, we only control the difference for \eqref{objects:eq-enhanced-1}, \eqref{objects:eq-enhanced-4}, and \eqref{objects:eq-enhanced-6}. \\ 

\emph{Difference estimate for \eqref{objects:eq-enhanced-1}:} The difference between the linear stochastic objects $\glinear[\leqN][r][]$ and $\linear[\leqN][r][]$ in \eqref{objects:eq-enhanced-1} can be controlled directly using \eqref{gauged:eq-glinear-decomposition} and \eqref{gauged:eq-Phi}, which yield 
\begin{equation*}
\Big\| \glinear[\leqN][r][]-  \linear[\leqN][r][] \Big\|_{C_t^0 \Cs_x^{-\kappa}}
\lesssim \Big\| \Phi_{\leq N} \Big\|_{C_t^0 \Cs_x^{-\kappa}} 
\lesssim N^{-\kappa} R^2 S^4. 
\end{equation*}
The differences between the quadratic nonlinearities in \eqref{objects:eq-enhanced-1} can be controlled using the product estimate (Lemma \ref{prelim:lem-para-product}), \eqref{gauged:eq-glinear-decomposition}, and \eqref{gauged:eq-Phi}, which yield that
\begin{align*}
\Big\| \Big[ \glinear[\leqN][r][i], \glinear[\leqN][r][j] \Big] 
- \Big[ \linear[\leqN][r][i], \linear[\leqN][r][j] \Big] \Big\|_{C_t^0 \Cs_x^{-\kappa}} 
\lesssim\, & \Big\| \Big[ \Phi^i_{\leq N}, \linear[\leqN][r][j] \Big] 
+ \Big[ \linear[\leqN][r][i], \Phi^j_{\leq N} \Big]
+ \Big[ \Phi^i_{\leq N}, \Phi^j_{\leq N} \Big] \Big\|_{C_t^0 \Cs_x^{-\kappa}} \\ 
\lesssim\, & \Big\| \Phi_{\leq N} \Big\|_{C_t^0 \Cs_x^{2\kappa}} 
\Big\| \linear[\leqN][r][] \Big\|_{C_t^0 \Cs_x^{-\kappa}} 
+ \Big\| \Phi_{\leq N} \Big\|_{C_t^0 \Cs_x^{2\kappa}}^2 \\ 
\lesssim\, & N^{-\kappa} R^3 S^4 + \Big( N^{-\kappa} R^2 S^4 \Big)^2  
\lesssim  N^{-\kappa} R^4 S^{8}. 
\end{align*}

\emph{Difference estimate for \eqref{objects:eq-enhanced-4}:} 
In order to utilize Lemma \ref{gauged:lem-quadratic}, we write
\begin{align}
\gquadratic[\leqN][r][j] - \quadratic[\leqN][r][j] 
&= 2 \bigg( \gquadratic[\leqN][d][j][i][j] - \quadratic[\leqN][d][j][i][j] \bigg) 
- \bigg( \gquadratic[\leqN][d][j][j][i] - \quadratic[\leqN][d][j][j][i] \bigg) \notag \\ 
&= 2 \bigg( \gquadratic[\leqN][d][j][i][j] - \Big( 
\quadratic[\leqN][d][j][i][j] 
- 2 \delta^i_j \theta_{\leq N} \Duh  \Kil h_{\leq N}^j  \Big) \bigg) 
\label{gauged:eq-enhanced-p1} \\
&- \bigg( \gquadratic[\leqN][d][j][j][i] - \Big(  \quadratic[\leqN][d][j][j][i] - 2 \delta^j_j \theta_{\leq N} \Duh \Kil  h_{\leq N}^i \bigg)
\label{gauged:eq-enhanced-p2} \\ 
&- 4 \delta^i_j \theta_{\leq N} \Duh \Kil h^j_{\leq N} + 2 \delta^j_j \theta_{\leq N} \Duh \Kil h^i_{\leq N}. \label{gauged:eq-enhanced-p3} 
\end{align}
Using Lemma \ref{gauged:lem-quadratic}, it holds that
\begin{equation*}
\Big\| \eqref{gauged:eq-enhanced-p1} \Big\|_{C_t^0 \Cs_x^{1-2\kappa} \,  \scalebox{0.8}{$\medcap$} \, S^{1-2\kappa}} 
+ \Big\| \eqref{gauged:eq-enhanced-p2} \Big\|_{C_t^0 \Cs_x^{1-2\kappa} \,  \scalebox{0.8}{$\medcap$} \, S^{1-2\kappa}} 
\lesssim N^{-(\kappa - 2\varep)} R^4 S^{8},
\end{equation*}
which is acceptable. By computing the implicit sum over $j\in [2]$, it holds that
\begin{equation*}
\eqref{gauged:eq-enhanced-p3} = - 4 \theta_{\leq N} \Duh \Kil h^i_{\leq N} + 4 \theta_{\leq N} \Duh \Kil h^i_{\leq N} = 0, 
\end{equation*}
and hence \eqref{gauged:eq-enhanced-p3} vanishes. 

\emph{Difference estimate for \eqref{objects:eq-enhanced-6}:} 
First, by Proposition \ref{prop:linear-objects-close-gauge-covariance} we have that
\[ \sigma^2_{\leq N} \big\|\Kil\big(\glinear[\leqN][r][] - \linear[\leqN][r][]\big)\big\|_{C_t^0 \Cs_x^{-\kappa}} \lesssim \sigma^2_{\leq N} \big\|\Phi_{\leq N}\|_{C_t^0 \Cs_x^{-\kappa}} \lesssim \sigma^2_{\leq N} N^{-\kappa} R^2 S^4 \lesssim N^{-(k-4\varep)} R^2 S^4, \]
where the final inequality follows since $\sigma^2_{\leq N} \lesssim \log N$. Next, let $i, j, k \in [2]$. We write
\begin{align}
&\Big[\gquadratic[\leqN][r][i] \parasim \ptl^j \glinear[\leqN][r][k]\Big] - \Big[\quadratic[\leqN][r][i] \parasim \ptl^j \linear[\leqN][r][k]\Big] \\
&=  \Big[\gquadratic[\leqN][r][i] \parasim \ptl^j \Phi^k_{\leq N}\Big] \label{gauged:eq-enhanced-p4}\\
&+ \Big[\gquadratic[\leqN][r][i]  -   \quadratic[\leqN][r][i] \parasim \ptl^j \linear[\leqN][r][k]\Big] \label{gauged:eq-enhanced-p5} 
\end{align}
By a paraproduct estimate, the difference estimate for \eqref{objects:eq-enhanced-4}, and \eqref{gauged:eq-Phi}, we have that
\[ \|\eqref{gauged:eq-enhanced-p4}\|_{C_t^0 \Cs_x^{-\kappa}} \lesssim \Big\|\gquadratic[\leqN][r][i]\Big\|_{C_t^0 \Cs_x^{1-2\kappa}} \|\ptl^j \Phi^k\|_{C_t^0 \Cs_x^{-\rho}} \lesssim N^{-\kappa} \big(R^2 + N^{-(\kappa -2 \varep)}R^4 S^8 \big) R^2 S^4.  \]
By the difference estimate for \eqref{objects:eq-enhanced-4}, a paraproduct estimate, and \eqref{eq:gauged-lem-quadratic-extra-data-bound}, we have that
\[\begin{split}
\|\eqref{gauged:eq-enhanced-p5}\|_{C_t^0 \Cs_x^{\varep}} &\lesssim \Big\|\gquadratic[\leqN][r][i]  -   \quadratic[\leqN][r][i]\Big\|_{C_t^0 \Cs_x^1} \big\|\ptl \linear[\leqN]\big\|_{C_t^0 \Cs_x^{-1+\varep}} \\
&\lesssim N^{2\varep} \Big\|\gquadratic[\leqN][r][i]  -   \quadratic[\leqN][r][i]\Big\|_{C_t^0 \Cs_x^1} \big\|\ptl \linear\big\|_{C_t^0 \Cs_x^{-1-\varep}} \lesssim N^{-(\kappa - 4 \varep)} R^5 S^{8}.
\end{split}\]
The desired result now follows.
\end{proof}

\subsection{Contraction mapping argument and the proof of Proposition \ref{prop:main-gauge-covariance}}\label{section:contraction-mapping-gauge-covariance}

In this section, we carry out the contraction mapping argument that gives the proof of Proposition \ref{prop:main-gauge-covariance}.

\begin{notation}
For brevity, let $\mc{X} := \Sc^{1-2\kappa}$, $\mc{Y} := \Sc^{2-5\kappa}$. We will sometimes write $(\mc{X} \times \mc{Y})([0, T])$ instead of $\mc{X}([0, T]) \times \mc{Y}([0, T])$. We will write $\|(X, Y)\|_{\mc{X} \times \mc{Y}} := \|X\|_{\mc{X}} + \|Y\|_{\mc{Y}}$.
\end{notation}

Recall the map $\Gamma = (\Gamma^X, \Gamma^Y)$ from Definition \ref{def:Gamma-contraction-map}, which was shown to be a (1/2)-contraction in the proof of Proposition \ref{nonlinear:prop-wellposedness-para}, and whose fixed points were the solutions to the paracontrolled SYM equation. Recall that for any $T \geq 0$, $\Gamma : (\mc{X} \times \mc{Y})([0, T]) \ra (\mc{X} \times \mc{Y})([0, T])$. Recall also that there is implicit dependence on $A^{(0)}$ and $\Xi$, i.e. $\Gamma(X, Y) = \Gamma(X, Y; A^{(0)}, \Xi)$. From Proposition \ref{nonlinear:prop-wellposedness-para} (and a standard bootstrapping argument), we have that on $[0, \maxtime)$,
\[ A = \linear + \quadratic + X + Y, \]
where for any $T \in [0, \maxtime)$, we have that $\Gamma(X, Y) = (X, Y)$ on $[0, T]$. \\

In the following , recall (from Definition \ref{def:U-and-h}) that $h = g^{-1} dg$. Recall equation \eqref{eq:gauged-A}, which is the equation which defines $\gauged{A}$. As usual, in order to solve this equation, we cast it as a fixed point problem, where now the contraction map is $\Gamma$ plus a perturbation (arising from the fact that equation \eqref{eq:gauged-A} is a modified version of the SYM equation) which is defined as follows.

\begin{definition}
Let $T \in [0, \maxtime)$. Define the map 
\[ \mbf{P}^Y : \mc{X}([0, T]) \times \mc{Y}([0, T]) \times C_t^0 \Cs_x^{1-2\kappa}([0, T] \times \T^2) \ra C_t^0 \Cs_x^{2-2\kappa}([0, T]) \]
by
\[ \mbf{P}^Y(\bar{X}, \bar{Y}, D) = \mbf{P}^Y(\bar{X}, \bar{Y}, D; X, Y, h) := \Duh(d [D_j, h^j]) + \Duh(d[X_j - \bar{X}_j, h^j]) + \Duh(d[Y_j - \bar{Y}_j, h^j]).\]
\end{definition}

\begin{remark}
The fact that $\mbf{P}^Y$ maps into $C_t^0 \Cs_x^{2-2\kappa}$ follows by Schauder.
\end{remark}

Having defined the perturbation $\mbf{P}^Y$, we may now define the contraction map corresponding to the equation for $\gauged{A}$.

\begin{definition}
Let $T \in [0, \maxtime)$. Let $D \in C_t^0 \Cs_x^{1-2\kappa}([0, T] \times \T^2)$. Define 
\[ \bar{\Gamma} : (\mc{X} \times \mc{Y})([0, T]) \ra (\mc{X} \times \mc{Y})([0, T])\]
by
\[ \bar{\Gamma}(\bar{X}, \bar{Y})  :=  \big(\Gamma^X(\bar{X}, \bar{Y}), \bar{\Gamma}^Y(\bar{X}, \bar{Y}) \big), \]
where $\Gamma^X$ is as before and
\[ \bar{\Gamma}^Y(\bar{X}, \bar{Y}) = \bar{\Gamma}^Y(\bar{X}, \bar{Y}; A^{(0)}, \Xi, X, Y, D, h) := \Gamma^Y(\bar{X}, \bar{Y}; A^{(0)}, \Xi) + \mbf{P}^Y(\bar{X}, \bar{Y}, D; X, Y, h). \]
Note that $\bar{\Gamma}$ has implicit dependence on $T, A^{(0)}, \Xi, X, Y, D, h$.
\end{definition}

By the same proof as for Proposition \ref{nonlinear:prop-wellposedness-para}, we have that $\gauged{A}$ may be expressed as
\[ \gauged{A} = \glinear + \gquadratic + \gauged{X} + \gauged{Y}, \]
where $(\gauged{X}, \gauged{Y})$ is a fixed point $(\gauged{X}, \gauged{Y}) = \bar{\Gamma}(\gauged{X}, \gauged{Y})$, or more precisely:
\[ (\gauged{X}, \gauged{Y}) = \bar{\Gamma}\Big(\gauged{X}, \gauged{Y}; A^{(0)}, \Xi^g, X, Y, \linear - \glinear + \quadratic - \gquadratic, h \Big), \]
at least up to a small enough time depending (inverse-polynomially) on the sizes of the various implicit parameters in the definition of $\bar{\Gamma}$. Then a standard bootstrapping argument shows that the above holds up to the simultaneous maximal existence time of $\gauged{A}, A$, and $g$. \\

The point now is that by Proposition \ref{gauged:prop-enhanced}, we have that $d(\Xi, \Xi^g) \ra 0$ as $N \toinf$, which in particular also implies that $\linear - \glinear + \quadratic - \gquadratic$ will be going to zero as $N \toinf$. Thus as $N \toinf$, the maps $\Gamma, \bar{\Gamma}$ essentially become the same map, and this is the reason why $A = \gauged{A}$ in the $N \toinf$ limit. In order to deduce that $A = \gauged{A}$ from the preceeding facts, we need the following stability estimate. The proof is very similar to the proof that $\Gamma$ is a contraction map in the proof of Proposition \ref{nonlinear:prop-wellposedness-para}, and thus it is omitted. First, recall the spaces $\Dc_R([0, T])$ \eqref{eq:Dc-R} and $\mc{B}_{\leq S}^{-\kappa}$ \eqref{eq:BR}.

\begin{lemma}\label{lemma:gamma-and-bar-gamma-close}
Let $C = C(\kappa, \theta)$, be a sufficiently large constant which depends only on\footnote{Recall that $\kappa, \theta$ enter into the definition of $\Sc$ (Definition \ref{prelim:def-solution-space})} $\kappa, \theta$. Let $T_0 \in [0, \maxtime)$. For $R \geq 1$, there exists $T \leq C^{-1} R^{-C}$ such that the following holds. Suppose $R \geq 1$ is such that
\[\begin{split}
&\Xi, \Xi^g \in \mc{D}_R([0, T_0]), ~~ \|(X, Y)\|_{(\mc{X} \times \mc{Y})([0, T_0])} \leq  R , ~~ A^{(0)}, \bar{A}^{(0)} \in \mc{B}_{\leq R}^{-\kappa}, ~~\|D\|_{C_t^0 \Cs_x^{1-2\kappa}([0, T_0])} \leq R.
\end{split}\]
Then, on $[0, T]$, we have that
\[\begin{split}
\|\Gamma(X_0, Y_0) - \bar{\Gamma}(\bar{X}_0, \bar{Y}_0)\|_{(\mc{X} \times \mc{Y})([0, T])} \leq \frac{1}{2} \|&(X_0, Y_0) - (\bar{X}_0, \bar{Y}_0)\|_{(\mc{X} \times \mc{Y})([0, T])} ~+ \\
&\tilde{C}(R) \big( \|A^{(0)} - \bar{A}^{(0)}\|_{\Cs_x^{-\kappa}} + d_T(\Xi, \bar{\Xi}) + \|D\|_{C_t^0 \Cs_x^{1-2\kappa}([0, T] \times \T^2)}\big),
\end{split}\]
where $\tilde{C}(R) = \tilde{C}(R, \kappa, \theta)$ is a constant depending only on $R, \kappa, \theta$. To be clear, here 
\[ \bar{\Gamma}(\bar{X}_0, \bar{Y}_0) = \bar{\Gamma}(\bar{X}_0, \bar{Y}_0; \bar{A}^{(0)}, \Xi^g, X, Y, D, h)\]
\end{lemma}

We may now prove Proposition \ref{prop:main-gauge-covariance}.

\begin{proof}[Proof of Proposition \ref{prop:main-gauge-covariance}]
Let $C$ be as in Lemma \ref{lemma:gamma-and-bar-gamma-close}. Fix $T_0 \in [0, \maxtime) \cap [0, 1]$. Let $R \geq 1$ be such that
\[\begin{split}
\sup_{N \in 2^{\N_0}} d_{T_0}(\Xi_{\leq N}, 0) \leq R, ~~ &\sup_{N \in 2^{\N_0}} d_{T_0}(\Xi_{\leq N}, \Xi^g_{\leq N}) \leq R ~~ \sup_{N \in 2^{\N_0}} \|(X_{\leq N}, Y_{\leq N})\|_{(\mc{X} \times \mc{Y})([0, T_0])} \leq  R, ~~ A^{(0)} \in \mc{B}_{\leq R}^{-\kappa}.
\end{split}\]
By Proposition \ref{gauged:prop-enhanced}, we have that $\lim_{N \toinf} d_{T_0}(\Xi_{\leq N}, \Xi^g_{\leq N}) \stackrel{a.s.}{=} 0$. Thus for large enough $N$, we have that $\Xi_{\leq N}^g \in \mc{D}_{2R}([0, T_0])$. Let $T \leq \min(T_0,  C^{-1}(2R)^{-C})$. By Lemma \ref{lemma:gamma-and-bar-gamma-close}, we have that
\[ \|\Gamma(X_{\leq N}, Y_{\leq N}) - \bar{\Gamma}(\gauged{X}_{\leq N}, \gauged{Y}_{\leq N})\|_{(\mc{X} \times \mc{Y})([0, T])} \leq \frac{1}{2} \|(X_{\leq N}, Y_{\leq N}) - (\gauged{X}_{\leq N}, \gauged{Y}_{\leq N})\|_{(\mc{X} \times \mc{Y})([0, T])} +  \tilde{C}(R) d_T(\Xi, \Xi^g).\]
We previously observed that
\[ A_{\leq N} = \linear[\leqN][r][] + \quadratic[\leqN][d][][] + X_{\leq N} + Y_{\leq N}, ~~ \gauged{A}_{\leq N} = \glinear[\leqN][r][] + \gquadratic[\leqN][d][][] + \gauged{X}_{\leq N} + \gauged{Y}_{\leq N}, \]
where
\[ (X_{\leq N}, Y_{\leq N}) = \Gamma(X_{\leq N}, Y_{\leq N}), ~~  (\gauged{X}_{\leq N}, \gauged{Y}_{\leq N}) = \bar{\Gamma} (\gauged{X}_{\leq N}, \gauged{Y}_{\leq N} ).\]
Combining, we thus obtain
\[ \frac{1}{2} \|(X_{\leq N}, Y_{\leq N}) - (\gauged{X}_{\leq N}, \gauged{Y}_{\leq N})\|_{(\mc{X} \times \mc{Y})([0, T])} \leq \tilde{C}(R) d_T(\Xi_{\leq N}, \Xi^g_{\leq N}), \]
and thus
\[\begin{split}
\|A_{\leq N} - \gauged{A}_{\leq N}\|_{C_t^0 \Cs_x^{-\kappa}([0, T] \times \T^2)} &\leq 2d_T(\Xi_{\leq N}, \Xi_{\leq N}^g) + 2 \tilde{C}(R) d_T(\Xi_{\leq N}, \Xi_{\leq N}^g) \stackrel{a.s.}{\ra} 0 \text{ as $N \toinf$}.
\end{split}\]
We can now bootstrap up to time $T_0$ by iterating.
\end{proof}

We close this section off by providing a proof of Lemma \ref{lemma:gtgauged-equation}.

\begin{proof}[Proof of Lemma \ref{lemma:gtgauged-equation}]
For brevity, write $\gauged{A} = \gauged{A}_{\leq N}$, and let $\gtgauged{A} = \gauged{A}^g = (\gauged{A}_{\leq N})^g$. Suppose first that $\Ad_g \xi$ is smooth. Then similar to the discussion at the beginning of Section \ref{section:gauge-covariance}, we have that
\[\begin{split}
\ptl_t \gtgauged{A} &= -D^*_{\gtgauged{A}} F_{\gtgauged{A}} - D_{\gtgauged{A}} D^*_{\gtgauged{A}} (\mrm{Ad}_g \gauged{A}) + D_{\gtgauged{A}}(\mrm{Ad}_g[A^j - \gauged{A}^j, g^{-1} \ptl_j g]) + D_{\gtgauged{A}} D^*_{A^g}((dg) g^{-1}) + P_{\leq N}(\Ad_g \xi) \\
&= -D^*_{\gtgauged{A}} F_{\gtgauged{A}} - D_{\gtgauged{A}} D^*_{\gtgauged{A}} \gtgauged{A} - D_{\gtgauged{A}} D_{\gtgauged{A}}^* ((dg) g^{-1}) +  D_{\gtgauged{A}}(\mrm{Ad}_g[A^j - \gauged{A}^j, g^{-1} \ptl_j g]) + D_{\gtgauged{A}} D^*_{A^g}((dg) g^{-1}) +  P_{\leq N}(\Ad_g \xi),
\end{split}\]
where in the second identity, we used that $D_{\gtgauged{A}} D^*_{\gtgauged{A}} (\mrm{Ad}_g \gauged{A}) = D_{\gtgauged{A}} D^*_{\gtgauged{A}}(\gtgauged{A} + (dg) g^{-1})$. To finish, it suffices to show that
\[ - D_{\gtgauged{A}}^*((dg) g^{-1}) + \mrm{Ad}_g[A^j - \gauged{A}^j, g^{-1} \ptl_j g] + D^*_{A^g}((dg) g^{-1}) = 0. \]
We have that
\[\begin{split}
D_{\gtgauged{A}}^*((dg) g^{-1}) - D^*_{A^g}((dg) g^{-1}) &= d^* ((dg) g^{-1}) - [\gtgauged{A}_j, (\ptl_j g) g^{-1}] - \big(d^* ((dg) g^{-1}) - [A^g_j, (\ptl_j g) g^{-1}] \big) \\
&= [A^g_j - \gtgauged{A}_j, (\ptl_j g) g^{-1}] = [\mrm{Ad}_g(A_j - \gauged{A}_j), (\ptl_j g) g^{-1}] \\
&= \mrm{Ad}_g[A_j - \gauged{A}_j, g^{-1} \ptl_j g],
\end{split}\]
as desired. This shows that if $\Ad_g \xi$ is smooth, then $\gtgauged{A}$ is a strong solution, and thus also a mild solution, to the claimed equation. For the result in full generality, we introduce a dyadic scale $M \in \dyadic$, and first replace $\Ad_g \xi$ by $(\Ad_g \xi)_{\leq M} := P_{\leq M, \leq M} (\Ad_g \xi)$. Let $\gauged{A}^{\leq M}$ be the solution to the equation \eqref{eq:gauged-A} for $\gauged{A}$, except with $\Ad_g \xi$ replaced by $(\Ad_g \xi)_{\leq M}$. By applying the result in the smooth case, we obtain that 
\[ (\gauged{A}^{\leq M})^g = \bA\big(A_0^{g_0}, P_{\leq N} ((\Ad_g \xi)_{\leq M})\big). \]
Recalling by Lemma \ref{lemma:U-xi-defined} that $\Ad_g \xi \in \Cs_{tx}^{(-1/2-100\kappa, -1-100\kappa)}$, we have that as $M \toinf$,
\[ (\Ad_g \xi)_{\leq M} \ra \Ad_g \xi \text{ in $\Cs_{tx}^{(-1/2-101\kappa, -1-101\kappa)}$}. \]
Using the fact that $\Duh : \Cs_{tx}^{(-1/2-101\kappa,-1-101\kappa)} \ra C_t^0 \Cs_x^{-500\kappa}$ is continuous (by Schauder, i.e. Lemma \ref{lemma:space-time-Besov-space-Schauder}), we then obtain that 
\[\lim_{M \toinf} \gauged{A}^{\leq M} = \lim_{M \toinf} \bA\big(A_0, P_{\leq N}((\Ad_g \xi)_{\leq M})\big) = \bA(A_0, P_{\leq N} (\Ad_g \xi)) = \gauged{A}, \]
\[ \lim_{M \toinf} \bA\big(A_0^{g_0}, P_{\leq N} ((\Ad_g \xi)_{\leq M})\big) = \bA(A_0^{g_0}, P_{\leq N} (\Ad_g \xi)). \]
By combining the previous few identities we obtain
\[ \gtgauged{A} = \gauged{A}^g = \lim_{M \toinf} (\gauged{A}^{\leq M})^g = \lim_{M \toinf} \bA\big(A_0^{g_0}, P_{\leq N} ((\Ad_g \xi)_{\leq M})\big) = \bA(A_0^{g_0}, P_{\leq N} (\Ad_g \xi)), \]
as desired.
\end{proof}

\begin{appendix}

\section{A  standard reduction}\label{section:standard-reduction}

\begin{lemma}[Standard reduction]\label{lemma:standard-reduction}
Let $F\colon [0,1]\times \T^2 \rightarrow \C$ be a continuous stochastic process. In addition, let $m\geq 1$, let $\epsilon>0$, let
$\alpha \in (0,\frac{1}{2})$, and let $C\geq 1$. Then, it holds for all $M\geq 1$ that
\begin{equation}\label{eq:standard-reduction-1}
\begin{aligned}
&\hspace{2ex}\sup_{p\geq 1} p^{-\frac{m}{2}} \E \Big[ \| F(t,x) \|_{L_t^\infty L_x^\infty([0,1]\times \T^2)}^p \Big]^{\frac{1}{p}} \\
&\lesssim_{\alpha,\epsilon,m,C}  M^\epsilon  \sup_{(t,x)\in [0,1] \times \T^2} \sup_{p\geq 1} p^{-\frac{m}{2}} \E \Big[ |F(t,x)|^p \Big]^{\frac{1}{p}} \\ 
&+ M^{-C}  
\sup_{\substack{(t,x),(s,y)\in [0,1] \times \T^2\colon \\ (t,x)\neq (s,y) }}
\sup_{p\geq 1} p^{-\frac{m}{2}} \E \bigg[ \bigg( \frac{\big| F(t,x)-F(s,y)\big|}{|t-s|^{\frac{\alpha}{2}}+|x-y|^{\alpha}} \bigg)^p \bigg]^{\frac{1}{p}}. 
\end{aligned}
\end{equation}
If, for all $(t,x)\in [0,1]\times \T^2$,  $F(t,x)$ is a multiple-stochastic integral of degree $m$, then it holds that
\begin{equation}\label{eq:standard-reduction-2} 
\begin{aligned}
\sup_{p\geq 1} p^{-\frac{m}{2}} \E \Big[ |F(t,x)|^p \Big]^{\frac{1}{p}} &\lesssim_m \E \Big[ |F(t,x)|^2 \Big]^{\frac{1}{2}}, \\ 
\sup_{p\geq 1} p^{-\frac{m}{2}} \E \bigg[ \bigg( \frac{\big| F(t,x)-F(s,y)\big|}{|t-s|^{\frac{\alpha}{2}}+|x-y|^{\alpha}} \bigg)^p \bigg]^{\frac{1}{p}} 
&\lesssim_m  \E \bigg[ \bigg( \frac{\big| F(t,x)-F(s,y)\big|}{|t-s|^{\frac{\alpha}{2}}+|x-y|^{\alpha}} \bigg)^2 \bigg]^{\frac{1}{2}}. 
\end{aligned}
\end{equation}
\end{lemma}

\begin{remark}
In all applications of Lemma \ref{eq:standard-reduction-1}, $F$ will be localized to a (spatial) frequency-scale $N\in \dyadic$. We then always choose $M\sim N$, which means that the pre-factor in the second summand of \eqref{eq:standard-reduction-1} gains an arbitrary power of $N$. Due to this gain, the second summand in \eqref{eq:standard-reduction-1} can always be estimated trivially, and we therefore often simply omit its estimate. For example, in the context of the linear object (Lemma \ref{lemma:linear-object}), we have that
\[ \linear[N][r][i](t, x) = \int f_{t, x, N}^i dW, \text{ where } f_{t, x, N}^i(\ell, n, u) = \delta^i_\ell \ind(u \leq t) \rho_N(n) \e_n(x) e^{-(t-u)\fnorm{n}^2}  I_{\cfrkg}, \]
and we may estimate (by Lemma \ref{lemma:multiple-stochastic-integral-second-moment-bound-l2})
\[\begin{split}
\E\Big[\big| \linear[N][r][i](t, x) - \linear[N][r][i](s, y)\big|_\frkg^2 \Big] \lesssim \|f^i_{t, x, N} - f^i_{s, y, N}\|_{L^2}^2 &= \int_{\indexset} (f^i_{t, x, N}(\ell, n, u) - f^i_{s, y, N}(\ell, n, u))^2 d\lebI \\
&\lesssim \sum_{n \in \Z^2} \rho_N(n) (|t-s| + |x-y|^2) \langle n \rangle^8\\
&\lesssim  N^{10} (|t - s| + |x-y|^2),
\end{split}\]
where we were very loose in obtaining the $N^{10}$ bound.
\end{remark}

\begin{proof}[Proof of Lemma \ref{lemma:standard-reduction}]
Let $0<h\leq 1$ remain to be chosen and let $\Lambda_h$ be a lattice discretization of $[0,1]\times \T^2$, whose lattice spacing in the time variable is of size $h^2$ and in the spatial variable is of size $h$. As a result, it holds that $\# \Lambda_h \sim h^{-4}$. For each $z\in [0,1]\times \T^2$, let $B_h(z)$ be the ball of radius $h$ with respect to the parabolic distance around $z$. Since $\{B_h(z)\colon z \in \Lambda_h\}$ covers $[0,1]\times \T^2$, it follows from estimates for maxima of random variables (see e.g. \cite{V18} or \cite[Lemma 2.4]{B20}) that 
\begin{align*}
\sup_{p\geq 1} p^{-\frac{m}{2}} \E \Big[ \| F(t,x) \|_{L_t^\infty L_x^\infty([0,1]\times \T^2)}^p \Big]^{\frac{1}{p}} 
\lesssim&\, \sup_{p\geq 1} p^{-\frac{m}{2}} \E \Big[ \max_{z\in \Lambda_h} \sup_{(t,x)\in B_h(z)} |F(t,x)|^p \Big]^{\frac{1}{p}} \\ 
\lesssim&\, \log(\#\Lambda_h)^{\frac{m}{2}} 
\sup_{p\geq 1} p^{-\frac{m}{2}} \E \Big[ \sup_{(t,x)\in B_h(z)} |F(t,x)|^p \Big]^{\frac{1}{p}}. 
\end{align*}
Since $\#\Lambda_h\sim h^{-4}$, it holds that $\log(\#\Lambda_h)^{\frac{m}{2}} \sim |\log(h)|^{\frac{m}{2}}$. Using Kolmogorov's continuity theorem (see e.g. \cite{S93}), it follows for all $p\geq p_0(\alpha)$, where $p_0(\alpha)$ is sufficiently large depending on $\alpha$, that
\begin{equation}
\begin{aligned}
 &\hspace{3ex} \E \Big[ \sup_{(t,x)\in B_h(z)} |F(t,x)|^p \Big]^{\frac{1}{p}} \\ 
 &\lesssim_\alpha \sup_{(t,x)\in [0,1]\times \T^2} \E \Big[ |F(t,x)|^p \Big]^{\frac{1}{p}} 
 + h^{\frac{\alpha}{2}} \sup_{\substack{(t,x),(s,y)\in [0,1] \times \T^2\colon \\ (t,x)\neq (s,y) }}
 \E \bigg[ \bigg( \frac{\big| F(t,x)-F(s,y)\big|}{|t-s|^{\frac{\alpha}{2}}+|x-y|^{\alpha}} \bigg)^p \bigg]^{\frac{1}{p}}. 
\end{aligned}
\end{equation}
By treating the regime $1\leq p \leq p_0(\alpha)$ using H\"{o}lder's inequality and combining the previous estimates, it follows that
\begin{equation} 
\begin{aligned}
&\hspace{2ex}\sup_{p\geq 1} p^{-\frac{m}{2}} \E \Big[ \| F(t,x) \|_{L_t^\infty L_x^\infty([0,1]\times \T^2)}^p \Big]^{\frac{1}{p}} \\
&\lesssim_{\alpha,m}   |\log(h)|^{\frac{m}{2}}   \sup_{(t,x)\in [0,1] \times \T^2} \sup_{p\geq 1} p^{-\frac{m}{2}} \E \Big[ |F(t,x)|^p \Big]^{\frac{1}{p}} \\ 
&+  |\log(h)|^{\frac{m}{2}} h^{\frac{\alpha}{2}}  
\sup_{\substack{(t,x),(s,y)\in [0,1] \times \T^2\colon \\ (t,x)\neq (s,y) }}
\sup_{p\geq 1} p^{-\frac{m}{2}} \E \bigg[ \bigg( \frac{\big| F(t,x)-F(s,y)\big|}{|t-s|^{\frac{\alpha}{2}}+|x-y|^{\alpha}} \bigg)^p \bigg]^{\frac{1}{p}}. 
\end{aligned}
\end{equation}
By choosing $h\sim M^{-\frac{2C}{\alpha}}$, this implies the desired estimate \eqref{eq:standard-reduction-1}. The second estimate \eqref{eq:standard-reduction-2} is a direct consequence of Gaussian hypercontractivity. 
\end{proof}

\section{Proofs of para-product and Schauder-type estimates}\label{section:proofs-para}

\begin{lemma}[\protect{Heat flow estimate \cite[Lemma 2.5]{CC18}}]\label{prelim:lem-heat-flow}
Let $\alpha,\beta\in \R$ satisfy $\alpha \geq \beta$, let $t>0$, and let $f\colon \T^2 \rightarrow \C$. Then, it holds that
\begin{equation*}
\big\| \Hc_t f \big\|_{\Cs^\alpha} \lesssim t^{-\frac{\alpha-\beta}{2}} \big\| f \big\|_{\Cs^\beta}. 
\end{equation*}
Furthermore, if also $\alpha < \beta+1$, then 
\begin{equation*}
\big\| \big( \Hc_t -1 \big) f \big\|_{\Cs^\beta} \lesssim t^{\frac{\alpha-\beta}{2}} \big\| f \big\|_{\Cs^\alpha}. 
\end{equation*}
\end{lemma}

\begin{lemma}[\protect{Heat commutator estimate \cite[Lemma 2.5]{CC18}}]\label{prelim:lem-heat-commutator}
Let $\alpha,\beta,\gamma\in \R$ satisfy $\alpha \geq \beta+\gamma$ and $\beta<1$. Then, it holds that
\begin{equation*}
\big\| \Ht \big( f \parall g \big) - f \parall \Ht g \big\|_{\Cs^\alpha} \lesssim t^{- \frac{\alpha-\beta-\gamma}{2}} \big\| f \big\|_{\Cs^\beta} \big\| g \big\|_{\Cs^\gamma}.
\end{equation*}
\end{lemma}

\begin{lemma}[Schauder estimate]\label{prelim:lem-Schauder}
Let $\alpha_1,\alpha_2 \in \R$, $\zeta \in [0,1)$, and $\nu_1,\nu_2 \in [0,1)$. 
\begin{enumerate}[label=(\roman*)]
\item If the parameter conditions
\begin{equation}\label{prelim:eq-Schauder-a1}
\alpha_2 \leq \alpha_1 < \alpha_2 +2 \qquad \text{and} \qquad \delta:= 1 + \nu_1 -\nu_2 - \frac{\alpha_1-\alpha_2}{2} >0 
\end{equation}
are satisfied, then it holds that
\begin{equation}\label{prelim:eq-Schauder-e1}
\big\| \Duh \big[ f \big] \big\|_{\Wc^{\alpha_1,\nu_1}([0,T])} \lesssim T^\delta  \big\| f \big\|_{\Wc^{\alpha_2,\nu_2}([0,T])}. 
\end{equation}
\item If the parameter conditions
\begin{equation}\label{prelim:eq-Schauder-a2}
\alpha_2 \leq \alpha_1 + 2 \zeta < \alpha_2 +2  \qquad \text{and} \qquad \delta:= 1 + \nu_1 -\nu_2 - \frac{\alpha_1+2\zeta-\alpha_2}{2}>0 
\end{equation}
are satisfied, then it holds that
\begin{equation}\label{prelim:eq-Schauder-e2}
\big\| \Duh \big[ f \big] \big\|_{\CWc^{\zeta,\alpha_1,\nu_1}([0,T])} \lesssim T^\delta  \big\| f \big\|_{\Wc^{\alpha_2,\nu_2}([0,T])}. 
\end{equation}
\end{enumerate}
\end{lemma}

\begin{proof}
We first prove \eqref{prelim:eq-Schauder-e1}. Using Lemma \ref{prelim:lem-heat-flow}, we obtain for all $t>0$ that
\begin{align*}
\Big\| t^{\nu_1} \int_0^t \ds \Hts f(s) \Big\|_{\Cs^{\alpha_1}} 
\leq t^{\nu_1} \int_0^t \ds \big\| \Hts f(s) \big\|_{\Cs^{\alpha_1}} 
\lesssim t^{\nu_1} \Big( \int_0^t \ds (t-s)^{-\frac{\alpha_1-\alpha_2}{2}} s^{-\nu_2} \Big) \big\| f \big\|_{\Wc^{\alpha_2,\nu_2}}.  
\end{align*}
Using the change of variables $s:=tu$, the pre-factor can be rewritten as 
\begin{equation}\label{prelim:eq-Schauder-p1}
t^{\nu_1} \Big( \int_0^t \ds (t-s)^{-\frac{\alpha_1-\alpha_2}{2}} s^{-\nu_2} \Big)  = t^{1+\nu_1 - \nu_2 - \frac{\alpha_1-\alpha_2}{2}} 
\int_0^1 \du (1-u)^{-\frac{\alpha_1-\alpha_2}{2}} u^{-\nu_2}. 
\end{equation}
Due to our assumptions from \eqref{prelim:eq-Schauder-a1} and our assumption $\nu_2<1$, \eqref{prelim:eq-Schauder-p1} yields an acceptable contribution. \\

We now turn the proof of \eqref{prelim:eq-Schauder-e2}, which is similar to the proof of \eqref{prelim:eq-Schauder-e1}. For any $t_2 \geq t_1 >0$, it holds that 
\begin{align}
\Duh \big[ f \big](t_2) - \Duh \big[ f \big] (t_1) 
&= \int_0^{t_2} \ds  \, \Hc_{t_2-s} f(s) -  \int_0^{t_1} \ds \, \Hc_{t_1-s} f(s)    \notag \\
&= \int_{t_1}^{t_2} \ds \, \Hc_{t_2-s} f(s) \label{prelim:eq-Schauder-p2}  \\ 
&+ \int_0^{t_1} \ds \, \big( \Hc_{t_2-t_1} - 1 \big) \Hc_{t_1-s} f(s) \label{prelim:eq-Schauder-p3}. 
\end{align}
We first estimate \eqref{prelim:eq-Schauder-p2}. Using Lemma \ref{prelim:lem-heat-flow}, we obtain that
\begin{align*}
t_1^{\nu_1} \big\| \eqref{prelim:eq-Schauder-p2} \big\|_{\Cs^{\alpha_1}} 
\lesssim \, t_1^{\nu_1} \Big( \int_{t_1}^{t_2}\ds \, (t_2-s)^{-\frac{\alpha_2-\alpha_1}{2}} s^{-\nu_2} \Big) \big\| f \big\|_{\Wc^{\alpha_2,\nu_2}}.
\end{align*}
Using $t_1 \leq t_2$, the integral can be estimated by 
\begin{align*}
 t_1^{\nu_1} \Big( \int_{t_1}^{t_2}\ds \, (t_2-s)^{-\frac{\alpha_2-\alpha_1}{2}} s^{-\nu_2} \Big)
&\lesssim  t_1^{\nu_1} (t_2-t_1)^\zeta \Big( \int_0^{t_2} \ds \, (t_2-s)^{-\frac{\alpha_2-\alpha_1}{2}-\zeta} s^{-\nu} \Big) \\
&\lesssim t_1^{\nu_1} t_2^{1-\nu_2-\frac{\alpha_1-\alpha_2}{2}-\zeta} (t_2-t_1)^{\zeta} \Big( \int_0^1 \du \, (1-u)^{-\frac{\alpha_2-\alpha_1}{2}-\zeta} u^{-\nu} \Big) \\
&\lesssim t_2^{1+\nu_1-\nu_2-\frac{\alpha_1-\alpha_2}{2}-\zeta} (t_2-t_1)^{\zeta}. 
\end{align*}
Due to our assumption \eqref{prelim:eq-Schauder-a2}, this yields an acceptable contribution. It now remains to estimate \eqref{prelim:eq-Schauder-p3}. Using Lemma \ref{prelim:lem-heat-commutator} and Lemma \ref{prelim:lem-heat-flow}, we obtain that
\begin{align*}
t_1^{\nu_1} \Big\| \int_0^{t_1} \ds \, \big( \Hc_{t_2-t_1}-1 \big) \Hc_{t_1-s} f(s) \Big\|_{\Cs_x^{\alpha_1}} 
&\lesssim t_1^{\nu_1} \int_0^{t_1} \ds \, \big\|  \big( \Hc_{t_2-t_1}-1 \big) \Hc_{t_1-s} f(s)  \big\|_{\Cs_x^{\alpha_1}} \\
&\lesssim t_1^{\nu_1} (t_2-t_1)^\zeta \int_0^{t_1} \ds  \big\| \Hc_{t_1-s} (f) \big\|_{\Cs_x^{\alpha_1+\zeta}} \\
&\lesssim t_1^{\nu_1} (t_2-t_1)^\zeta \Big( \int_0^{t_1} \ds \, (t_1-s)^{-\frac{\alpha_1-\alpha_2}{2}-\zeta} s^{-\nu_2} \Big) \big\| f \big\|_{\Wc^{\alpha_2,\nu_2}}. 
\end{align*}
Using a similar argument as in \eqref{prelim:eq-Schauder-p1}, this yields an acceptable contribution.
\end{proof}

Equipped with Lemma \ref{prelim:lem-Schauder}, we can now prove Proposition \ref{prelim:prop-Duhamel-Salpha}.

\begin{proof}[Proof of Proposition \ref{prelim:prop-Duhamel-Salpha}]
We first estimate the $\Wc^{\alpha^\prime,\frac{\alpha^\prime}{2}+\theta}$-terms in the $\Sc^\alpha$-norm. For any $\alpha^\prime \in \Reg \medcap [0,\alpha]$, it follows our parameter conditions \eqref{prelim:eq-Duhamel-Salpha-a}, $0\leq \alpha^\prime \leq \alpha$, $\beta\leq -\kappa$, and $\theta \geq 0$ that
\begin{equation}\label{prelim:eq-Duhamel-Salpha-p1}
\beta \leq \alpha^\prime < \beta + 2 \qquad \text{and} \qquad
\delta \leq 1 + \frac{\alpha^\prime}{2} + \theta - \frac{\alpha^\prime-\beta}{2} -\nu. 
\end{equation}
As a result, Lemma  \ref{prelim:lem-Schauder} implies that
\begin{equation*}
\big\| \Duh \big[ f \big] \big\|_{\Wc^{\alpha^\prime,\frac{\alpha^\prime}{2}+\theta}}
\lesssim T^\delta \big\| f \big\|_{\Wc^{\beta,\nu}}. 
\end{equation*}
The estimate for the $\CWc^{\frac{\alpha^\prime+\kappa}{2},-\kappa,\frac{\alpha^\prime}{2}+\theta}$-terms is similar and exhibits the same numerology. Finally, due to the identity 
\begin{equation*}
\big\| \Duh \big[ f \big] \big\|_{C_t^0 \Cs_x^{-\kappa}} =  \big\| \Duh \big[ f \big] \big\|_{\Wc^{-\kappa,0}}, 
\end{equation*}the estimate of the $C_t^0 \Cs_x^{-\kappa}$-norm follows from the same argument as long as
\begin{equation}\label{prelim:eq-Duhamel-Salpha-p2}
\beta \leq - \kappa < \beta+2 \qquad \text{and} \qquad 
\delta \leq 1 + \frac{\beta}{2} -\nu.
\end{equation}
Due to \eqref{prelim:eq-Duhamel-Salpha-a} and $\alpha \geq 0$, both conditions in \eqref{prelim:eq-Duhamel-Salpha-p2} are satisfied.
\end{proof}

The following lemma is a variant of \cite[Proposition 2.7]{CC18}.

\begin{lemma}[General integral commutator estimate]\label{prelim:lem-general-integral-commutator}
Let $\alpha,\alpha_1,\alpha_2,\alpha_1^\prime,\alpha_2^\prime \in \R$, $\zeta \in [0,1)$, and $\nu,\nu_1,\nu_2,\nu_1^\prime,\nu_2^\prime \in [0,1)$. Assume that
\begin{alignat}{3}
&\alpha_1 + \alpha_2 \leq \alpha < \alpha_1 + \alpha_2 +2 ,  \qquad 
&&\alpha_1^\prime + \alpha_2^\prime \leq \alpha < \alpha_1^\prime + \alpha_2^\prime +2\zeta +2, 
\label{prelim:eq-general-integral-commutator-a1} \\
&\alpha_1 <1, \qquad &&\alpha_1^\prime <0, \label{prelim:eq-general-integral-commutator-a2} \\
&\nu_1 + \nu_2 <1, \qquad && \nu_1^\prime + \nu_2^\prime <1. \label{prelim:eq-general-integral-commutator-a3} 
\end{alignat}
Furthermore, assume that
\begin{equation}\label{prelim:eq-general-integral-commutator-a4}
\delta := \min \Big( 1 + \nu - \nu_1 - \nu_2 - \frac{\alpha-\alpha_1-\alpha_2}{2}, \, 
1+ \nu - \nu_1^\prime - \nu_2^\prime + \zeta - \frac{\alpha-\alpha_1^\prime-\alpha_2^\prime}{2} \Big) >0. 
\end{equation}
Then, it holds that 
\begin{equation}\label{prelim:eq-general-integral-commutator}
\big\| \Duh \big[ f \parall g \big] - f \parall \Duh \big[ g \big] \big\|_{\Wc^{\alpha,\nu}} 
\lesssim T^\delta \Big( \big\| f \big\|_{\Wc^{\alpha_1,\nu_1}} \big\| g\big\|_{\Wc^{\alpha_2,\nu_2}} + 
\big\| f \big\|_{\CWc^{\zeta,\alpha_1^\prime,\nu_1^\prime}} 
\big\| g \big\|_{\Wc^{\alpha_2^\prime,\nu_2^\prime}} \Big). 
\end{equation}
\end{lemma}

\begin{proof}[Proof of Lemma \ref{prelim:lem-general-integral-commutator}]
We first decompose 
\begin{align}
\Duh \big[ f \parall g \big] - f \parall \Duh \big[ g \big] 
=& \, \int_0^t  \Hts \big( f(s) \parall g(s) \big) \ds 
- f(t) \parall \int_0^t \Hts g(s) \ds \notag \\
=& \, \int_0^t \Big( \Hts \big( f(s) \parall g(s) \big) - f(s) \parall \Hts g(s) \Big) \ds 
\label{prelim:eq-general-integral-commutator-p1} \\ 
+& \, \int_0^t \big( f(s) - f(t) \big) \parall \Hts g(s) \ds. 
\label{prelim:eq-general-integral-commutator-p2} 
\end{align}
We now estimate \eqref{prelim:eq-general-integral-commutator-p1} and \eqref{prelim:eq-general-integral-commutator-p2} separately. \\

\emph{Estimate of \eqref{prelim:eq-general-integral-commutator-p1}:} 
Using $\alpha_1<1$ and  Lemma \ref{prelim:lem-heat-commutator}, we obtain for all $t>0$ that
\begin{align*}
t^\nu \Big\| \eqref{prelim:eq-general-integral-commutator-p1} \Big\|_{\Cs_x^\alpha}
&\leq t^\nu 
\int_0^t \Big\| \Hts \big( f(s) \parall g(s) \big) - f(s) \parall \Hts g(s)  \Big\|_{\Cs_x^\alpha} \\
&\lesssim  t^{\nu} \int_0^t (t-s)^{-\frac{\alpha-\alpha_1-\alpha_2}{2}}
\big\| f(s) \big\|_{\Cs_x^{\alpha_1}} \big\| g(s) \big\|_{\Cs_x^{\alpha_2}} \\
&\lesssim  t^{\nu} \bigg( \int_0^t (t-s)^{-\frac{\alpha-\alpha_1-\alpha_2}{2}} s^{-\nu_1 - \nu_2}
\bigg)  \big\| f \big\|_{\Wc^{\alpha_1,\nu_1}} \big\| g\big\|_{\Wc^{\alpha_2,\nu_2}}. 
\end{align*}
Using our assumption \eqref{prelim:eq-general-integral-commutator-a1} and \eqref{prelim:eq-general-integral-commutator-a3}, the change of coordinates $s:=tu$, and $t\leq T$, it follows that
\begin{equation*}
 t^{\nu} \bigg( \int_0^t (t-s)^{-\frac{\alpha-\alpha_1-\alpha_2}{2}} s^{-\nu_1 - \nu_2}
\bigg) 
\lesssim  t^{1+\nu-\nu_1 - \nu_2 - \frac{\alpha-\alpha_1-\alpha_2}{2}} 
\bigg( \int_0^1 (1-u)^{-\frac{\alpha-\alpha_1-\alpha_2}{2}} u^{-\nu_1 - \nu_2} \, \mathrm{d}u 
\bigg) \lesssim  T^\delta. 
\end{equation*}
As a result, the contribution of \eqref{prelim:eq-general-integral-commutator-p1} can be controlled by the first summand in \eqref{prelim:eq-general-integral-commutator}. \\

\emph{Estimate of \eqref{prelim:eq-general-integral-commutator-p2}:} 
Using Lemma \ref{prelim:lem-para-product}, $\alpha_1^\prime <0$, and $\alpha \geq \alpha_1^\prime+\alpha_2^\prime$, it follows that
\begin{align*}
t^\nu \Big\| \eqref{prelim:eq-general-integral-commutator-p2} \Big\|_{\Cs^\alpha}
&\leq t^\nu \int_0^t \Big\| \big( f(s) - f(t) \big) \parall \Hts g(s) \Big\|_{\Cs_x^\alpha} \\
&\lesssim t^\nu \int_0^t \big\| f(s) - f(t) \big\|_{\Cs_x^{\alpha_1^\prime}} 
\big\| \Hts g(s) \big\|_{\Cs^{\alpha-\alpha_1^\prime}_x} \\
&\lesssim t^\nu \bigg( \int_0^t (t-s)^{\zeta-\frac{\alpha-\alpha_1^\prime-\alpha_2^\prime}{2}} s^{-\nu_1^\prime -\nu_2^\prime} \bigg) \big\| f \big\|_{\CWc^{\zeta,\alpha_1^\prime,\nu_1^\prime}} \big\| g\big\|_{\Wc^{\alpha_2^\prime,\nu_2^\prime}}. 
\end{align*}
Using our assumption \eqref{prelim:eq-general-integral-commutator-a1} and \eqref{prelim:eq-general-integral-commutator-a3}, the change of coordinates $s:=tu$, and $t\leq T$, it follows that
\begin{equation*}
    t^\nu \bigg( \int_0^t (t-s)^{\zeta-\frac{\alpha-\alpha_1^\prime-\alpha_2^\prime}{2}} s^{-\nu_1^\prime -\nu_2^\prime} \bigg)  
    \lesssim t^{1+\nu-\nu_1^\prime-\nu_2^\prime - \frac{\alpha-\alpha_1^\prime-\alpha_2^\prime}{2}+\zeta} 
    \bigg( \int_0^1 
    (1-u)^{\zeta-\frac{\alpha-\alpha_1^\prime-\alpha_2^\prime}{2}} u^{-\nu_1^\prime -\nu_2^\prime} \, \mathrm{d}u \bigg) \lesssim T^\delta. 
\end{equation*}
As a result, the contribution of \eqref{prelim:eq-general-integral-commutator-p2} can be controlled by the second summand in \eqref{prelim:eq-general-integral-commutator}.
\end{proof}

\begin{corollary}\label{cor:general-integral-commutator}
Let all parameters be as in Lemma \ref{prelim:lem-general-integral-commutator}. Then, it holds that
\[
\big\| \ptl_t \Duh \big[ f \parall g \big] - \ptl_t \big(f \parall \Duh \big[ g \big]\big) \big\|_{\Wc^{\alpha-2,\nu}} 
\lesssim T^\delta \Big( \big(\big\| f \big\|_{\Wc^{\alpha_1,\nu_1}} + \|\ptl_t f\|_{\Wc^{\alpha_1 - 2, \nu_1}} \big)\big\| g\big\|_{\Wc^{\alpha_2,\nu_2}} + 
\big\| f \big\|_{\CWc^{\zeta,\alpha_1^\prime,\nu_1^\prime}} 
\big\| g \big\|_{\Wc^{\alpha_2^\prime,\nu_2^\prime}} \Big). 
\]
\end{corollary}
\begin{proof}
Observe that
\[\begin{split}
\ptl_t \Duh \big[ f \parall g \big] &= \Delta  \Duh \big[ f \parall g \big] + f \parall g, \\
\ptl_t \big(f \parall \Duh \big[ g \big]\big) &= (\ptl_t f) \parall \Duh \big[ g \big] + f \parall \big(\Delta \Duh \big[ g \big] + g\big) .
\end{split}\]
It thus follows that
\[\begin{split}
\ptl_t \Duh \big[ f \parall g \big] - \ptl_t \big(f \parall \Duh \big[ g \big]\big) = \Delta \big(\Duh \big[ f \parall g \big] - f \parall \Duh\big[g\big]\big) +~ &(\Delta f - \ptl_t f) \parall \Duh\big[g\big] ~+ \\
&(\ptl_j f) \parall \ptl^j \Duh\big[g\big].
\end{split}\]
The first term on the right hand side above may be bounded in $\Wc^{\alpha-2, \nu}$ by Lemma \ref{prelim:lem-general-integral-commutator}. The second and third terms may be bounded by applying paraproduct estimates and Schauder (Lemma \ref{prelim:lem-Schauder}).
\end{proof}

Equipped with Lemma \ref{prelim:lem-general-integral-commutator}, we can now prove Lemma \ref{prelim:lem-trilinear-integral-commutator}.

\begin{proof}[Proof of Lemma \ref{prelim:lem-trilinear-integral-commutator}:]
We first prove the bilinear estimate
\begin{equation}\label{prelim:eq-integral-commutator}
\Big\| \Duh \big( f \parall g \big) - f \parall \Duh \big( g \big) \Big\|_{\Wc^{1+2\kappa,\nu}([0,T])} 
\lesssim \big\| f \big\|_{\Sc^{1-2\kappa}([0,T])} \big\| g \big\|_{\Wc^{-1-\kappa,0}([0,T])}. 
\end{equation}
In order to prove \eqref{prelim:eq-integral-commutator}, we want to use Lemma \ref{prelim:lem-general-integral-commutator}. We let $\nu$ be as in the statement of Lemma \ref{prelim:lem-trilinear-integral-commutator} and choose the remaining parameters as 
\begin{alignat}{5}
\alpha &:= 1+2\kappa,                        \qquad &\zeta   &:= \frac{\alpha_1+\kappa}{2}, \\ 
\alpha_1&:= 1-2\kappa,                       \qquad &\alpha_2&:=-1-\kappa,   
\qquad &
\alpha_1^\prime&:= - \kappa,                 \qquad &\alpha_2^\prime&:= -1-\kappa, \\
\nu_1 &:= \frac{\alpha_1}{2} + \theta,          \qquad &\nu_2 &:= 0,   
\qquad &
\nu_1^\prime &:= \frac{\alpha_1}{2} + \theta,    \qquad &\nu_2^\prime &:= 0.
\end{alignat}
It is easy to check that the parameters above satisfy the conditions \eqref{prelim:eq-general-integral-commutator-a1}, \eqref{prelim:eq-general-integral-commutator-a2}, and \eqref{prelim:eq-general-integral-commutator-a3}. Furthermore, a direct calculation reveals that
\begin{align*}
 &\min \Big( 1 + \nu - \nu_1 - \nu_2 - \frac{\alpha-\alpha_1-\alpha_2}{2}, \, 
1+ \nu - \nu_1^\prime - \nu_2^\prime + \zeta - \frac{\alpha-\alpha_1^\prime-\alpha_2^\prime}{2} \Big) \\
=& \nu - \theta - \frac{3}{2} \kappa.
\end{align*}
Thus, the condition \eqref{prelim:eq-general-integral-commutator-a4} is satisfied due to our assumption $\nu \geq \theta + 2\kappa$. Thus, all the assumptions of Lemma \ref{prelim:lem-general-integral-commutator} are satisfied, and the bilinear estimate \eqref{prelim:eq-integral-commutator} now follows from \eqref{prelim:eq-general-integral-commutator}. \\ 

The desired trilinear \eqref{prelim:rem-trilinear-integral-commutator} now follows directly from the 
para-product estimates (Lemma \ref{prelim:lem-para-product}) and the bilinear integral commutator estimate \eqref{prelim:eq-integral-commutator}. Indeed, it holds that
\begin{align*}
&\Big\| \Big( \Duh \big( f \parall g \big) - f \parall \Duh \big( g\big) \Big) \parasim h \Big\|_{\Wc^{-4\kappa,\nu}([0,T])} \\
\lesssim&\,  \Big\|  \Duh \big( f \parall g \big) - f \parall \Duh \big( g\big) \Big\|_{\Wc^{1+2\kappa,\nu}([0,T])} \big\| \,  h \big\|_{\Wc^{-1-\kappa,0}([0,T])} \\ 
 \lesssim& \, \big\| f \big\|_{\Sc^{1-2\kappa}([0,T])} \big\| g \big\|_{\Wc^{-1-\kappa,0}([0,T])} \big\| h\big\|_{\Wc^{-1-\kappa,0}([0,T])}. \qedhere
\end{align*}
\end{proof}

\section{Space-time function spaces}\label{appendix:space-time-besov-space}

Let $L$ be a dyadic scale such that $L > 1$. Define $\eta_L$ by
\[ \eta_L(t) := \int_{-\infty}^t \widecheck{\rho}_L(s) ds. \]
Note since $L > 1$, we have that $\widecheck{\rho}_L$ is mean zero, and thus $\eta_L$ is Schwartz. We also have that $\eta_L'(u) = \widecheck{\rho}_L$ and  $\hat{\eta}_L(u) = \frac{1}{iu} \rho_L$. The latter statement implies that convolution with $\eta_L$ acts as a Fourier multiplier supported on an annulus of scale $L$. Finally, note that $\|\eta_L\|_{L^1} \lesssim L^{-1}$. 

\begin{proof}[Proof of Lemma \ref{lemma:space-time-Besov-space-Schauder}]
It suffices to show that for all dyadic scales $L, N$, we have that
\[ \|\Duh(P_{L, N} f)\|_{C_t^0 \Cs_x^{\beta_x}([0, \infty))} \lesssim \max(L, N^2)^{-(2 + \alpha_x + 2\alpha_t - \beta_x)/2} \|f\|_{\Cs_{tx}^\alpha}, \]
Fix $t \in [0, \infty)$. First, suppose that $L \leq N^2$. We have that
\[ \|\Duh(P_{L, N} f)(t)\|_{\Cs_x^{\beta_x}} \lesssim N^{\beta_x} \|\Duh(P_{L, N} f)(t) \|_{L_x^\infty} \lesssim N^{\beta_x} \int_0^t \|\Hts (P_{L, N} f(s))\|_{L_x^\infty} ds. \]
Since $P_{L, N} f$ is localized to spatial frequency $N$, we have that $\|\Hts (P_{L, N} f(s))\|_{L_x^\infty} \lesssim e^{-c(t-s)N^2} \|P_{L, N} f(s)\|_{L_x^\infty}$ (note the fact that this is also true for $N = 1$ follows because $\Ht = e^{-t(1-\Delta)}$ and not $e^{t\Delta}$). We thus further obtain (using that $L \leq N^2$ and $\alpha_t < 0$)
\[\begin{split}
\|\Duh(P_{L, N} f)(t)\|_{\Cs_x^{\beta_x}} \lesssim N^{\beta_x} \int_0^t e^{-c(t-s)N^2} ds \|P_{L, N} f\|_{L_{tx}^\infty} &\lesssim N^{-2} N^{\beta_x} L^{-\alpha_t} N^{-\alpha_x} \|f\|_{\Cs_{tx}^\alpha} \\
&\lesssim N^{-(2 + \alpha_x +2\alpha_t - \beta_x)} \|f\|_{\Cs_{tx}^\alpha}.
\end{split}\]
Next, suppose that $L > N^2$. Now define $g := \eta_L * (P_N f)$, so that $\ptl_t g = \eta_L' * (P_N g) = P_{L, N} f$. Using the various properties of $\eta_L$ we listed at the beginning of this section, we may show that 
\[ \|g\|_{L^\infty_{tx}} \lesssim L^{-(1+\alpha_t)} N^{-\alpha_x} \|f\|_{\Cs_{tx}^\alpha},\]
which implies that (using that $L > N^2$ and $\beta_x - \alpha_x \geq 0$)
\[ \|g\|_{L_t^\infty \Cs_x^{\beta_x}([0, \infty))} \lesssim N^{\beta_x} \|g\|_{L_{tx}^\infty} \lesssim  L^{-(2 + \alpha_x + 2\alpha_t - \beta_x)/2} \|f\|_{\Cs_{tx}^\alpha}.   \]
Now, since $\ptl_t g = P_{L, N} f$, we have that
\[ \Duh(P_{L, N} f) = \Duh(\ptl_t g) = \Delta \Duh(g) + g(t) - \Ht g(0). \]
The latter two terms are controlled by our previous estimate on $\|g\|_{L_t^\infty \Cs_x^{\beta_x}}$. For the first term, by arguing similarly in the case $L \leq N^2$, we may obtain
\[ \| \Delta \Duh(g) (t) \|_{\Cs_x^{\beta_x}} \lesssim N^{2 + \beta_x} \|\Duh(g)(t)\|_{L_x^\infty} \lesssim N^{2 + \beta_x - 2} \|g\|_{L_{tx}^\infty}. \]
By our previous estimate on $\|g\|_{L_{tx}^\infty}$, the right hand side above is bounded by $L^{-(2 + \alpha_x + 2\alpha_t - \beta_x)/2} \|f\|_{\Cs_{tx}^\alpha}$, as desired.
\end{proof}

\begin{proof}[Proof of Lemma \ref{lemma:space-time-integral-commutator-estimate}]
Let $\eta > 0$ be such that
\[ \alpha' < 2 + \alpha + \beta_x + 2\beta_t - \eta, ~~ 1 + \beta_t - \frac{\eta}{2} > 0, ~~ 1 + \nu' - \nu - \frac{\alpha' - (\alpha + \beta_x + 2\beta_t - \eta)}{2} > 0.\]
It suffices to show that for all dyadic scales $L, N$, we have that 
\[ \|\Duh(f \parall P_{L, N} g) - f \parall \Duh(P_{L, N} g)\|_{\Wc^{\alpha', \nu'}} \lesssim \max(L, N^2)^{-\eta/2}   \big(\|f\|_{\Wc^{\alpha, \nu}} + \|\ptl_t f\|_{\Wc^{\alpha - 2, \nu}}\big) \|g\|_{\Cs_{tx}^\beta}, \]
First, suppose that $L \leq N^2$. We then have that (note $\beta_t < 0$)
\[ \|P_{L, N} g \|_{L_t^\infty \Cs_x^{\beta_x + 2\beta_t - \eta}} \lesssim N^{\beta_x + 2\beta_t -\eta} \|P_{L, N} g\|_{L_{tx}^\infty} \lesssim L^{-\beta_t} N^{-\beta_x + \beta_x + 2\beta_t - \eta} \|g\|_{\Cs_{tx}^\beta} \lesssim N^{-\eta} \|g\|_{\Cs_{tx}^{\beta}}, \]
Applying Lemma \ref{prelim:lem-general-integral-commutator} and the above estimate, we obtain
\[\begin{split}
\|\Duh(f \parall P_{L, N} g) - f \parall \Duh(P_{L, N} g)\|_{\Wc^{\alpha', \nu'}} & \lesssim \big(\|f\|_{\Wc^{\alpha, \nu}} + \|\ptl_t f\|_{\Wc^{\alpha - 2, \nu}}\big) \|P_{L, N} g\|_{L_t^\infty \Cs_x^{\beta_x + 2\beta_t - \eta}} \\
& \lesssim N^{-\eta} \big(\|f\|_{\Wc^{\alpha, \nu}} + \|\ptl_t f\|_{\Wc^{\alpha - 2, \nu}}\big) \|g\|_{\Cs_{tx}^\beta}. 
\end{split}\]
Next, suppose that $L > N^2$. Define $h := \eta_L * (P_N g)$, so that $\ptl_t h = \eta_L' * (P_N g) = P_{L, N} g$. Using the properties of $\eta_L$ listed at the beginning of this section, we may show that
\[ \|h\|_{L_{tx}^\infty} \lesssim L^{-(1+\beta_t)} N^{-\beta_x} \|g\|_{\Cs_{tx}^\beta}, \]
which implies (recall that $1 + \beta_t - \eta/2 > 0$)
\beq\label{eq:h-estimate-L-geq-N2} \|h\|_{L_t^\infty \Cs_x^{\beta_x + 2\beta_t + 2-\eta}} \lesssim N^{\beta_x + 2\beta_t + 2 - \eta} \|h\|_{L_{tx}^\infty} \lesssim L^{-(1+\beta_t)} N^{2(\beta_t + 1 - \eta/2)} \|g\|_{\Cs_{tx}^{\beta}} \lesssim L^{-\eta/2} \|g\|_{\Cs_{tx}^\beta}.  \eeq
Now, since $\ptl_t h = P_{L, N} g$, we may write
\[\begin{split}
\Duh(f \parall P_{L, N} g) &= \Duh(\ptl_t (f \parall h)) - \Duh(\ptl_t f \parall h) \\
&= \Delta \Duh(f \parall h)  + f(t) \parall h(t) -  e^{t\Delta} (f(0) \parall h(0)) - \Duh(\ptl_t f \parall h). 
\end{split}\]
Similarly, 
\[ f\parall \Duh(P_{L, N} g) = f \parall \Delta \Duh(h) + f(t) \parall h(t) - f(t) \parall e^{t \Delta} h(0) .\]
Note we have that
\[ f \parall \Delta \Duh(h) = \Delta (f \parall \Duh(h)) - (\Delta f) \parall \Duh(h) - (\nabla f) \parall \nabla \Duh(h).\]
Combining, we thus obtain
\[\begin{split}
\Duh(f \parall& P_{L, N} g) - f\parall \Duh(P_{L, N} g) = \Delta \big(\Duh(f \parall h) - f \parall \Duh(h) \big) ~- \\
&(\Delta f) \parall \Duh(h) - (\nabla f) \parall \nabla \Duh(h) - \Duh(\ptl_t f \parall h) - \big( e^{t\Delta} (f(0) \parall h(0)) - f(t) \parall e^{t \Delta} h(0)\big) .
\end{split}\]
By Lemma \ref{prelim:lem-general-integral-commutator} and the estimate \eqref{eq:h-estimate-L-geq-N2}, we have that
\[\begin{split}
\|\Delta (\Duh(f \parall h) - f \parall \Duh(h))\|_{\Wc^{\alpha', \nu'}} &\lesssim \|\Duh(f \parall h) - f \parall \Duh(h)\|_{\Wc^{\alpha' + 2, \nu'}} \\
&\lesssim L^{-\eta/2} \big(\|f\|_{\Wc^{\alpha, \nu}} + \|\ptl_t f\|_{\Wc^{\alpha - 2, \nu}}\big) \|g\|_{\Cs_{tx}^\beta} . 
\end{split}\]
The terms $(\Delta f) \parall \Duh(h)$, $(\nabla f) \parall \nabla \Duh(h)$, and $\Duh(\ptl_t f \parall h)$ may be bounded by paraproduct and Schauder estimates  combined with \eqref{eq:h-estimate-L-geq-N2}. For the final term, observe that
\[ e^{t\Delta} (f(0) \parall h(0)) - f(t) \parall (e^{t \Delta} h(0)) = e^{t \Delta} (f(0) \parall h(0)) - f(0) \parall e^{t \Delta} h(0) + \big((f(0) - f(t)\big) \parall e^{t \Delta} h(0).\]
By Lemma \ref{prelim:lem-heat-commutator} and \eqref{eq:h-estimate-L-geq-N2}, we have that
\[ \|e^{t \Delta} (f(0) \parall h(0)) - f(0) \parall e^{t \Delta} h(0)\|_{\Wc^{\alpha', \nu'}} \lesssim \|f\|_{\Wc^{\alpha -2\nu, 0}} \|h\|_{L_t^\infty \Cs_x^{\beta_x + 2\beta_t + 2 -\eta}} \lesssim L^{-\eta/2} \|f\|_{\Wc^{\alpha -2\nu, 0}} \|g\|_{\Cs_{tx}^\beta}.  \]
Finally, we have that (using that $\alpha < 1$ so that $\alpha - 2 < 0$)
\[\begin{split}
t^{\nu'} \|\big(f(0) - f(t)\big) \parall e^{t \Delta} h(0)\|_{\Cs_x^{\alpha'}} &\lesssim t^{\nu'} \|f(0) - f(t)\|_{\Cs_x^{\alpha - 2}} \|e^{t\Delta} h(0)\|_{\Cs_x^{\alpha' + 2-\alpha}}  \\
&\lesssim t^{\nu'} \bigg(\int_0^t \|\ptl_t f(s)\|_{\Cs_x^{\alpha -2}} ds\bigg) t^{-\frac{1}{2}(\alpha' + 2 - \alpha - (\beta_x + 2\beta_t + 2 - \eta))} \|h(0)\|_{\Cs_x^{\beta_x + 2\beta_t + 2 - \eta}} \\
&\lesssim t^{1 + \nu - \nu' - \frac{1}{2}(\alpha' - (\alpha + \beta_x + 2\beta_t - \eta))} L^{-\eta/2} \|\ptl_t f\|_{\Wc^{\alpha - 2, \nu}} \|g\|_{\Cs_{tx}^{\beta}}
\end{split}. \]
By our assumption on $\eta$, the exponent of $t$ above is positive.
Combining all the estimates, we obtain the desired inequality in the case $L \geq N^2$.
\end{proof}

\begin{proof}[Proof of Lemma \ref{lemma:schauder-space-time-besov-space-singular-time-weights}]
Let $t > 0$. For $s \in (0, t)$, we have that
\[ \Duh(f)(t) = \Hts (\Duh(f)(s)) + \Duh(f(\cdot + s))(t-s). \]
Using that $f \in \Cs_{tx}^\alpha$, by a Schauder estimate (Lemma \ref{lemma:space-time-Besov-space-Schauder}) we have that (using that $\theta' \leq \theta$)
\[ \big\|\Hts \big(\Duh(f)(s)\big)\big\|_{\Cs_x^{\gamma + 2\theta'}} \lesssim (t-s)^{-\theta'} \|\Duh(f)(s)\|_{\Cs_x^\gamma} \lesssim (t-s)^{-\theta} \|f\|_{\Cs_{tx}^\alpha}.  \]
Using that $f \in \Cs_{tx}^{\beta, \theta}$, we have that $\|f(\cdot + s)\|_{\Cs_{tx}^\beta((0, \infty))} = \|f\|_{\Cs_{tx}^\beta((s, \infty))} \lesssim s^{-\theta} \|f\|_{\Cs_{tx}^{\beta, \theta}}$, and thus by a Schauder estimate (Lemma \ref{lemma:space-time-Besov-space-Schauder}), we have that
\[ \|\Duh(f(\cdot + s))\|_{C_t^0 \Cs_x^{\gamma + 2\theta'}} \lesssim \|f(\cdot + s)\|_{\Cs_{tx}^\beta((0, \infty))} \lesssim s^{-\theta} \|f\|_{\Cs_{tx}^{\beta, \theta}}.\]
Taking $s = t/2$ and combining the previous estimates, we obtain
\[ \|\Duh(f)(t)\|_{\Cs_x^{\gamma + 2\theta}} \lesssim t^{-\theta} \|f\|_{\Cs_{tx}^\alpha \cap \Cs_{tx}^{\beta, \theta}}. \]
The desired result now follows.
\end{proof}

Next, before we prove Lemma \ref{lemma:f-Wc-ptl-t-f-Wc}, we need the following preliminary estimate.

\begin{lemma}\label{lemma:singular-time-weight-rho-estimate}
Let $\theta \in [0, 1)$. For $t \in \R$, we have that
\[\begin{split}
\int ds |\widecheck{\rho}_L(t-s)| \min(|s|, 1)^{-\theta} &\lesssim_\theta \min\big(L^\theta, \min(|t|, 1)^{-\theta}\big), \\
\int ds |\eta_L(t-s)| \min(|s|, 1)^{-\theta} &\lesssim_\theta L^{-1} \min\big(L^\theta, \min(|t|, 1)^{-\theta}\big).
\end{split}\]
\end{lemma}
\begin{proof}
We only prove the first estimate, as the proof of the second estimate is very similar. Assume without loss of generality that $t \geq 0$. We may split
\[ \int ds |\widecheck{\rho}_L(t-s)| \min(|s|, 1)^{-\theta} \lesssim \int_{-1}^1 ds |\widecheck{\rho}_L(t-s)| |s|^{-\theta}  + \int_{|s| > 1} ds |\widecheck{\rho}_L(t-s)| . \]
The second term on the right hand side is bounded by $\|\widecheck{\rho}_L\|_{L^1} \lesssim 1$, which is acceptable. Thus we focus on the first term, which for brevity we denote by $I$. Suppose first that $t \leq L^{-1}$. In this case, we may bound (using that $\|\widecheck{\rho}_L\|_\infty \lesssim L$ in the second inequality)
\[\begin{split}
I &\leq \int_{|s| \leq L^{-1}} |\widecheck{\rho}_L(t-s)| |s|^{-\theta} ds + L^\theta \int_{L^{-1} < |s| \leq 1} |\widecheck{\rho}_L(t-s)| ds  \lesssim L \int_{|s|\leq L^{-1}} |s|^{-\theta} ds + L^\theta \lesssim L^\theta.
\end{split}\]
Next, suppose that $t > L^{-1}$. Suppose also that $t \leq 1$. In this case, we split
\[ I \leq \int_{|s| < t/2} ds |\widecheck{\rho}_L(t-s)| |s|^{-\theta} + t^{-\theta} \int_{|s| > t/2} ds |\widecheck{\rho}_L(t-s)|, \]
Now for $|s| < t/2$, we may bound (since $t > L^{-1}$) $|\widecheck{\rho}_L(t-s)| \leq \sup_{u \geq t/2} |\widecheck{\rho}_L(u)| \lesssim t^{-1}$.
We thus obtain
\[ I \lesssim t^{-1} \int_{|s| < t/2} ds |s|^{-\theta} + t^{-\theta} \lesssim t^{-\theta}. \]
The case $t > 1$ may be handled similarly, by splitting into $|s| < 1/2$ and $|s| > 1/2$.
\end{proof}

\begin{proof}[Proof of Lemma \ref{lemma:f-Wc-ptl-t-f-Wc}]
Without loss of generality, let $t > 0$. We have that
\[ \|P_{L, N} f(t)\|_{L_x^\infty} \leq \int |\widecheck{\rho}_L(t-s)| \cdot \|P_N^x f(s)\|_{L_x^\infty} ds \lesssim N^{-\alpha} \|f\|_{\Wc^{\alpha, \theta}} \int |\widecheck{\rho}_L(t-s)| \min(|s|, 1)^{-\theta} ds.  \]
By Lemma \ref{lemma:singular-time-weight-rho-estimate}, we further obtain
\[ \|P_{L, N} f(t)\|_{L_x^\infty} \lesssim \min\big(L^\theta, \min(t, 1)^{-\theta}\big) N^{-\alpha} \|f\|_{\Wc^{\alpha, \theta}}.\]
To show the second estimate, it remains to show that
\[ \|P_{L, N} f(t)\|_{L_x^\infty} \lesssim L^{-1} \min\big(L^\theta, \min(t, 1)^{-\theta}\big) N^{-(\alpha-2)} \|\ptl_t f\|_{\Wc^{\alpha-2, \theta}} .\]
Using that $\eta_L' = \widecheck{\rho}_L$, we may write
\[ P_{L, N} f(t) = \int \widecheck{\rho}_L(t-s) P_N^x f(s) ds = -\int \eta_L(t-s) P_N^x \ptl_s f(s) ds. \]
We may then bound
\[ \|P_{L, N} f(t)\|_{L_x^\infty} \leq \int |\eta_L(t-s)| \|P_N^x \ptl_s f(s)\|_{L_x^\infty} ds \lesssim N^{-(\alpha-2)} \|\ptl_t f\|_{\Wc^{\alpha-2, \theta}} \int |\eta_L(t-s)| \min(|s|, 1)^{-\theta} ds . \]
Applying Lemma \ref{lemma:singular-time-weight-rho-estimate}, the desired estimate follows.
\end{proof}

\section{Proofs of Lemma \ref{lemma:P-N-lo-hi-commutator}, Lemma \ref{lemma:P-N-commutator-space}, and Lemma \ref{prelim:lem-Q}}\label{section:proofs-commutator}

Let $y \in \T^2$. Let $\Theta_y$ be the shift map on functions or distributions $f : \T^2 \ra \R$ which is defined by $(\Theta_y f)(x) := f(x - y)$, $x \in \T^2$.

\begin{proof}[Proof of Lemma \ref{lemma:P-N-lo-hi-commutator}]
We have that
\[ P_N(f \parall g) - f \parall P_N g = \sum_{M \sim N} P_N (P_{\ll M} f P_M g) - P_{\ll M} f P_N P_M g.  \]
Since the term $P_N (P_{\ll M} f P_M g) - P_{\ll M} f P_N P_M g$ is localized to frequency $M$, it suffices to show that for $M \sim N$, we have that
\[ \|P_N (P_{\ll M} f P_M g) - P_{\ll M} f P_N P_M g\|_{L_x^\infty} \lesssim N^{-\delta} M^{-(\alpha + \beta - \delta)} \|f\|_{\Cs_x^\alpha} \|g\|_{\Cs_x^\beta}.\]
Towards this end, we may write
\[ P_N (P_{\ll M} f P_M g) - P_{\ll M} f P_N P_M g = \int \widecheck{\rho}_N(y) \big((\Theta_y P_{\ll M} f) (\Theta_y P_M g) - P_{\ll M} f (\Theta_y P_N P_M g)\big) dy. \]
We may thus bound
\[\begin{split}
\|P_N (P_{\ll M} f P_M g) - P_{\ll M} f P_N P_M g\|_{L_x^\infty} &\leq \int |\widecheck{\rho}_N(y)| \cdot \|\Theta_y P_{\ll M} f - P_{\ll M} f\|_{L_x^\infty} \cdot \|\Theta_y P_N P_M g\|_{L_x^\infty} dy  \\
&\leq M^{-\beta} \|g\|_{\Cs_x^\beta} \int |\widecheck{\rho}_N(y)| \cdot \|\Theta_y P_{\ll M} f - P_{\ll M} f\|_{L_x^\infty} dy \\
&\lesssim M^{-\beta} \|P_{\ll M} f\|_{\Cs_x^1} \|g\|_{\Cs_x^\beta} \int |\widecheck{\rho}_N(y)| \cdot |y| dy \\
&\lesssim N^{-1} M^{(1-\alpha)} M^{-\beta} \|P_{\ll M} f\|_{\Cs_x^\alpha} \|g\|_{\Cs_x^{\beta}} \\
&\lesssim N^{-\delta} M^{-(\alpha + \beta - \delta)} \|f\|_{\Cs_x^\alpha} \|g\|_{\Cs_x^\beta}, 
\end{split}\]
where we used that $\alpha < 1$ in the second to last inequality. The desired result now follows.
\end{proof}

Before we prove Lemma \ref{lemma:P-N-commutator-space}, we first prove the following preliminary result.

\begin{lemma}\label{lemma:P-N-smooth-commutator}
Let $N\in \dyadic$ and let $\alpha \in (0,1)$. For all smooth $f,g\colon \T^2 \rightarrow \C$, it then follows that 
\[ \| P_{\leq N}(fg) - f P_{\leq N} g - \ptl_j f Q^j_{> N} g\|_{L_x^\infty} \lesssim N^{-(1+\alpha)} \|f\|_{\Cs_x^{1+\alpha}} \|g\|_{L_x^\infty}. \]
\end{lemma}
\begin{proof}
We write
\[ P_{\leq N}(fg) - f P_{\leq N} g - \ptl_j f Q^j_{> N} g = \int \widecheck{\rho}_{\leq N}(y) (\Theta_y f - f - \nabla f \cdot y) \Theta_y g dy. \]
We bound 
\[\begin{split}
\|P_{\leq N}(fg) - f P_{\leq N} g - \ptl_j f Q^j_{> N} g \|_{L_x^\infty} &\lesssim \int |\widecheck{\rho}_{\leq N}(y)| \cdot \|(\Theta_y f - f - \nabla f \cdot y) \Theta_y g\|_{L_x^\infty} dy \\
&\lesssim \int |\widecheck{\rho}_{\leq N}(y)| \cdot \|\Theta_y f - f - \nabla f \cdot y\|_{L_x^\infty} \|\Theta_y g\|_{L_x^\infty} dy \\
&\lesssim \|g\|_{L_x^\infty} \int |\widecheck{\rho}_{\leq N}(y)| \cdot \|\Theta_y f - f - \nabla f \cdot y \|_{L_x^\infty} dy.
\end{split}\]
For fixed $y$, we may write (by the mean value theorem)
\[ (\Theta_y f - f)(x) = (\nabla f)(x + O(|y|)) \cdot y, \]
and so 
\[\begin{split}
\|\Theta_y f - f - \nabla f \cdot y \|_{L_x^\infty} &= \|\big((\nabla f)(\cdot + O(|y|)) - \nabla f\big) \cdot y \|_{L_x^\infty} \leq |y|\cdot \|(\nabla f)(\cdot + O(|y|)) - \nabla f\|_{L_x^\infty} \\
&\lesssim |y|^{1+\alpha} \|\nabla f\|_{\Cs_x^\alpha} \lesssim |y|^{1+\alpha} \|f\|_{\Cs_x^{1+\alpha}}. 
\end{split} \]
Finally, note that (since $\widecheck{\rho}_{\leq N}(y) = N^2 \widecheck{\rho}(N y)$)
\[ \int |\widecheck{\rho}_{\leq N}(y)| |y|^{1+\alpha} dy \lesssim N^{-(1+\alpha)}.\]
The desired result now follows upon combining the previous few estimates.
\end{proof}

\begin{proof}[Proof of Lemma \ref{lemma:P-N-commutator-space}]
For $M \in \dyadic$, let 
\[ F_M   :=  P_{\leq N} (P_{\ll M} f P_M g) - P_{\ll M} f (P_{\leq N} P_M g) - (\ptl_j P_{\ll M} f) Q^j_{> N} (P_{\leq N} P_M g).\]
Observe that $F_M = 0$ for $M \gg N$ and $F_M$ is localized to frequency $M$. Thus, it suffices to show that for all $M \lesssim N$,
\[ \|F_M \|_{L_x^\infty} \lesssim N^{-\delta} M^{-(\alpha+ \beta - \delta)} \|f\|_{\Cs_x^{\alpha}} \|g\|_{\Cs_x^\beta}. \]
Towards this end, by Lemma \ref{lemma:P-N-smooth-commutator}, we have that (since $\alpha, \delta < 2$, we have that $\max(\alpha, \delta, 3/2) < 2$)
\[\begin{split}
\|F_M\|_{L_x^\infty} &\lesssim N^{-\max(\alpha,\delta, 3/2)} \|P_{\ll M} f\|_{\Cs_x^{\max(\alpha,3/2)}} \|P_M g\|_{L_x^\infty} \\
&\lesssim N^{-\max(\alpha, \delta, 3/2)} M^{-\beta} \|P_{\ll M} f\|_{\Cs_x^{\max(\alpha, \delta, 3/2)}} \|g\|_{\Cs_x^\beta} \\
&\lesssim N^{-\max(\alpha, \delta, 3/2)} M^{-\beta} N^{(\max(\alpha, \delta, 3/2) - \alpha)} \|P_{\ll M} f\|_{\Cs_x^\alpha} \|g\|_{\Cs_x^\beta} \\
&\lesssim N^{-\delta} M^{-(\alpha + \beta - \delta)} \|f\|_{\Cs_x^\alpha} \|g\|_{\Cs_x^\beta}. \qedhere
\end{split}\]
\end{proof}

\begin{proof}[Proof of Lemma \ref{prelim:lem-Q}:] 
Due to the interpolation estimate
\begin{equation*}
\big\| Q^\ell_{>N} f \big\|_{\Cs_x^\beta} \lesssim 
\big\| Q^\ell_{>N} f \big\|_{\Cs_x^\alpha}^{1-(\beta-\alpha)}
\big\| Q^\ell_{>N} f \big\|_{\Cs_x^{\alpha+1}}^{\beta-\alpha},
\end{equation*}
it suffices to prove \eqref{prelim:eq-Q-estimate} for $\beta=\alpha$ and $\beta=\alpha+1$. That is, it suffices to prove that 
\begin{equation}\label{appendix:eq-Q-p1}
\big\| Q^\ell_{>N} f \big\|_{\Cs_x^\alpha} \lesssim N^{-1} \big\| f \big\|_{\Cs_x^\alpha}
\qquad \text{and} \qquad 
\big\| Q^\ell_{>N} f \big\|_{\Cs_x^{\alpha+1}} \lesssim  \big\| f \big\|_{\Cs_x^\alpha}.
\end{equation}
Using that 
\begin{equation*}
\big\| Q^\ell_{>N} f \big\|_{\Cs_x^{\alpha+1}} \sim \big\| Q^\ell_{>N} f \big\|_{\Cs_x^{\alpha}} + \big\| \nabla Q^\ell_{>N} f \big\|_{\Cs_x^{\alpha}},
\end{equation*}
the estimates in \eqref{appendix:eq-Q-p1} can be reduced to the estimates 
\begin{equation}\label{appendix:eq-Q-p2}
\big\| Q^\ell_{>N} f \big\|_{\Cs_x^\alpha} \lesssim N^{-1} \big\| f \big\|_{\Cs_x^\alpha}
\qquad \text{and} \qquad 
\big\| \partial^k Q^\ell_{>N} f \big\|_{\Cs_x^{\alpha}} \lesssim  \big\| f \big\|_{\Cs_x^\alpha},
\end{equation}
where $k\in [2]$. Using the definition of the  $\Cs_x^\alpha$-spaces (Definition \ref{prelim:def-Hoelder}) and that $Q^\ell_{>N}$ is a Fourier multiplier, which necessarily commutes with the Littlewood-Paley operators, we have that 
\begin{align*}
&\,\big\| Q^\ell_{>N} f \big\|_{\Cs_x^\alpha}
= \sup_{M\in \dyadic} M^\alpha \big\| P_M Q^\ell_{>N} f \big\|_{L_x^\infty}
= \sup_{M\in \dyadic} M^\alpha \big\|  Q^\ell_{>N} P_M f \big\|_{L_x^\infty} \\
\leq&\,  \big\| Q^\ell_{>N} \big\|_{L_x^\infty \rightarrow L_x^\infty}
\sup_{M\in \dyadic} M^\alpha \big\| P_M f \big\|_{L_x^\infty}
=  \big\| Q^\ell_{>N} \big\|_{L_x^\infty \rightarrow L_x^\infty} \big\| f \big\|_{\Cs_x^\alpha},
\end{align*}
where $ \| Q^\ell_{>N} \|_{L_x^\infty \rightarrow L_x^\infty} $ denotes the operator-norm of $Q^\ell_{>N}$ on $L^\infty_x$. Similarly, by using that $\partial^k Q^\ell_{>N}$ is a Fourier multiplier, we also obtain that 
\begin{align*}
\big\| \partial^k Q^\ell_{>N} f \big\|_{\Cs_x^\alpha}
\leq \big\| \partial^k Q^\ell_{>N} \big\|_{L_x^\infty \rightarrow L_x^\infty} \big\| f \big\|_{\Cs_x^\alpha}.
\end{align*}
In order to prove \eqref{appendix:eq-Q-p2}, it therefore suffices to estimate the operator-norms of $Q^\ell_{>N}$ and $\partial^k Q^\ell_{>N}$ on $L_x^\infty$, i.e., it suffices to prove the estimates 
\begin{equation}\label{appendix:eq-Q-bound}
N \big\| Q^\ell_{>N} f \big\|_{L_x^\infty} \lesssim  \big\| f \big\|_{L_x^\infty} 
\qquad \text{and} \qquad 
\big\| \partial^k Q^\ell_{>N} f \big\|_{L_x^\infty} \lesssim \big\| f \big\|_{L_x^\infty}.
\end{equation}
Since $\rho_{>N}=1-\rho_{\leq N}$, it follows that the symbol of $ Q^\ell_{>N}$ is given by  
\begin{equation}\label{appendix:eq-Q-symbol}
- \frac{1}{\icomplex } \frac{1}{N} (\partial^\ell \rho)\Big( \frac{\xi}{N}\Big). 
\end{equation}
In particular, the symbol of $Q^{\ell}_{>N}$ is supported on frequencies $\sim N$ and it, therefore, suffices to prove the first estimate in \eqref{appendix:eq-Q-bound}. As a result of \eqref{appendix:eq-Q-symbol}, it holds that 
\begin{equation*}
N Q^\ell_{>N} f = - \frac{1}{\icomplex} \int_{\R^2} \dy \, N^2 \widecheck{(\partial^\ell \rho)}(Ny) f(x-y). 
\end{equation*}
The desired estimates \eqref{appendix:eq-Q-bound} then follow directly from Young's convolution inequality. 
\end{proof}

\end{appendix}

\bibliography{SYM_Library}
\bibliographystyle{myalpha}

\end{document}